\documentclass{amsbook}
\usepackage{amsmath}

\usepackage{graphics,graphicx,psfrag}
\newcommand{\vip}{\vskip0.15cm}
\newcommand{\ala}{\nonumber \\}
\newcommand{\indiq}{{\rm 1 \hskip-4pt 1}}
\newcommand{\lb}{[\![}
\newcommand{\rb}{]\!]}
\newcommand{\E}{{\mathbb{E}}}
\newcommand{\nn}{\mathbb{N}}
\newcommand{\zz}{\mathbb{Z}}
\newcommand{\rr}{{\mathbb{R}}}
\newcommand{\dd}{{\mathbb{D}}}
\newcommand{\supp}{{\mathrm{supp}\;}}

\newcommand{\zuz}{{\{0,1\}^\zz}}
\newcommand{\ba}{{\mathbf{a}}}
\newcommand{\bn}{{\mathbf{n}}}
\newcommand{\bm}{{\mathbf{m}}}
\newcommand{\bk}{{\mathbf{k}}}
\newcommand{\bd}{{\mathbf{d}}}
\newcommand{\bdelta}{{\boldsymbol{\delta}}}

\newcommand{\cI}{{\mathcal I}}
\newcommand{\cS}{{\mathcal S}}
\newcommand{\cA}{{\mathcal A}}
\newcommand{\cB}{{\mathcal B}}
\newcommand{\cC}{{\mathcal C}}

\newcommand{\cT}{{\mathcal T}}
\newcommand{\cX}{{\mathcal X}}
\newcommand{\cR}{{\mathcal R}}
\newcommand{\cU}{{\mathcal U}}

\newcommand{\ttau}{{\tilde \tau}}
\newcommand{\tS}{{\tilde S}}
\newcommand{\tV}{{\tilde V}}
\newcommand{\tY}{{\tilde Y}}
\newcommand{\tL}{{\tilde L}}

\newcommand{\tT}{{\tilde T}}
\newcommand{\tZ}{{\tilde Z}}

\newcommand{\tH}{{\tilde H}}
\newcommand{\tOmega}{{\tilde \Omega}}
\newcommand{\bOmega}{{\bar \Omega}}

\newcommand{\la}{{\lambda}}
\newcommand{\intot}{{\int_0^t}}
\newcommand{\sm}{{{s-}}}
\newcommand{\tm}{{{t-}}}
\newcommand{\e}{{\varepsilon}}

\newtheorem{theo}{\indent Theorem}[section]
\newtheorem{prop}[theo]{\indent Proposition}

\newtheorem{rem}[theo]{\indent Remark}
\newtheorem{lem}[theo]{\indent Lemma}
\newtheorem{defin}[theo]{\indent Definition}
\newtheorem{cor}[theo]{\indent Corollary}
\newtheorem{nota}[theo]{\indent Notation}

\newtheorem{algo}[theo]{\indent Algorithm}

\newenvironment{preuve}{\vip \noindent {\it Proof}}{\hfill$\square$ \vip}

\begin{document}

\title[General forest fire processes]
{One-dimensional general forest fire processes}
\author{Xavier Bressaud}
\author{Nicolas Fournier}
%\thanks{Acknowledgments: The second author was supported during this work by
%the grant from the Agence Nationale de la Recherche with reference 
%ANR-08-BLAN-0220-01.}
\address{Xavier Bressaud: Universit\'e P. Sabatier, Institut de 
Maths de Toulouse, F-31062 Toulouse Cedex, France}
\email{bressaud@math.univ-toulouse.fr}
\address{Nicolas Fournier:
Laboratoire d'Analyse et de Math\'ematiques Appliqu\'ees, CNRS UMR 8050,
Universit\'e Paris-Est, 
61 avenue du G\'en\'eral de Gaulle, 94010 Cr\'eteil Cedex, France}
\email{nicolas.fournier@univ-paris12.fr}

\begin{abstract}
We consider the one-dimensional generalized forest fire process:
at each site of $\zz$, seeds and matches fall according some i.i.d. 
stationary 
renewal processes. When a seed falls on an empty site, a tree grows immediately.
When a match falls on an occupied site, a fire starts and destroys
immediately the corresponding connected component of occupied sites.
Under some quite reasonable assumptions on the renewal processes,
we show that when matches become less and less frequent, the process converges,
with a correct normalization, to a limit forest fire model.
According to the nature of the renewal processes governing seeds, 
there are four possible limit forest fire models. The four limit
processes can be perfectly simulated.
This study generalizes consequently previous results of \cite{bf}
where seeds and matches were assumed to fall according to Poisson
processes.
\end{abstract}



\maketitle

\vip \vip \vip

Acknowledgments: The second author was supported during this work by
the grant from the Agence Nationale de la Recherche with reference 
ANR-08-BLAN-0220-01.

\setcounter{tocdepth}{1}
\tableofcontents

\part{Introduction}

\section{Introduction}
\setcounter{equation}{0}

Consider a graph $G=(S,A)$, $S$ being the set of vertices and $A$
the set of edges. Introduce the space of configurations $E=\{0,1\}^S$.
For $\eta\in E$, we say that $\eta(i)=0$ if the site $i\in S$ is vacant
and $\eta(i)=1$ if $i$ is occupied by a tree. 
Two sites are neighbors if there is an edge between them. We call forests the 
connected components of occupied sites. 
For $i \in S$ and $\eta \in E$, we denote by $C(\eta,i)$ the forest
around $i$ in the configuration $\eta$ 
(with $C(\eta,i)=\emptyset$ if $\eta(i)=0$).
We consider the following (vague) rules: 

\vip

$\bullet$
vacant sites become occupied (a seed falls and a tree immediately grows) 
at rate $1$;

\vip

$\bullet$ occupied sites take fire (a match falls) at rate $\la>0$;

\vip

$\bullet$ fires propagate to neighbors (inside the forest) at rate $\pi>0$. 

\vip

Such a model was introduced by Henley \cite{h} and 
Drossel and Schwabl \cite{ds} as a toy model for forest fire
propagation and as an example of a simple model intended to clarify the 
concept of {\it self-organized criticality}.  

\vip

The order of magnitude of the 
rate of growth is much smaller than the propagation rate, $\pi \gg 1$. We 
will focus here on the limit case where the propagation is instantaneous: when 
a tree takes fire, the whole forest (to which it belongs) is destroyed 
immediately. The model is thus:

\vip

$\bullet$ vacant sites become occupied (a seed falls and a tree immediately 
grows) at rate $1$;

\vip

$\bullet$ matches fall on occupied sites at rate $\la$ and then burn
instantaneously the corresponding forest.

\vip

The features of the model depend on the geometry of the graph; we only 
consider in this paper 
the case $S = \zz$ (with its natural set of edges). 
They also depend  on the laws of the processes 
governing seeds and matches; the standard case is when these are Poisson 
processes so that the forest fire 
process is Markov. We deal here with the most general 
(stationary) case; Poisson processes are replaced by stationary renewal 
processes. 

\vip

Our main preoccupation is the behavior of this model in the 
asymptotic of rare seeds, namely when $\la \to 0$. 
We present four possible limit processes 
(depending on the tail properties of the 
law of the stationary processes governing seeds) arising 
when we suitably rescale space and 
accelerate time while letting  $\la \to 0$. This is a considerable
generalization of the results obtained in \cite{bf}. 

\vip

This introduction consists of five subsections. \vip

(i) In Subsection \ref{intro1}, we briefly recall
the concept of {\it self-organized criticality} and recall
a certain number of models supposed to enjoy self-organized critical
properties.

(ii) We present in Subsection \ref{intro2} a quick history of the forest-fire 
process, its other possible interpretations and its links
with other models.

(iii) Subsection \ref{intro3} explains the importance of the geometry of the
underlying graph $G$ and the links of the forest-fire model 
with {\it percolation}.

(iv) In Subsection \ref{intro4}, we recall what has been done
for the (Markov) forest-fire process on $\zz$ from a rigorous mathematical
point of view.

(v) Subsection \ref{intro5} is devoted to a brief exposition of the main ideas 
of the present paper.

(vi) Finally, we give the plan of the paper in Subsection \ref{intro6}.

\subsection{Self-organized criticality}\label{intro1}
One of the successes of statistical mechanics is to explain how 
local interactions generate macroscopic effects through simple models 
on lattices. Among the most striking phenomena are those observed around 
so-called {\it critical values} of the parameters of such models, such 
as scale-free patterns, power laws, conformal invariance, critical 
exponents or universality. 

\vip
\subsubsection{Paradigm}
The study of self-organized critical systems has
become rather popular in physics since the end of the 80's.
These are simple models supposed to clarify
temporal and spatial randomness observed in a variety of natural 
phenomena showing {\it long range correlations}, like  sand piles, avalanches,
earthquakes, stock market crashes, forest fires, shapes of mountains,
clouds, etc. It is remarkable that such phenomena reminiscent of critical 
behavior 
arise so frequently in nature where nobody is here to finely tune the 
parameters to critical values.  

\vip

An idea proposed in 1987 by Bak-Tang-Wiesenfeld \cite{btw1} to tackle this 
contradiction is, roughly, that of
systems {\it growing} toward a {\it critical state}  
and relaxing through {\it catastrophic}  
events: avalanches, crashes, fires, etc. If the catastrophic 
events become more and more probable when approaching the critical state, 
the system spontaneously reaches an equilibrium {\it close} to the critical 
state.   This idea was developed in  \cite{btw1} through the study of 
the {\it archetypical} sand pile model. 

\vip 

This paradigm was used to investigate various phenomena, from physics 
to sociology through biology, epidemiology or economics. The pertinence 
of the conclusions are not always convincing. Discussion to decide if 
whether or not there is self-organized criticality 
in nature or in one or another model, or 
even to decide what self-organized criticality should exactly be, 
is beyond our purpose. Anyhow 
let us summarize the usual standard features of these models:  

\vip

$\bullet$  local dynamics but with possibly very long range effects (at high 
speed) through a simple mechanism;
\vip

$\bullet$ 
macroscopic states with scaling invariance properties, {\it a priori}
related to the critical state of a well-known system;\vip

$\bullet$ 
long range spatial correlations and power laws for natural 
observables at fixed times;\vip

$\bullet$ 
presence of $1/f^\beta$-noise with $0<\beta<1$ in the temporal 
fluctuations of natural observables, i.e. $S(f)$ of order $1/f^\beta$ 
where  $S$ is the spectral density.   

\vip

One of the specificities of these models is that the interaction is formally 
non local; it is local in general, but may, when close to the 
{\it critical region} ---whatever this means--- have long range effects. 
This, together with a lack of monotonicity, yields mathematical difficulties 
that justify a careful treatment. 

\vip

To understand, explain or illustrate these phenomena, a multitude of  models 
have been proposed to explore various mechanisms that would produce these 
effects. Simple models, non necessarily realistic, are nice for they try 
to catch the underlying mechanisms. They have often been treated numerically, 
in the spirit of Bak-Tang-Wiesenfield 
\cite{btw1}.   Forest fire 
models are among them and still need a mathematical rigorous study. 
Sand pile models, while somehow more complicated, have been more studied. 

\vip
\subsubsection{Sand pile models} Let us explain in a few words what a 
sand pile model is. First, we assume that we have a definition
of what a {\it stable} sand pile is. Sand grains fall at random
on sites. When a grain falls, if the new pile is
{\it unstable},  it is immediately
re-organized to become stable, through (possibly many) successive
elementary steps. Such events are called {\it avalanches}.  
This model was introduced by 
Bak-Tang-Wiesenfeld \cite{btw1} and studied by Dhar \cite{d2}. 
Since, there has been a huge amount of results and we 
will not try to be exhaustive; 
for surveys see for instance Holroyd-Levine-Meszaros-Peres-Propp-Wilson
\cite{hlmppw}, Goles-Latapy-Magnien-Morvan-Phan \cite{glmmp}
or Redig \cite{r1}.  

\vip

Let us give a slightly 
more precise description of the so-called  {\it Abelian} sand pile model.  
The state of the system 
is described by   $\eta \in \zz^S$, representing local {\it slopes}
of the sand pile. For instance, when $S=\zz$, think that 
$\eta(i)=h(i+1) - h(i)$ where $h(i)$ is the height of the sand pile on the
site $i$. A dynamic is defined on $\zz^S$ using a matrix $\Delta$  
indexed by $S \times S$, called {\it toppling matrix}. It has  
positive entries on the diagonal (think of $\Delta_{i,i} = \gamma$ constant), 
negative entries when $i,j \in S$ are neighbors and null entries
elsewhere. It is {\it dissipative} if  $\Delta_{i,i} + \sum_{j \neq i} 
\Delta_{i,j} < 0$. Then define the toppling of a site $i$ as the 
mapping $T_i: \zz^S \to \zz^S$ defined by 
\begin{eqnarray*}
T_i(\eta)(j) &=& \eta(j) - \Delta_{i,j} \;\; \forall \; j \in S \;
\hbox{ if } \; \eta(i) > \Delta_{i,i};\\
T_i(\eta)&=&\eta \hbox{ otherwise}.
\end{eqnarray*} 
Toppling at $i$ consists, whenever the slope is {\it too big} at $i$, of
spreading grains on neighboring sites 
(possibly in a non conservative way).  A pile is stable if for all 
$i \in S$, $ \eta(i) \leq  \Delta_{i,i}$ (then, no toppling has any effect).
Observe that successive topplings at different sites commute (which explains
the term {\it Abelian}). 

\vip

Now consider the situation where sand grains fall at random, on each site,
at rate $1$.  Each time a grain falls, immediately
topple (possibly many times) until stability is reached.
Some dissipativity assumptions guarantee that this is always possible.

\vip

At first glance, arrival of a new sand grain 
on a site has only a local effect: a non trivial toppling at $i$ may occur. 
But there can be a chain reaction creating an avalanche.   And indeed, the 
action may, in general, have a long range effect.  

\vip

These systems have a nice 
underlying group structure that depends on the size and geometry of the 
underlying lattice, see e.g. 
Le Borgne-Rossin \cite{lbr} for such an algebraic point of view. 
The thermodynamic limits of the sand-pile models have been investigated.
In particular, existence and uniqueness of a stationary measure have 
been proved.
See for instance Maes-Redig-Saada \cite{mrs} when $S=\zz$ and 
J\'arai \cite{j2} when $S=\zz^d$.  Some features of self-organized criticality 
have been observed for $d>1$, at least numerically, in the physical 
literature, see e.g.  L\"ubeck-Usadel \cite{lu}. For instance, they have 
studied the {\it sizes of avalanches} 
(number of topplings necessary to stabilize after a 
grain has been added).  A scaling limit
was obtained recently by D\"urre \cite{du}. 

\vip
\subsubsection{Other models.} 
The Abelian sand pile seems to be the most popular sand pile model. However 
it has a lot of variants: Zhang sand pile model (see Zhang \cite{z},
Pietronero-Tartaglia-Zhang \cite{ptz}), Oslo model 
(see Christensen-Corral-Frette-Feder-Jossang \cite{ccffj}, 
Amaral-Lauristsen \cite{al}), Oslo rice pile model (see Brylawski \cite{bry}),
chip firing game (see Tardos \cite{tardos}), etc.

\vip

Moreover, various different models have been introduced and studied 
with the eyes of self-organized criticality. 
There is of course the forest fire model that we are going to discuss in this
paper.
Let us mention briefly some other models:  rotor-router 
model (introduced by Priezzhev-Dhar-Dhar-Krishnamurthy \cite{pddk} 
under the name {\it Eulerian walkers model}), loop-erased 
random walks (Majumdar \cite{maj}), diffusion/aggregation models 
(Cafiero-Pietronero-Vespignani \cite{cpv}), Scheidegger's model of river 
basin (Scheidegger \cite{scheid}), 
models describing earthquakes (Olami-Feder-Christensen \cite{ofc}) 
or crashes in stock markets (Staufer-Sornette \cite{ss,s}), etc.

\vip

As we already mentioned those systems have often been subjected to numerical 
experimentations and studies. 
Of course this is a difficult task and it has sometimes been misleading:
long range effects need huge simulations, the
interpretation of which is not always meaningful. 

\vip

For surveys on self-organized criticality, 
see Bak-Tang-Wiesenfeld \cite{btw2}, Dhar \cite{d}, Jensen
\cite{j} and the references therein.

\subsection{Forest fire models}\label{intro2}
Here we consider the classical forest fire model on $G=(S,A)$.
Recall that on each site of $S$, seeds are falling at rate $1$ and matches
are falling at rate $\la$, according to some Poisson processes. 
A seed falling on a vacant site makes it
immediately occupied, and a match falling on an occupied site
makes instantaneously vacant the whole corresponding occupied connected
component. 
Thus the forest fire process is Markov
(at least if one is able to prove that it exists and is unique). 

\vip

\subsubsection{History and numerical studies}
The forest fire model was introduced independently by Henley \cite{h} and 
Drossel-Schwabl \cite{ds}. In the literature, it is generally 
referred to as the Drossel-Schwabl forest fire model. In their original 
paper, they consider the case where $S$ is a cube in $\zz^d$. 
They are interested in scaling laws and critical exponents for 
this model. Orders of magnitude of relevant quantities are derived by 
analytical computations using essentially mean field considerations. 
The results are {\it confirmed} by computer simulations. 
In Drossel-Clar-Schwabl \cite{dcs}, 
the asymptotic behavior of the density of vacant sites in 
the limit $\la \to 0$ is obtained when $S=\zz$ (using heuristic arguments, see
Subsubsection \ref{clusterssize} below). After this work,  numerous numerical or
semi-analytical studies have been produced. Among others, let us mention 
Henecker-Peschel \cite{hp} and Pruessner-Jensen \cite{pj}. 
Numerical studies were handled again by  Grassberger \cite{g},
who computes, when $S=\zz^2$, the density of occupied sites,
the fractal dimension of fires 
and the distribution of the fire sizes, in the limit $\la \to 0$. 
 
\vip

The first rigorous probabilistic treatment of this model
is the paper 
by van den Berg and J\'arai \cite{vdbj}. They give a 
rigorous description of the 
asymptotic density of vacant sites in the limit $\la \to 0$ for the forest 
fire process on $\zz$.  To our knowledge, all the rigorous results about 
the forest fire process
concern the case where seeds and matches fall according to Poisson processes. 
See
D\"urre \cite{du1,du2,du} (existence and uniqueness of the process on 
$\zz^d$ with $\la>0$ fixed), 
van den Berg-Brouwer \cite{vdbb} (behavior of the process near 
the critical time in dimension $2$, as $\la\to 0$) 
and Brouwer-Pennanen \cite{bp} (estimates on
the cluster size distribution in the asymptotic $\la\to 0$, in dimension $1$).
See also the papers by the authors \cite{bfold} (study of the invariant 
distribution when $\la=1$ in dimension $1$) and  
\cite{bf} (scaling limit of the one dimensional forest fire process 
in the asymptotic $\la\to 0$). 
We will discuss all these results more specifically in this introduction.

\vip
\subsubsection{Real forest fires}
Forest fires in real life are also a subject of preoccupation and of study 
from different point of views. In particular there are various statistical 
studies of sizes (and sometimes shapes) of {\it real} forest fires in 
different regions (see for instance Holmes-Hugget-Westerling \cite{hhw}). 
One of the recurrent 
observations is that the distributions of those fires have heavy tails 
(power laws) and pleasant scale invariance properties. Another one is 
the tentative description of the (fractal) geometry of fires 
(see for instance Mangiavillano \cite{man}). 
For references, connection with real 
life and practical interest of these studies, see Cui-Perera
\cite{cp}. 
A few studies relate the dynamics of real fires in a given region with 
theoretical models. One natural task was to compare real data and 
numerical experiments done with the toy models we have. On this aspect, 
let us mention the recent (and encouraging) works by Zinck-Grimm-Johst 
\cite{zg,zgj}. 
Other studies focus on the propagation of the fire itself, 
but this is not our main preoccupation here since we have assumed that the 
propagation is instantaneous.

\vip 

A direction of study suggested by works on real forest fires is to 
consider fires in {\it inhomogeneous}, for instance {\it random}, media.
To our knowledge, this aspect has not yet been investigated. Another 
one, that we address here, is to consider the non Markov case: 
seeds and matches may not (and actually should not) fall according
to Poisson processes.

\vip
\subsubsection{Other interpretations and variations}
The forest fire model has a very simple (and natural) dynamic. It may 
accept a variety of interpretations. And various modifications can 
make it fit the description of other phenomena. Indeed, we initially 
thought of it as a simplification of the avalanche process: snow flakes
fall on each site, a snow flake falling on a vacant site makes it 
occupied, and a snow flake falling on an occupied site makes vacant 
the whole connected component of occupied sites (such an event being
called {\it avalanche}). This is nothing but the 
forest fire process with $\la=1$, see \cite{bfold}.
More generally, the forest fire process 
may be used to model phenomena involving geometric relations and a 
common behavior on connected components; natural examples arise e.g. in 
epidemiology (change {\it fire} by {\it virus}).
From these points of view, some natural modifications could 
be explored such as making the growth process have effect only on sites  
which are neighbors of occupied sites (in the spirit of the so-called 
contact process). Such variants should be dominated by the standard 
contact process and by the forest fire process and may enjoy interesting 
features.  

\vip

In a different spirit, a {\it directed} version of the forest fire 
model has been studied as a toy model for neural networks. Roughly, the 
idea is to think of growth as {\it activation} and of fire as 
{\it signal emission}. 
The signal is transmitted along the (directed) connected component which is at 
the same time deactivated.  The difference is that the underlying graph is
a directed graph (usually a tree) and that the signal is (instantaneously) 
sent according to the directed edge (instead of all the connected component). 
Let us mention 
the work of van den Berg-Brouwer \cite{vdbb}, which 
include remarks about this model, and the work of van den Berg-T\'oth
\cite{vdbt}.

\vip
\subsubsection{Coagulation/Fragmentation}
\label{coagfrag}
A slight change of point of view about the forest fire model makes explicit a 
parallel with a class of coagulation/fragmentation processes. 
Assume e.g. that $S=\zz$. Say that each edge $(i,i+1)$ has mass $1$,
and that two neighbor edges $(i-1,i)$ and $(i,i+1)$ are connected (or
belong to the same particle) if $\eta(i)=1$.
Then each time a seed falls on a vacant site, 
this glues two particles (preserving the
total mass). And each time a match falls on a site 
(say, belonging to a forest containing $k\geq 1$ sites), this breaks
a particle of mass $k+1$ into $k+1$ particles with mass $1$.

\vip

We used this remark in \cite{bfold} 
to study the evolution of the sizes of particles when neglecting
correlation, using a deterministic coagulation-fragmentation equation. 
Of course, similar considerations can be handled on any graph $G$.

\vip
\subsubsection{Recent results for related models in dimension 1}
Let us mention two recent results about one-dimensional forest fire processes
with a somehow different flavor. 

\vip

In \cite{vo}, Volkov considers a version of the forest fire process on 
$\nn$ where ignition occurs only at $0$. He studies the weak limit of 
the distribution of the (suitably normalized) delay between to fires involving
$n$, as $n \to \infty$. 

\vip

In \cite{be}, Bertoin considers a modified version of Knuth's parking 
model where random fires burn connected components of cars. On a 
circle of size $n$, cars arrive at each site at rate $1$.
When a car arrives, it occupies
the first vacant site (turning clockwise). 
Molotov cocktails fall on each site at rate $n^{-\alpha}$ where 
$0< \alpha<1$ is fixed. 
Bertoin studies the asymptotic behavior of the saturation time as 
$n \to \infty$ and  observes a phase transition at $\alpha = 2/3$.

\vip
\subsubsection{Specific difficulties} As we already mentioned, one of 
the difficulties with forest fire models (and with self-organized critical 
systems in general) is that 
the interaction is not local. The process, whenever it is Markov,  is 
not Feller and some classical results fail. In dimension one, this 
difficulty does not yield real problems for the questions of existence 
and uniqueness of the process. This is essentially due to the fact
that obviously, the sizes of the
forests always remain finite (even when $\la$ is very small). This 
difficulty is more important in higher dimensions, because
in the absence of fires, clusters would become infinite in finite time
(due to the fact that in dimension $d\geq 2$, percolation occurs).
Fires prevent us from the existence of infinite clusters, but precisely,
these arbitrarily huge clusters burning make difficult the 
control of the range of interactions. This difficulty also makes
the usual proof of existence of stationary measures using compactness 
arguments fail (because indeed there is a lack of continuity). 

\vip

The lack  of monotonicity of these models, although not fundamental, makes
the use of standard intuitions and techniques impossible. 
Monotonicity allows one
to compare the processes started from two different ordered 
initial configurations (coupled in a suitable way). Monotonicity cannot
hold here, because a configuration with more trees will burn sooner.

\subsection{Geometry of the lattice}\label{intro3}
The geometry of the 
underlying lattice is crucial in statistical mechanics. Recall 
for instance that {\it phase transition} for the Ising model on $\zz^d$
appears only for $d\geq 2$ (see Velenik \cite{vel}). 
For the forest fire models, the influence of the geometry 
clearly comes through the behavior of the lattice with 
respect to percolation.  
This geometrical influence was already striking  in numerical studies. 

\vip

\subsubsection{Growth without fires/Percolation}
Consider a graph $G=(S,A)$. For all $0\leq p \leq 1$ consider an
i.i.d. family $\{\eta(i), i \in S\}$
of  Bernoulli random variables with parameter $p$ (a {\it percolation 
trial} with probability $p$). It is well known that there is 
$0 \leq p_c \leq 1$, depending on the graph, such that for all $p<p_c$, 
there are a.s. no infinite connected components of occupied sites, 
while for $p>p_c$,  there is at least one infinite connected component 
with probability $1$. The real number $p_c$ is called
percolation threshold of $G$. It is rather 
natural to consider (dynamical) {\it percolation processes} on $G$, 
that are couplings of percolation trials for all $0\leq p \leq 1$. 
For instance, consider a family $\{T_i, i \in S\}$ of i.i.d. random
variables on $\rr_+$ with exponential distribution with parameter
$1$. Put $\eta_t(i)=0$ if $t<T_i$ and $\eta_t(i)=1$ if $t\geq T_i$.
Then for all $t>0$, $\{\eta_t(i), i \in S\}$ is a percolation trial with 
probability $P(T_i \leq t)=1-e^{-t}$. Thus an infinite cluster appears at
time $t_c$ defined by $1-e^{-t_c}=p_c$.

\vip

It clearly appears that the percolation threshold plays a 
{\it crucial} role in understanding the behavior of 
the forest fire process on a given lattice. The simple observation 
is that the {\it growth process}, i.e. without fires ($\la =0$), is 
exactly a percolation process on the lattice. For $\la$ small, and 
{\it a fortiori} for $\la \to 0$ its  study is a necessary preliminary. For 
instance, one aspect  is the formation of infinite clusters (although 
in general those clusters will never appear since, taking fires into 
account,  they must burn before they become infinite). Recall 
that the percolation threshold is $1$ in dimension $1$. It is 
$0< p_c^{(d)}<1$ on $\zz^d$ and once there is an infinite cluster, 
there is a unique one. While, for instance on a $d$-regular tree, 
just after the percolation threshold, there are infinitely many infinite 
clusters: these situations are rather different and should yield different 
behaviors for the corresponding forest fire processes. 
Observe that though, for $\la$ small enough, the forest fire process is 
easy to define for small times,  things turn out to be more complicated 
when we reach the {\it critical time} $t_c$.
Even in dimension $1$ the separate study of the percolation 
process makes sense as we shall see further, Subsection \ref{kingman}. 

\vip
\subsubsection{Modified percolation models}
It has  also been fruitful to study modified (for instance dynamical) 
versions of percolation processes.  
%Indeed some questions have been 
%reformulated or related to questions about percolation and/or dynamical 
%models derived from percolation. The growth process (before a fire starts) 
%is indeed the dynamical percolation model. 
Models like frozen percolation 
(Aldous \cite{a2}, see also Brouwer \cite{br}), invasion percolation 
(see for instance V\'agv\"olgyi \cite{vag}), 
or  self-destructive percolation (see van den Berg-Brouwer \cite{vdbb} 
and more recently 
van den Berg-Brouwer-V\'agv\"olgyi \cite{vdbbv}) are
closely related to the forest fire processes. Let us focus one moment on 
this last example since it has direct implications on forest fire processes. 

\vip 

A typical configuration for the self-destructive percolation model on $\zz^2$
with parameter $(p,\delta)$
is generated in three steps: first generate a configuration for 
the ordinary percolation model 
with parameter $p$. Next, make all sites in the infinite occupied
cluster vacant. Finally, make occupied each vacant site with probability
$\delta$.
Let $\theta(p,\delta)$ be the probability that $0$ belongs, 
in the final configuration, to an infinite occupied cluster.
In a recent paper \cite{vdbbv}, van den Berg, Brouwer and V\'agv\"olgyi prove 
that this function  
is continuous outside of a set of the form 
$\{ (p_c, \delta): \delta < \delta_0 \}$. 
It is conjectured that this function has a discontinuity,
roughly meaning that there is $\delta>0$ such that for any $p>p_c$, 
the model with parameter $(p,\delta)$ 
is sub-critical (there a.s. is no infinite cluster). 

\vip

In \cite{vdbb}, van den Berg and Brouwer have proved that assumption of 
this conjecture yields a result for a $2$-dimensional forest fire process
after the critical time: there is $t>t_c$ such that for all $m\geq 1$,  
$$
\liminf_{\la\to 0} \liminf_{n \to \infty} \Pr\left[
\begin{array}{l}
\hbox{a tree in $[\![-m,m]\!]^2$ 
burns before $t$} \\
\hbox{in the  forest fire
process on $S_n=[\![-n,n]\!]^2$}
\end{array}
\right] \leq \frac12.
$$

\vip
\subsubsection{Thermodynamic limit}
The forest-fire process on a finite graph 
is a finite state space continuous time Markov 
chain (if matches and seeds fall according to Poisson processes). 
Existence and uniqueness of the process thus come for free. Existence of
an invariant measure as well. A basic argument also yields uniqueness of the 
invariant measure (because the configuration with all sites vacant
is recurrent). Hence interesting 
phenomena may arise only when we let the size of the lattice tend to infinity. 

\vip

When $S=\zz$, it is not very expensive to go directly to the limit: 
the process is naturally uniquely defined on $\zz$. This is easily 
seen through a graphical construction of the process (see \cite{bf}),
see also Proposition \ref{gcgff} below.

\vip

In dimension $d> 1$ the situation is more delicate. On $\zz^d$
(and actually on any graph with bounded vertex degree) 
existence has been proved recently by D\"urre \cite{du1}. He also 
proved uniqueness, but in two steps:  firstly, in \cite{du2}, he 
shows that, for $\la>0$ large enough (the bound is related to the percolation 
threshold), the forest-fire process is unique.
Only very recently the same author, in  \cite{du}, tackled the same
question on a  graph with bounded vertex degree and for all $\la>0$. 
This is a much more subtle task. To prove this result he has to 
introduce the so-called blur processes, to show that the influence
of matches falling far away from $0$ is negligible.

\vip
\subsubsection{Mean field model}\label{meanfield}
The mean field case is slightly different. Indeed, one has to adopt the 
dual point of view (on edges).
Furthermore, the process cannot be defined 
directly on an infinite lattice since we consider the complete graph. 
The point of view developed by R\'ath and T\'oth in \cite{rt} is based 
on the Erd\"os-Reyni construction \cite{er}. For all $n \geq 1$, let 
$S_n$ be a set (of vertices) with $|S_n|=n$, and consider 
the complete graph $G_n=(S_n,A_n)$. Start initially with all edges
vacant. Then edges appear independently at rate $1/n$. Matches fall
at rate $\la_n$ on each site and destroy instantaneously
the whole corresponding occupied connected component. 
We consider the asymptotic $n \to \infty$. 
The various regimes (see R\'ath-T\'oth \cite{rt}) are quite illuminating.

\vip

$\bullet$ (I) If $\la_n <<1/n$, then fires are (asymptotically) negligible. 
Thus we have the same asymptotics as in the Erd\"os-R\'eyni model: a 
giant component appears after some time $T_{gel}$ (the critical time in 
this formalism). 

\vip

$\bullet$ (II) If $\la_n \simeq \la/n$, then a giant component appears,
but is destroyed after some time. Only the giant component may burn:
there are no matches enough to burn finite size forests. 

\vip

$\bullet$ (III) If $1/n<< \la_n << 1$, there are not enough fires to burn
finite size forests, but too many to let any infinite forest appear. Hence
no giant component appears.

\vip

$\bullet$  (IV)  If $\la_n \simeq \la$, then matches may kill finite forests,
so that of course, no giant component emerges.

\vip

To formalize these statements rigorously,  R\'ath-T\'oth \cite{rt} consider 
the cluster size distributions: $\nu_{n,k}(t)$
is the number of vertices belonging to a connected component of size $k$ at time
$t$ divided by $n$. Consider also the {\it concentrations}
$c_{n,k}(t):=\nu_{n,k}(t) /k$.
As $n \to \infty$, the limit concentrations $(c_k(t))_{k\geq1}$
should satisfy a system of differential equations closely related to 
Smoluchowski's coagulation equations with 
multiplicative kernel and mono-disperse
initial condition:
$$\left\{\begin{array}{l}
c_1(0)=1, \quad c_k(0)=0, \quad k \geq 2, \\
\displaystyle\frac{d}{dt} c_k(t) = \frac 1 2 \sum_{i=1}^{k-1} 
i(k-i)c_i(t) c_{k-i}(t)
- k c_k(t) \sum_{i=1}^\infty ic_i(t), \quad k \geq 1.
\end{array}\right.
$$
Such equations, discussed in details in Aldous \cite{a}, have 
been introduced by Smoluchowski \cite{smol} in 1916. These equations
are subjected to a {\it phase transition} known as {\it gelation}: some
mass is lost at some positive finite instant $T_{gel}$, due to the emergence
of a giant particle. For $t>T_{gel}$, we have to decide what to do 
with the giant particle. It can e.g. interact with finite particles
(Flory's equation) or be removed from the system (Smoluchowski's equation).
See Aldous \cite{a3} and \cite{fl} for such considerations.

\vip

In the regime (I), the limit equations are the Flory equations: 
a giant particle appears at time $T_{gel}$ and then coexists  with other
particles (finite particles do coalesce with the giant particle). 
In the regime (II), the limit equations are closer to the Smoluchowski
equations: a giant particle appears at time $T_{gel}$ 
(the same one as previously) but once it is giant, it is replaced by
particles with mass $1$ (in a conservative way). In the regimes 
(III) and (IV), some other modifications
of the Smoluchowski equations appear.

\vip

The most interesting results obtained by R\'ath-T\'oth in \cite{rt} are that 
in the regime (III),
the modified Smoluchowski coagulation system has a unique solution 
which is the classical one for all $t < T_{gel}$ and has a particular 
(critical-like) form for $t>T_{gel}$, and $(c_{n,k}(t))_{t\geq 0,k \geq 1}$ 
converges to this unique solution as $n \to \infty$.
This shows that the complete graph exhibits self-organized criticality
in the sense that beyond $T_{gel}$, it 
remains critical forever: no giant component appears but, after $T_{gel}$,  
the size-distribution is, in some sense, critical.
 
%\vip
%In dimension $1$ or more, i.e. when it makes sense to define the process 
%directly on the infinitely extended lattice, it is natural to first study 
%regime $(IV)$ and then  let $\la \to 0$. But, then, one has to renormalize 
%time and/or space in order to see fires occur. 

\vip
\subsubsection{Stationary measures}
The existence of invariant measures for the 
forest-fire process in $\zz^d$ (with any $\la>0$ fixed) 
has been proved by Stahl \cite{stahl}. 
For the case of $\zz$ the situation is simpler, see the next subsection.

\subsection{Forest fire  on $\zz$}\label{intro4}
Let us review in details known results about the forest fire processes 
in dimension $1$. 
We still focus on the usual case where seeds and matches fall according
to i.i.d. Poisson processes, with respective rates $1$ and $\la>0$. 
We denote $\eta^\la_t \in \{0,1\}^\zz$ the configuration at time $t$ 
and, for $i\in\zz$,  $C(\eta^\la_t,i)$ is 
the connected component of occupied sites around $i$.
Observe that (possible) infinite clusters in the 
initial configuration would immediately disappear. 

\vip

From the point of view of self-organized criticality, the interesting
regime is the asymptotic behavior of the forest-fire process as $\la \to 0$:
then fires are very rare, but concern huge occupied components.

\vip
\subsubsection{Stationary measures}
Existence of a stationary measure does not immediately follow from standard 
compactness arguments since the process is not Feller. However, in \cite{bp}, 
Brouwer and Pennanen prove the existence of a stationary measure for all fixed 
$\la>0$. In \cite{bfold}, we proved the uniqueness of this invariant 
distribution, as well as the exponential convergence to equilibrium 
in the special case where $\la = 1$.
We also proved that the invariant distribution is
(spatially) exponentially mixing and can be graphically constructed. 
The methods in \cite{bfold} should be easily extended
to the case where $\la\geq 1$ (and actually to $\la>1-\e_0$ for some
rather small $\e_0>0$) but our proof completely breaks down 
for small values of $\la>0$.

\vip
\subsubsection{Asymptotic density}
Van den Berg and J\'arai study in \cite{vdbj} the asymptotic 
density of vacant sites in the limit $\la \to 0$. Their result states that 
there are two constants $0<c<C$ such that for any initial configuration,
for any $\la >0$ small enough, for $t$ large enough (of order 
$\log(1/\la)$), 
$$\frac{c}{\log(1/\la)} \leq \Pr\left(\eta^\la_t(0) = 0\right) \leq   
\frac{C}{\log(1/\la)}.$$
This is coherent with the intuition that the rarer fires are, the 
more space is occupied by trees 
(although because of the lack of monotonicity, this is not 
straightforward). 
We mentioned that such result was stated in Drossel-Clar-Schwabl
\cite{dcs}. But the proof in \cite{dcs}
is not rigorous: it is based on  the {\it ansatz}  that the cluster 
sizes were following a cutoff power law, for cluster-sizes
up to some $s_{max}^\la$ defined by $s_{max}^\la \log{s_{max}^\la} = 1/\la$, i.e.
$$
s_{max}^\la \simeq \frac 1 {\la \log(1/\la)}.
$$  
In \cite{vdbj}, van den Berg and J\'arai also show that
the cluster sizes cannot follow the predicted power law.

\vip
\subsubsection{Sizes of clusters, first results}
\label{clusterssize}
In \cite{bp}, Brouwer and Pennanen show that this last {\it ansatz} 
holds true up to 
$s_{max}^{1/3}$. More specifically, they show that there are constants 
$0<c<C$ such that for all $0 < \la < 1$ and all stationary measures $\mu_\la$
(invariant by translation) of the forest fire model on $\zz$ with parameter
$\la$, for all $x<(s_{max}^\la)^{1/3}$, 
$$
\frac{c}{(1+x)\log{(1/\la)}} \leq \mu_\la \left(|C(\eta,0)| = x \right) 
\leq \frac{C}{(1+x) \log{(1/\la)}}. 
$$
Observe that this estimate is valid for relatively small clusters 
that will not be seen after rescaling (microscopic clusters).  

\vip

\subsubsection{Kingman's Process} 
\label{kingman}
We detail a classical construction related to the Smoluchowski equation 
with constant kernel which is quite close to our point of view. Most ideas and 
references for proofs can be found in Aldous \cite{a}. 
Let us consider the following percolation process on 
$\zz$. Starting from the vacant configuration, we let appear trees at 
each site at some rate $r(t)$, that allows us to control the 
{\it speed} of the process. 
Say that each edge $(i,i+1)$ has mass $1$ (see Subsubsection \ref{coagfrag}).
Let a seed fall on each site $i$ at some random time $T_i$ with 
$P(T_i>t) = 2/(t+2)$
independently (this corresponds to the rate 
$r(t)= 1/(t+2)$). Call $D(t,i)$ the {\it particle} containing the edge
$(i,i+1)$ at time $t$ (say that two neighbor edges $(j-1,j)$ and $(j,j+1)$
are glued if $\eta_t(j)=1$).
At time $t$, the particle containing a given edge (e.g. $(0,1)$)
has mass $m$ with 
probability 
$$\rho_m(t) = m \left(\frac{2}{2+t}\right)^{2} \left(\frac{t}{2+t}\right)^{m-1}
$$ 
and hence the concentration of clusters with mass $m$ per unit length is 
nothing but
$$c_m(t)= \left(\frac{2}{2+t}\right)^{2} \left(\frac{t}{2+t}\right)^{m-1}.
$$
We recognize the solution to Smoluchowski's equation with constant
coagulation kernel and mono-disperse initial condition, see
Aldous \cite{a}.

\vip

Now consider a standard construction of the so-called Kingman coalescent 
process. Take independent exponential random variables $\{\xi_k, k \geq 2\}$ 
of rates $\binom{2}{k}$. Since $E[\sum_{k=2}^{\infty} \xi_k] = 2$, we can 
define random times $0 < \cdots < \tau_3< \tau_2 < \tau_1 < \infty$ by 
$\tau_i = \sum_{k=i+1}^{\infty} \xi_k$. Take $\{U_i, i \geq 1\}$ independent 
random variables uniformly distributed on $(0,1)$. For each $i$ draw a 
vertical segment  from $(U_i,\tau_i)$ to $(U_i,0)$. At time $t$ this 
construction splits $(0,1)$ into $i$ intervals, where 
$\tau_i < t < \tau_{i-1}$. Write $X(t)$ for the list of the lengths of 
these subintervals. This is a version of the stochastic coalescent called 
{\it Kingman's coalescent}. Observe that we also could have put the 
marks $\{(U_i,\tau_i), i \geq 1\}$ using a Poisson measure 
on $[0,1] \times \rr_+$ with a well-chosen intensity measure.

\vip

Straightforward computations show that Kingman's coalescent is a limit
of the previously defined percolation process in the following sense: 
consider the list of (distinct) normalized clusters
$\la D(t/\la, \lfloor x /\la  \rfloor)$ when $x$ runs 
along $[0,1]$ (cutoff the boundary clusters at $0$ and $1$)  
at time $t$. When $\la \to 0$, it converges to $X(t)$ in law (in an 
appropriate topology). This construction shows how the growth process 
behaves in the large scales. 
In some sense  we have identified $\{0, \ldots, \bn_\la\} \subset \zz$ 
with $[0,1] \subset \rr$ (here $\bn_\la = 1/\la$) and obtained
a limiting process for the rescaled percolation process.   

\vip

We stress the fact that the convergence holds globally only for the specific 
speed $r(t)= 1/(t+2)$ of the percolation process. This fact is related to 
the self-similarity of the percolation (coalescent) process. In particular it 
shows that for if the rate is constant (exponential times for seeds) then 
there is no hope for such a convergence to Kingman's coalescent: 
in this situation, 
after normalization the size of cluster at time $t$ is of order $\la^{1-t}$ 
and converges to $0$ or $\infty$ according to whether $t<1$ or $t>1$. 
Conversely, if the rate of growth has a polynomial decay,
there is a hope to have a limit process.

\vip
\subsubsection{Asymptotic regime: relevant space/time scales}
As already mentioned, we are interested in the 
behavior of the system in the large space and time scales in the limit 
$\la \to 0$. Hence the first difficulty is to decide what the relevant  
scales are.  Let us recall the heuristic developed in \cite{bf}.
We need a time scale for which tree clusters see about
one fire per unit of time. But for $\la$ very small, clusters will
be very large just before they burn. We thus also have to rescale space, 
in order that just before burning, clusters have a size of order $1$.

\vip

Consider the cluster $C(\eta^\la_t,0)$ around the site $0$ (for example)
at time $t$. For $\la>0$ very small and 
for $t$ not too large, one might neglect fires and consider only the growth
process; it follows that $|C(\eta^\la_t,0)|\simeq e^{t}$ for $t$ not too large
(because since seeds fall according to Poisson processes with rate $1$,
each site is vacant at time $t$ with probability $e^{-t}$). 
Then the cluster $C(\eta^\la_t,0)$ burns at rate $\la |C(\eta^\la_t,0)|
\simeq \la e^t$, so that we decide to accelerate time by 
a factor $\ba_\la:=\log (1/\la)$. 
By this way, $\la |C(\eta^\la_{\ba_\la},0)|\simeq 1$.

\vip

Now we rescale space in such a way that during a time interval of order
$\ba_\la$, something like one match falls per unit of (space) length.
Since matches fall at rate $\la$ on each site, 
our space scale has to be of order 
$\bn_\la:= 1/(\la\ba_\la)$: this means that we will identify
$\{ 0,\dots,\bn_\la\} \subset\zz$ with $[0,1]\subset\rr$.
Observe that there holds $\bn_\la \simeq s_{max}^\la$.

\vip

Consider now the time/space rescaled cluster around $0$
$$ 
D_t^\la(0) = \frac{1}{\bn_\la} C(\eta^\la_{ \ba_\la t}, 0).
$$
The same difficulty as in Subsubsection \ref{kingman} appears: neglecting
fires (which is roughly valid for small values of $t$), we see that
$$
|D^\la_t(0)| \simeq \bn_\la^{-1} e^{\ba_\la t}=\la^{1-t} \log (1/\la),
$$ 
which goes to $0$ for $t<1$ and to $\infty$ for $t\geq 1$. 
For $t\geq 1$, we hope that fires will be in effect, which will
limit the size of clusters. But for $t<1$, $|D^\la_t(0)|$ will indeed
tend to $0$. This means that we have lost some information. 
To describe the limit process, we have to keep in mind more information and
thus introduce another quantity (a sort of {\it degree of smallness}) 
which measures the order of magnitude of the {\it microscopic} clusters,
that is clusters that we can not see at macroscopic scales (of which the
sizes are much smaller than $\bn_\la$). 

\vip
\subsubsection{Limit processes}
We have proved in  \cite{bf} that in the asymptotic of rare 
matches, the forest fire process converges, under the previously described
normalization, to some limit forest fire process.
We described precisely the dynamics of this limit process and have shown 
that it is unique, that it can be built by using a graphical construction and
thus can be perfectly 
simulated. Using the limit process, we have also estimated 
the size of clusters. Very roughly, we have proved
that in a very weak sense, for $\la$ small enough and for $t$ large enough
(of order $\log(1/\la)$), the cluster-size distribution resembles
$$
\Pr\left[ C(\eta^\la_{t}, 0) = x\right]  
\simeq \frac{a}{(x+1) \log(1/\la)}\indiq_{\{x << \bn_\la\}} + 
\frac{b e^{-x / \bn_\la}}{\bn_\la},
$$
where $a,b$ are two positive constants. Very roughly, 
we are able to replace the condition $x<(s_{max}^\la)^{1/3}$ of \cite{bp} 
by the condition $x<(s_{max}^\la)^{1-\e}$ for any $\e\in (0,1)$
(but our result is weaker, in the sense that it holds when integrated in $x$,
and we have to take the limit $\la\to 0$).
This means that there are two types
of clusters: {\it microscopic clusters}, described by a power-like law
and {\it macroscopic} clusters, described by an exponential-like law.
This shows a {\it phase transition} around the {\it critical size} 
$\bn_\la$.

\vip

\subsubsection{No self-organized criticality}

From the qualitative point of view the conclusion is rather different from 
that of R\'ath and T\'oth \cite{rt} 
(presented in Subsubsection \ref{meanfield}).
Here, the (asymptotic) cluster-size distribution does not
exhibit self-organized criticality  features. We proved the presence 
of a power law, but
this power law describes clusters which are much smaller than the critical
size. Large clusters (clusters near the {\it critical size}) 
have a law with fast decay.

\subsection{Main ideas of the present paper}\label{intro5}

From the modelling point of view, the Poisson assumption is clearly not
well justified. Thus it seems interesting to study what happens
when seeds and matches are driven by other renewal processes. 
The goal of this paper is to extend the previous study \cite{bf} described 
above to a more general class of renewal processes. 
We assume that the renewal processes are stationary for simplicity,
but this can be more or less justified by the fact
that it is the only way that time $0$ does not play a special 
role.

\vip

We thus consider the case where seeds (respectively matches) 
fall on each site of $\zz$ independently, according to some stationary
renewal processes, with {\it stationary delay} distributed according
to some law $\nu_S$ (respectively $\nu^\la_M$). 
This means that for any time $t\geq 0$
and on any site $i\in\zz$, the time we have to wait for the next seed
is a $\nu_S$-distributed random variable. 
We have an assumption saying that as $\la \to 0$, matches are rarer
and rarer. We also assume that $\nu_S$ has a bounded support or a tail
with fast or regular or slow variations. We prove that,
after re-scaling, the corresponding forest fire process
converges, as $\la\to 0$, to a limit process. And we show that there
are four classes of limit processes, according to the fact that

\vip
$\bullet$ $\nu_S$ has a bounded support ($HS(BS)$),

\vip

$\bullet$ $\nu_S$ has a tail with fast decay ($HS(\infty)$),

\vip

$\bullet$ $\nu_S$ has a tail with polynomial decay ($HS(\beta)$),

\vip

$\bullet$  $\nu_S$ has a tail with logarithmic decay ($HS(0)$).

\vip

As we will see, the limit forest fire process
built in \cite{bf} is quite universal: it describes the asymptotics of
a large class (roughly exponential decay for $\nu_S$) of forest fire processes.
A similar limit process arises when $\nu_S$ has bounded support. 
But some quite different limit
processes arise when $\nu_S$ has a heavy tail.
We also develop the necessary tools to study the cluster size distributions.
Let us mention at once that there is indeed presence of a {\it critical size} 
under $(HS(BS))$ and $(HS(\infty))$ but not under $(HS(\beta))$ or $(HS(0))$.
In the latter situation, there are only macroscopic clusters. This
is related to Subsubsection \ref{kingman}.

\vip

It is striking that in \cite{bf} we made repeated use of the Markov property 
of Poisson processes while it turns out the result still holds 
without this assumption (and with no significant increase of the complexity). 
Indeed, proofs remain essentially elementary except maybe 
from the combinatorial and computational point of view. 

\vip

From the qualitative point of view, the main novelty is the rise of a new 
class of processes (those corresponding to polynomial tails), reminiscent 
of the Kingman coalescent (with deaths). But for this case as for the 
others, the conclusion is that, as expected, 
self-organized criticality features do not show up 
for this model in dimension 1.

\vip

Let us finally insist on the fact that surprisingly (in view of the
complexity and length of the proofs), our assumptions are really light. 
Consider e.g. the case where $\nu_S$ has an unbounded support and 
a fast decay, which means (for us) that for any $t>0$,
$$
\lim_{x\to \infty} \frac{\nu_S((x,\infty))}{\nu_S((tx,\infty))} = t^\infty,
$$
where $t^\infty=0$ if $t<1$,  $1^\infty=1$, and $t^\infty=\infty$ if $t>1$.
We do not need the least additional condition. 

\subsection{Plan of the paper}\label{intro6}

Part 2 is devoted to a complete exposition of our results.
We start in Section \ref{sex1} with notation and with the definitions of the 
objects under study, and we state our assumptions. 
In Section \ref{hscales}, we explain the heuristic scales
and the relevant quantities 
(rescaled macroscopic clusters and measure of microscopic
clusters). Then we describe precisely our results in Sections 
\ref{mri} (case with fast decay), \ref{mrbs} (case with bounded support),
\ref{mrbeta} (case with polynomial decay) and \ref{mrz} (case with logarithmic
decay). We conclude this part with a quick discussion about our modeling
choices.
Part 3 (Sections \ref{exunidisc} to \ref{prcsd}) contains all the proofs.
We handle a few numerical simulations to illustrate our results
in Part 4.
Finally, Part 5 contains an appendix about regularly varying
functions and coupling.

\part{Notation and results}

\section{Definitions, notation and assumptions}\label{sex1}

\subsection{Stationary Renewal processes}\label{srp}

We first fix notation about stationary renewal processes.
We refer to Cocozza-Thivent \cite{ct} for many precisions.

\begin{defin}\label{sr1}
For $\mu$ a probability measure on $(0,\infty)$ with finite expectation
$m_\mu$,
set $\nu_\mu(dt)=m_\mu^{-1}\mu((t,\infty))dt$, which is also a probability 
measure on $(0,\infty)$.
Let $T_1$ be a $\nu_\mu$-distributed random variable and let
$(X_k)_{k\geq 1}$ be a sequence of i.i.d. random variables
with law $\mu$, independent of $T_1$. Set
$T_{k+1}=T_k+X_k$ for all $k\geq 1$ and $N_t=\sum_{k\geq 1} 
\indiq_{\{T_k\geq t\}}$ for all $t\geq 0$.
We say that $(N_t)_{t\geq 0}$ is a stationary 
renewal process with parameter $\mu$, or a $SR(\mu)$-process in short.
\end{defin}

It is well-known, see e.g. \cite[Corollaire 6.19 p 169]{ct}, 
that for $(N_t)_{t\geq 0}$ a $SR(\mu)$-process in the sense of 
Definition \ref{sr1},
the law of $T_{N_{t}+1}-t$ (i.e. the time we have to wait
for the next mark at time $t$)
is $\nu_\mu$ for all $t\geq 0$.  Another possible definition is the following.

\begin{defin}\label{sr2}
For $\mu$ a probability measure on $(0,\infty)$ with finite expectation
$m_\mu$, set $\nu_\mu(dt)=m_\mu^{-1}\mu((t,\infty))dt$
and $\zeta_\mu(dt)=m_\mu^{-1} t \mu(dt)$,
which are also probability measures on $(0,\infty)$.
Consider some random variables $(X_i)_{i \in \zz\setminus \{0\}}$ with
law $\mu$. Consider also $X_0$ with law $\zeta_\mu$ and $U$ uniformly
distributed on $[0,1]$. Assume that all these random variables are independent.
Define $T_0=-(1-U)X_0$, $T_1= U X_0$ and then, for $n\geq 1$, $T_{n+1}=T_n+X_n$
and $T_{-n}=T_{-(n-1)} - X_{-n}$. Then we say that 
$(T_n)_{n\in \zz}$ is a $SR(\mu)$-process. 
\end{defin}

\vip

If $(T_n)_{n\in \zz}$ is a $SR(\mu)$-process in the sense of Definition
\ref{sr2} and if one considers the
associated counting process $N_t=\sum_{n\geq 1}\indiq_{\{T_n\leq t\}}$, 
it is indeed a 
$SR(\mu)$-process in the sense of Definition \ref{sr1}.
This can
be checked immediately: it suffices to observe that the law of $T_1$ is
$\nu_\mu$.

\vip

If we have a $SR(\mu)$-process $(N_t)_{t\geq 0}$ 
as in Definition \ref{sr1} and if we denote by $(T_n)_{n\geq 1}$ its successive
instants of jump, one can easily build $(T_n)_{n \leq 0}$ in such a way
that $(T_n)_{n\in \zz}$ is a $SR(\mu)$-process as in Definition \ref{sr2}.

\vip

It also holds that for $(T_n)_{n\in\zz}$ a $SR(\mu)$-process 
as in Definition \ref{sr2}, 
for any $t\in \rr$, the random sets $\cup_{n\in \zz} \{T_n\}$,  
$\cup_{n\in \zz} \{-T_n\}$ and 
$\cup_{n\in \zz} \{T_n+t\}$ have the same law. Thus if we introduce
$n_t$ such that $T_{n_t}+t<0<T_{n_t+1}+t$, the process $(T_{n_t+n}+t)_{n\in\zz}$ 
is a $SR(\mu)$-process. By the same way, the process
$(-T_{1-n})_{n\in \zz}$ is a $SR(\mu)$-process.

\subsection{The discrete model}

Next, we introduce the forest fire model.
For $a,b\in \zz$, we set $\lb a, b \rb =\{a,\dots,b\}\subset \zz$.
For $\eta \in \zuz$ and $i \in \zz$, we define the occupied 
connected component around $i$ as
$$
C(\eta,i)=\left\{ \begin{array}{lll}
\emptyset & \hbox{ if } & \eta(i)=0, \\
\lb l(\eta,i),r(\eta,i) \rb & \hbox{ if } & \eta(i)=1,
\end{array}\right.
$$ 
where $l(\eta,i)=\sup\{k< i:\; \eta(k)=0\}+1$ and
$r(\eta,i)=\inf\{k > i:\; \eta(k)=0\}-1$.

\begin{defin}\label{gff}
Let $\mu_S$ and $\mu_M$ be two laws on $(0,\infty)$ with some finite
expectations.
For each $i\in \zz$, we consider a $SR(\mu_S)$-process $(N^S_t(i))_{t\geq 0}$
and a $SR(\mu_M)$-process $(N^M_t(i))_{t\geq 0}$, all these processes being
independent. A $\{0,1\}$-valued process 
$(\eta_t(i))_{i \in \zz, t\geq 0}$ such that
$(\eta_t(i))_{t\geq 0}$
is a.s. c\`adl\`ag for all $i\in\zz$
is said to be a $FF(\mu_S,\mu_M)$-process
if a.s., for all $t\geq 0$, all $i\in \zz$,
$$
\eta_t(i)=\intot \indiq_{\{\eta_\sm(i)=0\}} dN^S_s(i) 
- \sum_{j\in\zz}\intot \indiq_{\{j\in C(\eta_\sm,i)\}} dN^M_s(j).
$$
\end{defin}

Formally, we say that $\eta_t(i)=0$ if there is no tree at site $i$
at time $t$ and $\eta_t(i)=1$ else.
Thus the forest fire process starts
from an empty initial configuration, seeds fall according to some i.i.d. 
$SR(\mu_S)$-processes and matches fall  according to some i.i.d. 
$SR(\mu_M)$-processes. When a seed falls on an empty site, a tree 
appears immediately. When a match falls on an occupied site,
it burns immediately  the corresponding connected component of occupied sites.
Seeds falling on occupied sites and matches falling on vacant sites
have no effect.
\vip

Assume for a moment that the support of $\mu_S$ is unbounded (thus so is that
of $\nu_{\mu_S}$). Then the $FF(\mu_S,\mu_M)$-process
can be shown to exist and to be unique (for almost every 
realization of $(N^S_t(i),N^M_t(i))_{i\in \zz,t\geq 0}$), 
by using a genuine {\it graphical construction}.
Indeed, to build the process until a given time $T>0$,
it suffices to work between sites $i$ which are vacant until time $T$ 
(because $N_T^S(i)=0$). Interaction cannot cross such sites. Since
such sites are a.s. infinitely many (because
$\Pr(N_T^S(i)=0)=\nu_{\mu_S}((T,\infty))>0$ by assumption),
this allows us to handle 
a graphical construction. This is illustrated by Figure \ref{GC}.
See Liggett \cite{l} for many examples of graphical constructions.

\begin{figure}[b] 
\fbox{
\begin{minipage}[c]{0.95\textwidth}
\centering
\includegraphics[width=9cm]{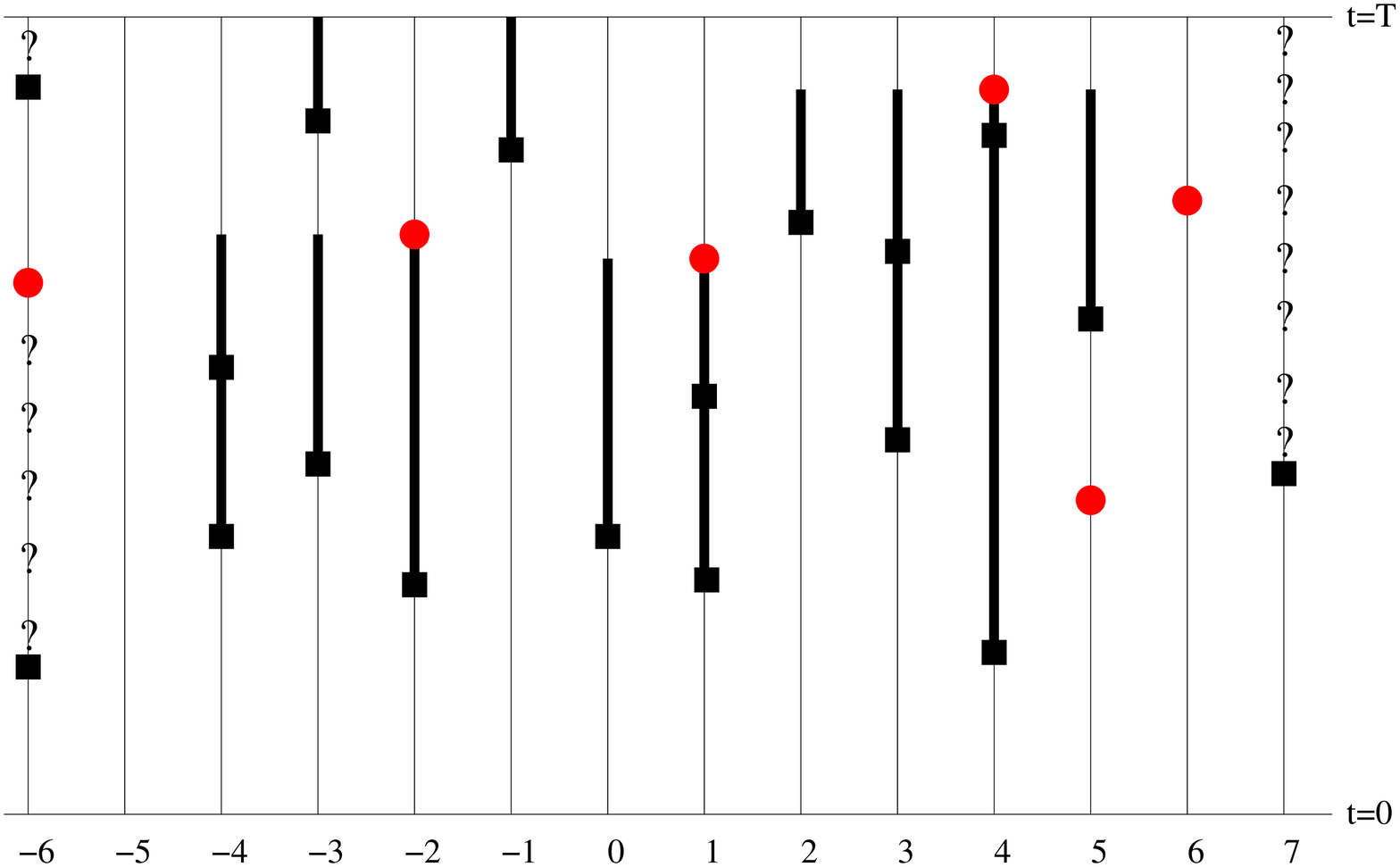}
\caption{Graphical construction of the $FF(\mu_S,\mu_M)$-process.}
\label{GC}
\vip
\parbox{13.3cm}{
\footnotesize{
Matches are represented as bullets and seeds as squares. On the sites
$-5$ and $6$, no seed fall during $[0,T]$, so that these
sites remain vacant until $T$. One can thus clearly deduce the values
of the process in $\lb -5,6\rb$ during $[0,T]$ using only
the bullets and squares inside $\lb -5,6\rb$.
}}
\end{minipage}
}
\end{figure}

\vip

We will also study the more complicated case 
where $\mu_S$ has a bounded support and this will lead to the 
following general result.

\begin{prop}\label{gcgff}
Let $\mu_S$ and $\mu_M$ be two laws on $(0,\infty)$ with some finite
expectations.
For each $i\in \zz$, we consider a $SR(\mu_S)$-process $(N^S_t(i))_{t\geq 0}$
and a $SR(\mu_M)$-process $(N^M_t(i))_{t\geq 0}$, all these processes being
independent.
Almost surely, there exists a unique $FF(\mu_S,\mu_M)$-process.
\end{prop}

This proposition is proved in Section \ref{exunidisc}.

\subsection{Assumptions}\label{ass}

We now state the assumptions we will impose on the laws $\mu_S$ and
$\mu_M$. First, we want to express the fact that matches are less
and less frequent. To do so, we consider a family of laws
$\mu_M^\la$, for $\la \in (0,1]$, as follows.

\vip
\begin{center}
\fbox{\begin{minipage}[c]{0.9\textwidth}
$(H_M)$: For each $\la \in (0,1]$, $\mu_M^\la$
is the image measure of $\mu_M^1$ by the map $t \mapsto t/\la$
and the probability measure $\mu_M^1$ on $(0,\infty)$
satisfies $\int_0^\infty t \mu_M^1(dt)=1$.
We set
$$
\nu_M^\la(dt)=\nu_{\mu_M^\la}(dt)= \la \mu_M^\la((t,\infty))dt =
\la \mu_M^1((\la t, \infty))dt.
$$
\end{minipage}}
\end{center}
\vip

The idea we have in mind is that we slow down matches:
for $(N^M_t)_{t\geq 0}$ a $SR(\mu^1_M)$-process, $(N_{\la t}^M)_{t\geq 0}$ is a 
$SR(\mu_M^\la)$-process. 

\vip

Assume that $\int_0^\infty t \mu_M^1(dt)=\kappa \in (0,\infty)$. Then
$\tilde \mu_M^{\la}=\mu_M^{\kappa \la}$ satisfies $(H_M)$. We thus
may of course assume that $\kappa=1$ without loss of generality.

\vip

Next, we put some conditions about $\mu_S$. 

\vip
\begin{center}
\fbox{\begin{minipage}[c]{0.9\textwidth}
$(H_S)$: The probability measure $\mu_S$ on $(0,\infty)$
has a finite mean $m_S=\int_0^\infty t \mu_S(dt)$. We set
$$
\nu_S(dt)=\nu_{\mu_S}(dt)= m_S^{-1} \mu_S((t,\infty))dt.
$$
Either $\mu_S$ has a bounded 
support or $\mu_S$ has an unbounded support and
$$
\forall \; t>0, \quad
\lim_{x\to \infty}
\frac{\nu_S((x,\infty))}{\nu_S((tx,\infty))} \in [0,\infty)\cup\{\infty\}
\quad \hbox{exists}.
$$
\end{minipage}}
\end{center}
\vip

Surprisingly, we will consider these assumptions in full generality: 
no supplementary technical condition is needed.
In the whole paper, we admit the following convention:
$$
t^\infty= \left\{ \begin{array}{lll}
0 & \hbox{if} & t\in (0,1)\\
1 & \hbox{if} & t=1\\
\infty & \hbox{if} & t\in (1,\infty).
\end{array}\right.
$$

As proved in Lemma \ref{hsimplieshsbeta}, $(H_S)$ 
implies either 

\vip
\begin{center}
\fbox{\begin{minipage}[c]{0.9\textwidth}
$(H_S(BS))$: The probability measure $\mu_S$ on $(0,\infty)$
has a bounded support. We denote by $m_S$ the expectation of 
$\mu_S$ and define $T_S= \max \; \supp  \mu_S$ and
$\nu_S(dt)=m_S^{-1}\mu_S((t,\infty))dt$. 
Observe that $\supp  \nu_S = [0,T_S]$.
\end{minipage}}
\end{center}
\vip

\noindent or, for some $\beta \in [0,\infty) \cup \{\infty\}$,

\vip
\begin{center}
\fbox{\begin{minipage}[c]{0.9\textwidth}
$(H_S(\beta))$: The probability measure $\mu_S$ on $(0,\infty)$
has an unbounded support, a finite mean $m_S$ and for
$\nu_S(dt)=m_S^{-1}\mu_S((t,\infty))dt$,
$$ 
\forall \; t>0, \quad
\lim_{x\to \infty}\frac{\nu_S((x,\infty))}{\nu_S((tx,\infty))}=t^\beta.
$$
\end{minipage}}
\end{center}
\vip

We finally introduce the following notation.

\begin{nota}\label{phipsi}
(i) Assume $(H_S(\beta))$ for some $\beta \in [0,\infty)$. 
We denote by $\phi_S$ the 
inverse function of $t \mapsto t/\nu_S((t,\infty))$. Note that
$\phi_S:(0,\infty) \mapsto (0,\infty)$ is an increasing continuous 
bijection.

(ii) Assume $(H_S(\infty))$. We denote by $\phi_S$ the 
inverse function of $t \mapsto t/\nu_S((t,\infty))$ and by 
$\psi_S$ the inverse function of 
$t\mapsto \nu_S((0,t)) $. The functions 
$\phi_S:(0,\infty) \mapsto (0,\infty)$ and 
$\psi_S:(0,1) \mapsto (0,\infty)$ are increasing bijections.

(iii) Assume $(H_S(BS))$. We denote by $\psi_S$ the 
inverse function of $t \mapsto \nu_S((0,t))$.
The function $\psi_S:(0,1)\mapsto (0,T_S)$ is an increasing continuous
bijection.
\end{nota}

\subsection{Examples}
Concerning $(H_M)$,
the situation is clear. The Poisson case studied
in \cite{bf} corresponds to $\mu_M^1(dt)=e^{-t} \indiq_{\{t>0\}}dt$,
whence $\mu_M^\la(dt)=\nu_M^\la(dt)=\la e^{-\la t} \indiq_{\{t>0\}}dt$. We 
study here a much more general case. However, this is not the main
point of the paper, since it will not generate some very interesting
behaviors. Concerning $(H_S)$, we present here four classes of
examples, that will lead to different behaviors. 

\vip

{\bf Example 1.} If $\mu_S= \delta_{T_S}$, whence 
$\nu_S(dt)=T_S^{-1}\indiq_{[0,T_S]}(t)dt$, then $(H_S(BS))$ holds and
$\psi_S(z)=T_S z$.

\vip

{\bf Example 2.} Assume that $\mu_S((t,\infty))\stackrel \infty \sim
e^{-t^\alpha}$ for some $\alpha>0$, so that 
$\nu_S((t,\infty))\stackrel  \infty \sim c t^{1-\alpha} e^{-t^\alpha}$.
Then $(H_S(\infty))$ holds. Furthermore, 
$\phi_S(z) \stackrel \infty \sim (\log z)^{1/\alpha}$ and
$\psi_S(z) \stackrel 1 \sim [\log(1/(1-z))]^{1/\alpha}$.

\vip

{\bf Example 3.} Assume that $\mu_S((t,\infty)) 
\stackrel \infty \sim t^{-1-\beta}$ for some $\beta>0$,
whence $\nu_S((t,\infty))\stackrel \infty \sim c t^{-\beta}$.
Then $(H_S(\beta))$ holds and 
$\phi_S(z)\stackrel \infty \sim  (cz)^{1/(\beta+1)}$.

\vip

{\bf Example 4.} If $\mu_S((t,\infty))\stackrel \infty \sim  
t^{-1} (\log t)^{-1-\gamma}$ for some $\gamma>0$, then  
$\nu_S((t,\infty)) \stackrel \infty \sim  c (\log t)^{-\gamma}$, so that 
$(H_S(0))$ is satisfied and 
$\phi_S(z)\stackrel \infty \sim cz (\log z)^{-\gamma}$.

\vip

The Poisson
case treated in \cite{bf}, which corresponds to the case where
$\mu_S((t,\infty))=e^{-t}=\nu_S((t,\infty))$, is thus included in
Example 2. Example 1 might seem slightly strange from the modelling point 
of view, but it can happen e.g. if seeds are thrown by a machine.

\vip

Observe that $(H_S)$ is not very restrictive, since it is satisfied
by all reasonable laws.  Anyway, our results (not only the proofs) 
clearly break down without such an assumption.

\vip

It is not so easy to build a law $\mu_S$ not
meeting $(H_S)$, because the function $t\mapsto \nu_S((t,\infty))$
is automatically quite smooth (Lipschitz continuous, decreasing and convex). 
One can however
verify that $(H_S)$ is not holding for 
$\mu_S(dt)=\indiq_{\{t>0\}}
[20 -3\cos\log(1+t) + \sin \log(1+t)]/[9(1+t)^3)]dt$, 
for which $\nu_S((t,\infty))=[10+\sin\log(1+t)]/[10(1+t)]$. One 
easily checks that 
$\nu_S((x,\infty))/\nu_S((x e^{\pi/2},\infty))$ has no limit as $x\to \infty$,
choosing e.g. the sequences $x_n=e^{2n\pi}$ and $x_n=e^{2n\pi+\pi/2}$.

\subsection{Notation}\label{nocv}

In the whole paper, we denote, for $I\subset \zz$, by $|I|=\#I$
the number of elements in $I$. For $I=\lb a,b\rb=\{a,\dots,b\} \subset \zz$ and 
$\alpha>0$, we will set $\alpha I := [\alpha a, \alpha b]\subset \rr$.
For $\alpha>0$, we of course take the convention that 
$\alpha \emptyset=\emptyset$.

\vip

For $J=[a,b]$ an interval of $\rr$,
$|J|=b-a$ stands for the length of $J$ and for $\alpha>0$, we set
$\alpha J = [\alpha a, \alpha b]$. 

\vip

For $x\in \rr$, $\lfloor x \rfloor$ stands for the integer part of $x$.

\vip

We denote by $\cI=\{[a,b],a\leq b\}$ the
set of all closed finite intervals of $\rr$.
For two intervals $[a,b]$ and $[c,d]$, we set
$$
\bdelta([a,b],[c,d])=|a-c|+|b-d|, \quad 
\bdelta([a,b],\emptyset)=|b-a|.
$$

\vip

For two functions $I,J:[0,T]\mapsto \cI\cup\{\emptyset\}$,
we set 
$$
\bdelta_{T}(I,J)=\int_0^T \bdelta(I_t, J_t) dt.
$$

For $(x,I), (y,J)$ in $\dd([0,T],
\rr_+ \times \cI\cup\{\emptyset\})$, the set of c\`adl\`ag functions from
$[0,T]$ into $\rr_+\times \cI\cup\{\emptyset\}$,  we define
$$
\bd_T((x,I),(y,J))=\sup_{t\in [0,T]}|x(t)-y(t)|+\bdelta_{T}(I,J).
$$

\section{Heuristic scales and relevant quantities}\label{hscales}

For $\mu_S,\mu^\la_M$ satisfying $(H_S)$ and $(H_M)$, 
we consider the $FF(\mu_S,\mu^\la_M)$-process $(\eta_t^\la(i))_{t\geq 0,i\in\zz}$.
We look for some time scale for which tree clusters see about
one fire per unit of time. But for $\la$ very small, clusters will
be very large just before they burn. We thus also have to rescale space. 

\vip

{\it Time scale.} 
For $\la>0$ very small and 
for $t$ not too large, one might neglect fires,
so that roughly, each site is vacant with probability $\nu_S((t,\infty))$.
Indeed, the time we have to wait for the first seed follows,
on each site, the law $\nu_S$.
Thus $C(\eta_t^\la,0)\simeq \lb -X,Y\rb$, where $X,Y$ are geometric
random variables with parameter $\nu_S((t,\infty))$.
Consequently, for $t$ not too large,
$$
|C(\eta_t^\la,0)|\simeq 1/\nu_S((t,\infty)).
$$

Under $(H_S(BS))$, $|C(\eta_t^\la,0)|$ becomes infinite at time $T_S$,
so there is no really need to accelerate time: 
we are sure that $|C(\eta_t^\la,0)|$
will be involved in a fire before $T_S$. We will accelerate time by 
a factor $T_S$ (in some sense, this allows us to assume that
$T_S=1$).

\vip

Next we assume $(H_S(\beta))$ for some $\beta \in [0,\infty)\cup\{\infty\}$.
We observe that thanks to $(H_M)$, $\nu_M^\la((t,\infty))\simeq 
1 - \la \int_0^{t} \mu_M^1((\la s,\infty))ds \simeq 1-\la t$. Hence
the probability that at least one match falls in
the cluster $C(\eta^\la,0)$ during $[0,t]$ is roughly similar, under $(H_M)$, 
to  
$$
1- \left(\nu_M^\la((t,\infty)) \right)^{|C(\eta_t^\la,0)|}\simeq \la t 
|C(\eta_t^\la,0)|\simeq \la t /\nu_S((t,\infty)).
$$
We decide to accelerate time by a factor $\ba_\la$, where $\ba_\la$
solves $\la \ba_\la =  \nu_S((\ba_\la,\infty))$. By this way, 
the probability that a match falls in
$C(\eta^\la,0)$ during $[0,\ba_\la]$ should tend to some nontrivial value.

\vip

To summarize, we have set, recalling Notation \ref{phipsi},
\begin{align}\label{ala}
\left\{\begin{array}{l}
\hbox{under } (H_S(BS)), \;\;  \ba_\la = T_S,\\
\hbox{under } (H_S(\beta)) \hbox{ with }\beta \in [0,\infty)\cup\{\infty\}, 
\;\; 
\ba_\la = \phi_S(1/\la), \hbox{ which solves } 
\la \ba_\la =  \nu_S((\ba_\la,\infty)).
\end{array}\right.
\end{align} 

Under $(H_S(\beta))$ for some $\beta \in [0,\infty)\cup\{\infty\}$, 
one easily checks that
$$
\lim_{\la \to 0}\ba_\la = \infty \quad \hbox{and thus} \quad 
\lim_{\la\to 0} \la \ba_\la = \lim_{\la\to 0} \nu_S((\ba_\la,\infty))=0.
$$

\vip

{\it Space scale.}
Now we rescale space in such a way that during a time interval with length 
of order
$\ba_\la$, something like one
fire starts per unit of (space) length.
Since on each site, the probability that (at least) 
one match falls during $[0,\ba_\la]$
equals $\nu_M((0,\ba_\la))=\la \int_0^{\ba_\la}\mu_M^1((\la t,\infty))dt 
\simeq  \la \ba_\la$, 
our space scale has to be of order
\begin{align}\label{nla}
\bn_\la= \lfloor 1/ (\la \ba_\la) \rfloor. 
\end{align}
This means that we will identify
$\lb 0,\bn_\la\rb \subset\zz$ with $[0,1]\subset\rr$. We always have
$\lim_{\la\to 0} \bn_\la=\infty$.

\vip

{\it Rescaled clusters.} We thus set, for $\la \in (0,1)$, 
$t\geq 0$ and $x\in \rr$, 
recalling Subsection \ref{nocv},
\begin{align}\label{dlambda}
D^\la_t(x):= \frac{1}{\bn_\la} C(\eta^\la_{\ba_\la t},
\lfloor \bn_\la x \rfloor) \subset \rr.
\end{align}
By the previous study, we know that roughly, when neglecting fires,
$$
|D_t^\la(x)|\simeq \frac{1}{\bn_\la \nu_S((\ba_\la t,\infty))}\simeq
\frac{\la \ba_\la}{\nu_S((\ba_\la t,\infty))}.
$$
Under $(H_S(\beta))$ for some $\beta \in [0,\infty)\cup\{\infty\}$, one gets
$$
|D_t^\la(x)|\simeq
\frac{\nu_S((\ba_\la ,\infty))}{\nu_S((\ba_\la t,\infty))} \simeq t^\beta.
$$
Under $(H_S(BS))$, we obtain roughly (assume that $t\ne 1$)
$$
|D_t^\la(x)|\simeq t^\infty.
$$
Indeed, $\nu_S((\ba_\la t,\infty))=\nu_S((T_S t,\infty))$ does
not depend on $\la$ and is positive if and only if $t<1$.

\vip

{\it Case $\beta \in [0,\infty)$.} In this case, everything is fine:
for all times of order
$\ba_\la t$, the good space scale is indeed $\bn_\la$. Thus
we will describe the $FF(\mu_S,\mu_M^\la)$-process through
$(D^\la_t(x))_{x\in \rr,t\geq 0}$. 

\vip

{\it Case $\beta \in \{\infty,BS\}$.}
Then we have a difficulty as in \cite{bf}:
the previous estimate (neglecting fires) suggests that 
for all $x\in \rr$, for $t<1$, $|D_t^\la(x)|\to 0$ and
for $t>1$,  $|D_t^\la(x)|\to \infty$.
For $t >1$, fires might be in effect
and we hope that this will make finite the possible limit of  $|D^\la_t(x)|$.
But fires can only reduce the size of clusters, so that for $t<1$,
the limit of $|D^\la_t(x)|$ will really be $0$.

\vip

Since we would like to have an idea of the sizes of microscopic 
clusters, we have to keep some information about the
{\it degree of smallness} of microscopic clusters.
We adopt a different strategy than in \cite{bf}, which is more
adapted to the case where $\beta=BS$ and which leads us to a slightly
more direct proof (even in the Poisson case).
We consider a function
$\bm_\la:(0,1]\mapsto \nn$ satisfying
\begin{align}\label{mla}
\left\{\begin{array}{l}
\lim_{\la\to 0} \bm_\la = \infty, \;\;
\; \lim_{\la\to 0} (\bm_\la/\bn_\la) = 0, \;
\la \mapsto \bm_\la \hbox{ is non-increasing}\\
\hbox{and additionally, under $(H_S(\infty))$, } 
\forall z \in [0,1), \lim_{\la \to 0} \bm_\la \nu_S((\ba_\la z,\infty))=\infty.
\end{array}\right.
\end{align}

Such a function exists: under $(H_S(\infty))$, see Lemma \ref{mlaexist} 
and under $(H_S(BS))$, 
choose for example $\bm_\la = \lfloor \sqrt{1/\la}\rfloor$.

\vip

We introduce, for
$\la>0$, $x\in \rr$, $t>0$, recall Subsection \ref{nocv} and Notation
\ref{phipsi}, 
\begin{align}\label{zlambda}
\left\{ \begin{array}{rl}
K^\la_t(x)&:= 
\displaystyle\frac{ \left|\left\{i \in \lb \lfloor \bn_\la x \rfloor - \bm_\la 
, \lfloor \bn_\la x\rfloor + \bm_\la  \rb\,:\;   
\eta^\la_{\ba_\la t} (i)=1
\right\}\right|}{2 \bm_\la+1} \in [0,1],\\
Z^\la_t(x)&:= \displaystyle\frac{\psi_S (K^\la_t(x))}{\ba_\la} \land 1
\in [0,1]. \end{array}\right.
\end{align}
Observe that $K^\la_t(x)$ stands for the {\it local density of occupied sites}
around $ \lfloor \bn_\la x \rfloor $ at time $\ba_\la t$. This density is
{\it local} because $\bm_\la << \bn_\la$.
We hope that for $t<1$, neglecting fires, 
$K^\la_t(x)\simeq \nu_S((0,\ba_\la t))$, whence $Z^\la_t(x)\simeq t$.

\vip

For all $\la>0$ small enough (we need that $2\bm_\la+1<\bn_\la$), 
it also holds that  $Z^\la_t(x)=1$ if and only if  
$K^\la_t(x)=1$, i.e. if and only if  all the sites are 
occupied around $\lfloor \bn_\la x \rfloor$.
Indeed, under $(H_S(BS))$, $Z^\la_t(x)=1$ implies that
$\psi_S(K^\la_t(x))=T_S$, so that $K^\la_t(x)=\nu_S((0,T_S))=1$.
Under $(H_S(\infty))$, $Z^\la_t(x)=1$ implies that
$\psi_S(K^\la_t(x))\geq \ba_\la$, so that $K^\la_t(x) \geq \nu_S((0,\ba_\la))
=1-\nu_S((\ba_\la,\infty))=1- \la \ba_\la \geq 1-1/\bn_\la$, whence
$K^\la_t(x)=1$. This last assertion comes from the facts that
$K^\la_t(x)$ takes its values in 
$\{k/(2\bm_\la+1) : k\in \{0,\dots,2\bm_\la+1\}$ and that $2\bm_\la+1<\bn_\la$.

\vip

Since we will allow $\bm_\la$ to be arbitrarily close to $\bn_\la$,
$Z^\la_t(x)=1$ will imply, roughly, that the cluster 
containing $\lfloor \bn_\la x \rfloor$ is macroscopic, i.e. has a length
of order $\bn_\la$.

\vip
 
We will study the $FF(\mu_S,\mu_M^\la)$-process through
$(D^\la_t(x),Z_t^\la(x))_{x\in \rr,t\geq 0}$. 
The main idea is that for $\la>0$ very small:

\vip

$\bullet$ if $Z^\la_t(x)=z\in(0,1)$, 
then $|D^\la_t(x)|\simeq 0$ and the (rescaled)
cluster containing
$x$ is microscopic (in the sense that the non-rescaled cluster
is small when compared to $\bn_\la$), but we control the local
density of occupied sites around $x$, which resembles $\nu_S((0,\ba_\la z))$.
Observe that this density tends to $1$ as $\la \to 0$ 
for all $z\in (0,1)$ under $(H_S(\infty))$, while it remains bounded
as $\la \to 0$ for all $z\in (0,1)$ under $(H_S(BS))$.

\vip

$\bullet$ if $Z^\la_t(x)=1$ and $D_t^\la(x)=[a,b]$, then
the (rescaled) cluster containing $x$ is macroscopic and has a length
equal to $b-a$, or $|C(\eta^\la_{\ba_\la t},
\lfloor \bn_\la x \rfloor)|\simeq (b-a)\bn_\la$ in the original scales.

\subsection*{Summary}
Assume $(H_S(\beta))$ for some $\beta \in [0,\infty)\cup\{\infty,BS\}$. 

\vip

$\bullet$ We accelerate time by the factor $\ba_\la$,
defined by  $\la\ba_\la=\nu_S((\la \ba_\la,\infty))$ if 
$\beta \in [0,\infty)\cup\{\infty\}$ and by 
$\ba_\la=T_S$ if $\beta=BS$.

\vip

$\bullet$ Our space scale is $\bn_\la= \lfloor 1/(\la\ba_\la) \rfloor$.

\vip

$\bullet$ If $\beta \in [0,\infty)$, we will only study 
the rescaled clusters $(D^\la_t(x))_{t\geq 0,x\in\rr}$, see
(\ref{dlambda}). 

\vip

$\bullet$ If $\beta \in \{\infty,BS\}$, we will study the rescaled clusters
$(D^\la_t(x))_{t\geq 0,x\in\rr}$, as well as the local densities of occupied
sites $(Z^\la_t(x))_{t\geq 0,x\in\rr}$, see (\ref{mla}-\ref{zlambda}).

\section{Main result in the case $\beta=\infty$}\label{mri}
\setcounter{equation}{0}

\subsection{Definition of the limit process}

We describe the limit process in the case where
$\beta=\infty$.  It is exactly the same process as in the Poisson
case studied in \cite{bf}. 
We consider a Poisson measure $\pi_M(dt,dx)$ on $[0,\infty) \times \rr$, with 
intensity measure $dt dx$, whose marks correspond to matches.

\begin{defin}\label{dflffp}
A process 
$(Z_t(x),D_t(x),H_t(x))_{t\geq 0,x\in \rr}$ 
with values in $\rr_+\times \cI\times \rr_+$ such that a.s., 
for all $x\in\rr$, $(Z_t(x),H_t(x))_{t\geq 0}$ is c\`adl\`ag,
is said to be a $LFF(\infty)$-process
if a.s., for all $t\geq 0$, all $x \in \rr$,
\begin{align}\label{eqlffp}
\left\{ 
\begin{array}{l} 
Z_t(x)=  \displaystyle \intot \indiq_{\{Z_s (x) < 1\}}ds - 
\intot \int_\rr \indiq_{\{ Z_\sm(x)=1,y \in D_{\sm}(x)\}}\pi_M(ds,dy),\\
H_t(x)= \displaystyle 
\intot Z_\sm(x)\indiq_{\{Z_\sm(x)<1\}} \pi_M(ds\times \{x\}) 
- \intot  \indiq_{\{H_s (x) > 0 \}}ds, 
\end{array}
\right.
\end{align}
where $D_t(x) = [L_t(x),R_t(x)]$, with
\begin{align*}
L_t(x) =& \sup\{ y\leq x:\; Z_t(y)<1 \hbox{ or } H_t(y)>0 \},\\
R_t(x) =& \inf\{ y\geq x:\; Z_t(y)<1 \hbox{ or } H_t(y)>0 \} 
\end{align*}
and where $D_{t-}(x)$ is defined in the same way.
\end{defin}

\subsection{Formal dynamics}

Let us explain the dynamics of this process. We consider $T>0$ fixed
and set $\cA_T= \{x \in \rr:\; \pi_M([0,T]\times\{x\})> 0\}$.
For each $t\geq 0$, $x\in \rr$, $D_t(x)$ stands for the
occupied cluster containing $x$. We call this cluster is {\it microscopic} 
if $D_t(x)=\{x\}$. We have $D_t(x)=D_t(y)$ for all 
$y \in D_t(x)$.
\vip

{\it 1. Initial condition.}
We have $Z_0(x)=H_0(x)=0$
and $D_0(x)=\{x\}$ for all $x\in \rr$.

\vip

{\it 2. Occupation of vacant zones.}
We consider here $x\in \rr\setminus \cA_T$. Then we have 
$H_t(x)=0$ for all $t\in [0,T]$.
When $Z_t(x)<1$, then $D_t(x)=\{x\}$ 
and $Z_t(x)$ stands for the {\it local density of occupied sites}
around $x$ (or rather for a suitable function of this local density). Then 
$Z_t(x)$ grows linearly until it reaches $1$, as described by the first term
on the RHS of the first equation in (\ref{eqlffp}). When $Z_t(x)=1$, 
the cluster containing $x$ is macroscopic and is described
by $D_t(x)$.

\vip

{\it 3. Microscopic fires.} Here we assume that $x\in\cA_T$ and that the 
corresponding mark
of $\pi_M$ happens at some time $t$ where $Z_\tm(x)<1$. In such a case,
the cluster containing $x$ is microscopic. Then we set $H_t(x)=Z_\tm(x)$,
as described by the first term on the RHS of the second equation of
(\ref{eqlffp}) and we leave unchanged the value of $Z_t(x)$. 
We then let $H_t(x)$ decrease linearly until it reaches $0$,
see the second term on the RHS of the second equation in (\ref{eqlffp}).
At all times where $H_t(x)>0$, the site $x$ acts like a barrier (see Point 5.
below).

\vip

{\it 4. Macroscopic fires.} 
Here we assume that $x\in\cA_T$ and that the corresponding mark
of $\pi_M$ happens at some time $t$ where $Z_\tm(x)=1$. This means that
the cluster containing $x$ is macroscopic and thus this mark destroys
the whole component $D_\tm(x)$, that is for all $y\in D_\tm(x)$, we set 
$D_t(y)=\{y\}$, $Z_t(y)=0$. This is described by the second term on the RHS
of the first equation in (\ref{eqlffp}). 

\vip

{\it 5. Clusters.} Finally the definition of the clusters
$(D_t(x))_{x\in\rr}$ becomes more clear: these
clusters are delimited by zones with local density smaller than $1$
(i.e. $Z_t(y)<1$)
or by sites where a microscopic fire has (recently) started (i.e. $H_t(y)>0$).

\vip

For $A>0$, we call $(Z_t^A(x),D_t^A(x),H_t^A(x))_{t\geq 0,x \in [-A,A]}$
the finite box version of the $LFF(\infty)$-process: 
it has the same dynamics as the 
true $LFF(\infty)$-process, but we restrict the space of tree positions
to $x\in [-A,A]$. See Section \ref{pri} for a more precise definition.
On Figure \ref{figLFFinfty}, a typical path of this finite box 
$LFF(\infty)$-process is discussed.
See also Algorithm \ref{algo1} (with the function $F_S(z,v)=z$) below.

\begin{figure}[b] 
\fbox{
\begin{minipage}[c]{0.95\textwidth}
\centering
\includegraphics[width=13cm]{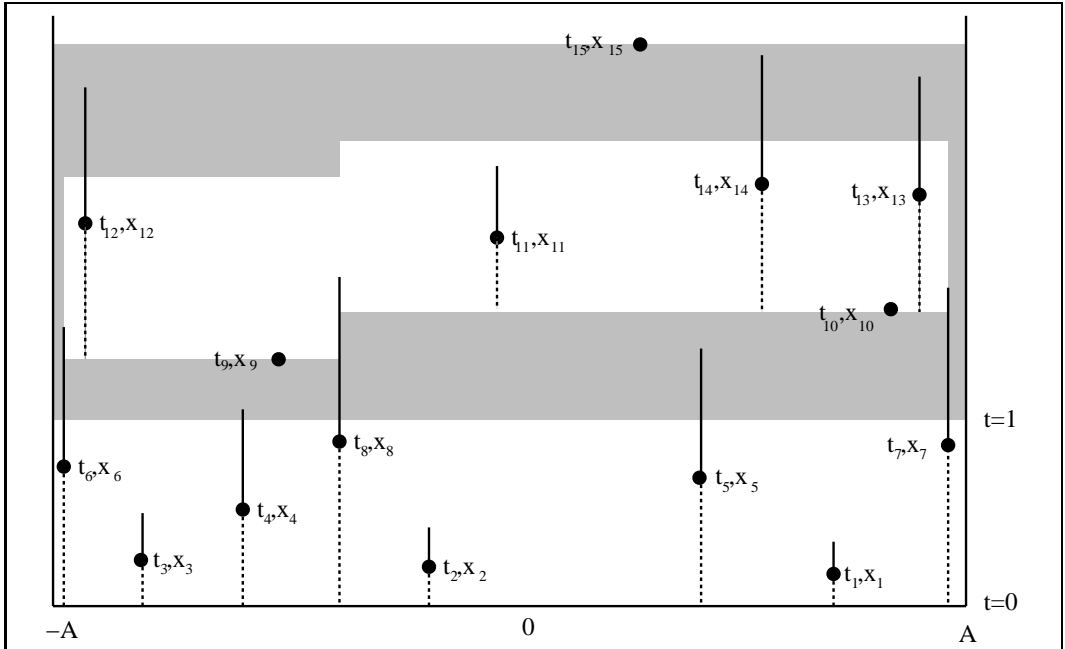}
\caption{$LFF(\infty)$-process in a finite box.}
\label{figLFFinfty}
\vip
\parbox{13.3cm}{
\footnotesize{
The marks of $\pi_M$ (matches) are represented as $\bullet$'s.
The filled zones represent zones in which $Z_t^A(x)=1$ and $H_t^A(x)=0$,
that is macroscopic clusters. The plain vertical segments represent
the sites where $H_t^A(x)>0$. In the rest of the space, we always have
$Z_t^A(x)<1$.
Until time $1$, all the particles are microscopic. The $8$ first
matches fall in that zone. Thus 
at each of these marks, a process $H^A$ starts and its life-time
equals the instant where it has started. For example
the segment above $(t_1,x_1)$ ends at time $2t_1$: we draw a dotted segment
from $(0,x_1)$ to $(t_1,x_1)$ and then a plain vertical segment above
$(t_1,x_1)$ with the same length.
At time $1$, all the clusters where
there has been no mark become macroscopic and merge together.
But this is limited by vertical segments. Here we have at time $1$ the
clusters $[-A,x_6]$, $[x_6,x_4]$, $[x_4,x_8]$, $[x_8,x_5]$, $[x_5,x_7]$ and
$[x_7,A]$. The segment above $(t_4,x_4)$ ends at time $2t_4$ and thus at this
time the clusters $[x_6,x_4]$ and $[x_4,x_8]$ merge into 
$[x_6,x_8]$.
The $9$-th mark
falls in the (macroscopic) zone $[x_6,x_8]$ and thus destroys it immediately.
This zone $[x_6,x_8]$ will become macroscopic again only at time
$t_9+1$. Then a process $H^A$ starts at $x_{12}$ at time $t_{12}$:
we draw a dotted segment from $(t_9,x_{12})$ to $(t_{12},x_{12})$ and
then  a plain vertical segment above $(t_{12},x_{12})$
with the same length ($Z^A_{t_{12}-}(x_{12})=t_{12}-t_9$ 
because $Z^A_{t_9}(x_{12})$ has been set to $0$).
The segment $[x_8,x_7]$ has been destroyed at time $t_{10}$
and thus will remain microscopic until $t_{10}+1$. As a consequence,
the only macroscopic clusters at time $t_9+1$ are $[-A,x_{12}]$,
$[x_{12},x_8]$ and  $[x_7,A]$. Then the zone $[x_8,x_7]$ becomes macroscopic
(but there have been marks at $x_{13}, x_{14}$), so that at time $t_{10}+1$, 
we get the macroscopic clusters $[-A,x_{12}]$, $[x_{12},x_{14}]$,  
$[x_{14},x_{13}]$ and $[x_{13},A]$. These clusters merge by pairs, at times
$2t_{12}-t_9$, 
$2t_{13}-t_{10}$ and $2t_{14}-t_{10}$, etc.

Here we have $0 \in (x_{11},x_{15})$ and thus 
$Z^A_t(0)= t$ for $t \in [0,1]$, $Z^A_t(0)=1$ for $t \in [1,t_{10})$, then
$Z^A_t(0)= t-t_{10}$ for $t \in [t_{10},t_{10}+1)$, then 
$Z^A_t(0)=1$ for $t \in [t_{10}+1,t_{15})$,\dots
We also see that $D^A_t(0)=\{0\}$ for $t \in [0,1)$, 
$D^A_t(0)=[x_8,x_5]$ for $t\in [1,2t_5)$, 
$D^A_t(0)=[x_8,x_7]$ for $t\in [2t_5,t_{10})$, 
$D^A_t(0)=\{0\}$ for $t\in [t_{10},t_{10}+1)$, $D^A_t(0)=[x_{12},x_{14}]$
for $t \in [t_{10}+1,2t_{12}-t_9)$, $D^A_t(0)=[-A,x_{14}]$
for $t \in [2t_{12}-t_9,2t_{14}-t_{10})$, ...
Of course, $H^A_t(0)=0$ for all $t\geq 0$, but for example $H^A_t(x_{11})=0$ for
$t\in[0, t_{11})$, $H^A_t(x_{11})=2t_{11}-t_{10}-t$ for 
$t\in [t_{11},2t_{11}-t_{10})$
and then $H^A_t(x_{11})=0$ for $t\in [2t_{11}-t_{10},\infty)$.
}}
\end{minipage}
}
\end{figure}

\subsection{Well-posedness}
The existence and 
uniqueness of the $LFF(\infty)$-process has been proved 
in \cite[Theorem 3]{bf}. We will provide here a simpler proof, which
also works for the case where $\beta=BS$.

\begin{theo}\label{wpinfty}
For any Poisson measure $\pi_M(dt,dx)$ 
on $[0,\infty)\times\rr$
with intensity measure $dt dx$, there a.s. exists a unique 
$LFF(\infty)$-process.
Furthermore, it can be constructed graphically
and its restriction to any finite
box $[0,T]\times[-n,n]$ can be perfectly simulated.
\end{theo}

The $LFF(\infty)$-process $(Z_t(x),D_t(x),H_t(x))_{t\geq 0,x\in \rr}$
is furthermore Markov, since it solves a well-posed time homogeneous
Poisson-driven S.D.E.

\subsection{The convergence result}
Recall Subsection \ref{nocv}.

\begin{theo}\label{converge1}
Assume $(H_M)$ and $(H_S(\infty))$. Recall that $\ba_\la$, $\bn_\la$
and $\bm_\la$ were defined in (\ref{ala})-(\ref{nla})-(\ref{mla}).
Consider, for each $\la\in (0,1]$, the process
$(Z^\la_t(x),D^\la_t(x))_{t\geq 0,x\in\rr}$ associated with
the $FF(\mu_S,\mu_M^\la)$-process,
see Definition \ref{gff}, (\ref{dlambda}) and (\ref{zlambda}). 
Consider also the $LFF(\infty)$-process
$(Z_t(x),D_t(x),H_t(x))_{t\geq 0,x\in \rr}$.

(a) For any $T>0$, any finite subset $\{x_1,\dots,x_p\}\subset \rr$, 
$(Z^\la_t(x_i),D^\la_t(x_i))_{t\in [0,T],i=1,\dots,p}$ goes in law to 
$(Z_t(x_i),D_t(x_i))_{t\in [0,T],i=1,\dots,p}$, in 
$\dd([0,T], \rr\times\cI\cup\{\emptyset\})^p$, 
as $\la$ tends to $0$. Here
$\dd([0,\infty), \rr\times \cI\cup\{\emptyset\})$ 
is endowed with the distance $\bd_T$.

(b) For any finite subset  $\{(t_1,x_1),\dots,(t_p,x_p)\}\subset 
[0,\infty)\times \rr$, with $t_k\ne 1$ for $k=1,\dots,p$,
$(Z^\la_{t_i}(x_i),D^\la_{t_i}(x_i))_{i=1,\dots,p}$ goes in law to 
$(Z_{t_i}(x_i),D_{t_i}(x_i))_{i=1,\dots,p}$ in $(\rr\times\cI\cup\{\emptyset\})^p$.
Here $\cI\cup\{\emptyset\}$ is endowed with $\bdelta$.

(c) Recall Notation \ref{phipsi}-(ii). For all $t>0$, 
$$
\left(\frac{\psi_S\left(1-1/|C(\eta^\la_{\ba_\la t},0)| \right)}{\ba_\la} 
\indiq_{\{|C(\eta^\la_{\ba_\la t},0)|\geq 1\}}\right) \land 1 
$$
goes in law to $Z_t(0)$ as $\la\to 0$.
\end{theo}

Point (c) will allow us to check some estimates on the cluster-size 
distribution.
Since we deal with finite-dimensional
marginals in space, it is quite clear that the process $H$ does not
appear in the limit, since for each $x\in \rr$, a.s., for all $t\geq 0$,
$H_t(x)=0$. (Of course, it is false that a.s., for all $x\in \rr$, 
all $t\geq 0$, $H_t(x)=0$).

\vip

We cannot guarantee the convergence
in law of $D_t^\la(0)$ to $D_t(0)$ at time $t=1$. 
This is due to the fact that for a zone $I_\la$ of length of order $\bn_\la$,
the probability that a seed falls on each site of $I_\la$ during $[0,\ba_\la t]$
tends to $0$ if $t<1$, to some nontrivial value if $t=1$,
and to $1$ if $t>1$.
We believe that this is really not important and we decided to keep
this definition of the $LFF(\infty)$-process despite this light defect.

\subsection{Heuristic arguments}

Let us explain here roughly the reasons why Theorem \ref{converge1}
holds true. We consider, for $\la>0$ very small, a
$FF(\mu_S,\mu_M^\la)$-process
$(\eta^\la_t(i))_{t\geq 0,i\in\zz}$ and the associated processes
$(Z^\la_t(x),D^\la_t(x))_{t\geq 0,x\in\rr}$.

\vip

{\it 0. Matches.} The times and positions at which matches fall
will tend, in our scales, to the marks of a Poisson measure with
intensity measure $1$. A hint for this is the following.
Consider e.g. the domain
$[0,T]\times [0,1]$, which corresponds to 
$[0,\ba_\la T] \times \lb 0,\bn_\la\rb$.
The probability that two matches fall on the same site during $[0,\ba_\la T]$
is very small. Thus the number of matches falling in 
$[0,\ba_\la T] \times \lb 0,\bn_\la\rb$ has approximately
a Binomial distribution with
parameters $\bn_\la$ and $\nu_M([0,\ba_\la T])$. Since
$$\bn_\la\nu_M([0,\ba_\la T])\simeq \frac{1}{\la \ba_\la}\left[\int_0^{\ba_\la T} 
\la \mu_M^1((\la \ba_\la T,\infty)) dt\right]
\to T$$ 
as $\la\to 0$, 
the asymptotic number of matches falling in $[0,T]\times [0,1]$ should
have a Poisson distribution with parameter $T$.

\vip

{\it 1. Initial condition.}
For all $x\in\rr$, 
$(Z^\la_0(x),D^\la_0(x))= (0,\emptyset) \simeq (0,\{x\})$
(recall that $\psi_S(0)=0$).

\vip

{\it 2. Occupation of vacant zones.} 
Assume that a zone $[a,b]$ becomes 
completely vacant at some time $t$ (because it has been destroyed
by a fire). 

\vip

(i) For $s\in [0,1)$ and if no fire starts on $[a,b]$ during 
$[t,t+s]$, we have $D^\la_{t+s}(x) \simeq [x \pm 1/(\bn_\la
\nu_S(\ba_\la s, \infty))]\simeq \{x\}$ 
and $Z^\la_{t+s}(x)\simeq s$ for all $x\in [a,b]$.

\vip

Indeed, $D^\la_{t+s}(x)\simeq [x-X/\bn_\la, x+Y/\bn_\la]$,
where $X$ and $Y$ are approximately geometric random variables with parameter
$\nu_S((\ba_\la s,\infty))$. (Recall that for any $t\geq 0$ and for any site,
$\nu_S$ is the law of the time we have to wait until the next seed 
falls). Thus $D^\la_{t+s}(x) \simeq [x \pm 1/(\bn_\la
\nu_S((\ba_\la s, \infty))]\simeq \{x\}$ due to $(H_S(\infty))$, since
$\nu_S((\ba_\la s, \infty)) >> \nu_S((\ba_\la,\infty))\simeq \bn_\la$.
For the same reasons, it holds that $K^\la_{t+s}(x) \simeq \nu_S((0,\ba_\la s))$,
whence $Z^\la_{t+s}(x) \simeq s$.

\vip

(ii) If no fire starts on $[a,b]$ during $[t,t+1]$, then 
$Z^\la_{t+1}(x)\simeq 1$ and all the sites in $[a,b]$ are occupied 
(with very high probability) just after time $t+1$. 

\vip

Indeed,
we have $(b-a)\bn_\la$ sites and each of them is occupied
at time $t+1+\e$ with approximate probability $\nu_S((0,\ba_\la(1+\e)])$,
so that all of them are occupied with approximate
probability $(\nu_S((0,\ba_\la(1+\e))))^{(b-a)\bn_\la}\simeq 
\exp(-(b-a)\nu_S((\ba_\la(1+\e),\infty ))/\nu_S((\ba_\la,\infty)))$,
which tends to 
$1$ as $\la\to 0$ for any $\e>0$ by $(H_S(\infty))$.

\vip

{\it 3. Microscopic fires.} 
Assume that a fire starts at some place $x$
at some time $t$, with $Z_{\tm}^\la(x)=z\in (0,1)$.
Then the possible clusters on the left and
right of $x$ cannot be connected during (approximately) $[t,t+z]$,
but can be connected after (approximately) $t+z$.

\vip

Indeed, the match falls in a zone with approximate density 
$\nu_S((\ba_\la z, \infty))$, so that it should destroy a zone 
$A$ of approximate length 
$1/\nu_S((\ba_\la z, \infty))<<\bn_\la$.
The probability that a fire starts again 
in $A$ after $t$ is very small.  Thus the probability that
$A$ is completely occupied at time $t+s$
is approximately equal to
$(\nu_S((0,\ba_\la s]))^{1/\nu_S((\ba_\la z,\infty))}
\simeq \exp\left(-  \nu_S((\ba_\la s,\infty))/\nu_S((\ba_\la z,\infty))\right)$.
When $\la\to 0$, this quantity 
tends to $0$ if $s<z$ and to $1$ if $s>z$ thanks to $(H_S(\infty))$.

\vip

{\it 4. Macroscopic fires.} 
Assume now that a fire starts at some place $x$,
at some time $t$ and that $Z^\la_{t-}(x) \simeq 1$, so that
$D^\la_{t-}(x)$ is macroscopic (that is its length is of order $1$
in our scales, or of order $\bn_\la$ in the original process). 
This will thus make vacant the zone $D^\la_{t-}(x)$. Such a 
(macroscopic) zone
needs a time of order $1$ to be completely occupied, see
Point 2.

\vip

{\it 5. Clusters.} For $t\geq 0$, $x\in \rr$, the cluster $D_t^\la(x)$
resembles $[x\pm 1/(\bn_\la\nu_S((\ba_\la z, \infty)))]\simeq \{x\}$ if 
$Z_t^\la(x)=z\in (0,1)$.
We then say that $x$ is microscopic.
Macroscopic clusters are delimited either by
microscopic zones, or by 
sites where there has been recently a microscopic fire.

\subsection{Cluster-size distribution}
We will deduce from Theorem \ref{converge1} 
the following estimates on the cluster-size distribution.

\begin{cor}\label{co1}
Assume $(H_M)$ and $(H_S(\infty))$. Recall that
$\ba_\la$ and $\bn_\la$ were defined in (\ref{ala}) and (\ref{nla}). Let
$(Z_t(x),D_t(x),H_t(x))_{t\geq 0, x\in \rr}$ be 
a $LFF(\infty)$-process.
For each $\la\in (0,1]$, let $(\eta^\la_t(i))_{t\geq 0, i \in \zz}$ be a 
$FF(\mu_S,\mu_M^\la)$-process.

(i) For some $0<c_1<c_2$,
for all $t\geq 5/2$, all $0< a < b < 1$,
\begin{align*}
&\lim_{\la\to 0}  
\Pr\left(|C(\eta^\la_{\ba_\la t},0)| \in 
[1/\nu_S((\ba_\la a,\infty)), 1/\nu_S((\ba_\la b,\infty))]\right) \\
=&\Pr\left(Z_t(0)\in [a,b] \right)
\in [c_1(b-a),c_2 (b-a)].
\end{align*}

(ii) For some $0<c_1<c_2$ and $0< \kappa_1 <\kappa_2$, for all
$t \geq 3/2$, all $B>0$,
$$
\lim_{\la\to 0}  
\Pr\left(|C(\eta^\la_{\ba_\la t},0)|
\geq B \bn_\la  \right)=\Pr\left(|D_t(0)|\geq B \right)
\in [c_1 e^{-\kappa_2 B}, c_2 e^{-\kappa_1 B}].
$$
\end{cor}

This results shows that there is a {\it phase transition} around the 
{\it critical size} $\bn_\la$: the cluster-size distribution changes 
of shape at $\bn_\la$. 

\vip

Consider the case of Example 2, where 
$\mu_S((t,\infty)) \stackrel \infty \sim e^{-t^\alpha}$. 
Then 
$\ba_\la \sim  (\log (1/\la))^{1/\alpha}$
and $\bn_\la \sim 1/ [\la (\log (1/\la))^{1/\alpha} ]$. 
Very roughly, Corollary \ref{co1} proves that when $\la\to 0$, the law of 
$|C(\eta^\la,0)|$, for large times, resembles
$$
\frac{[\log(1+x)]^{1/\alpha-1}}{(1+x)[\log(1/\la)]^{1/\alpha}}
\indiq_{\{x\in [0,\bn_\la]\}}dx
+ (1/\bn_\la) e^{- x/\bn_\la } \indiq_{\{x\geq 0\}} dx.
$$
The first term corresponds approximately 
to the law of $1/\nu_S((\ba_\la U,\infty))$,
for $U$ uniformly distributed on $[0,1]$ and the second term is
an exponential law with mean $\bn_\la$.

\vip

The main idea is that two types of clusters are present: macroscopic clusters,
of which the size is of order $\bn_\la\sim \la^{-1}[\log(1/\la)]^{-1/\alpha}$, 
with an exponential-like distribution;
and microscopic clusters, of which the size is smaller than $\bn_\la$,
with a law with shape $\log(1+x)^{1/\alpha-1}/(1+x)$.

\section{Main result in the case $\beta=BS$}\label{mrbs}
\setcounter{equation}{0}

This case is slightly more complicated than the case $\beta=\infty$.
The limit process is essentially the same, except that 
the height of the barriers (vertical segments in Figure \ref{figLFFinfty}) 
are more random.

\subsection{Law of the heights of the barriers}
Start at time $0$ with all sites vacant. Let $u \in (0,1)$. 
Assume that a match falls at site $0$ at time $T_S u$ 
and neglect all other fires. Call $\Theta_u$ the time needed for the
destroyed zone to be completely regenerated and $\theta_u$ the law
of $\Theta_u / T_S$. Clearly, $\theta_u$ is supported by $[0,1]$. 
We will show in Lemma \ref{deftheta} 
below that $\theta_u$ can be defined as follows.

\begin{defin}\label{defF}
Assume $(H_S(BS))$.
For $t,s \in [0,\infty)$, we denote by 
$$
g_S(t,s)=\Pr[N^S_{T_S t}>0,N^S_{T_S(t+s)}>N^S_{T_S t}],
$$
where $(N^S_t)_{t\geq 0}$ is a $SR(\mu_S)$-process. For $u \in (0,1)$,
we consider the probability measure $\theta_u$ on $[0,1]$ defined
by
$$
\forall \; h\in [0,1], \quad \theta_u([0,h])=\nu_S((T_S u, T_S))
+ \left(\frac{\nu_S((T_S u, T_S))}{1-g_S(u,h)}\right)^2
g_S(u, h).
$$
Finally, we consider a function $F_S:[0,1]\times[0,1]\mapsto [0,1]$
such that for each $u\in [0,1]$ and for $V$ a uniformly distributed
random variable on $[0,1]$, the law of $F_S(u,V)$ is $\theta_u$.
We can choose $F_S$ in such a way that for each $u\in[0,1]$,
$v\mapsto F_S(u,v)$ is nondecreasing.
\end{defin}

Let $u\in [0,1]$ be fixed. Since $\mu_S([0,T_S])=1$, there holds 
$g_S(u, 1)=\nu_S([0,T_Su])$, whence $\theta_u([0,1])=1$.
To check that $h\mapsto \theta_u([0,h])$ is nondecreasing, it suffices
to observe that $h\mapsto g(u,h)$ is nondecreasing.
Notice that $\theta_u(\{0\})=\nu_S((T_S u, T_S))$: this corresponds to the
situation where nothing has been destroyed because the match has fallen 
on an empty site. For $F_S(u,.)$, one can e.g. use
the generalized inverse function of $\theta_u([0,.])$.

\subsection{Definition of the limit process}

Let $\pi_M(dt,dx)$ be a Poisson measure
on $[0,\infty)\times\rr$
with intensity measure $dt dx$, whose marks correspond to matches.
We also consider an i.i.d. sequence $(V_k)_{k \geq 1}$ of uniformly
distributed random variables on $[0,1]$, independent
of $\pi_M$. If $\pi_M(dt,dx)=\sum_{k\geq 1}\delta_{(T_k,X_k)}$, we (abusively)
write $\pi_M(dt,dx,dv)=\sum_{k\geq 1}\delta_{(T_k,X_k,V_k)}$. 
Observe that $\pi_M(dt,dx,dv)$ is a Poisson measure 
on $[0,\infty)\times\rr \times [0,1]$ with intensity measure $dt dx dv$.

\begin{defin}\label{dflffpbs}
A process 
$(Z_t(x),D_t(x),H_t(x))_{t\geq 0,x\in \rr}$ 
with values in $\rr_+\times \cI\times \rr_+$ such that a.s., 
for all $x\in\rr$, $(Z_t(x),H_t(x))_{t\geq 0}$ is c\`adl\`ag,
is said to be a $LFF(BS)$-process
if a.s., for all $t\geq 0$, all $x \in \rr$,
\begin{align}\label{eqlffpbs}
\left\{ 
\begin{array}{l} 
Z_t(x)=  \displaystyle \intot \indiq_{\{Z_s (x) < 1\}}ds - 
\intot \int_\rr \indiq_{\{ Z_\sm(x)=1,y \in D_{\sm}(x)\}}\pi_M(ds,dy),\\
H_t(x)= \displaystyle 
\intot \int_0^1 F_S(Z_\sm(x),v)\indiq_{\{Z_\sm(x)<1\}} 
\pi_M(ds\times \{x\}\times dv) 
- \intot  \indiq_{\{H_s (x) > 0 \}}ds, 
\end{array}
\right.
\end{align}
where $D_t(x) = [L_t(x),R_t(x)]$, with
\begin{align*}
L_t(x) =& \sup\{ y\leq x:\; Z_t(y)<1 \hbox{ or } H_t(y)>0 \},\\
R_t(x) =& \inf\{ y\geq x:\; Z_t(y)<1 \hbox{ or } H_t(y)>0 \} 
\end{align*}
and where $D_{t-}(x)$ is defined in the same way.
\end{defin}

The difference with the $LFF(\infty)$-process is that when a match falls
at $(t,x)$ with $Z_{t-}(x)<1$, we choose $H_t(x)$
according to the law $\theta_{Z_{t-}(x)}$, instead
of simply setting $H_t(x)=Z_{t-}(x)$.

\subsection{Formal dynamics}

Let us explain the dynamics of this process. We consider $T>0$ fixed
and set $\cA_T= \{x \in \rr:\; \pi_M([0,T]\times\{x\})> 0\}$.
For each $t\geq 0$, $x\in \rr$, $D_t(x)$ stands for the
occupied cluster containing $x$. We call this cluster is {\it microscopic} 
if $D_t(x)=\{x\}$. We have $D_t(x)=D_t(y)$ for all 
$y \in D_t(x)$.
\vip

{\it 1. Initial condition.}
We have $Z_0(x)=H_0(x)=0$
and $D_0(x)=\{x\}$ for all $x\in \rr$.

\vip

{\it 2. Occupation of vacant zones.}
We consider here $x\in \rr\setminus \cA_T$. Then we have 
$H_t(x)=0$ for all $t\in [0,T]$.
When $Z_t(x)<1$, then $D_t(x)=\{x\}$ 
and $Z_t(x)$ stands for the local density of occupied sites around $x$
(or rather for a suitable function of this density)
Then 
$Z_t(x)$ grows linearly until it reaches $1$, as described by the first term
on the RHS of the first equation in (\ref{eqlffpbs}). When $Z_t(x)=1$,
the cluster containing $x$ is macroscopic and is described
by $D_t(x)$.

\vip

{\it 3. Microscopic fires.} Here we assume that $x\in\cA_T$ and that the 
corresponding mark
of $\pi_M$ happens at some time $t$ where $Z_\tm(x)<1$. In such a case,
the cluster containing $x$ is microscopic. Then we set $H_t(x)=F_S(Z_\tm(x),V)$,
for some uniformly distributed $V$ on $[0,1]$
as described by the first term on the RHS of the second equation of
(\ref{eqlffpbs}).
We then let $H_t(x)$ decrease linearly until it reaches $0$,
see the second term on the RHS of the second equation in (\ref{eqlffpbs}).
At all times where $H_s(x)>0$, the site $x$ acts like a barrier (see Point 5.
below). All this means that at $x$, there is a barrier during
$[t,t+H_t(x))$, where $H_t(x)$ is chosen at random, according to the
law $\theta_{Z_{t-}(x)}$.

\vip

{\it 4. Macroscopic fires.} 
Here we assume that $x\in\cA_T$ and that the corresponding mark
of $\pi_M$ happens at some time $t$ where $Z_\tm(x)=1$. This means that
the cluster containing $x$ is macroscopic and thus this mark destroys
the whole component $D_\tm(x)$, that is for all $y\in D_\tm(x)$, we set 
$D_t(y)=\{y\}$, $Z_t(y)=0$. This is described by the second term on the RHS
of the first equation in (\ref{eqlffpbs}). 

\vip

{\it 5. Clusters.} Finally the clusters
$(D_t(x))_{x\in\rr}$ are delimited by zones with density smaller than $1$
(i.e. $Z_t(y)<1$)
or by sites where a microscopic fire has (recently) started (i.e. $H_t(y)>0$).

\vip

A typical path of a finite-box version
$(Z_t^A(x),D_t^A(x),H_t^A(x))_{t\geq 0,x\in [-A,A]}$ 
of the $LFF(BS)$-process is discussed
on Figure \ref{figLFFBS}. It is very similar to Figure \ref{figLFFinfty}:
the only difference is that each time there is a bullet
falling at some $(t,x)$ in a white zone, the height of the segment
above $(t,x)$ is chosen at random, according to the law $\theta_{Z_{t-}(x)}$.
And $Z_{t-}(x)$ equals the time passed since $x$ was involved in a macroscopic
fire (the case $LFF(\infty)$ corresponds to the law $\theta_z=\delta_z$).
See also Algorithm \ref{algo1} below.

\begin{figure}[b] 
\fbox{
\begin{minipage}[c]{0.95\textwidth}
\centering
\includegraphics[width=13cm]{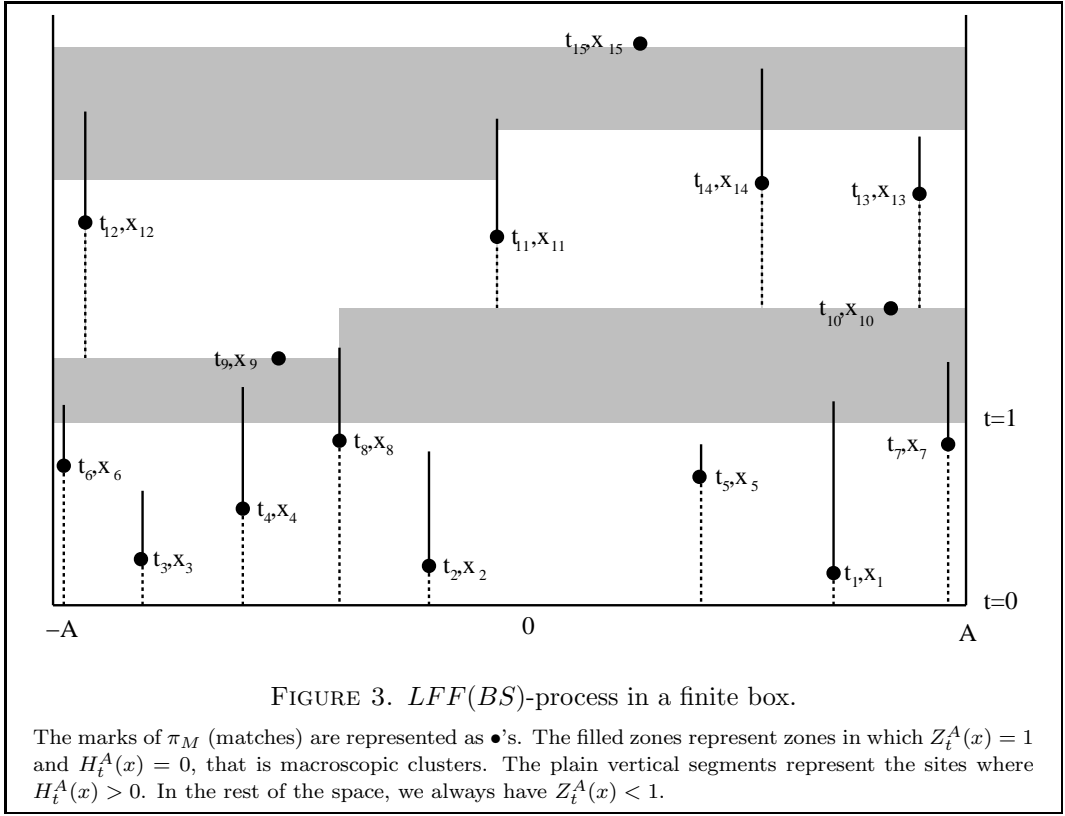}
\caption{$LFF(BS)$-process in a finite box.}
\label{figLFFBS}
\vip
\parbox{13.3cm}{
\footnotesize{
The marks of $\pi_M$ (matches) are represented as $\bullet$'s.
The filled zones represent zones in which $Z_t^A(x)=1$ and $H_t^A(x)=0$,
that is macroscopic clusters. The plain vertical segments represent
the sites where $H_t^A(x)>0$. In the rest of the space, we always have
$Z_t^A(x)<1$.
}}
\end{minipage}
}
\end{figure}

\subsection{Well-posedness}
We will prove the following result.

\begin{theo}\label{wpbs}
For any Poisson measure $\pi_M(dt,dx,dv)$  on $[0,\infty)\times\rr \times [0,1]$
with intensity measure $dt dx dv$ (and for $\pi_M(dt,dx)=\int_{v\in [0,1]}
\pi_M(dt,dx,dv)$), there a.s. exists a unique 
$LFF(BS)$-process.
Furthermore, it can be constructed graphically
and its restriction to any finite
box $[0,T]\times[-n,n]$ can be perfectly simulated.
\end{theo}

The $LFF(BS)$-process $(Z_t(x),D_t(x),H_t(x))_{t\geq 0,x\in \rr}$
is furthermore Markov, since it solves a well-posed time homogeneous
Poisson-driven S.D.E.

\subsection{The convergence result}
We are now in a position to state the main result of this section.
Recall Subsection \ref{nocv}.

\begin{theo}\label{convergebs}
Assume $(H_M)$ and $(H_S(BS))$. Recall that $\ba_\la=T_S$, 
$\bn_\la=\lfloor 1/(\la T_S) \rfloor$
and let $\bm_\la$ satisfy (\ref{mla}).
Consider, for each $\la\in (0,1]$, the process
$(D^\la_t(x),Z^\la_t(x))_{t\geq 0,x\in\rr}$ associated with
the $FF(\mu_S,\mu_M^\la)$-process $(\eta^\la_t(i))_{t\geq 0, i\in \zz}$,
see Definition \ref{gff}, (\ref{dlambda}) and (\ref{zlambda}). 
Consider also the $LFF(BS)$-process
$(Z_t(x),D_t(x),H_t(x))_{t\geq 0,x\in \rr}$.

(a) For any $T>0$, any finite subset $\{x_1,\dots,x_p\}\subset \rr$, 
$(Z^\la_t(x_i),D^\la_t(x_i))_{t\in [0,T],i=1,\dots,p}$ goes in law to 
$(Z_t(x_i),D_t(x_i))_{t\in [0,T],i=1,\dots,p}$, in 
$\dd([0,T], \rr\times\cI\cup\{\emptyset\})^p$, 
as $\la$ tends to $0$. Here
$\dd([0,\infty), \rr\times \cI\cup\{\emptyset\})$ 
is endowed with the distance $\bd_T$.

(b) For any finite subset  $\{(t_1,x_1),\dots,(t_p,x_p)\}\subset 
[0,\infty)\times \rr$, 
$(Z^\la_{t_i}(x_i),D^\la_{t_i}(x_i))_{i=1,\dots,p}$ goes in law to 
$(Z_{t_i}(x_i),D_{t_i}(x_i))_{i=1,\dots,p}$ in $(\rr\times\cI\cup\{\emptyset\})^p$.
Here $\cI\cup\{\emptyset\}$ is endowed with $\bdelta$.

(c) For any $t\geq 0$, any $k \in \nn$, there holds
$$
\lim_{\la\to 0} \Pr\left[|C(\eta^\la_{T_S t},0)|=k \right] = 
\E \left[q_k(Z_t(0)) \right],
$$
where, for $z\in [0,1]$, 
\begin{align}\label{defqk}
\left\{
\begin{array}{lcr}
q_0(z)=\nu_S((zT_S,T_S)),&&\\
q_k(z)=k [\nu_S((zT_S,T_S))]^2 [\nu_S((0,zT_S))]^k &\hbox{ if }& k\geq 1.
\end{array}
\right.
\end{align}
\end{theo}

Here we have no problem with $t=1$: for the discrete process
(in the absence of fires), all the sites are occupied at time $T_S$ (which 
corresponds to time $1$ after normalization). Point (c) will be useful
to prove some estimates about the cluster-size distribution.
Observe that for $z\in (0,1)$, $q_k(z)$ is 
the probability that the cluster around $0$
has the size $k$ at time $T_S z$ in the absence of fires, 
if seeds fall according to i.i.d. $SR(\mu_S)$-processes.

\subsection{Heuristic arguments}

Let us explain roughly the reasons why Theorem \ref{convergebs}
holds true.
We consider a
$FF(\mu_S,\mu_M^\la)$-process
$(\eta^\la_t(i))_{t\geq 0, i \in \zz}$ and the corresponding  processes
$(Z^\la_t(x),D^\la_t(x))_{t\geq 0,x\in\rr}$.
We assume below that $\la$ is very small.

\vip

{\it 0. Matches.} As in the case $\beta=\infty$, 
the times and positions at which matches fall
will tend, in our scales, to the marks of a Poisson measure with
intensity measure $1$.

\vip

{\it 1. Initial condition.}
We have, for all $x\in\rr$, 
$(Z^\la_0(x),D^\la_0(x))= (0,\emptyset) \simeq (0,\{x\})$.

\vip

{\it 2. Occupation of vacant zones.} 
Assume that a zone $[a,b]$ becomes 
completely vacant at some time $t$ (because it has been destroyed
by a fire). 

\vip

(i) For $s\in [0,1)$ and if no fire starts on $[a,b]$ during $[t,t+s]$
(or $[T_S t, T_S(t+s)]$ in the original scales)
the density of vacant sites in $[a,b]$ at time $t+s$ should clearly
resemble $\nu_S((0,T_S s))$. Hence for $x\in [a,b]$, $Z^\la_t(x)\simeq
\psi_S(\nu_S((0,T_S s)))=s$ and $D^\la_{t+s}(x) \simeq \{x\}$.

\vip

(ii) If no fire starts on $[a,b]$ during $[t,t+1]$ (or $[T_S t, T_S(t+1)]$ 
in the original scales), then all the sites of $[a,b]$ become occupied
at time $t+1$ (recall that $\nu_S((0,T_S])=1$).

\vip

{\it 3. Microscopic fires.} 
Assume that a fire starts at some place $x$
at some time $t$, with $Z_{\tm}^\la(x)=z\in (0,1)$.
Then the possible clusters on the left and
right of $x$ cannot be connected during (approximately) $[t,t+\Theta_z T_S]$,
but can be connected after (approximately) $t+\Theta_z T_S$, where $\Theta_z$ 
follows
approximately the law $\theta_z$. Indeed, $\theta_z$ is designed for that:
consider a zone where the density of occupied sites is $z$
and assume that the sites are exchangeable in this zone.
Pick at random a cluster in this zone. The law of its size depends on $z$.
Then $\theta_z$ is the law of the time needed for a seed to fall on
each sites of this cluster (divided by $T_S$).

\vip

{\it 4. Macroscopic fires.} 
Assume now that a fire starts at some place $x$,
at some time $t$ and that $Z^\la_{t-}(x) \simeq 1$, so that
$D^\la_{t-}(x)$ is macroscopic (that is its length is of order $1$
in our scales, or of order $\bn_\la$ in the original process). 
This will thus make vacant the zone $D^\la_{t-}(x)$. Such a 
(macroscopic) zone
needs a time of order $1$ to be completely occupied, see
Point 2.

\vip

{\it 5. Clusters.} For $t\geq 0$, $x\in \rr$, there are some
vacant sites in the neighborhood of $x$ if $Z^\la_t(x)<1$ (then we say that
$x$ is microscopic),
or if there has been (recently) a microscopic fire at $x$ (see Point 3).
Now macroscopic clusters are delimited either by
microscopic zones, or by 
sites where there has been recently a microscopic fire.

\subsection{Cluster-size distribution}
We will deduce from Theorem \ref{convergebs} 
the following estimates on the cluster-size distribution.

\begin{cor}\label{cobs}
Assume $(H_M)$ and $(H_S(BS))$. Recall that
$\ba_\la$ and $\bn_\la$ were defined in (\ref{ala}) and (\ref{nla}). Let
$(Z_t(x),D_t(x),H_t(x))_{t\geq 0, x\in \rr}$ be 
a $LFF(BS)$-process.
For each $\la\in (0,1]$, let $(\eta^\la_t(i))_{t\geq 0,i\in\zz}$ be a 
$FF(\mu_S,\mu_M^\la)$-process.

(i) For some $0<c_1<c_2$,
for all $t\geq 5/2$, all $k \in \{0,1,\dots\}$,
$$
\lim_{\la\to 0}  
\Pr\left(|C(\eta^\la_{T_S t},0)|=k\right) \in [c_1 q_k,c_2 q_k],
$$
where $q_0=\int_0^1 \nu_S((T_S z, T_S))dz$ and 
$q_k=k\int_0^1 [\nu_S((T_S z, T_S))]^2[\nu_S((0,T_S z))]^k dz$ for $k\geq 1$.

(ii) For some $0<c_1<c_2$ and $0< \kappa_1 <\kappa_2$, for all
$t \geq 3/2$, all $B>0$,
$$
\lim_{\la\to 0}  
\Pr\left(|C(\eta^\la_{T_S t},0)|
\geq B \bn_\la  \right)=\Pr\left(|D_t(0)|\geq B \right)
\in [c_1 e^{-\kappa_2 B}, c_2 e^{-\kappa_1 B}].
$$
\end{cor}

Consider the case of Example 1, where $\mu_S=\delta_1$, $T_S=1$ 
and $\nu_S(dt)=\indiq_{[0,1]}(t)dt$.
Then $\bn_\la  \sim 1/\la$ and one can check that
$q_0= 1/2$ and $q_k=2k/[(k+1)(k+2)(k+3)]$ for $k\geq 1$.

Corollary \ref{cobs} shows the presence of two regimes: for $\la>0$
very small, there are some finite (uniformly in $\la$)
clusters, as described in Point (i) and some clusters of order $1/\la$,
as described in Point (ii). Roughly, for $\la>0$ very small,
the cluster-size distribution resembles, for large times,
$$
\sum_{k\geq 0} q_k \delta_k(dx) + \la e^{- \la x}\indiq_{\{x\geq 0\}} dx.
$$

\section{Main results when $\beta\in (0,\infty)$}\label{mrbeta}
\setcounter{equation}{0}

\subsection{Definition of the limit process} \label{dflffbeta}
Surprisingly, the limit process in this case is more natural 
than in the previous cases, in the sense 
that there are only macroscopic clusters and thus no microscopic fires: 
heavy tails can sometimes produce natural objects. This is due to the fact that
for $\beta < \infty$, the scale space $\bn_\la$ is correct for all times.
We describe the limit forest fire process 
by a {\it graphical construction}.
The limit
forest fire process $(Y_t(x))_{x\in \rr, t\geq 1}$ will take its
values in $\{0,1\}$. In some sense, $Y_t(x)=0$ means that 
there is no tree at $x$ at time $t$. 

\vip

For $(Y(x))_{x\in \rr}$ with values in $\{0,1\}$, we define
the occupied component around $x\in \rr$ as
\begin{align}\label{defcyx}
C(Y,x):=[l(Y,x),r(Y,x)]
\end{align}
where $l(Y,x)=\sup\{y\leq x:\;Y(y)=0\}$ and $r(Y,x)=\inf\{y\geq x:\;Y(y)=0\}$.
If $Y(x)=0$, this implies $C(Y,x)=\{x\}$.

\vip

We consider a Poisson measure $\pi_M(dt,dx)$
on $[0,\infty)\times\rr$
with intensity measure $dt dx$, whose marks correspond to matches.
We also introduce a Poisson measure $\pi_S(dt,dx,dl)$ on 
$[0,\infty)\times\rr\times[0,\infty)$, independent of $\pi_M$, with intensity
measure $dt dx \beta (\beta+1) l^{-\beta-2} dl$. 
Roughly, when $\pi_S$ has mark $(\tau,X,L)$, this means that no seed fall on
$X$ during $[\tau-L,\tau]$. In all the other zones, seeds fall 
{\it continuously}.

\vip

We now handle the construction on a fixed time interval $[0,T]$.

\vip

First, we set $Y_t^0(x)=\indiq_{\{\pi_S(\{(s,x,l)\ : \; s>t,s-l<0\})=0\}}$ 
for all $t\in [0,T]$, all $x\in \rr$. Observe that for all $x\in \rr$, 
$t\mapsto Y_t^0(x)$ is non-decreasing on $[0,T]$.
Since $\int_0^\infty\int_0^\infty \indiq_{\{s>T,s-l<0\}} \beta(\beta+1)l^{-\beta-2}
dl ds >0$, one can clearly find an unbounded family $\{\chi_i\}_{i\in \zz}
\subset \rr$ such that for all $t\in [0,T]$,  all $i\in \zz$, 
$Y_t^0(\chi_i)=0$. 
We take the convention that for all $i\in \zz$, $\chi_i \leq \chi_{i+1}$,
$\chi_0 \leq 0 <\chi_1$, 
$\lim_{-\infty} \chi_i=-\infty$
and $\lim_{\infty} \chi_i=\infty$. 

\vip

We now handle the construction on each box $[0,T]\times[\chi_i,\chi_{i+1}]$ 
separately.
Let thus $i$ be fixed. The Poisson measure $\pi_M$ has a.s. a finite
number $n_i$ of marks 
$(\rho_{1}^{i}, \alpha_{1}^{i}),\dots,(\rho_{n_i}^{i}, \alpha_{n_i}^{i})$ 
in $[0,T]\times[\chi_i,\chi_{i+1}]$,
ordered in such a way that $0<\rho_{1}^{i}<\dots< \rho_{n_i}^{i}$.

\vip

We consider the occupied cluster $I^i_1=C(Y^0_{\rho_{1}^i-}, \alpha_{1}^{i})$ 
(which is included in $[\chi_i,\chi_{i+1}]$ by construction). 
For $(t,x)\in [0,T]\times[\chi_i,\chi_{i+1}]$, we set
$Y_t^1(x)=\indiq_{\{\pi_S(\{(s,x,l)\, : \; s>t,s-l<\rho_1^i\})=0\}}$
if $(t,x)\in [\rho^i_1,T]\times I^i_1$ and $Y^1_t(x)=Y^0_t(x)$ else. 

\vip

Assume that for some $k=2,\dots,n_i$,  
$(Y^{k-1}_t(x))_{t\in[0,T],x\in[\chi_i,\chi_{i+1}]}$ 
has been built and consider the occupied cluster
$I^i_{k}=C(Y^{k-1}_{\rho_{k}^i-}, \alpha_{k}^{i})$ (which is still included in
$[\chi_i,\chi_{i+1}]$).
For $(t,x)\in [0,T]\times[\chi_i,\chi_{i+1}]$, we define $Y^k_t(x)$ 
by setting
$Y^k_t(x)= \indiq_{\{\pi_S(\{(s,x,l)\, : \; s>t,s-l<\rho_k^i\})=0\}}$ 
if $(t,x)\in [\rho^i_k,T]\times I^i_k$
and $Y^k_t(x)=Y^{k-1}_t(x)$ else.

\vip

We finally set $Y_t(x)=Y^{n_i}_t(x)$ for all $t\in [0,T]$, 
all $x\in [\chi_i,\chi_{i+1}]$.
Doing this for each $i$, this defines a process 
$(Y_t(x))_{t\in [0,T],x\in\rr}$. 

\vip

A typical path of the $LFF(\beta)$-process 
is drawn and discussed on Figure \ref{figLFFbeta}, from which 
the following remark is clear.

\begin{figure}[b]
\fbox{
\begin{minipage}[c]{0.95\textwidth}
\centering
\includegraphics[width=13cm]{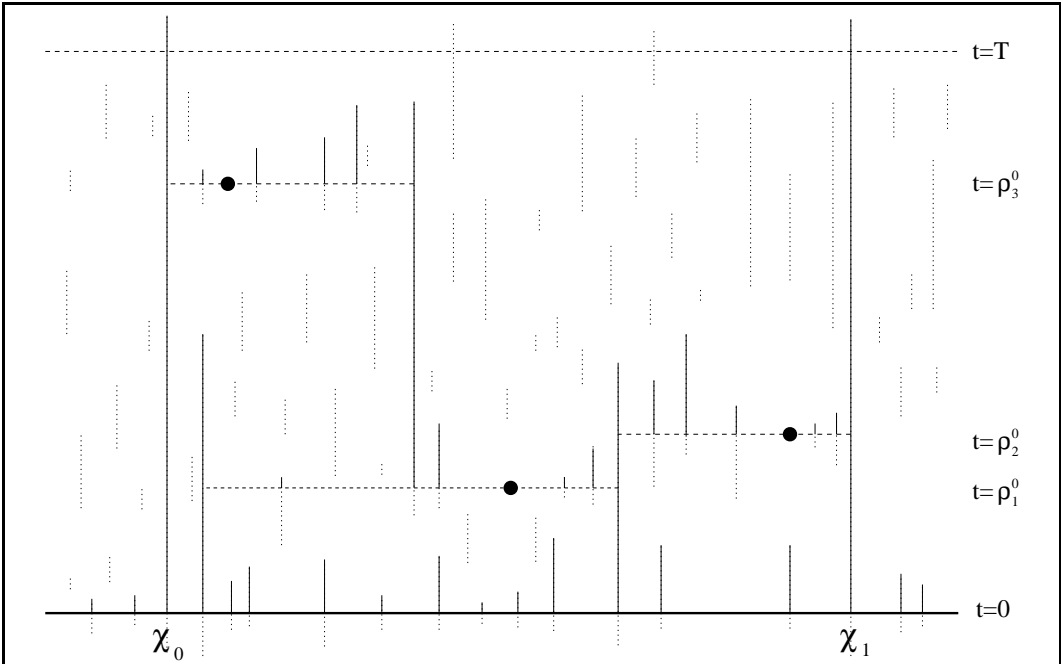}
\caption{$LFF(\beta)$-process with $\beta\in(0,\infty)$.}
\label{figLFFbeta}
\vip
\parbox{13.3cm}{
\footnotesize{
The plain segments represent vacant sites and the occupied
clusters are delimited by these segments. The marks of $\pi_M$
(matches) are represented as $\bullet$'s.

{\bf Step 0.} First, we draw on the whole space $[0,\infty)\times \rr$ 
all the $\bullet$'s and we draw a vertical dotted segment from 
$(\tau-L,X)$ to $(\tau,X)$ when $\pi_S$ has a mark at $(\tau,X,L)$.
Of course, such segments are infinitely many so that it is not possible
to draw all of them on a figure.

{\bf Step 1.} For each of these dotted segments that encounter the axis
$t=0$, we overwrite in plain its part above $t=0$.
Then we denote by $\chi_0$ and $\chi_1$ the first places on the left 
and right of $0$ such that plain segments go beyond $T$. At this stage, 
we have built $(Y^0_t(x))_{t\in [0,T], x \in \rr}$.

{\bf Step 3.} At time $\rho^0_1$, we consider the component $I_1^0$ 
(between plain segments) where the match $\bullet$ falls. 
Then, for each dotted 
segment (lying in $I_1^0$) that encounters the axis $t=\rho^0_1$, 
we overwrite in plain its part above $t=\rho_1^0$. At this stage, 
we have built $(Y^1_t(x))_{t\in [0,T], x \in [\chi_0,\chi_1]}$.

{\bf Step 3.}  At time $\rho^0_2$, we consider the component $I_2^0$ 
(between plain segments) where the match $\bullet$ falls. 
Then, for each dotted 
segment (lying in $I_2^0$) that encounters the axis $t=\rho^0_2$, 
we overwrite in plain its part above $t=\rho_2^0$. 
We have built $(Y^2_t(x))_{t\in [0,T], x \in [\chi_0,\chi_1]}$.

And so on...

{\bf Remark.} If we draw a vertical dotted segment from
$(\tau-L,X)$ to $(\tau,X)$ when $\pi_S$ has a mark at $(\tau,X,L)$ only
if $L>\delta$, and if $\delta>0$ is smaller than 
$\min\{\rho_1^0,\rho_2^0-\rho_1^0,
\rho_3^0-\rho_2^0\}$, then we get the exact values of $Y_t(x)$
for all $x \in [\chi_0,\chi_1]$ and all $t \in [0,T]\setminus ([0,\delta]\cup
[\rho_1^0,\rho_1^0+\delta]\cup[\rho_2^0,\rho_2^0+\delta]
\cup[\rho_3^0,\rho_3^0+\delta])$.

}}
\end{minipage}
}
\end{figure}

\begin{rem}\label{ribeta}
(i) If we build the process 
using some larger final time $T'>T$, this does not change the values of the
process on 
$[0,T]\times\rr$. 
Thus the process can be extended to $[0,\infty)\times\rr$.

(ii) For $\delta>0$, denote by 
$\pi_S^\delta$ the restriction of $\pi_S$ to $[0,\infty)\times \rr\times
[\delta,\infty)$. The sequence $(\chi_i)_{i\in \zz}$ clearly depends
only on $\pi_S^T$. Then for each $i\in \zz$, we denote by
$\cT_M^{i,T}=\{t\in [0,T]\, : \; \pi_M(\{t\}\times 
[\chi_i,\chi_{i+1}])>0)\cup \{0\}$
and by $\delta_{i,T}= \inf_{s,t \in \cT_M^{i,T}, s\ne t} |t-s|$.
Then for all $\delta\in (0,\delta_{i,T}\land T)$,
all $x\in [\chi_i,\chi_{i+1}]$ and all $t\in [0,T]\setminus 
\cup_{s\in \cT_M^{i,T}}[s,s+\delta]$, the value of
$Y_t(x)$ depends only on $\pi_M,\pi_S^\delta$.
\end{rem}

Observe that for all $t\geq 0$, $\{Y_t = 0\}$ is countable
and for all $t>0$ such that $\pi_M(\{t\}\times \rr)=0$,
$\{Y_t = 0\}$ is discrete (it has no accumulation point).

\begin{prop}\label{wpbeta}
Let $\pi_M,\pi_S$ be two independent Poisson measures 
on $[0,\infty)\times\rr$ and $[0,\infty)\times\rr\times[0,\infty)$
with intensity measures $dt dx$ and
$dt dx \beta (\beta+1) l^{-\beta-2} dl$. 
There a.s. exists a unique 
$LFF(\beta)$-process $(Y_t(x))_{t\geq 0,x\in \rr}$.
It can be simulated exactly on any finite box $[0,T]\times[-n,n]$.
For each $t\geq 0$ and $x\in \rr$, we will denote by
$D_t(x)=C(Y_t,x)$, recall (\ref{defcyx}).
\end{prop}

This proposition is obvious from the previous construction.
Of course, we can build {\it exactly} the process on any finite
box, but we cannot draw it {\it exactly}: when a match falls
in some occupied cluster $I$ at some time $t$, the set 
$\{x\in I:\;Y_t(x)=0\}$
is dense in $I$ (but $\{x\in I:\;Y_{t+\e}(x)=0\}$ is finite for all small
$\e>0$).

\subsection{On the Markov property}
The $LFF(\beta)$-process $(Y_t(x))_{t\geq 0}$
is clearly not Markov, in particular because
the heights of the barriers are not exponentially
distributed. The aim of this subsection is to
build a Markov process that contains more information
than $(Y_t(x))_{t\geq 0}$. 

\vip

Let the Poisson measures $\pi_M$ and $\pi_S$ be given. 
Write $\pi_S=\sum_{k\geq 1} \delta_{(t_k,x_k,l_k)}$ and introduce
$\pi_S^1=\sum_{k\geq 1} \delta_{(t_k-l_k,x_k,l_k)}\indiq_{\{t_k-l_k>0\}}$
and $\pi^0_S=\sum_{k\geq 1} \delta_{(t_k,x_k,l_k)}\indiq_{\{t_k-l_k<0\}}$.
Observe that $\pi^0_S$ and $\pi^1_S$ are independent. Furthermore,
$\pi^1_S$ has a mark $(\tau,X,L)$ if and only if there is a dotted
vertical segment from $(\tau,X)$ to $(\tau+L,X)$ (with $\tau>0$) and
$\pi^0_S$ has a mark $(\tau,X,L)$ if and only if there is a dotted
vertical segment from $(\tau-L,X)$ to $(\tau,X)$ (with $\tau-L<0<\tau$). 
One can easily check that $\pi^1_S$ is a Poisson measure
on $[0,\infty)\times\rr\times(0,\infty)$ with intensity
measure $dt dx \beta(\beta+1)l^{-\beta-2}dl$.
We set, for $x\in \rr$,
$$
\Gamma_0(x)= \int_0^\infty\int_0^\infty s \pi^0_S(ds\times\{x\}\times dl),
$$
which represents the height above $0$ of the dotted (or plain)
vertical segment at $x$ that crosses the axis $t=0$, with of course
$\Gamma_0(x)=0$ if there is no such dotted segment.
We then introduce, for $x\in \rr$ and $t\geq 0$,
$$
\Gamma_t(x)= \Gamma_0(x) + \intot \int_0^\infty \max\{l-\Gamma_\sm(x),0\}
\pi_S^1(ds\times\{x\}\times dl)  - \intot \indiq_{\{\Gamma_s(x)>0\}} ds,
$$
which represents the height above $t$ of the dotted (or plain)
vertical segment at $x$ that crosses the horizontal axis with ordinate
$t$, with $\Gamma_t(x)=0$ if there is no such dotted segment.
Indeed, $\Gamma_t(x)$ clearly decreases linearly when it is positive,
and jumps from $\Gamma_\sm(x)$ to $\max\{\Gamma_\sm(x),l\}$
when $\pi^1_S$ has a mark at $(s,x,l)$.
Using the fact that a.s., for all $x\in \rr$, there is at
most one dotted segment at $x$, it is possible to 
replace $\max\{l-\Gamma_\sm(x),0\}$ by $l$.
Finally, we define, for $x\in \rr$ and $t\geq 0$,
$$
\left\{
\begin{array}{l}
H_t(x)=\Gamma_0(x) +\displaystyle \intot \int_\rr \indiq_{\{y \in \stackrel \circ
C(Y_\sm,x)\}} 
\Gamma_\sm(x) \pi_M(ds,dy) -\displaystyle\intot \indiq_{\{H_s(x)>0\}} ds ,\\
Y_t(x)=\indiq_{\{H_t(x)=0\}},
\end{array}
\right.
$$
where $\stackrel \circ C(Y_\sm,x)$ stands for the interior of $C(Y_\sm,x)$.
Then  $H_t(x)$ is the height above $t$ of the
plain segment at $x$ that crosses the horizontal axis with ordinate $t$
(with $H_t(x)=0$ if there no such plain segment), and thus
$(Y_t(x))_{t\geq 0}$ is the $LFF(\beta)$-process.
Indeed, since we overwrite in plain all the dotted segments that cross the axis
$t=0$, we clearly have $H_0(x)=\Gamma_0(x)$.
Then $H_t(x)$ decreases linearly when it is positive, and jumps
to $\Gamma_\sm(x)$ when $x$ is involved in a fire at some time $s$
(whence necessarily $H_\sm(x)=0$): 
recall that we then overwrite in plain the dotted
segment at $x$ that crosses the horizontal axis with ordinate $s$,
of which the height above $s$ it given by $\Gamma_\sm(x)$.

\vip

The process $(\Gamma_t(x),H_t(x),Y_t(x))_{t\geq 0,x\in \rr}$ is Markov,
since it solves a well-posed homogeneous Poisson-driven S.D.E.

\subsection{The convergence result}

We now state our main result in the case $\beta \in (0,\infty)$.
We use Subsection \ref{nocv}.

\begin{theo}\label{converge2}
Assume $(H_M)$ and $(H_S(\beta))$ for some $\beta \in (0,\infty)$.
Consider, for each $\la\in (0,1]$, the process
$(D^\la_t(x))_{t\geq 0,x\in\rr}$ associated with
the $FF(\mu_S,\mu_M^\la)$-process,
see Definition \ref{gff} and (\ref{dlambda}).
Consider also a $LFF(\beta)$-process $(Y_t(x))_{t\geq 0, x\in \rr}$ 
and the associated $(D_t(x))_{t\geq 0, x\in \rr}$.

(a) For any $T>0$, any finite subset $\{x_1,\dots,x_p\}\subset \rr$, 
$(D^\la_t(x_i))_{t\in [0,T],i=1,\dots,p}$ goes in law to 
$(D_t(x_i))_{t\in [0,T],i=1,\dots,p}$ in $\dd([0,T],\cI)^p$,
as $\la \to 0$. Here
$\dd([0,T],\cI)$ is endowed with $\bdelta_{T}$.

(b) For any finite subset  $\{(t_1,x_1),\dots,(t_p,x_p)\}\subset 
(0,\infty)\times \rr$, 
$(D^\la_{t_i}(x_i))_{i=1,\dots,p}$ goes in law to 
$(D_{t_i}(x_i))_{i=1,\dots,p}$ in $\cI ^p$, $\cI$ being endowed with $\bdelta$.
\end{theo}

\subsection{Heuristic arguments}

We assume below that $\la>0$ is very small.

\vip

{\it 0. Matches.} Exactly as in the case $\beta=\infty$, 
we hope that matches will fall, in our scales, according
to a Poisson measure with intensity $1$ (in mean, $1$ match per
unit of time per unit of space, which corresponds to $1$ match per
$\bn_\la$ sites during $[0,\ba_\la]$ in the original scales).

\vip

{\it 1. Occupation of vacant zones.} Consider a zone $[a,b]$ (or
$\lb \lfloor a \bn_\la\rfloor,\lfloor b \bn_\la\rfloor\rb$ in the original
scales). At time $0$, this zone is completely empty.
In this zone, each site will be empty at time $t$ if no seed has
fallen during $[0,t]$ (or $[0,\ba_\la t]$ in the original scale).
This occurs with probability $\nu_S((\ba_\la t,\infty))$.
Thus in the absence of fires, 
the number of empty sites in $[a,b]$ at time $t$ follows
a binomial distribution with parameters $(b-a)\bn_\la$ and 
$\nu_S((\ba_\la t,\infty))$. Recalling (\ref{ala}), (\ref{nla}) and 
$(H_S(\beta))$,
we see that $(b-a)\bn_\la\nu_S((\ba_\la t,\infty))
\simeq (b-a)\nu_S((\ba_\la t,\infty))/\nu_S((\ba_\la,\infty))
\to (b-a) t^{-\beta}$.
Hence the number of empty sites in $[a,b]$ at time $t$ follows
approximately a Poisson law with parameter $(b-a)t^{-\beta}$
(when neglecting fires).

\vip

The link with the $LFF(\beta)$-process is simple: for any $a<b$ and
any $t>0$, the random variable 
$\pi_S(\{(s,x,l)\, : \; x\in [a,b], s>t, s-l<0)\})$
follows a Poisson law with parameter
$\int_{t}^\infty ds \int_a^b dx \int_{s}^\infty \beta(\beta+1)l^{-\beta-2}dl
=(b-a)t^{-\beta}$.

\vip

{\it 2. Fires.} Now when a match falls at some place, this destroys
the whole occupied cluster. The destroyed cluster is then treated as in 
Point 1.

\subsection{Cluster-size distribution}

We aim here to estimate the law of the occupied cluster around $0$.
No phase transition occurs here.

\begin{cor}\label{co2}
Let $\beta \in (0,\infty)$. Assume $(H_M)$ and $(H_S(\beta))$. Recall that
$\ba_\la$ and $\bn_\la$ were defined in (\ref{ala}) and (\ref{nla}). 
Consider the $LFF(\beta)$-process $(Y_t(x))_{t\geq 0, x\in \rr}$
and the associated $(D_t(x))_{t\geq 0, x\in \rr}$. 
For each $\la\in (0,1]$, let $(\eta^\la_t(i))_{t\geq 0,i\in\zz}$ be a 
$FF(\mu_S,\mu_M^\la)$-process. 
There are
some constants $0<c_1<c_2$ and $0<\kappa_1<\kappa_2$
such that for all $t\geq 1$, all $B>0$,
$$
\lim_{\la\to 0} \Pr\left[|C(\eta^\la_{\ba_\la t},0)| \geq B\bn_\la \right]=
\Pr\left[|D_t(0)| \geq B \right] \in [c_1 e^{-\kappa_2 B}, c_2 e^{-\kappa_1 B}]. 
$$
\end{cor}

\section{Main results when $\beta=0$}\label{mrz}
\setcounter{equation}{0}

\subsection{Definition of the limit process}

In this case, the limiting process is trivial:
we consider a Poisson measure $\pi_S$ on $\rr$ with intensity measure
$dx$ and we put, for all $t\geq 0$, all $x\in \rr$,
$$
Y_t(x)=\indiq_{\{\pi_S(x)=0\}}.
$$
Denote by $\{\chi_i\}_{i\in \zz}$ the marks of $\pi_S$
with the convention that 
$\dots<\chi_{-1}<\chi_0< 0 < \chi_1<\chi_2<\dots$.
Then for all $t\geq 0$, all $i\in \zz$, 
recalling (\ref{defcyx}), $C(Y_t,x)=[\chi_i,\chi_{i+1}]$
for all $x\in (\chi_i,\chi_{i+1})$ and $C(Y_t,\chi_i)=\{\chi_i\}$.
Matches fall according to a Poisson measure $\pi_M(dt,dx)$ 
on $[0,\infty)\times\rr$
with intensity measure $dt dx$. 

\vip
The $LFF(0)$-process $(Y_t(x))_{t\geq 0,x\in \rr}$ is obviously Markov
and the following statement is trivial.

\begin{prop}\label{wpzero}
Let $\pi_S$ be a Poisson measure on $\rr$ with intensity measure $dx$.
There a.s. exists a unique $LFF(0)$-process $(Y_t(x))_{t\geq 0,x\in \rr}$.
It can be simulated exactly on any finite box $[0,T]\times[-n,n]$.
For each $t\geq 0$ and $x\in \rr$, we will denote by $D_t(x)=C(Y_t,x)$
the occupied cluster around $x$ (see (\ref{defcyx})).
\end{prop}

Of course, fires do not appear in the construction. Hence
it is not necessary to introduce $\pi_M$. However, it allows us to
keep in mind that fires do occur. But these fires generate empty zones
that are immediately regenerated.
The main idea is that in our scales: 
on the great majority of sites, seeds fall almost continuously
for all times; but there are {\it rare} sites where the first seed 
will never fall.
Hence when there is a fire, this always concerns a zone where seeds fall
{\it continuously}, so that one does not observe the fire at the limit.
A typical path of the $LFF(0)$-process is commented
on Figure \ref{figLFFzero}.

\begin{figure}[b] 
\fbox{
\begin{minipage}[c]{0.95\textwidth}
\centering
\includegraphics[width=10cm]{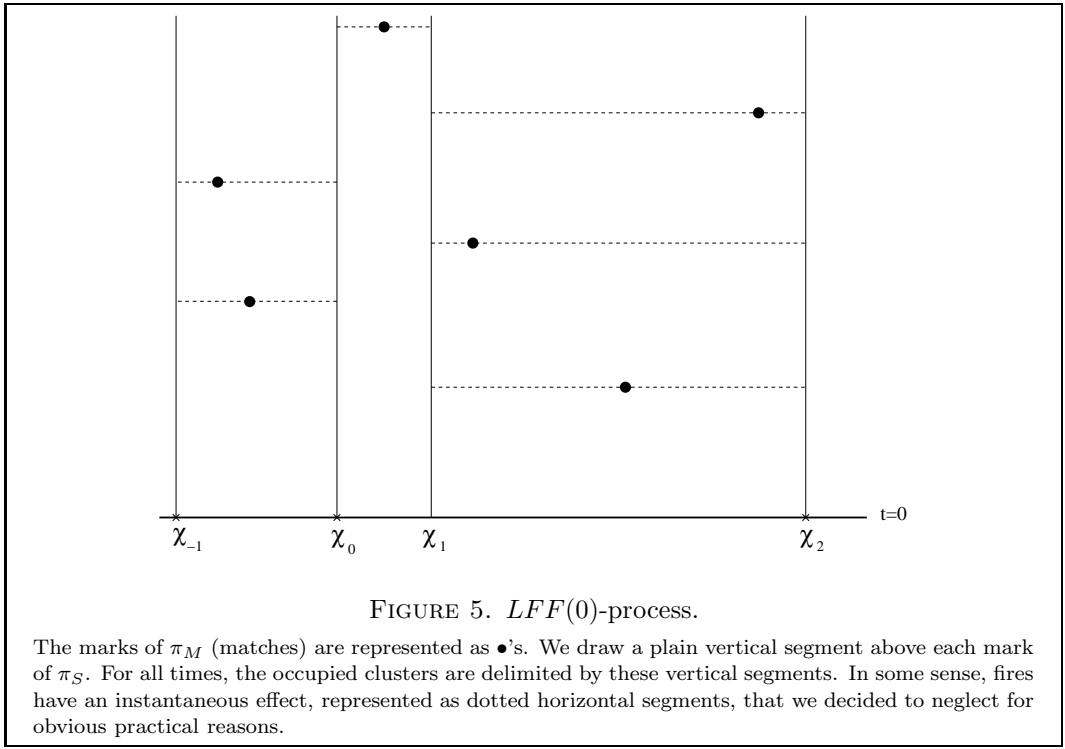}
\caption{$LFF(0)$-process.}
\label{figLFFzero}
\vip
\parbox{13.3cm}{
\footnotesize{
The marks of $\pi_M$ (matches) are represented as $\bullet$'s.
We draw a plain vertical segment above each mark of $\pi_S$.
For all times, the occupied clusters are delimited by these vertical
segments. In some sense,  
fires have an instantaneous effect,
represented 
as dotted horizontal segments, that we decided to neglect for 
obvious practical reasons. 
}}
\end{minipage}
}
\end{figure}

\subsection{The convergence result}
We now state our last main result, using Subsection
\ref{nocv}.

\begin{theo}\label{converge3}
Assume $(H_M)$ and $(H_S(0))$.
Consider, for each $\la\in (0,1]$, the process
$(D^\la_t(x))_{t\geq 0,x\in\rr}$ associated to
the $FF(\mu_S,\mu_M^\la)$-process,
see Definition \ref{gff} and (\ref{dlambda}).
Consider also the $LFF(0)$-process $(Y_t(x))_{t\geq 0, x\in \rr}$ 
and the associated $(D_t(x))_{t\geq 0, x\in \rr}$.

(a) For any $T>0$, any finite subset $\{x_1,\dots,x_p\}\subset \rr$, 
$(D^\la_t(x_i))_{t\in [0,T],i=1,\dots,p}$ goes in law to 
$(D_t(x_i))_{t\in [0,T],i=1,\dots,p}$ in $\dd([0,T],\cI)^p$
as $\la \to 0$. Here
$\dd([0,\infty),\cI)$ is endowed with $\bdelta_{T}$.

(b) For any finite subset  $\{(t_1,x_1),\dots,(t_p,x_p)\}\subset 
(0,\infty)\times \rr$, 
$(D^\la_{t_i}(x_i))_{i=1,\dots,p}$ goes in law to 
$(D_{t_i}(x_i))_{i=1,\dots,p}$ in $\cI ^p$, $\cI$ being endowed with $\bdelta$.
\end{theo}

\subsection{Heuristic arguments}

The only difference with the case $\beta\in (0,\infty)$ is the following.
In some sense, for each site $i$, in our scales,
either seeds fall {\it continuously}
on $i$, or the first seed never falls on $i$. 
A first hint for this is
the following.

\vip

Consider a zone $[a,b]$. At time $0$, this zone is completely vacant.
Fix $T>0$. Then in the absence of fires,
the number of vacant sites in $[a,b]$ at time $T$
(or in $\lb \lfloor a \bn_\la\rfloor,\lfloor b \bn_\la\rfloor\rb$ 
at time $\ba_\la T$ 
in the original scales)
follows a binomial distribution with parameters $(b-a)\bn_\la$
and $\nu_S((\ba_\la T,\infty))$. Observe now that for any value of $T>0$,
using $(H_S(0))$, (\ref{ala}) and (\ref{nla}),
$(b-a)\bn_\la\nu_S((\ba_\la T,\infty))\simeq (b-a)\nu_S((\ba_\la T,\infty))
/ \nu_S((\ba_\la ,\infty)) \to (b-a)$.
Hence the number of sites that are still vacant at time $T$ 
follows approximately a Poisson
distribution with parameter $(b-a)$. Since this parameter does not decrease
with $T$, this means that in our scales, sites are either immediately occupied
or vacant forever.

\subsection{Cluster-size distribution}

Since the $LFF(0)$-process is very simple, we obtain of course some 
more precise information on the asymptotic cluster-size distribution.

\begin{rem}\label{co3}
Assume $(H_M)$ and $(H_S(0))$.
For each $\la\in (0,1]$, let a $FF(\mu_S,\mu_M^\la)$-process
$(\eta^\la_t(i))_{t\geq 0,i\in\zz}$ be given,
see Definition \ref{gff}.
Consider the $LFF(0)$-process $(Y_t(x))_{t\geq 0, x\in \rr}$ 
and the associated $(D_t(x))_{t\geq 0, x\in \rr}$.
Then for $t>0$, $B>0$,
$$
\lim_{\la\to 0} \Pr\left[|C(\eta^\la_{\ba_\la t},0)|\geq B\bn_\la \right]
= \Pr\left[|D_t(0)| \geq B \right]= \int_B^\infty xe^{-x}dx=(B+1)e^{-B}.
$$
\end{rem}

No proof is needed here:  $xe^{-x}\indiq_{\{x>0\}}$ is just
the density of $|D_t(0)|=\chi_1-\chi_0$.
The convergence in law of $|C(\eta^\la_{\ba_\la t},0)|/\bn_\la=
|D^\la_t(0)|$ to $|D_t(0)|$ follows from Theorem \ref{converge3}.

\section{On some other modelling choices}\label{discu}
\setcounter{equation}{0}

For $\mu$ a probability law on $(0,\infty)$, we say that 
$N_t=\sum_{n\geq 1} \indiq_{\{X_1+\dots+X_n \leq t\}}$ is a natural renewal process
with parameter $\mu$, or a $NR(\mu)$-process in short, if the random variables
$X_i$ are i.i.d. with law $\mu$. 
When extending the traditional forest fire model (where all the renewal
processes are Poisson), we had to make some choices.

\vip

1. Matches can fall according to some i.i.d.
(i) $SR(\mu^\la_M)$-processes,
(ii) $NR(\mu^\la_M)$-processes.

\vip

2. Seeds can fall according to some i.i.d.
(i) $SR(\mu_S)$-processes,
(ii) $NR(\mu_S)$-processes.

\vip

3. When a fire destroys an occupied component $\lb a,b\rb$, we can
(i) keep the i.i.d. renewal processes governing seeds as they are,
(ii) forget everything and make start some new i.i.d. renewal processes 
governing seeds in the zone $\lb a,b\rb$.

\vip

Recall that when dealing with Poisson processes,
choosing (i) or (ii)  in Points 1, 2, 3 does not change the law of the 
$FF(\mu_S,\mu^\la_M)$-process.

\vip

From the point of view of modelling, it seemed more natural to choose
(i) in Points 1 and 2: this is the only way that time $0$ does not
play a special role. We also decided to choose (i) in Point 3, because
its seems more close to applications. 
Let us discuss briefly what could happen with other choices.

\vip

First, for matches (Point 1), choosing (i) does not play a fundamental role.
Indeed, in our scales,
only $0$ or $1$ match can fall on each site. Thus our results
should extend, without difficulty, to the choice 1-(ii), replacing $(H_M)$
by the assumption $\mu_M^\la((0,t))\simeq \la t$ as $\la \to 0$ 
(together with some additional regularity conditions
if we want a strong coupling as in Proposition \ref{coupling1}).

\vip

Next, we believe that in Point 2, our results should still hold
if choosing (ii) when $\beta=\infty$. 
In the case where $\mu_S$ has a bounded support, one would have
to assume some regularity on $\mu_S$ (the case $\mu_S=\delta_1$ is trivial)
and to modify the dynamics of the $LFF(BS)$-process (the law $\theta_u$
should also depend on time).
Our study would completely break down when $\beta\in [0,\infty)$. In the latter
case, the situation would be quite intricate and we are not
able to predict scales (and, {\it a fortiori}, to predict some limit process).
Let us explain briefly the situation. If $\beta=\infty$, then
$\nu_S$ and $\mu_S$ have a similar tail (see example 2). Thus the
time and space scales we have considered will fit both $\nu_S$ and $\mu_S$.
On the contrary, if $\beta\in [0,\infty)$, the tails of $\mu_S$ and $\nu_S$
are really different. Consequently, if we accelerate time according to $\mu_S$
(in order that for a $NR(\mu_S)$-process, the cluster containing the site 
$0$ burns before time $1$ with a positive probability), 
then this will be too slow for larger times
(when a fire destroys a cluster $(a,b)$, this zone 
$(a,b)$ will never regenerate).

\vip

Finally, in Point 3, we also believe that choosing (ii) would not change 
our results when $\beta=\infty$ and not change too much the situation
when $\mu_S$ has a bounded support.
When $\beta\in [0,\infty)$, we expect
that this would not change time/space scales, but we would have to modify 
the limit processes. For example if $\beta=0$, we expect that 
each time a fire burns a zone $(a,b)$, this zone would 
regenerate immediately, {\it except} in a random number of sites,
that follows a Poisson distribution with parameter $(b-a)$.
Next if $\beta\in (0,\infty)$, then when a fire burns a zone $(a,b)$ at some
time $t$, we would have to pick another  
Poisson measure $\pi_S^{(a,b),t}(ds,dx,dl)$
on $[t,\infty)\times(a,b)\times(0,\infty)$, independent of everything else
and use this Poisson measure above $(a,b)$ instead of the original
$\pi_S$.

\part{Proofs}

\section{Graphical construction of the discrete process}\label{exunidisc}
\setcounter{equation}{0}

The goal of this section is to prove Proposition \ref{gcgff} by
using a {\it graphical construction}.

\begin{preuve} {\it of Proposition \ref{gcgff}.}
Our aim is to prove that for any $T>0$, a.s., the values of the
$FF(\mu_S,\mu_M)$-process $(\eta_t(i))_{t\in [0,T],i\in \zz}$ 
are uniquely determined by $(N^S_t(i),N^M_t(i))_{t\geq 0,i\in \zz}$.
Recall that $\nu_S(dx)=m_S^{-1}\mu_S((x,\infty))dx$ and 
$\nu_M(dx)=m_M^{-1}\mu_M((x,\infty))dx$, where $m_S$ and $m_M$ are the 
expectations of $\mu_S$ and $\mu_M$.
We consider $h_0>0$ such that $\nu_S([2h_0,\infty))>0$ (if $\nu_S$ 
has an unbounded support, any value of $h_0$ is possible) and we
put $c_0=\nu_S((2h_0,\infty)) \nu_M((0,h_0))>0$. We also set
$K=\lfloor T/h_0 \rfloor$.

\vip

For $n\in \zz$, we consider the event $\Omega_{n,T}$, on which
the following conditions are
satisfied:

\vip

(i) $N^S_{h_0}(n)=0$,

\vip 
(ii) $\forall$ $i\in \lb 1,K\rb$, $N^S_{(i+1) h_0}(n+i)=N^S_{(i-1) h_0}(n+i)$,

\vip 
(iii) $\forall$ $i\in \lb 1,K\rb$, $N^M_{i h_0}(n+i)>N^M_{(i-1) h_0}(n+i)$.

\vip

We first observe that for any $n \in \zz$, using the
stationarity of the renewal processes,
$$
\Pr[\Omega_{n,T}]= \nu_S((h_0,\infty)) c_0^K =:c_T>0.
$$

Next we prove that necessarily, 
$$
\Omega_{n,T} \subset \left\{\forall t \in [0,T], \exists i \in \lb n,n+K\rb,
\eta_t(i)=0 \right\}.
$$ 
This is not hard: (i) implies that $\eta_t(n)=0$ for $t \in [0,h_0]$,
since no seed falls on $n$ during this interval. Point (iii) implies
that for $i\in  \lb 1,K\rb$, a match falls on 
$n+i$ during $((i-1)h_0,i h_0]$ and Point (ii) guarantees us that no seed 
falls on $n+i$
during  $((i-1)h_0,(i+1)h_0]$, whence the site $n+i$ is necessarily vacant
during (at least) $(i h_0,(i+1)h_0]$.
Consequently, on $\Omega_{n,T}$, there is always at least one vacant site
in $\lb n,n+K\rb$ during $[0,h_0] \cup (\cup_{i=1,\dots,K} (i h_0,(i+1)h_0])
\supset [0,T]$ (with our choice for $K$, we have $(K+1)h_0\geq T$).

\vip

Hence conditionally on $\Omega_{n,T}$, during $[0,T]$,
the fires starting on the right of $n+K$ do not affect the values
of the forest fire process on the left of $n$;
and the fires starting on the left of $n$ do not affect the values
of the forest fire process on the right of $n+K$.
 
\vip

Since $\Pr[\Omega_{n,T}]=c_T>0$, 
we can find $\dots<n_{-1}<n_0<0<n_1<n_2\dots$ 
such that $\cap_{l\in \zz}\Omega_{n_l,T}$ is realized
(use that $\Omega_{n,T}$ is independent of $\Omega_{m,T}$ if $|m-n|>K$).

\vip

We deduce that for any $i \in \zz$,
the values of $(\eta_t(i))_{t\in [0,T]}$ are entirely determined
by the values of $(N^S_t(j),N^M_t(j))_{t\in [0,(K+1)h_0]}$ for a
finite number of $j$'s, namely (at most) $j\in  \lb n_k,n_l+K \rb$,
where $k<l$ satisfy $n_k+K < i < n_l$.

\vip

We have shown that for any $T>0$, $(\eta_t(i))_{t\geq 0,i\in\zz}$ 
is entirely and uniquely defined by the values of 
$(N^S_t(i),N^M_t(i))_{t\in [0,(K+1)h_0],i\in \zz}$.
\end{preuve}

\section{Convergence of matches}\label{pr1}
\setcounter{equation}{0}

In this section, we consider
any function $\la \mapsto \ba_\la$ bounded from below and such that
$\bn_\la=\lfloor 1/(\la \ba_\la)\rfloor \to \infty$.
For $A>0$, we set $A_\la = \lfloor A \bn_\la \rfloor$ and 
$I_A^\la=\lb -A_\la,A_\la\rb$.
For $i\in \zz$, we set $i_\la=[i/\bn_\la, (i+1)/\bn_\la)$.
The following result will be used to prove our four main theorems.

\begin{prop}\label{coupling1}
Assume $(H_M)$. Let $A>0$ and $T>0$ be fixed. We can find, 
for any $\la \in (0,1]$, 
a coupling between a Poisson measure 
$\pi_M(dt,dx)$ on $[0,\infty)\times \rr$ with intensity measure
$dtdx$
and a family of i.i.d.
$SR(\mu_M^\la)$-processes 
$(N_t^{M,\la}(i))_{i\in \zz,t\geq 0}$ such that for
$$
\Omega^M_{A,T}(\la):=\left\{\forall t\in [0,T],\;
\forall i\in I_A^\la, \;
\Delta N^{M,\la}_{\ba_\la t}(i) \ne 0 \; \hbox{iff} \; 
\pi_M(\{t\}\times i_\la) \ne 0
\right\},
$$
there holds $\lim_{\la\to 0}\Pr[\Omega^M_{A,T}(\la)]=1$.
\end{prop}

This means that in our scales, with a high probability, the matches 
used in the discrete processes can be prescribed by a Poisson measure,
as in the limit processes.

\begin{proof}
We divide the proof into several steps. Observe that 
$$
B_\la:=\cup_{i\in I_A^\la} i_\la
=[-A_\la/\bn_\la,(A_\la+1)/\bn_\la )
$$
(which is approximately $[-A,A]$).
It of course suffices to build 
$\pi_M$ restricted to $[0,T]\times B_\la$
and the family $(N_t^{M,\la}(i))$ for $i\in I_A^\la$ and 
$t\in [0,\ba_\la T]$. 

\vip

{\bf Step 1.} First, we observe that a possible way to build $\pi_M$ 
(restricted to $[0,T]\times B_\la$) is the following: 

\vip

(i) Consider a family of i.i.d. r.v. $(Z_i^\la)_{i\in I_A^\la}$
following a Poisson distribution
with parameter $T |i_\la|= T/\bn_\la$.

\vip

(ii) For each $i$ with $Z_i^\la > 0$, 
pick some i.i.d. r.v. $(T^{i,\la}_1,X^{i,\la}_1)
,\dots,(T^{i,\la}_{Z_i^\la},X^{i,\la}_{Z_i^\la})$ 
with uniform law on $[0,T]\times i_\la$ (conditionally on $Z_i^\la$).

\vip

Set finally $\pi_M=\sum_{i\in I_A^\la} \sum_{k=1}^{Z_i^\la}
\delta_{(T^{i,\la}_k,X^{i,\la}_k)}$.

\vip

{\bf Step 2.} Next, we note it is possible to build the family 
$(N_t^{M,\la}(i))_{i\in I_A^\la, t\in [0,\ba_\la T]}$ as follows:
introduce $q_k(\la,T)=\Pr[N^{M,\la}_{\ba_\la T}(i)=k]$ and 
$\zeta_k^{\la,T}(dt_1,\dots,dt_k)$ the law of the $k$ jump instants of 
$N^{M,\la}(i)$ in $[0,\ba_\la T]$ conditionally on $\{N^{M,\la}_{\ba_\la T}(i)=k\}$.

\vip

(i) Consider a family of i.i.d. r.v. $(\tZ_i^\la)_{i\in I_A^\la}$ 
with law $(q_k(\la,T))_{k \geq 0}$.

\vip

(ii) For each $i$ with $\tZ_i^\la > 0$, 
pick $(\tT^{i,\la}_1,\dots,\tT^{i,\la}_{\tZ_i^\la})$ according to 
$\zeta_{\tZ_i^\la}^{\la,T}(dt_1,\dots,dt_{\tZ_i^\la})$ 
(conditionally on $\tZ_i^\la$).

\vip

Set finally 
$N_t^{M,\la}(i)=\sum_{k=1}^{\tZ^\la_i} \indiq_{\{t \geq \tT^{i,\la}_k\}}$ 
for $t\in [0,\ba_\la T]$, $i\in I_A^\la$.

\vip

{\bf Step 3.} We show in this step that for each 
$i\in I_A^\la$, one can couple $Z_i^\la$
(as in Step 1-(i)) and $\tZ_i^\la$ (as in Step 2-(i)) in such a way that
$$
\Pr[Z_i^\la=\tZ^\la_i =0] \geq  1 - \la \ba_\la T(1+\e_T(\la))
\quad \hbox{ and } \quad  
\Pr[Z_i^\la=\tZ^\la_i=1] \geq \la \ba_\la T (1-\e_T(\la)),
$$
where $\lim_{\la \to 0}\e_T(\la)=0$. Below, the function $\e_T$ may change
from line to line.

\vip

It is classically possible (see Lemma \ref{gcou}-(i)) 
to build a coupling in such a way that
\begin{align*}
\Pr[Z_i^\la=\tZ^\la_i=0] \geq& \Pr[Z_i^\la=0] \land \Pr [\tZ_i^\la=0], \\
\Pr[Z_i^\la=\tZ^\la_i=1] \geq&  \Pr[Z_i^\la=1] \land \Pr [\tZ_i^\la=1].
\end{align*}
We now use $(H_M)$: recalling that $\int_0^\infty \mu^1_M((s,\infty))ds=1$,
\begin{align*}
\Pr [\tZ_i^\la=0] =& \nu_M^\la((\ba_\la T,\infty))=\la\int_{\ba_\la T}^\infty 
\mu^1_M((\la s,\infty))ds = \int_{\la \ba_\la T}^\infty 
\mu^1_M((u,\infty))du \\
=& 1 - \int_0^{\la \ba_\la T}\mu^1_M((u,\infty))du
\geq 1 - \la\ba_\la T.
\end{align*}
Since $\Pr [Z_i^\la=0]=e^{-T/\bn_\la}=e^{-T /\lfloor 1/(\la\ba_\la)\rfloor} 
= 1 - \la\ba_\la T(1+\e_T(\la))$, this concludes the proof of the 
first lower-bound. Next, recalling Definition \ref{sr1}
and $(H_M)$,
\begin{align*}
\Pr [\tZ_i^\la=1]=&\int_0^{\ba_\la T} \mu_M^\la((\ba_\la T-s,\infty)) 
\nu^\la_M(ds)\\
=&\int_0^{\ba_\la T} \mu_M^1((\la (\ba_\la T-s),\infty)) 
\la \mu^1_M((\la s,\infty))ds \\
=& \int_0^{\la \ba_\la T} \mu_M^1((\la \ba_\la T- u,\infty)) 
\mu^1_M((u,\infty))du = \la \ba_\la T (1-\e_T(\la)),
\end{align*}
since $\la \ba_\la \to 0$ as $\la \to 0$.
But now $\Pr [Z_i^\la=1]=(T/\bn_\la)e^{-T/\bn_\la}= \la\ba_\la T(1-\e_T(\la))$, 
because $\bn_\la =
\lfloor 1/(\la\ba_\la) \rfloor$ and this concludes the step.

\vip

{\bf Step 4.} We now check that for each 
$i\in I_A^\la$, conditionally 
on $\{Z_i^\la=\tZ_i^\la=1\}$, we can couple $T_1^{i,\la}$ and $\tT_1^{i,\la}$ 
(see Steps 1-(ii) and 2-(ii))
in such a way that for 
$$
r_T(\la)= \Pr\left[T_1^{i,\la}=\tT_1^{i,\la} /\ba_\la
\;\vert\;Z_i^\la=\tZ_i^\la=1 \right],
$$ 
there holds $\lim_{\la\to 0}r_T(\la)=1$. We first recall that $T_1^{i,\la}$
is uniformly distributed on $[0,T]$ (conditionally on $\{Z^\la_i=1\}$). 
We next remark that the conditional law
of $\tT_1^{i,\la}$ knowing $\{\tZ_i^\la=1\}$ 
(which we called $\zeta_1^{\la,T}$)
is nothing but 
\begin{align*}
\zeta_1^{\la,T}(dt)=&\frac{\nu_M^\la(dt)\mu_M^\la((\ba_\la T-t, \infty))
\indiq_{\{t \in [0,\ba_\la T]\}}}{\int_0^{\ba_\la T} \mu_M^\la((\ba_\la T-s,\infty)) 
\nu^\la_M(ds)} \\
=& \frac{\mu_M^1((\la(\ba_\la T - t),\infty))\la \mu^1_M((\la t, \infty)) 
\indiq_{\{t\in [0,\ba_\la T]\}}}{\la \ba_\la T (1-\e_T(\la))}dt,
\end{align*}
where we used the same computations as in Step 3. Consequently, the 
conditional law of $\tT_1^{i,\la}/\ba_\la $ knowing $\{\tZ_i^\la=1\}$ 
has a density $g_{\la,T}$ of the form
$$
g_{\la,T}(t)= \frac{1+\e_T(\la)}{T}\mu_M^1((\la \ba_\la (T - t),\infty))
\mu^1_M((\la \ba_\la t, \infty)) \indiq_{\{t\in [0,T]\}}.
$$
Observe that $\lim_{\la \to 0} g_{\la,T}(t)=T^{-1}\indiq_{\{t\in[0,T]\}}$,
since $\la\ba_\la\to 0$.
Hence, classical arguments (see Lemma \ref{gcou}-(ii)) 
show that conditionally on
$\{Z_i^\la=\tZ_i^\la=1\}$, we can couple $T_1^{i,\la}$ and $\tT_1^{i,\la}$ in such
a way that
$$
\Pr\left[T_1^{i,\la}=\tT_1^{i,\la} /\ba_\la
\;\vert\;Z_i^\la=\tZ_i^\la=1 \right] \geq \int_0^T \min\left(\frac{1}{T},
g_{\la,T}(t) \right) dt,
$$
which tends to $1$ as $\la\to 0$ by dominated convergence.

\vip

{\bf Step 5.} We finally may build the complete coupling.

\vip

(i) For each $i\in I_A^\la$, we consider some coupled 
random variables $(Z_i^\la,\tZ_i^\la)$ as in Step 3.

\vip

(ii) For  $i \in I_A^\la$ such that $Z_i^\la=\tZ_i^\la=0$,
there is nothing to do.

\vip

(iii) For $i \in I_A^\la$ such that $Z_i^\la=\tZ_i^\la=1$,
couple $T_1^{i,\la}$ and $\tT_1^{i,\la}$ as in Step 4 and pick $X_1^{i,\la}$
uniformly in $i_\la$. 

\vip

(iv) If $i \in I_A^\la$ does not meet one of the two
above conditions (ii) and (iii), then we build $(T^{i,\la}_1,X^{i,\la}_1),
\dots,(T^{i,\la}_{Z_i^\la},X^{i,\la}_{Z_i^\la})$
and $\tT^{i,\la}_1,\dots,\tT^{i,\la}_{\tZ_i^\la}$ 
in any way (e.g., follow the rules
of Step 1-(ii) and Step 2-(ii) independently).

\vip

(v) Set $\pi_M=\sum_{i\in I_A^\la} \sum_{k=1}^{Z_i^\la}
\delta_{(T^{i,\la}_k,X^{i,\la}_k)}$ and 
$N_t^{M,\la}(i)=\sum_{k=1}^{\tZ^\la_i} \indiq_{\{t \geq \tT^{i,\la}_k\}}$ 
for $i\in I_A^\la$, $t\in [0,T \ba_\la]$.

\vip

{\bf Step 6.} Define the event
$$
\tOmega^M_{A,T}(\la)= \bigcap_{i \in I_A^\la}
\left(\{Z_i^\la=\tZ_i^\la=0\}\cup \{Z_i^\la=\tZ_i^\la=1,\;
T_1^{i,\la}=\tT^{i,\la}_1/\ba_\la\} \right).
$$
Then we have $\tOmega^M_{A,T}(\la) \subset \Omega^M_{A,T}(\la)$ 
(where $\Omega^M_{A,T}(\la)$ was defined in the statement).
Indeed, on $\tOmega^M_{A,T}(\la)$, for any $i\in I_A^\la$, $t\in [0,T]$, we have
$\Delta N_{\ba_\la t}^{M,\la}(i) \ne 0$ iff ($\tZ_i^\la=1$ and 
$\ba_\la t = \tT^{i,\la}_1$) iff  ($Z_i^\la=1$ and 
$t = T^{i,\la}_1$) iff $\pi_M(\{t\}\times i_\la)>0$.

\vip

Finally, using Steps 3 and 4 and that $|I_A^\la|=2A_\la+1$,
\begin{align*}
\Pr[\Omega^M_{A,T}(\la)] \geq& 
\left( 
\Pr[Z_0^\la=\tZ_0^\la=0] + \Pr\left[Z_0^\la=\tZ_0^\la=1,\;\;
T_1^{0,\la}=\tT_1^{0,\la} /\ba_\la \right]
\right)^{2A_\la+1}\\
\geq& \big(1-\la\ba_\la T (1+\e_T(\la))
+ \la\ba_\la T (1-\e_T(\la))r_T(\la)  \big)^{2A_\la+1}.
\end{align*}
Recall that $\lim_{\la\to 0} \e_T(\la) =0$, that 
$\lim_{\la\to 0} r_T(\la) =1$ and that 
$A_\la \leq A /(\la\ba_\la)$. Hence for some (other) function $\e_T$
with limit $0$ at $0$,
$$
\Pr[\Omega^M_{A,T}(\la)]  \geq \left(1-\la\ba_\la T
\e_T(\la)\right)^{2A/(\la\ba_\la)+1}.
$$
This last quantity tends to $1$ as $\la\to 0$, which concludes the proof.
\end{proof}

\section{Convergence proof when $\beta \in (0,\infty)$}\label{prbeta}
\setcounter{equation}{0}

We split this section into three parts. First, we handle some
preliminary computations on $SR(\mu_S)$-processes. Next, we
show how to couple the set of times/locations where no seed fall 
(in the discrete model) with
the Poisson measure $\pi_S$. Then we conclude the convergence proof.
In the whole section, we assume $(H_M)$ and $(H_S(\beta))$ for some 
$\beta \in (0,\infty)$. We recall that $\ba_\la$ and $\bn_\la$ are defined
in (\ref{ala}) and (\ref{nla}).
For $A>0$, we set $A_\la = \lfloor A \bn_\la \rfloor$ and 
$I_A^\la=\lb -A_\la,A_\la\rb$.
For $i\in \zz$, we set $i_\la=[i/\bn_\la, (i+1)/\bn_\la)$.

\subsection{Preliminary computations}

First, we will need the following estimate.

\begin{lem}
\label{estimus}
For any $l \in(0,\infty)$ fixed, $\mu_S((\ba_\la l,\infty)) 
\sim m_S \beta l^{-\beta-1} \la $ as $\la \to 0$.
\end{lem}

\begin{proof}
Recall that $\mu_S((t,\infty))dt=m_S \nu_S(dt)$. 
For $\alpha>0$, one may write, using the monotonicity of $x\mapsto
\mu_S((x,\infty))$,
\begin{align*}
\frac{\mu_S((\ba_\la l, \infty))}{\la} \geq& \frac{1}{\alpha \la \ba_\la} 
\int_{\ba_\la l}^{\ba_\la (l + \alpha)} \mu_S((x,\infty))dx \\
=& \frac{m_S}{\alpha \la \ba_\la}
[\nu_S ((\ba_\la l, \infty)) - \nu_S((\ba_\la (l+\alpha), \infty))] \\
=&  \frac{m_S}{\alpha} \left[\frac{\nu_S ((\ba_\la l, \infty))}
{\nu_S ((\ba_\la , \infty))} -\frac{\nu_S ((\ba_\la (l+\alpha), \infty))}
{\nu_S ((\ba_\la , \infty))} 
\right].
\end{align*}
For the last equality, we used that by definition, 
$\nu_S((\ba_\la,\infty))=\la \ba_\la$. Due to $(H_S(\beta))$, we deduce that
for any $\alpha>0$,
$$
\liminf_{\la \to 0} \frac{\mu_S((\ba_\la l, \infty))}{\la} 
\geq\frac{m_S}{\alpha} \left[l^{-\beta}- (l+\alpha)^{-\beta}\right]
\geq m_S \beta (l + \alpha)^{-\beta-1}.
$$
One gets an upper bound by the same way: for any 
$\alpha \in (0,l)$, 
$$
\limsup_{\la \to 0}
\frac{\mu_S((\ba_\la l, \infty))}{\la} \leq \limsup_{\la \to 0}
\frac{1}{\alpha \la \ba_\la} \int_{\ba_\la (l-\alpha)}^{\ba_\la l} 
\mu_S((x,\infty))dx \leq  
m_S \beta (l - \alpha)^{-\beta-1}.
$$
We have proved that for any $\alpha\in (0,l)$, 
$$
m_S \beta (l + \alpha)^{-\beta-1}
\leq \liminf_{\la \to 0} \frac{\mu_S((\ba_\la l, \infty))}{\la} 
\leq \limsup_{\la \to 0} \frac{\mu_S((\ba_\la l, \infty))}{\la} \leq  
m_S \beta (l - \alpha)^{-\beta-1}.
$$ 
Making $\alpha$ tend to $0$ allows us to conclude. 
\end{proof}

Next, we compute the asymptotic probability that on a given site, 
no seed
fall during some {\it large} time interval. By large, we mean
with a length of order $\ba_\la$.

\begin{lem}\label{ilestbalezefourniax}
Let $(T_n)_{n \in \zz}$ be a $SR(\mu_S)$-process (see Subsection \ref{srp}). 
For $\la>0$, $t\geq 0$
and $l>0$, we set
$$
S^{\la}_t(l)= \# \{n \in \zz\, : \; T_n \in [0,\ba_\la t], \; T_{n}-T_{n-1} 
\geq \ba_\la l  \},
$$
which represents the number of {\it delays} with length greater than  
$\ba_\la l$ that end in $[0,\ba_\la t]$. 

(i) For $t>0$ and $l>0$ fixed, as $\la \to 0$, 
$\Pr\left[S^{\la}_t(l)=1\right] \sim t \la \ba_\la \beta l^{-\beta-1}$.

(ii) For $t>0$ and $l>0$ fixed, $\limsup_{\la \to 0} 
(\la \ba_\la)^{-2} \Pr\left[S^{\la}_t(l) \geq 2\right] < \infty$.

(iii) On the event $\{S^\la_t(l)=1\}$, we put $\tau:=T_n$ and 
$L=T_n-T_{n-1}$, where 
$n$ is the unique index such that
$T_n \in [0,t]$ and $T_n-T_{n-1}\geq \ba_\la l$. For all 
$s\in [0,t]$, all $x \in (l,\infty)$,
$\lim_{\la \to 0}\Pr[\tau/\ba_\la \leq s,\; L/\ba_\la \geq x \;
\vert \; S^\la_t(l)=1]= (s/t) (x/l)^{-\beta-1}$ .
\end{lem}

\begin{proof}
Let us recall that the $SR(\mu_S)$-process $(T_n)_{n \in \zz}$ 
is built as follows: one considers an i.i.d. sequence 
$(X_i)_{i \in \zz\setminus \{0\}}$ of $\mu_S$-distributed r.v.,
$X_0$ a $x\mu_S(dx)/m_S$-distributed r.v. and $U$ uniformly distributed
on $[0,1]$. Then we set $T_0=-(1-U) X_0$, $T_1=UX_0$ and  for all $n \geq 1$,
$T_{n+1}=T_n+X_{n}$ and $T_{-n}=T_{-n+1}-X_{-n}$. We also introduce,
for $\la>0$, $l>0$ and $0\leq s \leq t$
$$
S^{\la}_{s,t}(l)= \# \{n \in \zz\, : \; T_n \in [\ba_\la s,\ba_\la t], \; 
T_{n}-T_{n-1} \geq \ba_\la l  \}.
$$

\vip

{\bf Step 1.} First assume that $l\geq t$. Then by construction, 
$S^\la_t(l)\in\{0,1\}$ and 
$\{S^\la_t(l)= 1\}=\{T_1 \leq \ba_\la t,\; T_1-T_0 \geq \ba_\la l\}=
\{UX_0 \leq \ba_\la t,\; X_0 \geq \ba_\la l\}$. Hence
\begin{align*}
\E[S^\la_t(l)] =& \Pr \left[ UX_0 \leq \ba_\la t,\; X_0 \geq \ba_\la l\right]\\
=& \int_{\ba_\la l}^\infty \frac{x \mu_S(dx)}{m_S} \int_0^1 du \indiq_{\{ux
\leq \ba_\la t\}} \\
=& \int_{\ba_\la l}^\infty \frac{x \mu_S(dx)}{m_S} \frac{\ba_\la t}{x}
= \frac{\ba_\la t}{m_S} \mu_S((\ba_\la l,\infty)).
\end{align*}
We used here that since $l\geq t$, for $x\geq \ba_\la l$, there holds 
$\ba_\la t / x \leq 1$.

\vip

{\bf Step 2.} We now show that for any $l>0$, any $t\geq 0$,
$$
\E[S^\la_t(l)]= \frac{\ba_\la t}{m_S} \mu_S((\ba_\la l,\infty)).
$$
Consider $n\geq 1$ such that $t/n\leq l$ and observe that $S^\la_t(l)=
\sum_{i=0}^{n-1} S^\la_{it/n,(i+1)t/n}(l)$. By stationarity, 
we have $\E[S^\la_{it/n,(i+1)t/n}(l)]=\E[S^\la_{t/n}(l)]$ for $i=0,\dots,n-1$, 
which is nothing but
$\frac{\ba_\la t}{n m_S} \mu_S((\ba_\la l,\infty))$ by Step 1. 
The conclusion follows by linearity of expectation.

\vip

{\bf Step 3.} We now check Point (ii). Let
$\rho_1=\inf \{T_n \; : \; n\in \nn, T_n-T_{n-1}\geq \ba_\la l, T_n>0 \}$ and
$\rho_2=\inf \{T_n \; : \; n\in \nn, T_n-T_{n-1}\geq \ba_\la l, T_n>\rho_1\}$. 
Then
$\Pr[S^\la_t(l)\geq 2]= \Pr[\rho_2 \leq \ba_\la t]$. We also observe that
$\Pr[\rho_1 \leq \ba_\la t]=\Pr[S^\la_t(l)\geq 1]\leq \E[S^\la_t(l)]
= \ba_\la t \mu_S((\ba_\la l,\infty))/m_S$.
Denote by $\zeta_{\la,l}$ the law of $\rho_1/\ba_\la$. Then a renewal 
argument shows that
$$
\Pr[S^\la_t(l)\geq 2]=\int_0^{t} \zeta_{\la,l}(dr) f(\la,l,t-r),
$$
where 
$$
f(\la,l,s)= \Pr[\exists\; n \geq 1; X_n\geq \ba_\la l; 
X_1+\dots+X_n\leq \ba_\la s].
$$
We can rewrite this as (recall that $T_1=UX_0\sim \nu_S$)
\begin{align*}
f(\la,l,s)=& \Pr[\exists\; n \geq 1; X_n\geq \ba_\la l; 
UX_0+X_1+\dots+X_n\leq \ba_\la s+UX_0]\\
\leq&  \Pr[\exists\; n \geq 0; X_n\geq \ba_\la l; 
UX_0+X_1+\dots+X_n\leq \ba_\la(s+1)] + \Pr[UX_0\geq \ba_\la]\\
= &   \Pr[S_{s+1}^\la(l)\geq 1] + \nu_S((\ba_\la,\infty))\\
= &   \frac{\ba_\la (s+1)}{m_S} 
\mu_S((\ba_\la l,\infty)) + \la \ba_\la
\end{align*}
thanks to Step 2. As a conclusion,
\begin{align*}
\Pr[S^\la_t(l)\geq 2]\leq& \left(\frac{\ba_\la (t+1)}{m_S} 
\mu_S((\ba_\la l,\infty)) + \la \ba_\la \right) 
\int_0^{t} \zeta_{\la,l}(dr) \\
=& \left(\frac{\ba_\la (t+1)}{m_S} 
\mu_S((\ba_\la l,\infty)) + \la \ba_\la \right) \Pr[\rho_1 \leq \ba_\la T]\\
\leq& \left(\frac{\ba_\la (t+1)}{m_S} \mu_S((\ba_\la l,\infty)) 
+ \la \ba_\la \right)
\frac{\ba_\la t}{m_S} \mu_S((\ba_\la l,\infty)).
\end{align*}
Due to Lemma \ref{estimus}, this last term is equivalent to 
$(\la \ba_\la)^2[(t+1)\beta l^{-\beta-1} +1 ]t\beta l^{-\beta-1}$, 
from which Point (ii) follows.

\vip

{\bf Step 4.} Steps 2 and 3 imply Point (i). Indeed, we clearly have
$\Pr[S^\la_t(l)=1] \leq \E[S^\la_t(l)]=\frac{\ba_\la t}{m_S} 
\mu_S((\ba_\la l,\infty))\sim t \la \ba_\la \beta l^{-\beta-1}$
by Lemma \ref{estimus}. 
Next, using that $S_t^\la(l) \leq 1+t/l$ by construction, 
\begin{align*}
\Pr[S^\la_t(l)=1] =& \E[S^\la_t(l)\indiq_{\{S^\la_t(l) =1\}}]
= \E[S^\la_t(l)]-\E[S^\la_t(l)\indiq_{\{S^\la_t(l) \geq 2\}}] \\
\geq& \frac{\ba_\la t}{m_S} \mu_S((\ba_\la l,\infty)) 
- (1+t/l)\Pr[S^\la_t(l) \geq 2].
\end{align*}
Point (ii) allows us to conclude easily.

\vip

{\bf Step 5.} It remains to check (iii). We thus fix
$0\leq s \leq t$ and $0<l<x$. Then as $\la\to 0$,
$$
\Pr\left[\tau/\ba_\la <s,\;L/\ba_\la>x \;\vert\; S^\la_t(l)=1 \right]
= \frac{\Pr\left[S^\la_s(x)=1 \right]}{\Pr\left[ S^\la_t(l)=1 \right]}
\sim \frac{s \la \ba_\la \beta x^{-\beta-1}}{t \la \ba_\la \beta l^{-\beta-1}}
=(s/t)(x/l)^{-\beta-1}
$$
due to Point (i).
\end{proof}

\subsection{Coupling of seeds}

We aim to couple the Poisson measure 
$\pi_{S}(dt,dx,dl)$ used to define the $LFF(\beta)$-process with
times/places where seeds do not fall in the $FF(\mu_S,\mu_M^\la)$-process.
We would like that roughly, $\pi_S(\{t\}\times i_\la \times\{l\})=1$ if and
only if no seed falls on $i$ during $[\ba_\la (t-l), \ba_\la t]$ 
(and if this is the maximal interval, that is seeds fall in $i$ at times
$\ba_\la(t-l)$ and $\ba_\la t$). We have to consider the finite
Poisson measure $\pi_S$ restricted to the set $\{l>\delta\}$, for some 
arbitrarily small $\delta>0$. 

\begin{prop}\label{coupseed}
Let $A>0$, $T>0$, $\alpha>0$  and 
$\delta>0$ be fixed. For any $\la \in (0,1]$,
it is possible to find a coupling between a Poisson
measure $\pi_S(dt,dx,dl)$ on $[0,\infty)\times
\rr \times  [0,\infty)$ with intensity measure 
$\beta(\beta+1) l^{-\beta-2}dt dx dl$
and an i.i.d. family of $SR(\mu_S)$-processes
$(\tT_n^i)_{i\in \zz,n \in \zz}$ (recall Subsection \ref{srp})
in such a way that for
\begin{align*}
&S^{\la}_T(\delta,i)= \pi_S([0,T]\times i_\la \times [\delta,\infty)),\\
&\tS^{\la}_T(\delta,i)= \#\left\{n \geq 1\, : \; \tT_n^i \in [0,\ba_\la T], \;
\tT_n^i-\tT_{n-1}^i \geq \ba_\la \delta \right\},
\end{align*}
setting 
\begin{align*}
\Omega^S_{A,T,\delta,\alpha}(\la):=
\bigcap_{i \in I_A^\la} \Big(
\left\{ S^{\la}_T(\delta,i)=\tS^{\la}_T(\delta,i)=0 \right\}
\cup \Big\{S^{\la}_T(\delta,i)=\tS^{\la}_T(\delta,i)=1, \hskip1cm& \\
\left|\tau^\la_T(\delta,i)-\frac{\ttau^\la_T(\delta,i)}{\ba_\la}\right|+
\left|L^\la_T(\delta,i)-\frac{\tL^\la_T(\delta,i)}{\ba_\la}\right|<\alpha
\Big\}\Big),&
\end{align*}
there holds 
$$
\lim_{\la \to 0} \Pr(\Omega^S_{A,T,\delta,\alpha}(\la))=1.
$$
On the event $\{S^{\la}_T(\delta,i)=\tS^{\la}_T(\delta,i)=1\}$, we have denoted
by $(\tau^\la_T(\delta,i),L^\la_T(\delta,i))$ the unique element 
$(t,l) \in [0,T]\times [\delta,\infty)$ such that 
$\pi_S(\{t\}\times i_\la \times \{l\} )=1$ and we have put 
$\ttau^\la_T(\delta,i)=\tT_n^i$ 
and $\tL^\la_T(\delta,i)=\tT_n^i-\tT_{n-1}^i$, where
$n\geq 1$ is the unique element of $\nn$ 
such that $\tT_n^i \in [0,\ba_\la T]$ and
$\tT_n^i-\tT_{n-1}^i \geq \ba_\la \delta$.
\end{prop}

\begin{proof} We fix $T>0$, $A>0$, $\delta>0$ and $\alpha>0$.
We divide the proof into several steps. Observe that 
$B_\la:=\cup_{i\in I_A^\la} i_\la
=[-A_\la/\bn_\la,(A_\la+1)/\bn_\la )$
(which is approximately $[-A,A]$).
It of course suffices to build 
$\pi_S$ restricted to $[0,T]\times B_\la \times [\delta,\infty)$
(we abusively still denote by $\pi_S$ this restriction)
and the family $(\tT_n^i)$ for $i\in I_A^\la$ and 
$n \geq 0$ (with our notation, we have $\tT_0^i\leq 0\leq \tT_1^i$).

\vip

{\bf Step 1.} A possible way to build $\pi_S$
(restricted to $[0,T]\times B_\la \times [\delta,\infty)$) is the following.

\vip

(i) Consider a family of i.i.d. r.v. 
$(S_T^{\la}(\delta,i))_{i\in I_A^\la}$ 
following a Poisson distribution
with parameter $ T |i_\la| \int_\delta^\infty \beta(\beta+1) l^{-\beta-2}
dl=  \beta \delta^{-\beta-1} T/\bn_\la$.

\vip

(ii) For each $i\in I_A^\la$ with $S_T^{\la}(\delta,i) > 0$,
pick some i.i.d. r.v. $\{(T^{i,\la}_k,X^{i,\la}_k,L_k^{i,\la})\}_{k=1, \dots, 
S_T^{\la}(\delta,i)}$ with density
$\indiq_{\{t\in [0,T], x\in i_\la, l>\delta\}} (\beta+1) \bn_\la \delta^{\beta+1} 
l^{-\beta-2}/T$.

\vip

Put 
$\pi_S=\sum_{i\in I_A^\la} \sum_{k=1}^{S_T^{\la}(\delta,i)}
\delta_{(T^{i,\la}_k,X^{i,\la}_k, L^{i,\la}_k)}$.

\vip

{\bf Step 2.} Next, we note it is possible to build the family 
$(\tT_n^i)_{i\in I_A^\la, n \geq 0}$ as follows: denote
by $q_k(\la)=\Pr[\tS^\la_T(\delta,i)=k]$ and by
$\Lambda_k^\la$ the law of $(\tT_n^i)_{n\geq 0}$ conditionally on
$\{\tS^\la_T(\delta,i)=k\}$.

\vip

(i) Consider a family of i.i.d. r.v. 
$(\tS^\la_T(\delta,i))_{i\in I_A^\la}$
with law $(q_k(\la))_{k \geq 0}$.

\vip

(ii) For each $i\in I_A^\la$, pick $(\tT_n^i)_{n\geq 0}$
according to $\Lambda_{\tS^\la_T(\delta,i)}^\la$ (conditionally on 
$\tS^\la_T(\delta,i)$).

\vip

{\bf Step 3. } For each $i$, it is possible to couple 
$S^\la_T(\delta,i)$ and $\tS^\la_T(\delta,i)$, distributed  
as in Step 1-(i) and Step 2-(i),
in such a way that
\begin{align*}
\Pr[S^\la_T(\delta,i)=\tS^\la_T(\delta,i)=0]\geq& 1-\la\ba_\la
\beta \delta^{-\beta-1} T (1+\e_{T,\delta}(\la)),\\
\Pr[S^\la_T(\delta,i)=\tS^\la_T(\delta,i)=1] \geq& \la\ba_\la
\beta \delta^{-\beta-1} T (1-\e_{T,\delta}(\la)),
\end{align*}
where $\lim_{\la \to 0} \e_{T,\delta}(\la)=0$. It is classically possible
(see Lemma \ref{gcou}-(i)) to build a coupling in such a way that
\begin{align*}
\Pr[S^\la_T(\delta,i)=\tS^\la_T(\delta,i)=0] &\geq 
\Pr(S^\la_T(\delta,i)=0) \land \Pr(\tS^\la_T(\delta,i)=0),\\
\Pr[S^\la_T(\delta,i)=\tS^\la_T(\delta,i)=1] &\geq 
\Pr(S^\la_T(\delta,i)=1) \land \Pr(\tS^\la_T(\delta,i)=1).
\end{align*}
First, we infer from Lemma \ref{ilestbalezefourniax} that
$\Pr(\tS^\la_T(\delta,i)=0)\geq 1-\la\ba_\la
\beta \delta^{-\beta-1} T (1+\e_{T,\delta}(\la))$ and 
$\Pr(\tS^\la_T(\delta,i)=1)\geq \la\ba_\la
\beta \delta^{-\beta-1} T (1-\e_{T,\delta}(\la))$. Next, since 
$S^\la_T(\delta,i)$ follows a Poisson distribution with parameter 
$\beta \delta^{-\beta-1} T/\bn_\la\sim \la \ba_\la \beta \delta^{-\beta-1}T$, 
we have
$\Pr(S^\la_T(\delta,i)=0)= e^{- \beta \delta^{-\beta-1} T/\bn_\la}\geq 1-\la\ba_\la
\beta \delta^{-\beta-1} T (1+\e_{T,\delta}(\la))$ and there holds
$\Pr(S^\la_T(\delta,i)=1)= [\beta \delta^{-\beta-1} T/\bn_\la] 
e^{- \beta \delta^{-\beta-1} T/\bn_\la}\geq \la\ba_\la \beta \delta^{-\beta-1} T 
(1-\e_{T,\delta}(\la))$. This concludes the step.

\vip 

{\bf Step 4.} We now check that for each 
$i\in I_A^\la$, conditionally 
on $\{S^{\la}_T(\delta,i)=\tS^{\la}_T(\delta,i)=1\}$, we can couple 
$(T^{i,\la}_1,L^{i,\la}_1,X^{i,\la,1})$ and 
$(\tT_n^i)_{n\geq 0}$ in such a way
that for (see the statement)
$$
r_{T,\delta,\alpha}(\la)=
\Pr\left[\left.
\left|\tau^\la_T(\delta,i)-\frac{\ttau^\la_T(\delta,i)}{\ba_\la}\right|+
\left|L^\la_T(\delta,i)-\frac{\tL^\la_T(\delta,i)}{\ba_\la}\right|<\alpha 
\; \right\vert \; Z_i^\la=\tZ_i^\la=1 \right],
$$
there holds $\lim_{\la \to 0} r_{T,\delta,\alpha}(\la)=1$.

\vip

To this end, consider $(\tT_n^i)_{n\geq 0}$ with law $\Lambda_1^\la$ 
(recall Step 2). Denote by $p_{\delta,T}^\la(dt,dl)$ the law of 
$(\ttau_T^\la(\delta,i)/\ba_\la,\tL^\la_T(\delta,i)/\ba_\la)$ 
(under $\Lambda_1^\la$). We know from Lemma \ref{ilestbalezefourniax}-(iii) 
that $p_{\delta,T}^\la(dt,dl)$ goes weakly, as $\la\to 0$, to 
$p_{\delta,T}(dt,dl):=
T^{-1}(\beta+1)\delta^{\beta+1}l^{-\beta-2} \indiq_{\{t\in [0,T], l\geq \delta\}} 
dtdl$. Indeed, observe that $p_{\delta,T}([0,s]\times [x,\infty))=
(s/T)(x/\delta)^{-\beta-1}$ for $s\in[0,T]$ and $x>\delta$. 

\vip

But $p_{\delta,T}(dt,dl)$ is nothing but the law of 
$(\tau_T^\la(\delta,i),L^\la_T(\delta,i))=(T^{i,\la}_1,L^{i,\la}_1)$ 
conditionally on $\{S_T^\la(\delta,i)=1\}$ (recall Step 1-(ii)). 
We easily conclude: 
first, we couple 
$(\ttau_T^\la(\delta,i)/\ba_\la,\tL^\la_T(\delta,i)/\ba_\la)$ and
$(\tau_T^\la(\delta,i),L^\la_T(\delta,i))$ in such a way
that they are close to each other (with a distance smaller than $\alpha$)
with high probability (tending to $1$ when $\la \to 0$), using Lemma 
\ref{gcou}-(iii). Then we choose
$X^{i,\la}_1$ at random, uniformly in $i_\la$, 
independently of everything else
and finally, we pick $(\tT_n^i)_{n\geq 0}$ 
conditionally on  $\{\tS_T^\la(\delta,i)=1\}$ and
$(\ttau_T^\la(\delta,i),\tL^\la_T(\delta,i))$.

\vip

{\bf Step 5.} We finally may build the complete coupling.

\vip

(i) For each $i\in I_A^\la$, consider some coupled r.v.
$(S_T^\la(\delta,i),\tS_T^\la(\delta,i))$ as in Step 3.

\vip

(ii) For $i \in I_A^\la$ such that  
$S^{\la}_T(\delta,i)=\tS^{\la}_T(\delta,i)=1$, couple 
$(T^{i,\la}_1,L^{i,\la}_1,X^{i,\la}_1)$ and $(\tT_n^i)_{n\geq 0}$ as in Step 4.

\vip

(iii) For $i \in I_A^\la$ not meeting the above 
condition (ii), follow the rules of Step 1-(ii) to build 
$(T^{i,\la}_k,X^{i,\la}_k,L^{i,\la}_k)_{1\leq k 
\leq S_T^{\la}(\delta,i)}$ and the rules of Step 2-(ii) to build 
$\{\tT^i_n\}_{n\geq 0}$ (e.g. independently).

\vip

This defines $\{\tT^i_n\}_{n\geq 0,i\in I_A^\la}$ and
$\pi_S:=\sum_{i\in I_A^\la} \sum_{k=1}^{S_T^{\la}(\delta,i)}
\delta_{(T^{i,\la}_k,X^{i,\la}_k, L^{i,\la}_k)}$.

\vip

{\bf Step 6.} With this coupling, using Steps 3 and 4 
and that $|I_A^\la|=2A_\la+1$,
\begin{align*}
\Pr[\Omega^S_{A,T,\delta,\alpha}] \geq&
\Big( 
\Pr\left[S_T^\la(\delta,i)=\tS_T^\la(\delta,i)=0\right]
+ \Pr\Big[S_T^\la(\delta,i)=\tS_T^\la(\delta,i)=1, \\
&\hskip2.5cm |\tau^\la_T(\delta,i)-\ttau^\la_T(\delta,i)/\ba_\la| 
+ L^\la_T(\delta,i)-\tL^\la_T(\delta,i)/\ba_\la|< \alpha  \Big]
\Big)^{2A_\la+1}\\
&\geq \big(1-\la\ba_\la\beta \delta^{\beta-1} T (1+\e_{T,\delta}(\la))
+\la\ba_\la\beta \delta^{\beta-1} T (1-\e_{T,\delta}(\la))r_{T,\delta,\alpha}(\la)  
\big)^{2A_\la+1}. 
\end{align*}
Recall that $\lim_{\la\to 0} \e_{T,\delta}(\la) =0$, that 
$\lim_{\la\to 0} r_{T,\delta,\alpha}(\la) =1$ and that 
$A_\la \leq A /(\la\ba_\la)$. Hence for some function $\e_{T,\delta,\alpha}$
with limit $0$ at $0$,
$$
\Pr[\Omega^S_{A,T,\delta,\alpha}]  \geq \left(1-\la\ba_\la\beta \delta^{\beta-1} 
T \e_{T,\delta,\alpha}(\la)\right)^{2A/(\la\ba_\la)+1}.
$$
This last quantity tends to $1$ as $\la\to 0$, which concludes the proof.
\end{proof}

\subsection{Convergence}

We are now able to conclude. Intuitively, the situation is clear: 
using Proposition \ref{coupling1}, we couple the time/positions 
at which matches fall in the $LFF(\beta)$-process with those of the
$FF(\mu_S,\mu^\la_M)$-process; and using Proposition \ref{coupseed},
we couple the time/positions at which no seed fall in the $LFF(\beta)$-process
with time/positions at which no seed fall during a time interval
of length of order $\ba_\la$ in the $FF(\mu_S,\mu^\la_M)$-process.
Then we only have to show carefully that this is sufficient to couple
the $FF(\mu_S,\mu^\la_M)$-process and the $LFF(\beta)$-process in such
a way that they remain close to each other. But there are
many technical problems: our couplings concern only finite boxes $[0,T]\times
[-A,A]$, do not allow to treat {\it small} time intervals with no seed falling,
etc.
We thus have to localize the processes in space and time
and to work on an event (with
high probability) on which everything works as desired.

\vip

\begin{preuve} {\it of Theorem \ref{converge2}.}
We fix $T>0$, $x_1<\dots<x_p$ and $t_1,\dots,t_p\in [0,T]$. 
We introduce $B>0$ such that $-B<x_1<x_p<B$. We fix $\e>0$ and $a>0$.
Our aim is to check that for all $\la>0$
small enough, there exists a coupling between a $FF(\mu_S,\mu^\la_M)$-process
$(\eta^\la_t(i))_{t\geq 0,i\in \zz}$
and a $LFF(\beta)$-process $(Y_t(x))_{t\geq 0, x\in \rr}$ 
such that, recalling
(\ref{dlambda})
and Proposition \ref{wpbeta}, there holds
\begin{align}\label{cqv}
\Pr\left[\sum_{k=1}^p\bdelta_{T}(D^\la(x_k),D(x_k)) 
+\sum_{k=1}^p \bdelta(D^\la_{t_k}(x_k),D_{t_k}(x_k)) \geq a
\right] \leq \e.
\end{align}
This will of course conclude the proof.

\vip

{\bf Step 1.} Consider two independent 
Poisson measures $\pi_S(dt,dx,dl)$ with intensity measure 
$\beta(\beta+1)l^{-\beta-2} dtdxdl$ and $\pi_M(dt,dx)$ 
with intensity measure $dtdx$. Set, for $A>B$, 
\begin{align*}
\Omega^{S,1}_{A,T}:= &\{\pi_S(\{(t,x,l)\, : \;x\in [B,A], t>T+1,l>t+1\})>0\}\\
&\cap \{\pi_S(\{(t,x,l)\, : \;x\in [-A,-B], t>T+1,l>t+1\})>0\}.
\end{align*}
A simple computation shows that
$$
\Pr[\Omega^{S,1}_{A,T}] \geq 1-2 \exp\left(-\int_B^A dx \int_{T+1}^\infty dt 
\int_{t+1}^\infty \beta(\beta+1)l^{-\beta-2} \right),
$$
so that we can choose $A$ large enough in such a way that 
$\Pr[\Omega^{S,1}_{A,T}]\geq 1-\e/6$. This will ensure us that there
are $\chi_g \in [-A,-B]$ and $\chi_d \in [A,B]$ with $Y_t(\chi_g)=Y_t(\chi_d)=0$
for all $t\in [0,T+1]$ (recall Figure \ref{figLFFbeta}). 
This fixes the value of $A$ for the whole proof.

Next we consider $T_0>T+1$ large enough, so that for
$$
\Omega^{S,2}_{A,T,T_0}=\left\{\pi_S(\{(t,x,l)\, 
: \;t>T_0,t-l<T+1,x\in [-A,A]\}) =0
\right\},
$$
$\Pr[\Omega^{S,2}_{A,T,T_0}]\geq 1-\e/6$. This is possible,
because 
$$
\Pr[\Omega^{S,2}_{A,T,T_0}]=\exp\left(-\int_{-A}^A dx \int_{T_0}^\infty dt 
\int_{t-(T+1)}^\infty \beta(\beta+1)l^{-\beta-2}dl \right),
$$ 
which clearly tends
to $1$ as $T_0$ increases to infinity.
This will ensure us that all the dotted vertical segments in $[-A,A]$
that intersect $[0,T+1]$ end before $T_0$ (see Figure \ref{figLFFbeta}).
This fixes the value of $T_0$ for the whole
proof.

\vip

Next we call $\cX_M = \{x \in [-A,A]\,: \; \pi_M([0,T]\times \{x\})>0\}$ 
and $\cT_M=\{t \in [0,T]\,: \; \pi_M(\{t\}\times[-A,A])>0\}\cup\{0\}$.
Classical results about Poisson measures allow us to choose
$K_M>0$ (large) and $c_M>0$ (small) in such a way that for
\begin{align*}
\Omega^{M,1}_{K_M,c_M} = \Big\{|\cT_M| \leq K_M, \, 
\min_{t,s \in \cT_M, s\ne t} |t-s| > c_M,\; 
\min_{t\in\cT_M,k=1,\dots,p}|t-t_k|>c_M, \\
\; \min_{x\in\cX_M,k=1,\dots,p}|x-x_k|>c_M 
\Big\},
\end{align*}
there holds $\Pr\left[\Omega^{M,1}_{K_M,c_M} \right] \geq 1-\e/6$.

\vip

We can now fix $\delta>0$ for the whole proof, in such a way that 
$$
\delta<c_M/4 \quad \hbox{and} \quad \delta<a/(8ApK_M).
$$
We use this $\delta$ to cutoff the Poisson measure $\pi_S$ (in order
that it has only a finite number of marks) without affecting the
values of the $LFF(\beta)$-process in the zone under study. 

\vip

Next, we consider the finite Poisson measure
$\pi_S^{A,\delta,T_0}$ defined as the restriction of 
$\pi_S$ to the set $[0,T_0]\times[-A,A]\times[\delta,\infty)$.
We define $\cX_S^\delta=\{x\in [-A,A]\,:\; \pi_S([0,T_0]\times \{x\}
\times[\delta,\infty))>0 \}$ and 
$$
\cT^\delta_S=  
\left(\bigcup_{(t,x,l) \in \;{\rm supp}\; \pi_S^{A,\delta,T_0}} \{t,t-l\} \right)
\cap[0,T].
$$
Then for $K_S>0$ large enough and $c_S>0$ small enough, the event
\begin{align*}
\Omega^{S,3}_{K_S,c_S,\delta} =& \Big\{|\cT_S^\delta| \leq  K_S, \, 
\min_{t \in \cT_M, s \in \cT_S^\delta} {|t-s| } > c_S,\; 
\min_{t \in \cT_S^\delta, k=1,\dots,p} {|t-t_k| } > c_S, \\
&\min_{x,y \in \cX_S^\delta, x\ne y} {|x-y| } > c_S,\;
\min_{x \in \cX_S^\delta, y \in \cX_M} {|x-y| } > c_S,\;
\min_{x \in \cX_S^\delta, k=1,\dots,p} {|x-x_k| } > c_S
\Big\}. 
\end{align*}
satisfies $\Pr[\Omega^{S,3}_{K_S,c_S,\delta}] \geq 1 - \e/6$.

\vip

Finally, we fix $\alpha>0$ in such a way that 
$$
\alpha<c_S/4, \quad \alpha<c_M/4,  \quad
\alpha<1/2 \quad \hbox{and} \quad
\alpha< a/(8Ap(2K_S+K_M)).
$$

\vip

{\bf Step 2.} Using Proposition \ref{coupling1}, we know that
for all $\la>0$ small enough, it is possible to couple a family
of i.i.d. $SR(\mu^\la_M)$-processes $(N^{M,\la}_t(i))_{t\geq 0, 
i \in \zz}$ with $\pi_M$ in such a way that
$$
\Omega^M_{A,T}(\la):=\left\{\forall t\in [0,T],\;
\forall i\in I_A^\la, \;
\Delta N^{M,\la}_{\ba_\la t}(i) \ne 0 \; \hbox{iff} \; 
\pi_M(\{t\}\times i_\la) \ne 0
\right\}
$$
satisfies $ \Pr[\Omega^M_{A,T}(\la)]\geq 1-\e/6$. We infer from 
Proposition \ref{coupseed} that for all $\la>0$
small enough, it is possible to couple an i.i.d. family of 
$SR(\mu_S)$-processes $(\tT_n^i)_{i\in \zz,n \geq 0}$ with $\pi_S$
in such a way that for
\begin{align*}
&S^{\la}_{T_0}(\delta,i)= \pi_S([0,T_0]\times\{i_\la\}\times[\delta,\infty)), \\
&\tS^{\la}_{T_0}(\delta,i)= \#\left\{n \geq 1\, : 
\; \tT_n^i \in [0,\ba_\la T_0], \;
\tT_n^i-\tT_{n-1}^i \geq \ba_\la \delta \right\},
\end{align*}
setting 
\begin{align*}
\Omega^S_{A,T_0,\delta,\alpha}(\la):=
\bigcap_{i \in I_A^\la} \Big(
\left\{ S^{\la}_{T_0}(\delta,i)=\tS^{\la}_{T_0}(\delta,i)=0 \right\}
\cup \Big\{S^{\la}_{T_0}(\delta,i)=\tS^{\la}_{T_0}(\delta,i)=1, \hskip1cm& \\
\left|\tau^\la(\delta,i)-\frac{\ttau^\la(\delta,i)}{\ba_\la}\right|+
\left|L^\la(\delta,i)-\frac{\tL^\la(\delta,i)}{\ba_\la}\right|<\alpha
\Big\}\Big),&
\end{align*}
there holds
$\Pr(\Omega^S_{A,T_0,\delta,\alpha}(\la))\geq 1- \e/6$.
On the event $\{S^{\la}_{T_0}(\delta,i)=\tS^{\la}_{T_0}(\delta,i)=1\}$, 
we have denoted
by $(\tau^\la(\delta,i),L^\la(\delta,i))$ the unique element
$(t,l) \in [0,T_0]\times [\delta,\infty)$ such that
$\pi_S(\{t\}\times i_\la \times \{l\} )=1$ and we have put 
$\ttau^\la(\delta,i)=\tT_n^i$ 
and $\tL^\la(\delta,i)=\tT_n^i-\tT_{n-1}^i$, where
$n\geq 1$ is the unique element of $\nn$ 
such that $\tT_n^i \in [0,\ba_\la T_0]$ and
$\tT_n^i-\tT_{n-1}^i \geq \ba_\la \delta$.
We put $N_t^S(i)=\sum_{n\geq 1} \indiq_{\{\tT^i_n\geq t\}}$ 
for all $i\in \zz$, all $t\geq 0$,
which is a family of i.i.d. $SR(\mu_S)$-processes in the sense of 
Definition \ref{sr1}, see Subsection \ref{srp}.

\vip

{\bf Step 3.} We work with the $FF(\mu_S,\mu^{\la}_M)$-process 
$(\eta^\la_t(i))_{t\geq 0,i\in\zz}$ built from $(N^S_t(i))_{t\geq 0,i\in \zz}$
and $(N^{M,\la}_t(i))_{t\geq 0,i\in \zz}$ and the $LFF(\beta)$-process 
$(Y_t(x))_{t\geq 0, x\in \rr}$ built from $\pi_S$ and $\pi_M$,
all these processes being coupled as in Step 2. We consider the 
associated clusters $(D^\la_t(x))_{t\geq 0,x\in \rr}$ and 
$(D_t(x))_{t\geq 0,x\in \rr}$, see (\ref{dlambda}) and Proposition \ref{wpbeta}.
We will work on the event
$$
\Omega_\la=\Omega^{S,1}_{A,T}\cap \Omega^{S,2}_{A,T,T_0} \cap 
\Omega^{M,1}_{K_M,c_M}\cap  
\Omega^{S,3}_{K_S,c_S,\delta}\cap \Omega^M_{A,T}(\la)\cap 
\Omega^{S}_{A,T_0,\delta,\alpha}(\lambda).
$$
Thanks to the previous steps, we know that $\Pr[\Omega_\la]\geq 1-\e$
for all $\la>0$ small enough.
We introduce 
$$
\cS= (\cup_{t \in \cT_M} [t,t+\delta+\alpha] ) 
\cup (\cup_{t \in \cT_S^\delta} [t-\alpha,t+\alpha] ). 
$$
We will prove in the next steps that for $\la>0$ small enough,  
on $\Omega_\la$, 
for all $k=1,\dots,p$, for all $t\in [0,T]$,
\begin{align}\label{objb}
\bdelta(D^\la_t(x_k),D_t(x_k))\leq 4 / \bn_\la + 2A \indiq_{\{t \in \cS\}},
\end{align}
which will imply that
$$
\bdelta_T(D^\la(x_k),D(x_k))\leq 4T / \bn_\la + 2A |\cS|.
$$
This will conclude the proof. Indeed, on $\Omega_\la$, we know that
$t_1,\dots,t_p$ do not belong to $\cS$ 
(thanks to $\Omega^{S,3}_{K_S,c_S,\delta}$ and $\Omega^{M,1}_{K_M,c_M}$ and 
since $c_S>\alpha$ and $c_M>\delta+\alpha$) and that the Lebesgue measure of
$\cS$ is smaller than $K_M \delta + (2K_S+K_M) \alpha$. Thus on $\Omega_\la$,
since $\delta< a/(8ApK_M)$ and $\alpha< a/(8Ap(2K_S+K_M))$,
\begin{align*}
&\sum_{k=1}^p\bdelta_{T}(D^\la(x_k),D(x_k)) 
+\sum_{k=1}^p \bdelta(D^\la_{t_k}(x_k),D_{t_k}(x_k))\\
\leq& p[2A(K_M \delta + (2K_S+K_M) \alpha)+ 4T/\bn_\la + 4/\bn_\la]
\leq a/2 +(4T+4)p/\bn_\la,
\end{align*}
which is smaller than $a$ for all $\la>0$ small enough. This implies
(\ref{cqv}) for all $\la>0$ small enough.

\vip

{\bf Step 4.} Here we localize the processes, on the event $\Omega_\la$. 
Due to $\Omega^{S,1}_{A,T}$,
we know that $\pi_S$ has some marks $(\tau_g,\chi_g,L_g)$ and 
$(\tau_d,\chi_d,L_d)$ such that $-A<\chi_g<-B$, $B<\chi_d<A$, $\tau_g>T+1$,
$\tau_d>T+1$, $L_g>\tau_g+1$ and $L_d>\tau_d+1$. This implies,
by definition of the $LFF(\beta)$-process, that
$Y_t(\chi_g)=Y_t(\chi_d)=0$ for all $t\in [0,T+1]$.
Consequently, for all $t\in [0,T]$ and all $x\in [\chi_g,\chi_d]\supset 
[-B,B]$, we 
have $D_t(x)\subset [\chi_g,\chi_d]$. 

\vip

Set now $g_\la= \lfloor \bn_\la\chi_g \rfloor$ and 
$d_\la=  \lfloor \bn_\la\chi_d \rfloor$. These are those sites 
of $I_A^\la \subset \zz$ 
such that $\chi_g \in (g_\la)_\la$ and $\chi_d \in (d_\la)_\la$. 
We claim
that on $\Omega_\la$, for all $t\in [0,\ba_\la T]$, $\eta_{t}^\la(g_\la)=
\eta_{t}^\la(d_\la)=0$. Consequently on $\Omega_\la$,
we clearly have $C(\eta_{\ba_\la t}^\la, i) 
\subset \lb g_\la+1,d_\la-1\rb$ for all $t\in [0,T]$ and all 
$i \in \lb g_\la+1,d_\la-1\rb$.

\vip

Indeed, consider e.g. the case of $d_\la$. 
Due to $\Omega^S_{A,T_0,\delta,\alpha}(\la)$
and since $S^\la_{T_0}(\delta,d_\la)>0$ (because
$\pi_S$ has the mark $(\tau_d,\chi_d,L_d)$ that falls in
$[0,T_0]\times (d_\la)_\la\times[\delta,\infty)$),
we deduce that 
$S^\la_{T_0}(\delta,d_\la)=\tS^\la_{T_0}(\delta,d_\la)=1$ and that
$|\ttau^\la(\delta,d_\la)/\ba_\la - \tau_d|+
|\tL^\la(\delta,d_\la)/\ba_\la - L_d|<\alpha<1$. But no seed falls
on $d_\la$, by definition, during
$(\ttau^\la(\delta,d_\la)-\tL^\la(\delta,d_\la),\ttau^\la(\delta,d_\la))$.
This last interval contains $[0,\ba_\la T]$: since $\alpha<1/2$,
$\ttau^\la(\delta,d_\la)\geq\ba_\la(\tau_d-\alpha)
\geq\ba_\la(T+1-\alpha)>\ba_\la T$ 
and $\ttau^\la(\delta,d_\la)-\tL^\la(\delta,d_\la)
\leq \ba_\la(\tau_d-L_d+2\alpha)  
\leq  \ba_\la( -1+2\alpha)<0$. This proves that  
$\eta^\la_t(d_\la)=0$ for all $t\in [0,\ba_\la T]$.

\vip

Using furthermore $\Omega^{S,2}_{A,T,T_0}(0)$,
we deduce that on $\Omega_\la$,
$(Y_t(x),D_t(x))_{t\in [0,T], x\in [\chi_g,\chi_d]}$ is completely 
determined by the values of
$\pi_S$ and $\pi_M$ restricted to the boxes 
$[0,T_0]\times [\chi_g,\chi_d] \times
(0,\infty)$ and  $[0,T]\times [\chi_g,\chi_d]$.
By the same way, 
$(\eta^\la_t(i))_{t \in [0, \ba_\la T], i \in \lb g_\la,d_\la\rb}$
is completely determined
by $(N_t^S(i),N_t^{M,\la}(i))_{t \in [0,\ba_\la T], i \in \lb g_\la,d_\la\rb}$.
And we recall that $[-B,B]\subset  [\chi_g,\chi_d] \subset
[-A,A]$.

\vip

{\bf Step 5.} In this whole step, we work on $\Omega_\la$.
We denote by $(\rho_i,\alpha_i)_{i=1,\dots,n}$ the marks of $\pi_M$
in $[0,T]\times[\chi_g,\chi_d]$, 
ordered chronologically ($0=\rho_0<\rho_1<\dots<\rho_n<T$).
For each $k$, we recall that in the $FF(\mu_S,\mu_M^\la)$-process,
there is match falling at time $\ba_\la \rho_k$ on the site
$\lfloor \bn_\la \alpha_k \rfloor$ (recall $\Omega^M_{A,T}(\la)$
and that $x\in i_\la$ iff $i= \lfloor \bn_\la x \rfloor$).
Furthermore, these are the only fires in 
$[0,\ba_\la T]\times \lb g_\la, d_\la \rb$.
For $k=0,\dots,n$, let us  consider the properties
\begin{align*}
(H_k):& \quad \forall \; i \in \lb g_\la, d_\la\rb, \quad
\eta^\la_{\ba_\la \rho_k}(i) = \inf_{\; x \in i_\la} Y_{\rho_k}(x);\\
(H^*_k):& \quad \forall \;i \in \lb g_\la, d_\la\rb, \; \forall \;
t \in [\rho_k ,\rho_{k+1})\setminus \cS, \quad
\eta^\la_{\ba_\la t}(i) = \inf_{x \in i_\la} Y_{t}(x).
\end{align*}
We observe that $(H_0)$ holds: for any $i\in \zz$, $\eta^\la_{0}(i)=0$
and $\inf_{\; x \in i_\la} Y_0(x)=0$
because the set $\{x\in \rr\, : \;Y_0(x)=0\}$ is a.s. dense in $\rr$. Indeed,
recall that 
$Y_0(x)=0$ as soon as $\pi_S(\{(t,x,l)\, : \; l>t\})>0$ and that 
$\int_0^\infty dt \int_t^\infty \beta (\beta+1)l^{-\beta-2}dl = \infty$.

\vip

We are going 
to prove that for $k\in\{0,\dots,n-1\}$, $(H_k)$ implies $(H^*_k)$ and
$(H_{k+1})$. Assume thus that 
$(H_k)$ holds for some $k \in\{0,\dots,n-1\}$. 
We first prove that $(H^*_k)$ holds.

\vip

We recall that 
for all $i\in \lb g_\la,d_\la\rb$,
$S^\la_{T_0}(\delta,i)=\tS^\la_{T_0}(\delta,i)$ is either $0$ or $1$. 
On $\{S^\la_{T_0}(\delta,i)=\tS^\la_{T_0}(\delta,i)=1\}$,
we have $|\tau^\la(\delta,i) - \ttau^\la(\delta,i)/\ba_\la| < \alpha$ 
and  $|L^\la(\delta,i) - \tL^\la(\delta,i)/\ba_\la| < \alpha$. 
Recalling furthermore $\Omega^{M,1}_{K_M,c_M}$ 
and $\Omega^{S,3}_{K_S,c_S,\delta}$, 
using that $\alpha < c_M/4$, we deduce that: 

\vip

$\bullet$ either $\tau^\la(\delta,i)$ 
and $\ttau^\la(\delta,i)/\ba_\la$ both belong to the same interval 
$(\rho_{q(i)},\rho_{q(i)+1})$ for some $q(i)\in \{0,\dots,n-1\}$ or are
both greater than $\rho_{n}$ (then we say that $q(i)=n$);

\vip

$\bullet$ either $\tau^\la(\delta,i)-L^\la(\delta,i)$  and  
$(\ttau^\la(\delta,i)-\tL^\la(\delta,i))/\ba_\la$ both
belong to the same interval $(\rho_{q'(i)},\rho_{q'(i)+1})$ 
for some $q'(i)\in \{0,\dots,n-1\}$ or are 
both greater than $\rho_{n}$ (then we adopt the convention
that $q'(i)=n$), or are both smaller than $0$ (then we say that $q'(i)=-1$).

\vip

We next observe that since $\delta<c_M$, the values of 
$(Y_t(x),D_t(x))_{t \in [0,T]
\setminus \cup_{s \in \cT_M}[s,s+\delta],x\in [\chi_g,\chi_d]}$ 
depends on $\pi_S$ only through its restriction to
$[0,T_0]\times [\chi_g,\chi_d]\times[\delta,\infty)$. Furthermore,
for any $t\in [0,T]\setminus\cup_{s \in \cT_M}[s,s+\delta]$ 
and any $x\in [\chi_g,\chi_d]$, $D_t(x)$ has its extremities in $\cX^\delta_S$.
Have a look
at Figure \ref{figLFFbeta} 
and use the fact that all the dotted segments with length 
smaller than $\delta$ cannot concern two fires. See also 
Remark \ref{ribeta}-(ii).

\vip

We now distinguish several situations 
to prove $(H_k^*)$. We use, in all the cases below, that 
there are no fires in the time interval $(\rho_k,\rho_{k+1})$ 
in the $LFF(\beta)$-process
in the box $[\chi_g,\chi_d]$ and no fire during 
$(\ba_\la\rho_k,\ba_\la\rho_{k+1})$ for the
$FF(\mu_S,\mu^\la_M)$-process in the box $\lb g_\la,d_\la\rb$, 
recall $\Omega^M_{A,T}(\la)$. Let $i \in \lb g_\la,d_\la\rb$.

\vip

{\bf Case (a):} $\eta^\la_{\ba_\la \rho_k}(i) = 1$. 
Then by $(H_k)$, $\inf_{x\in i_\la}Y_{\rho_k}(x)=1$. An obvious
monotonicity argument shows that 
for all $t \in (\rho_k, \rho_{k+1})$, 
$\eta^\la_{\ba_\la t }(i)=\inf_{x\in i_\la}Y_{t}(x)=1$.

\vip

{\bf Case (b):} $\eta^\la_{\ba_\la \rho_k}(i)=0$ and $S^\la_{T_0}(\delta,i)=
\tS^\la_{T_0}(\delta,i)=0$. Then $\inf_{x\in i_\la}Y_t(x)=1$ 
for all $t\in [\rho_k+\delta,\rho_{k+1})$, because in $i_\la$,
there is no dotted segment with length greater than $\delta$ that 
intersect $[0,T]$ (see Figure \ref{figLFFbeta}).
Next, $\tS^\la_{T_0}(\delta,i)=0$ means that all the delays
we wait for a seed (on the site 
$i$ during $[0,\ba_\la T_0]$) 
are smaller than $\ba_\la \delta$.
Consequently, $\eta^\la_{\ba_\la t}(i)=1$ for all $t \in [\rho_k+\delta,
\rho_{k+1} )$. Hence $\eta^\la_{\ba_\la t }(i)=\inf_{x\in i_\la}Y_{t}(x)=1$ 
for all $t\in [\rho_k+\delta, \rho_{k+1})\supset (\rho_k, \rho_{k+1})
\setminus \cS$.

\vip

{\bf Case (c):} $\eta^\la_{\ba_\la \rho_k}(i)=0$ and $S^\la_{T_0}(\delta,i)=
\tS^\la_{T_0}(\delta,i)=1$ and $q(i)<k$. 
Then $\inf_{x\in i_\la}Y_t(x)=1$
for all $t\in [\rho_k+\delta,\rho_{k+1})$, because the only dotted segment
in $i_\la$ with length greater than $\delta$ that intersects $[0,T]$ has
ended before $\rho_k$ (because $q(i)<k$).
Next, the only delay (between two seeds on $i$ during $[0,\ba_\la T]$) 
greater than $\ba_\la \delta$
is ended before $\ba_\la \rho_k$ (because $q(i)<k$), so that
$\eta^\la_{\ba_\la t}(i)=1$ for all $t \in [\rho_k+\delta,
\rho_{k+1} )$. Hence $\eta^\la_{\ba_\la t }(i)=\inf_{x\in i_\la}Y_{t}(x)=1$ 
for all $t\in [\rho_k+\delta, \rho_{k+1})\supset (\rho_k, \rho_{k+1})
\setminus \cS$.

\vip

{\bf Case (d):} $\eta^\la_{\ba_\la \rho_k}(i)=0$, $S^\la_{T_0}(\delta,i)=
\tS^\la_{T_0}(\delta,i)=1$ and $q'(i)\geq k$.  
Then $\inf_{x\in i_\la}Y_t(x)=1$
for all $t\in [\rho_k+\delta,\rho_{k+1})$. Indeed, the only dotted segment
in $i_\la$ with length greater than $\delta$ that intersects $[0,T]$ starts
(strictly) after $\rho_{k}$ (because $q'(i)\geq k$). 
Next, the only delay (between two seeds on $i$ during $[0,\ba_\la T]$) 
greater than $\ba_\la \delta$
will start strictly after $\ba_\la \rho_{k}$ (because $q'(i)\geq k$), so that
$\eta^\la_{\ba_\la t}(i)=1$ for all $t \in [\rho_k+\delta,
\rho_{k+1} )$. Hence $\eta^\la_{\ba_\la t }(i)=\inf_{x\in i_\la}Y_{t}(x)=1$ 
for all $t\in [\rho_k+\delta, \rho_{k+1})\supset (\rho_k, \rho_{k+1})
\setminus \cS$.

\vip

{\bf Case (e):} $\eta^\la_{\ba_\la \rho_k}(i)=0$ and $S^\la_{T_0}(\delta,i)=
\tS^\la_{T_0}(\delta,i)=1$ and $q'(i)< k \leq q(i)$.
Then $\eta^\la_{\ba_\la t}(i)=0$ for all $t\in [\rho_k,
(\ttau^\la(\delta,i)/\ba_\la) \land \rho_{k+1})$ 
and $\eta^\la_{\ba_\la t}(i)=1$ for all 
$t\in [(\ttau^\la(\delta,i)/\ba_\la)\land \rho_{k+1},\rho_{k+1})$ 
(because no seed fall
on $i$ during $[\ttau^\la(\delta,i)-\tL^\la(\delta,i),\ttau^\la(\delta,i))
\ni \rho_k$
and a seed falls on $i$ at time $\ttau^\la(\delta,i)$).
By $(H_k)$, we also know that $\inf_{x\in i_\la} Y_{\rho_{k}}(x)=0$.
Calling $(\tau^\la(\delta,i),x_0,L^\la(\delta,i))$ the only mark
of $\pi_S$
that falls in $[0,T_0]\times i_\la \times [\delta,\infty)$, we claim
that necessarily, $Y_{\rho_{k}}(x_0)=0$. Indeed, all the other dotted segments
in $i_\la$ that intersect $[0,T]$ have a length smaller than $\delta<c_M\leq
\rho_k-\rho_{k-1}$.
Thus if $\inf_{x\in i_\la}Y_{\rho_k-}(x)=0$, necessarily,
$Y_{\rho_k-}(x_0)=0$ and thus  $Y_{\rho_k}(x_0)=0$. If now 
$\inf_{x\in i_\la}Y_{\rho_k-}(x)=1$, then $i_\la$ is connected at time time
$\rho_k-$, whence the fire at time $\rho_k$ 
burns completely $i_\la$ (because $\inf_{x\in i_\la} Y_{\rho_{k}}(x)=0$
by assumption), so that in particular, $Y_{\rho_k}(x_0)=0$.
Then we have to separate two situations.

\vip

$\bullet$ If $\tau^\la(\delta,i)< \rho_{k}+\delta$, then we easily deduce
that $\inf_{x\in i_\la} Y_t(x)=1$ for $t\in [\rho_k+\delta,\rho_{k+1})$.
Recalling that  $\eta^\la_{\ba_\la t}(i)=1$ for all 
$t\in [(\ttau^\la(\delta,i)/\ba_\la)\land \rho_{k+1},\rho_{k+1})$ and that
$|\ttau^\la(\delta,i)/\ba_\la -\tau^\la(\delta,i)|<\alpha$, 
we easily conclude that $\eta^\la_{\ba_\la t}(i)=1$ for 
$t\in [\rho_k+\alpha+\delta, \rho_{k+1})$. Thus $\eta^\la_{\ba_\la t}(i)
= \inf_{x\in i_\la} Y_{t}(x)$ for $t \in [\rho_{k}+\delta+\alpha,\rho_{k+1})
\supset [\rho_{k},\rho_{k+1})\setminus \cS$.

\vip

$\bullet$ If now $\tau^\la(\delta,i) \geq \rho_{k}+\delta$, 
then we have, by construction,  
$\inf_{x\in i_\la} Y_t(x)=0$ for 
$t\in [\rho_k, \tau^\la(\delta,i) \land \rho_{k+1})$
and $\inf_{x\in i_\la} Y_t(x)=1$ for $t\in [\tau^\la(\delta,i) \land \rho_{k+1},
\rho_{k+1})$. Recalling the values of $\eta^\la_{\ba_\la t}(i)$ and that 
$|\ttau^\la(\delta,i)/\ba_\la -\tau^\la(\delta,i)|<\alpha$, one easily
concludes that $\eta^\la_{\ba_\la t}(i)
= \inf_{x\in i_\la} Y_{t}(x)$ for $t \in [\rho_{k},\rho_{k+1})\setminus \cS$
(because $\tau^\la(\delta,i) \in \cT^\delta_S$ whence 
$[\tau^\la(\delta,i)-\alpha,\tau^\la(\delta,i)+\alpha] \subset \cS$).

\vip

We have proved $(H_k^*)$ and this implies that 
$$
\forall \; i\in \lb g_\la,d_\la\rb, \quad
\eta^\la_{\ba_\la \rho_{k+1}-}(i)
=\inf_{x\in i_\la} Y_{\rho_{k+1}-}(x).
$$

It remains to prove $(H_{k+1})$.

\vip

Consider the ignited cluster
$[a,b]=D_{\rho_{k+1}-}(\alpha_{k+1})$ in the $LFF(\beta)$-process. 
Then the ignited cluster in the $FF(\mu_S,\mu^\la_M)$-process
at time $\ba_\la \rho_{k+1}$ (due to a match falling on the site
$\lfloor \bn_\la \alpha_{k+1} \rfloor$)
is nothing but 
$I_{k+1}^\la:=\{i \in \lb g_\la,d_\la\rb
\,: \;i_\la \subset D_{\rho_{k+1}-}(\alpha_{k+1})\}$,
at least if $\la$ is small enough (such that $1/\bn_\la<c_S$). 
Indeed, we have
$\eta^\la_{\ba_\la \rho_{k+1}-}(i)=\inf_{x\in i_\la}Y_{\rho_{k+1}-}(x)=1$ 
for all $i$ such that $i_\la \subset D_{\rho_{k+1}-}(\alpha_{k+1})$
and (on the two boundary sites) 
$\eta^\la_{\ba_\la \rho_{k+1}-}(i)=\inf_{x\in i_\la}Y_{\rho_{k+1}-}(x)=0$
for $i$ such that $i_\la \not\subset D_{\rho_{k+1}-}(\alpha_{k+1})$
with $i_\la \cap D_{\rho_{k+1}-}(\alpha_{k+1})\ne \emptyset$. And
for $\la>0$ small enough (such that $1/\bn_\la < c_S$),
$\lfloor \bn_\la \alpha_{k+1} \rfloor
\in I_{k+1}^\la$ (because $[a+1/\bn_\la,b-1/\bn_\la]\subset I_{k+1}^\la$ 
by the previous study,
because $D_{\rho_{k+1}-}(\alpha_{k+1})=[a,b]$ has its extremities $a,b$ in 
$\cX^\delta_S$, because
$\alpha_{k+1}\in \cX_M$ and because the
distance between $\cX^\delta_S$ and $\cX_M$ is greater than $c_S$,
recall $\Omega^{S,3}_{K_S,c_S,\delta}$, so that
actually, $\alpha_{k+1}\in [a+c_S,b-c_S]$).

\vip

Then on the one hand, for all
$i\in \lb g_\la,d_\la\rb$, we have 
$\inf_{x\in i_\la}Y_{\rho_{k+1}}(x)=\inf_{x\in i_\la}Y_{\rho_{k+1}-}(x)$ if
$i_\la \cap D_{\rho_{k+1}-}(\alpha_{k+1}) =\emptyset$
and $\inf_{x\in i_\la}Y_{\rho_{k+1}}(x)=0$ if 
$i_\la \cap D_{\rho_{k+1}-}(\alpha_{k+1}) \neq \emptyset$:
the first case is obvious
and the second one follows from
the fact that a.s., $\pi_S(\{(t,x,l)\, : \; 
t \geq \rho_{k+1}, x\in i_\la\cap D_{\rho_{k+1}-}(\alpha_{k+1}), t-l<\rho_{k+1}
\})=\infty$ (but this concerns marks $(t,x,l)$
with a very small length $l>0$).

\vip

On the other hand,  for all $i\in \lb g_\la,d_\la\rb$, we have
$\eta^\la_{\rho_{k+1}}(i)=\eta^\la_{\rho_{k+1}-}(i)$ if $i \notin I_{k+1}^\la$
and $\eta^\la_{\rho_{k+1}}(i)=0$ if $i \in I_{k+1}^\la$.

\vip

As a conclusion, for all $i\in \lb g_\la,d_\la\rb$, 

\vip

$\bullet$ if $i_\la \subset D_{\rho_{k+1}-}(\alpha_{k+1})$, i.e.
if $i \in I_{k+1}^\la$, then we have seen that
$\eta^\la_{\ba_\la \rho_{k+1}}(i)= 0 = \inf_{x\in i_\la}Y_{\rho_{k+1}}(x)$;

\vip

$\bullet$ if $i_\la \cap D_{\rho_{k+1}-}(\alpha_{k+1}) = \emptyset$
(hence $i\notin I_{k+1}^\la$), then we
have seen that 
$\eta^\la_{\ba_\la \rho_{k+1}}(i)= \eta^\la_{\ba_\la \rho_{k+1}-}(i) 
=\inf_{x\in i_\la}Y_{\rho_{k+1}-}(x) = \inf_{x\in i_\la}Y_{\rho_{k+1}}(x)$;

\vip

$\bullet$ if $i \notin I_{k+1}^\la$ but 
$i_\la \cap D_{\rho_{k+1}-}(\alpha_{k+1})\neq \emptyset$,
then we have seen that $\inf_{x\in i_\la} Y_{\rho_{k+1}}(x)=0$
and $\eta^\la_{\ba_\la \rho_{k+1}}(i)=0$ because
$\eta^\la_{\ba_\la \rho_{k+1}-}(i)=0$ (since then $i$ lies at the boundary
of $I_{k+1}^\la$).

\vip

Hence $(H_{k+1})$ holds.

\vip

{\bf Step 6.}  We finally can prove (\ref{objb})
on $\Omega_\la$ and this will conclude the proof.
First, we know from Step 4 that for all $t\in [0,T]$, all $k=1,\dots,p$,
$D_t(x_k)\subset[\chi_g,\chi_d]\subset [-A,A]$ and that $C(\eta^\la_{\ba_\la t},
\lfloor \ba_\la x_k \rfloor)\subset \lb g_\la+1,d_\la-1\rb$
whence $D^\la_t(x_k)\subset [-A,A]$ (because 
$(g_\la+1)/\bn_\la \geq \chi_g \geq -A$
and $(d_\la-1)/\bn_\la \leq \chi_d \leq A$). This obviously implies that 
$\bdelta(D_t(x_k),D^\la_t(x_k))\leq 2A$.

\vip

Next, Step 5 implies that for all $t\in [0,T]\setminus \cS$
(or rather for all $t\in[0,\rho_n)\setminus \cS$, 
but the extension is straightforward),
for all $i\in \lb g_\la,d_\la\rb$, 
$\eta^\la_{\ba_\la t}(i) = \inf_{x \in i_\la} (Y_{t}(x))$.
This implies that for all $t\in [0,T]\setminus \cS$,
for all $k=1,\dots,p$, $\bdelta(D_t^\la(x_k),D_t(x_k))\leq
4/\bn_\la$ as desired. 

Indeed, assume that $D_t(x_k)=[a,b]\subset [\chi_g,\chi_d]$
for some $t\in [0,T]\setminus \cS$. Recall that
$a,b \in \cX_S^\delta$.
We have $Y_t(y)=1$ for all $y\in (a,b)$ and $Y_t(a)=Y_t(b)=0$. 
Hence we deduce
that $\eta^\la_{\ba_\la t}(i)=1$ for all 
$i\in \lb \lfloor a \bn_\la \rfloor+1,\lfloor b \bn_\la \rfloor-1\rb$
and that
$\eta^\la_{\ba_\la t}(\lfloor a \bn_\la \rfloor)=
\eta^\la_{\ba_\la t}(\lfloor b \bn_\la \rfloor )=0$. Next, we observe that
for $\la>0$ small enough,
$\lfloor a \bn_\la \rfloor < \lfloor x_k \bn_\la \rfloor 
< \lfloor b \bn_\la \rfloor$. Indeed, on $\Omega_\la$, we have, 
since $a,b\in \cX_S^\delta$, $|x_k-a|>c_S$ and $|b-x_k|>c_S$.
We finally obtain
$C(\eta^\la_{\ba_\la t},
\lfloor x_k \bn_\la \rfloor)=\lb \lfloor a \bn_\la \rfloor+1
,\lfloor b \bn_\la \rfloor-1\rb$, whence 
$D^\la_t(x_k)=[(\lfloor a \bn_\la \rfloor+1)/\bn_\la, 
(\lfloor b \bn_\la \rfloor-1)/\bn_\la ]$. Recalling that 
$D_t(x)=[a,b]$, one easily deduces that
$\bdelta(D_t^\la(x_k),D_t(x_k))\leq 4/\bn_\la$.
\end{preuve}

\section{Cluster-size distribution when $\beta\in (0,\infty)$}

This section is entirely devoted to the 

\begin{preuve} {\it of Corollary \ref{co2}.}
We thus fix $\beta \in (0,\infty)$ and assume $(H_M)$ and $(H_S(\beta))$. 
For each $\la>0$, 
we consider a $FF(\mu_S,\mu_M^\la)$-process $(\eta^\la_t(i))_{t\geq 0,i\in\zz}$.
Let also  $(Y_t(x))_{t\geq 0, x\in \rr}$ be a $LFF(\beta)$-process.
We know from Theorem \ref{converge2} that $|C(\eta^\la_t,0)|/\bn_\la$
goes in law to $|D_t(0)|$, for any $t>0$.
In Step 1 below, we will check that for $t>0$, the law of $|D_t(0)|$
does not charge points. Thus for any $B\geq 0$, $t>0$, we will have
$\lim_{\la \to 0} \Pr[|C(\eta^\la_t,0)| \geq \bn_\la B] = \Pr[|D_t(0)|\geq B]$.
In Steps 2 to 6, we will check that there are some
constants $0<c_1<c_2$ and $0<\kappa_1<\kappa_2$ such that if $t>1$, for any 
$B\geq 2$, $\Pr[|D_t(0)|>B] \in [c_1 e^{-\kappa_2 B}, c_2 e^{-\kappa_1 B}]$.
One immediately checks that this implies 
$\Pr[|D_t(0)|>B] \in [c_1 e^{-2\kappa_2} e^{-\kappa_2 B}, (c_2 \lor e^{2\kappa_1})
e^{-\kappa_1 B}]$
for all $t>1$,
$B>0$ and this will conclude the proof.

\vip

{\bf Step 1.} The goal of this step is to check that for any $t>0$ fixed,
the law of $|D_t(0)|$ does not charge points.

\vip

Consider the first mark $(T_d,\chi_d,L_d)$ of $\pi_S$ 
on the right of $0$ ($\chi_d>0$) such that $[0,t]\subset [T_d-L_d,T_d]$.
Consider a similar mark  $(T_g,\chi_g,L_g)$ of $\pi_S$ with $\chi_g<0$.
\vip
Then $Y_s(\chi_g)=Y_s(\chi_d)=0$ for all $s\in [0,t]$,
so that fires falling outside $[\chi_g,\chi_d]$ cannot affect $0$
during $[0,t]$.
\vip
Next, denote by $(T_M,X_M)$ the instant/position of the last match falling 
before $t$ in $[\chi_g,\chi_d]$. Then a.s., $t-T_M>0$,
and $D_t(0)$ is of the form $[a,b]$, for some 
marks $(T_a,a,L_a)$ and $(T_b,b,L_b)$ of $\pi_S$ satisfying 
$\chi_g\leq a < 0< b < \chi_d$, $T_a-L_a<T_M$, $T_b-L_b<T_M$,
$T_a>t$ and $T_b>t$. There are a.s. a finite number of such marks
(because a.s., $\int_t^\infty ds \int_{s-T_M}^\infty  \beta(\beta+1)l^{-\beta-2}dl 
=(t-t_M)^{-\beta} <\infty$),
and their (spatial) positions clearly have densities, whence the result.

\vip

{\bf Step 2.} For $t>1$, $a\in \rr$, we consider the event
$\Omega_{t,a}$ defined as follows, see Figure \ref{figOmegatA}:

\vip

(i) $\pi_M$ has exactly one mark $(T_M,X_M)$ in $[t-1,t]\times [a,a+1]$
and there holds $(T_M,X_M) \in [t-2/3,t-1/2]\times [a+1/4,a+3/4]$;

\vip

(ii) $\pi_S$ has one mark $(T_g,X_g,L_g)$ such that $T_g-L_g<t-1<t<T_g$ and
$X_g \in [a,a+1/4]$ and one mark 
$(T_d,X_d,L_d)$ such that $T_d-L_d<t-1<t<T_d$ and
$X_d \in [a+3/4,a+1]$ (recalling Figure \ref{figLFFbeta}, there are
dotted vertical segments in $[a,a+1/4]$ and in 
$[a+3/4,a+1]$ that run across $[t-1,t]$);

\vip

(iii) all the other marks $(T,X,L)$ of $\pi_S$ with $X \in [a,a+1]$
and $[T-L,T]\cap [t-1,t] \ne \emptyset$ satisfy $L<1/4$ (recalling Figure 
\ref{figLFFbeta}, 
all the other vertical dotted segments in $[a,a+1]$ that intersect 
$[t-1,t]$ have a length smaller than $1/4$).

\begin{figure}[b] 
\fbox{
\begin{minipage}[c]{0.95\textwidth}
\centering
\includegraphics[width=8cm]{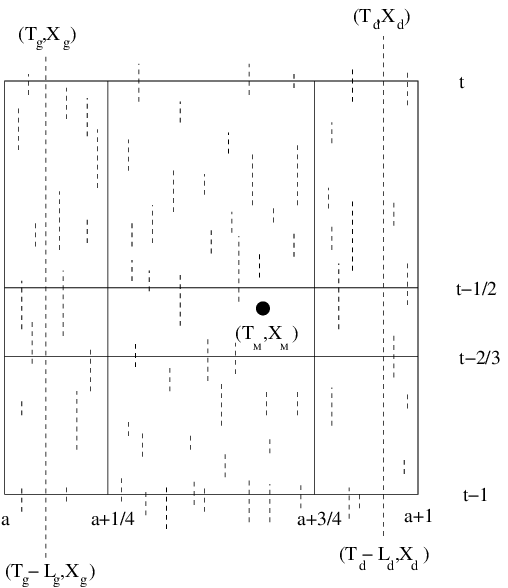}
\caption{The event $\Omega_{t,A}$.}
\label{figOmegatA}
\end{minipage}
}
\end{figure}

\vip

{\bf Step 3.} In this step, we prove 
that on $\Omega_{t,a}$, we have either $Y_s(X_g)=0$
for all $s\in [t-1/2,t]$ or $Y_s(X_d)=0$ for all $s\in [t-1/2,t]$. 
We distinguish two situations.

\vip

$\bullet$ First assume that $[X_g,X_d]$ is
connected at time $T_M-$ (that is $Y_{T_M-}(x)=1$ for all $x\in [X_g,X_d]$). 
Since $X_M \in [X_g,X_d]$, the fire 
destroys the cluster and thus we deduce that $Y_s(X_g)=0$ for
all $s \in [T_M,T_g) \supset[t-1/2,t] $ and that $Y_s(X_d)=0$ for
all $s \in [T_M,T_d) \supset[t-1/2,t]$.

\vip

$\bullet$ Next assume that $[X_g,X_d]\subset [a,a+1]$
is not connected at time $T_M-$ (that is, there is some
$x_0 \in [X_g,X_d]$ such that $Y_{T_M-}(x_0)=0$).
Then we claim that either $Y_{T_M-}(X_g)=0$ 
(then $Y_s(X_g)=0$ for all $s \in [T_M,T_g)\supset [t-1/2,t]$) or
$Y_{T_M-}(X_d)=0$ 
(then $Y_s(X_d)=0$ for all $s \in [T_M,T_d)\supset [t-1/2,t]$).
Indeed, recall that all the dotted segments that intersect $[t-1,t]$
in $(X_g,X_d)$ have a length smaller than $1/4$.
Thus if $[X_g,X_d]$ is disconnected at time $T_M-$ due to a fire
that started before $t-1$, it can be only with $x_0=X_g$ or $x_0=X_d$,
whence the conclusion. 
But if now $[X_g,X_d]$ is disconnected at time $T_M-$ due to a fire
that started at some time $\tau \in [t-1,T_M)$ at some place
$\chi \notin [a,a+1]$ (since there are no fires in
$[a,a+1]$ during $[t-1,T_M)$), this necessarily also concerns one of the
extremities $X_g$ or $X_d$ of $[X_g,X_d]$. Thus in any case,
we obtain either $Y_{T_M-}(X_g)=0$ or $Y_{T_M-}(X_d)=0$ as desired.

\vip
{\bf Step 4.} Let us prove that $p:=\Pr[\Omega_{t,a}]>0$.
This value will obviously does not depend on $a\in \rr$, $t\geq 1$,
by homogeneity in $(s,x)$ of the Poisson measures $\pi_M(ds,dx)$
and $\pi_S(ds,dx,dl)$.
Define the zones 
\begin{align*}
&A_M=(t-2/3,t-1/2)\times (a+1/4,a+3/4),\\
&B_M= ((t-1,t)\times (a,a+1)) \setminus A_M, \\
&A_S=\{(s,x,l), \; x\in (a,a+1/4), s>t>t-1>s-l\},\\
&B_S=\{(s,x,l), \; x\in (a+3/4,a+1), s>t>t-1>s-l\},\\
&C_S=\{(s,x,l), \; x\in (a+1/4,a+3/4), s>t>t-1>s-l\},\\
&D_S=\{(s,x,l), \; x\in (a,a+1), 
[s-l,s]\cap[t-1,t]\ne \emptyset, l>1/4 \} \setminus (A_S\cup B_S\cup C_S).
\end{align*}
The zones $A_M$ and $B_M$ are disjoint and for  
$\zeta_M(ds,dx)=dsdx$,
$\zeta_M(A_M) = 1/12$ and $\zeta_M(B_M) = 11/12$.
The zones $A_S,B_S,C_S,D_S$ are also disjoint and simple computations
show that, for $\zeta_S(ds,dx,dl)=\beta(\beta+1)l^{-\beta-2}dsdxdl$,
$\zeta_S(A_S)=\zeta_S(B_S)=1/4$, $\zeta_S(C_S)= 1/2$
and $\zeta_S(D_S)=4^\beta(5\beta+1)-1$. Consequently,
recalling that $\pi_M$ and $\pi_S$ are independent Poisson measures
with intensity measures $\zeta_M$ and $\zeta_S$,
\begin{align*}
\Pr[\Omega_{t,a,\delta}]&=\Pr\left(\pi_M(A_M)=1,\pi_M(B_M)=0,
\pi_S(A_S)=\pi_S(B_S)=1,\pi_S(C_S)=\pi_S(D_S)=0 \right)\\
&=\zeta(A_M)e^{-\zeta_M(A_M)} e^{-\zeta_M(B_M)}\zeta_S(A_S)e^{-\zeta_S(A_S)}
\zeta_S(B_S)e^{-\zeta_S(B_S)} e^{-\zeta_S(C_S)}e^{-\zeta_S(D_S)}\\
&=(1/12)e^{-1/12}e^{-11/12}(1/4)^2e^{-1/2}e^{-4^\beta(5\beta+1)+1}=:p>0.
\end{align*}

\vip

{\bf Step 5.} We clearly have, for any $t\geq 1$, any
$B\geq 2$, 
$$
\{|D_t(0)|\geq B\}\subset \{\forall \; x \in [0,B/2], \; Y_t(x)=1\} 
\cup \{\forall \; x \in [-B/2,0], \; Y_t(x)=1\},
$$
whence $\Pr[|D_t(0)|\geq B] \leq 2\Pr[\forall \; x \in [0,B/2], \; Y_t(x)=1]$ 
by symmetry.
Furthermore, Step 3 implies that 
$$
\{\forall \; x \in [0,B/2], \; Y_t(x)=1\} 
\subset \Omega_{t,0}^c\cap \Omega_{t,1}^c\cap 
\dots \cap \Omega_{t,\lfloor B/2 -1\rfloor}^c.
$$
Using then Step 4 (and some obvious independence arguments), we get
$$
\Pr[|D_t(0)|\geq B] \leq 2 (1-p)^{\lfloor B/2-1 \rfloor+1}
\leq 2 (1-p)^{B/2-1}.
$$
Consequently, for all $t\geq 1$, all $B\geq 2$, $\Pr[|D_t(0)|\geq B]
\leq c_2 e^{-\kappa_1 B}$, with $c_2=2/(1-p)$ and $\kappa_1=-[\log(1-p)]/2$.

\vip

{\bf Step 6.} Next, we consider the event $\tOmega_{t,B}$ on which:

\vip

(i) $\pi_M([t-1/2,t]\times [0,B])=0$;

\vip

(ii) all the marks $(T,X,L)$ of $\pi_S$ with $X \in [0,B]$ satisfy
either $T<t$ or $T-L>t-1/2$)
(this means that there is no dotted vertical segment running across
$[t-1/2,t]$ in $[0,B]$).

\vip

An easy computation as in Step 4 implies that
\begin{align*}
\Pr[\tOmega_{t,B}]=&\exp\left(-\int_{t-1/2}^t\int_0^B dsdx - 
\int_t^\infty ds \int_0^B dx \int_{s-t+1/2}^\infty dl \beta(\beta+1)l^{-\beta-2}
\right)\\
=& \exp\left(-B/2- 2^\beta B \right).
\end{align*}

We claim that on 
$\Omega_{t,-1}\cap\tOmega_{t,B}\cap \Omega_{t,B}$, there holds
$[0,B]\subset D_t(0)$, whence $|D_t(0)|\geq B$.
Indeed, we know from Step 3 that there is $\chi_0\in [-1,0]$
and $\chi_1 \in [B,B+1]$ such that $Y_s(\chi_0)=Y_s(\chi_1)=0$
for all $s\in [t-1/2,1]$. Thus the fires starting outside $[\chi_0,\chi_1]$
do not affect the zone $[\chi_0,\chi_1]$ during $[t-1/2,t]$.
Furthermore, there are no fires starting
in $[\chi_0,\chi_1]$ during $[t-1/2,t]$. At last, since all the dotted segments
in $[0,B]$ intersecting $\{t\}$ have started after $t-1/2$.
We easily conclude that $Y_t(x)=1$ for all $x\in [0,B]$.

Using finally some obvious
independence arguments, we get
$$
\Pr[|D_t(0)|\geq B] \geq \Pr[\Omega_{t,-1}\cap\tOmega_{t,B}\cap \Omega_{t,B}]
\geq p^2  \exp\left(-B/2- 2^\beta B \right)= c_1e^{-\kappa_2 B},
$$
with $c_1=p^2$ and $\kappa_2=1/2+2^\beta$.
\end{preuve}

\section{Convergence proof when $\beta=0$}\label{przero}
\setcounter{equation}{0}

This case is simpler than the case $\beta\in (0,\infty)$, 
but a little work is however needed.
We also divide the section into three parts: preliminaries,
coupling of seeds and convergence proof.
In the whole section, we assume $(H_M)$ and $(H_S(0))$.
We recall that $\ba_\la$ and $\bn_\la$ are defined
in (\ref{ala}) and (\ref{nla}).
For $A>0$, we set $A_\la = \lfloor A \bn_\la \rfloor$ and 
$I_A^\la=\lb -A_\la,A_\la\rb$.
For $i\in \zz$, we set $i_\la=[i/\bn_\la, (i+1)/\bn_\la)$.

\subsection{Preliminaries}

The proof will use the following estimate.

\begin{lem}\label{estimusz}
For any $l \in(0,\infty)$ fixed, we have 
$\lim_{\la \to 0} \la^{-1} \mu_S((\ba_\la l,\infty))=0$.
\end{lem}

\begin{proof}
Using the monotonicity of $\mu_S((x,\infty))$
and since $\mu_S((x,\infty))dx=m_S \nu_S(dx)$,
\begin{align*}
\frac{\mu_S((\ba_\la l, \infty))}{\la} \leq& \frac{2}{\la \ba_\la l} 
\int_{\ba_\la l/2}^{\ba_\la l} \mu_S((x,\infty))dx \\
=& \frac{2m_S}{\la \ba_\la l}
[\nu_S ((\ba_\la l /2, \infty)) - \nu_S((\ba_\la l, \infty))] \\
=& \frac{2m_S}{l}
\left[\nu_S ((\ba_\la l /2, \infty))/\nu_S((\ba_\la , \infty))  - 
\nu_S ((\ba_\la l , \infty))/\nu_S((\ba_\la , \infty))\right].
\end{align*}
For the last equality, we used that by definition, 
$\nu_S((\ba_\la,\infty))=\la \ba_\la$. 
Using $(H_S(0))$, we easily conclude.
\end{proof}

The following statement contains some crucial
facts about accelerated $SR(\mu_S)$-processes under $(H_S(0))$.

\begin{lem}\label{seedszero}
Let $(T_n)_{n \geq 1}$ be a $SR(\mu_S)$-process (see
Subsection \ref{srp}).
For $\la>0$, $t\geq 0$ and $l>0$, we set
$$
R^{\la}_t(l)= \# \{n \geq 1\,: \; T_n \in [0,\ba_\la t], \; T_{n+1}-T_{n} 
\geq \ba_\la l  \},
$$
which represents the number of {\it delays} with length greater than  
$\ba_\la l$ that start in $[0,\ba_\la t]$. 

(i) For any $T>0$, $\Pr [T_1\geq \ba_\la T]=\nu_S((\ba_\la T,\infty))
\sim \la \ba_\la$ as $\la \to 0$.

(ii) For any $T>0$, any $l>0$, $\E[R^\la_T(l)]=\ba_\la T 
\mu_S((\ba_\la l,\infty))/m_S=o(\la\ba_\la)$ as $\la \to 0$.
\end{lem}

\begin{proof}
Point (i) is immediate: $\nu_S$ is the law of $T_1$ and since
$\la\ba_\la=\nu_S((\ba_\la,\infty))$ by definition, one has 
$\nu_S((\ba_\la T,\infty))= \la \ba_\la
\nu_S((\ba_\la T,\infty))/ \nu_S((\ba_\la,\infty))$.
One concludes
using $(H_S(0))$.
Point (ii) is slightly more delicate. First, we complete the
$SR(\mu_S)$-process $(T_n)_{n \geq 1}$ in $(T_n)_{n \in \zz}$, see Subsection
\ref{srp}. Then we observe that since $T_0<0<T_1$,
$$
R^{\la}_T(l)= \# \{n \in \zz\, : \; T_n \in [0,\ba_\la T], \; T_{n+1}-T_{n} 
\geq \ba_\la l  \}.
$$
Next, we set $\tau_n=\ba_\la T -T_{-n}$ and we introduce $n_0$ 
such that $\tau_{n_0}<0<\tau_{n_0+1}$. We put
$\tT_n=\tau_{n_0+n}$. Then  $(\tT_{n})_{n \in \zz}$ 
is also a $SR(\mu_S)$-process (see Subsection \ref{srp}). 
We have
\begin{align*}
R^{\la}_T(l)
=& \# \{n \in \zz \, : \; \ba_\la T-T_n \in [0,\ba_\la T], \; (\ba_\la T-T_{n})-
(\ba_\la T -T_{n+1}) \geq \ba_\la l  \} \\
=&\# \{n \in \zz \, : \; \tT_{-n-n_0} \in [0,\ba_\la T], \; \tT_{-n-n_0}-
\tT_{-n-1-n_0} \geq \ba_\la l  \} \\
=&\# \{n \in \zz \, :\; \tT_{n} \in [0,\ba_\la T], \; \tT_{n}-
\tT_{n-1} \geq \ba_\la T  \}\\
=& \# \{n \geq 1 \, : \; \tT_{n} \in [0,\ba_\la T], \; \tT_{n}-
\tT_{n-1} \geq \ba_\la l  \}=:\tS^\la_T(l).
\end{align*}
We used that $\tT_0<0<\tT_1$ by construction. But $\tS^\la_T(l)$ 
is the number of
delays with length greater than $\ba_\la l$ that end in $[0,\ba_\la T]$,
for the $SR(\mu_S)$-process $(\tT_n)_{n\in\zz}$. Thus exactly as in the
proof of Lemma \ref{ilestbalezefourniax} (Steps 1 and 2), we get
$\E[\tS^\la_T(l)]=m_S^{-1}\ba_\la T \mu_S((\ba_\la l,\infty))$, so that 
$\E[R^\la_T(l)]=m_S^{-1}\ba_\la T \mu_S((\ba_\la l,\infty))$. 
Finally, Lemma \ref{estimusz} implies that $\E[R^\la_T(l)]=o(\la\ba_\la)$.
\end{proof}

\subsection{Coupling of seeds}

We aim here to couple the Poisson measure $\pi_S(dx)$ used to build
the $LFF(0)$-process with a family of $SR(\mu_S)$-processes, in such a way that
roughly:

\vip

$\bullet$ if $\pi_S(i_\la)>0$, then the first seed never falls on $i$;

\vip

$\bullet$ if $\pi_S(i_\la)=0$, then seeds fall almost continuously on $i$.

\vip

The precise statement is as follows.

\begin{prop}\label{csz}
Let $A>0$, $T>0$, $\delta>0$ be fixed. For any
$\la \in (0,1]$, it is possible to
find a coupling between a Poisson measure $\pi_S$ on $\rr$ with intensity 
measure $dx$ and a family $(N_t^S(i))_{t\geq 0, i \in \zz}$ of 
$SR(\mu_S)$-processes in such a way that for
\begin{align*}
\Omega^{S}_{A,T,\delta}(\la)=\bigcap_{i \in I_A^\la} \Big(  
\left\{\pi_S(i_\la)=0 , \inf_{t\in [0,T-\delta]} [N^S_{\ba_\la (t+\delta)}(i)-
N^S_{\ba_\la t}(i)]>0 \right\} \hskip1cm\\
\cup \left\{\pi_S(i_\la)=1, N^S_{\ba_\la T}(i)=0 
\right\}
\Big),
\end{align*}
there holds $\lim_{\la \to 0} \Pr[\Omega^{S}_{A,T,\delta}(\la)]=1$.
\end{prop}

\begin{proof}
We split the proof in several steps. As usual, it suffices to
build $\pi_S$ on $B_\la = \cup_{i \in I_A^\la} i_\la \simeq [-A,A]$ and 
to build $N^S_t(i)$ for $t\in [0,\ba_\la T]$ and $i \in I_A^\la$.
\vip

{\bf Step 1.} Denote by $(N^S_t)_{t\geq 0}$ a $SR(\mu_S)$-process and by
$(T_n)_{n\geq 1}$ its jump instants. Recall the notation of 
Lemma \ref{seedszero}. Then we observe that 
$\{N^S_{\ba_\la T}=0\}=\{T_1>\ba_\la T\}$ and
$$
\left\{\inf_{t\in [0,T-\delta]} [N_{\ba_\la (t+\delta)}^S-N^S_{\ba_\la t}]>0\right\}
=\{T_1<\ba_\la \delta,R_T^\la(\delta)=0\}.
$$
These two events are furthermore disjoint. By Lemma \ref{seedszero}, we deduce
that for some functions $\e_T(\la)$ and
$\e_{T,\delta}(\la)$ tending to $0$ when $\la\to 0$
$$
\Pr\left[\inf_{t\in [0,T-\delta]} [N_{\ba_\la (t+\delta)}^S-N^S_{\ba_\la t}]>0
\right]\geq 1 - \Pr[T_1>\ba_\la \delta]-\E[R_T^\la(\delta)] 
\geq 1-\la \ba_\la (1+\e_{T,\delta}(\la))
$$
and
$$
p_T(\la):=\Pr[N^S_{\ba_\la T}=0] = \Pr[T_1 > \ba_\la T]
=\la \ba_\la (1+ \e_{T}(\la)).
$$

{\bf Step 2.} Next, we prove that it is possible to couple a family
$(Z_i^\la)_{i \in I_A^\la}$ of i.i.d. Poisson-distributed random variables
with parameter $|i_\la|=1/\bn_\la$ and a family of 
$(\tZ_i^\la)_{i \in I_A^\la}$ of i.i.d. Bernoulli random variables with
parameter $p_T(\la)$ (see Step 1) in such a way that for
$$
\tOmega_{T,A}(\la)=\{\forall i \in I_A^\la,Z_i^\la=\tZ_i^\la \in \{0,1\}\},
$$
there holds $\lim_{\la \to 0} \Pr[\tOmega_{T,A}(\la)]=1$. As usual, this follows
from Lemma \ref{gcou}-(ii) and relies on the straightforward computations
(here the function $\e_T$ changes from line to line)
$$
\Pr[Z_i^\la=0]\land \Pr[\tZ_i^\la=0] = (e^{-1/\bn_\la})\land (1-p_T(\la))
\geq 1- \la\ba_\la (1+\e_{T}(\la)),
$$
recall that $\bn_\la \sim 1/(\la\ba_\la)$, and
$$
\Pr[Z_i^\la=1]\land \Pr[\tZ_i^\la=1] = (e^{-1/\bn_\la}/\bn_\la)
\land p_T(\la) \geq \la\ba_\la (1-\e_{T}(\la))
$$ 
from which
$$
\Pr[\tOmega_{T,A}(\la)] \geq \left[1- \la\ba_\la (1+\e_{T}(\la))
+\la\ba_\la (1-\e_{T}(\la)) \right]^{|I_A^\la|}
\geq \left[1- \la\ba_\la \e_{T}(\la) \right]^{|I_A^\la|}.
$$
This last quantity tends to $1$ as $\la\to 0$, because
$|I_A^\la|\sim 2A/(\la\ba_\la)$.

\vip

{\bf Step 3.} We finally build the complete coupling.

\vip

(a) Consider $(Z_i^\la,\tZ_i^\la)_{i \in I_A^\la}$ as in Step 2.

\vip

(b) For each $i \in I_A^\la$ such that $Z_i^\la>0$, pick 
some i.i.d. random variables $(X_1^{i,\la},\dots,X^{i,\la}_{Z_i^\la})$ 
uniformly distributed in $i_\la$. Then 
$\pi_S=\sum_{i \in I_A^\la} \sum_{k=1}^{Z_i^\la}\delta_{X_k^{i,\la}}$ is a Poisson
measure with intensity measure $dx$ on $B_\la=\cup_{i\in I_A^\la}i_\la$.

\vip

(c) For each $i \in I_A^\la$ such that $\tZ_i^\la=1$, set $N^S_{\ba_\la T}(i)=0$.
For each $i \in I_A^\la$ such that $\tZ_i^\la=0$, pick $(N^S_t(i))_{t\in
[0,\ba_\la T]}$ conditionally on $N^S_{\ba_\la T}(i)\ne 0$. This defines
a family of i.i.d. $SR(\mu_S)$ processes on $[0,\ba_\la T]$ (because
$\Pr[\tZ_i^\la=1]=p_T(\la)=\Pr[N^S_{\ba_\la T}(i)=0]$).

\vip

{\bf Step 4.} With this coupling, we have 
$\tOmega_{T,A}(\la) \cap \bOmega^S_{A,T,\delta}(\la) 
\subset \Omega^S_{A,T,\delta}(\la)$, where 
$$
\bOmega^S_{A,T,\delta}(\la) = \bigcap_{i\in I_A^\la} \left(
\inf_{t\in [0,T-\delta]} [N_{\ba_\la (t+\delta)}^S(i) -N^S_{\ba_\la t}(i)]>0
\hbox{ or }  N^S_{\ba_\la T}(i)=0 \right).
$$
It thus only remains to check that $\lim_{\la \to 0} 
\Pr[ \bOmega^S_{A,T,\delta}(\la)  ]=1$. But using Step 1 and recalling that 
$|I_A^\la|\sim 2A/(\la\ba_\la)$, we get
$$
\Pr[\bOmega^S_{A,T,\delta}(\la)] \geq \left[1-\la \ba_\la (1+\e_{T,\delta}(\la))
+ \la \ba_\la (1+ \e_T(\la)) \right]^{|I_A^\la|},
$$
which tends to $1$ as $\la\to 0$, as usual, since 
$|I_A^\la|\sim 2A/(\la\ba_\la)$.
\end{proof}

\subsection{Convergence}
We may now prove the convergence result in the case $\beta=0$.

\begin{preuve} {\it of Theorem \ref{converge3}.}
We fix $T>0$, $x_1<\dots<x_p$ and $t_1,\dots,t_p\in (0,T]$. 
We introduce $B>0$ such that $-B<x_1<x_p<B$. We fix $\e>0$ and $a>0$.
Our aim is to check that for $\la>0$
small enough, there exists a coupling between a $FF(\mu_S,\mu^\la_M)$-process
$(\eta^\la_t(i))_{t\geq 0, i \in \zz}$
and a $LFF(0)$-process $(Y_t(x))_{t\geq 0,x\in\rr}$ such that, 
recalling 
(\ref{dlambda})
and Proposition \ref{wpzero}, there holds
\begin{align}\label{cqv2}
\Pr\left[\sum_{k=1}^p\bdelta_{T}(D^\la(x_k),D(x_k)) 
+\sum_{k=1}^p \bdelta(D^\la_{t_k}(x_k),D_{t_k}(x_k)) \geq a
\right] \leq \e.
\end{align}
This will conclude the proof.

\vip

{\bf Step 1.} Consider two independent 
Poisson measures $\pi_S(dx)$ and $\pi_M(dt,dx)$ with intensity 
measures $dx$ and $dtdx$.
First, we consider $A>B$ large enough, in such a way that for
$$
\Omega^{S,1}_A=\{\pi_S([-A,-B])>0,\pi_S([B,A])>0\},
$$
there holds $\Pr\left(\Omega^{S,1}_A\right)\geq 1- \e/4$. 
This fixes the value of $A$.

Next we call $\cX_S=\{x\in [-A,A], \pi_S(\{x\})>0\}$,
$\cT_M = \{t \in [0,T] : \pi_M(\{t\}\times [-A,A])>0\}\cup\{0\}$ and
$\cX_M=\{x\in [-A,A], \pi_M([0,T]\times \{x\})>0\}$.
Classical results about Poisson measures allow us to choose
$K>0$ (large) and $c>0$ (small) in such a way that for
\begin{align*}
\Omega_{K,c} = \big\{|\cT_M| +|\cX_S| \leq K, \, 
\min_{t\in \cT_M, k=1,\dots,p} {|t-t_k| } > c, \min_{x,y\in \cX_S\cup \cX_M,x\ne y}
|x-y|>c, \\
\min_{x\in \cX_S\cup \cX_M, k=1,\dots,p}|x-x_k|>c
 \big\},
\end{align*}
there holds $\Pr\left[\Omega_{K,c} \right] \geq 1-\e/4$.
\vip

{\bf Step 2.} Next, we know from Proposition \ref{coupling1}
that for all $\la>0$ small enough, it is possible to couple a family
of i.i.d. $SR(\mu^\la_M)$-processes $(N^{M,\la}_t(i))_{t\geq 0, 
i \in \zz\}}$ with $\pi_M$
in such a way that for
$$
\Omega^M_{A,T}(\la):=\left\{\forall t\in [0,T],\;
\forall i\in I_A^\la, \;
\Delta N^{M,\la}_{\ba_\la t}(i) \ne 0 \; \hbox{iff} \; 
\pi_M(\{t\}\times i_\la) \ne 0
\right\},
$$
there holds  $\Pr[\Omega^M_{A,T}(\la)]\geq 1-\e/4$.

\vip

We now fix $\delta>0$ such that 
$$
\delta<c/4 \quad \hbox{and} \quad \delta<a/(4AKp).
$$

Proposition \ref{csz} tells us how to couple, for all $\la>0$ small
enough, a family
of i.i.d. $SR(\mu_S)$-processes $(N^{S}_t(i))_{t\geq 0, 
i \in \zz}$ with $\pi_S$ in such a way that for
\begin{align*}
\Omega^{S}_{A,T,\delta}(\la)=\bigcap_{i \in I_A^\la} \Big(  
\left\{\pi_S(i_\la)=0 , \inf_{t\in [0,T-\delta]} [N^S_{\ba_\la (t+\delta)}(i)-
N^S_{\ba_\la t}(i)]>0 \right\} \hskip1cm\\
\cup \left\{\pi_S(i_\la)=1, N^S_{\ba_\la T}(i)=0 
\right\}
\Big),
\end{align*}
there holds  $\Pr[\Omega^S_{A,T,\delta}(\la)]\geq 1-\e/4$.

\vip

{\bf Step 3.} We consider $\pi_M$, $\pi_S$, 
$(N^{S}_t(i))_{t\geq 0, i \in \zz\}}$ and $(N^{M,\la}_t(i))_{t\geq 0,i \in \zz\}}$
coupled as in Step 2. Then we build the corresponding 
$FF(\mu_S,\mu_M^{\la})$-process $(\eta_t^\la(i))_{t\geq 0, i \in \zz}$ 
and the associated rescaled clusters
$(D^\la_t(x))_{t\geq 0,x\in\rr}$, see (\ref{dlambda}) 
and we build the $LFF(0)$-process
associated to $\pi_S$ and the corresponding clusters 
$(D_t(x))_{t\geq 0,x\in\rr}$. We will work on the event
$$
\Omega_\la=\Omega^{S,1}_A\cap \Omega_{K,c}\cap \Omega^M_{A,T}(\la)\cap
\Omega^S_{A,T,\delta}(\la).
$$
We know that for all $\la>0$ small enough, $\Pr[\Omega_\la]\geq 1-\e$.
We introduce
$$
\cS=\cup_{t\in \cT_M} [t,t+\delta].
$$
We will prove in the next step 
that on $\Omega_\la$, for all $\la>0$ small enough,
for all $k\in\{1,\dots,p\}$, for all
$t\in [0,T]\setminus \cS$, 
\begin{align}\label{obj}
\bdelta(D^\la_t(x_k),D_t(x_k)) \leq 4/\bn_\la + 
2A\indiq_{\{t \in \cS\}},
\end{align}
which will imply that
$$
\bdelta_T(D^\la(x_k),D(x_k)) \leq 4T/\bn_\la + 
2A |\cS|.
$$
This will conclude the proof, since for $k=1,\dots,p$, $t_k \notin \cS$
(recall $\Omega_{K,c}$ and that $\delta<c$)
and since the Lebesgue measure of $\cS$ is smaller than $K\delta$
(recall $\Omega_{K,c}$ and that $\delta<c$). Thus (\ref{obj}) implies, since
$\delta<a/(4AKp)$,
\begin{align*}
\sum_{k=1}^p\bdelta_{T}(D^\la(x_k),D(x_k)) 
+\sum_{k=1}^p \bdelta(D^\la_{t_k}(x_k),D_{t_k}(x_k))
\leq& p \left[ 4T/\bn_\la +2A K\delta +4/\bn_\la  \right]\\
\leq& a/2+ 4p(T+1)/\bn_\la,
\end{align*}
which is smaller than $a$ for all $\la>0$ small enough.
Thus (\ref{cqv2}) holds for all $\la>0$ small enough.

\vip

{\bf Step 4.} It remains to check (\ref{obj}). 
In the whole step, we work on $\Omega_\la$.
Let thus $k\in \{1,\dots,p\}$ be fixed. Consider
the first marks $\chi_g,\chi_d$ of $\pi_S$ on the left and right of $x_k$.
Then by definition, we have $D_t(x_k)=[\chi_g,\chi_d]$ for all $t\in [0,T]$.
By $\Omega^{S,1}_A$ and since $x_k\in (-B,B)$, we know that 
$-A<\chi_g<\chi_d<A$.  Define $g_\la=\lfloor \bn_\la \chi_g\rfloor$ and 
$d_\la=\lfloor \bn_\la \chi_d\rfloor$. Due to $\Omega^S_{A,T,\delta}(\la)$
and since $\pi_S(\{\chi_g\})=\pi_S(\{\chi_d\})=1$ and $\pi_S((\chi_g,\chi_d))=0$
by construction,
we know that 

\vip

(i) $N^S_{\ba_\la T}(g_\la)=N^S_{\ba_\la T}(d_\la)=0$ 
(because $\chi_g \in (g_\la)_\la$ and $\chi_d \in (d_\la)_\la$),

\vip

(ii) for all $i \in \lb g_\la+1,d_\la-1\rb$, 
$\inf_{t\in [0,T-\delta]} [N^S_{\ba_\la (t+\delta)}(i)-N^S_{\ba_\la t}(i)]>0$
(because $i_\la \subset (\chi_g,\chi_d)$).

\vip
Observe now that for $\la>0$ small enough (it suffices that
$1/\bn_\la < c$), there holds $g_\la < \lfloor x_k \bn_\la \rfloor < d_\la$
(use that $\chi_g,\chi_d \in \cX_S$ and that $\chi_g<x_k<\chi_d$
so that due to $\Omega_{K,c}$, $\chi_g+c<x_k<\chi_d-c$).

\vip

Point (i) implies that $\eta^\la_{\ba_\la t}(g_\la)=\eta^\la_{\ba_\la t}(d_\la)=0$
for all $t\in [0,\ba_\la T]$. Consequently, for all $t\in [0,T]$,
there holds $C(\eta^\la_{\ba_\la t},\lfloor x_k \bn_\la \rfloor )
\subset \lb g_\la+1,d_\la-1\rb$. This implies that
$D^\la_t(x_k)\subset [(g_\la+1)/\bn_\la, (d_\la-1)/\bn_\la]\subset 
[\chi_g,\chi_d]$. Recalling that $D_t(x_k)=[\chi_g,\chi_d]$
and that $-A<\chi_g<\chi_d<A$, we deduce that 
$\bdelta(D_t(x_k),D^\la_t(x_k))\leq 2A$ for all $t\in [0,T]$.

\vip

Another consequence is that 
the matches
falling outside $\lb g_\la,d_\la\rb$ (and {\it a fortiori} 
outside $I_A^\la$) have no influence on 
$\lfloor x_k \bn_\la \rfloor$ during $[0,\ba_\la T]$. 

\vip

It only remains to check that for $t\in [0,T]\setminus \cS$,
if $\la>0$ is small enough, $\bdelta(D_t(x_k),D^\la_t(x_k))\leq 4/\bn_\la$.
We thus fix $t\in [0,T]\setminus \cS$ and consider 
$t_0=\max\{s \in \cT_M\, : \; s<t \}$. Then by definition of $\cS$,
$t-t_0>\delta$. Consequently,
point (ii) guarantees us that for all $i\in \lb g_\la+1,d_\la-1\rb$,
$N^S_{\ba_\la t}-N^S_{\ba_\la t_0}>0$: a seed falls on each of these
sites during $[\ba_\la t_0,\ba_\la t]$. Furthermore, 
there are no matches falling on $\lb g_\la+1,d_\la-1\rb$
during $[\ba_\la t_0,\ba_\la t]$, by definition of $t_0$ and due to 
$\Omega^M_{A,T}(\la)$. 
Consequently,
we have $\eta^\la_{\ba_\la t}(i)=1$ for all $i\in \lb g_\la+1,d_\la-1\rb$.
All this implies that $C(\eta^\la_{\ba_\la t}, \lfloor x_k \bn_\la \rfloor)
=\lb g_\la+1,d_\la-1\rb$, whence 
$D^\la_t(x_k)=[(g_\la+1)/\bn_\la, (d_\la-1)/\bn_\la]=
[(\lfloor \bn_\la \chi_g\rfloor+1)/\bn_\la,
(\lfloor \bn_\la \chi_d\rfloor-1)/\bn_\la ]$. Recalling that 
$D_t(x_k)=[\chi_g,\chi_d]$, we easily conclude.
\end{preuve}

\section{Well-posedness of the limit process when $\beta\in \{\infty,BS\}$}
\setcounter{equation}{0}

The aim of this section is to prove Theorems \ref{wpinfty} and \ref{wpbs},
and to localize the limit processes. All the results below
have already been proved in \cite{bf} for the $LFF(\infty)$-process.
We provide here a consequently simpler proof, that allows us to
treat simultaneously the cases $\beta=BS$ and $\beta=\infty$.

\begin{rem}\label{i=bs}
Under $(H_S(\infty))$, we put $\theta_u=\delta_u$ and $F_S(u,v)=u$
for all $u\in [0,1]$, all $v\in [0,1]$.
Using this function $F_S$, the $LFF(BS)$-process is nothing but
the $LFF(\infty)$-process.
\end{rem}

We consider a Poisson measure $\pi_M(dt,dx,dv)$ on
$[0,\infty)\times\rr\times[0,1]$ with intensity measure
$dtdxdv$ and abusively write
$\pi_M(dt,dx)=\int_{v\in [0,1]} \pi_M(dt,dx,dv)$, which is 
a Poisson measure on $[0,\infty)\times\rr$ with intensity measure $dtdx$.

\begin{defin}\label{dflffpA}
Let $\beta \in \{\infty,BS\}$. If $\beta=\infty$, consider $F_S$
as in Remark \ref{i=bs}. If $\beta=BS$, consider $F_S$ as in Definition
\ref{defF}.
Let $A>0$ be fixed. A $\rr_+\times\cI\times\rr_+$-valued process 
$(Z_t^A(x),D_t^A(x),H_t^A(x))_{t\geq 0,x\in [-A,A]}$
such that a.s., for all $x\in [-A,A]$,
$(Z_t^A(x),H_t^A(x))_{t\geq 0}$ is c\`adl\`ag,
is  called a $LFF_A(\beta)$-process if a.s., for all $t\geq 0$, 
all $x \in [-A,A]$,
\begin{align*}
Z_t^A(x)=&  \displaystyle \intot \indiq_{\{Z_s^A (x) < 1\}}ds - 
\intot \int_{[-A,A]} \indiq_{\{ Z_\sm^A(x)=1,y \in D_{\sm}^A(x)\}}\pi_M(ds,dy),\\
H_t^A(x)=& \displaystyle 
\intot \int_0^1 
F_S(Z_\sm^A(x),v)\indiq_{\{Z_\sm^A(x)<1\}} \pi_M(ds\times \{x\}\times dv) 
- \intot  \indiq_{\{H_s^A (x) > 0 \}}ds, 
\end{align*}
where $D_t^A(x) = [L_t^A(x),R_t^A(x)]$, with
\begin{align}\label{dfcc}
\left\{
\begin{array}{rl}
L_t^A(x) =& (-A) \lor \sup\{ y \in [-A,x];\; Z_t^A(y)<1 
\hbox{ or } H_t^A(y)>0 \}\\
R_t^A(x) =& A \land \inf\{ y\in [x,A];\; Z_t^A(y)<1 \hbox{ or } H_t^A(y)>0 \}
\end{array}
\right.
\end{align}
and where $D_{t-}^A(x)$ is defined similarly.
\end{defin}

Observe that for $\beta\in\{\infty,BS\}$, for any $A>0$,
the $LFF_A(\beta)$-process is obviously well and uniquely defined and can
be built as follows.

\begin{algo}\label{algo1} \rm
Consider the marks $(T_k,X_k,V_k)_{k=1,\dots,n}$ of
$\pi_M$ in $[0,T]\times[-A,A]\times[0,1]$, ordered
chronologically and set $T_0=0$. 

\vip

{\bf Step 0.} Put $Z_0^A(x)=H_0^A(x)=0$ and $D_0^A(x)=\{x\}$ 
for all $x\in [-A,A]$. 

\vip

Assume that for some $k\in \{0,\dots,n-1\}$,  
$(Z^A_t(x),D_t^A(x),H^A_t(x))_{t\in[0,T_k],x\in[-A,A]}$ has been built.

\vip

{\bf Step k+1.} Then for $t\in (T_k,T_{k+1})$ and $x\in [-A,A]$,
put $Z^A_t(x)=\min(1,Z^A_{T_k}(x)+t-T_k)$, set 
$H^A_t(x)=\max(0,H^A_{T_k}(x)-t+T_k)$ and then define $D_t^A(x)$
as in (\ref{dfcc}).
Finally, build $(Z^A_{T_{k+1}}(x),D_{T_{k+1}}^A(x),H^A_{T_{k+1}}(x))$ as follows.

\vip

$\bullet$ 
If $Z_{T_{k+1}-}^A(X_{k+1})=1$, set $H_{T_{k+1}}^A(x) =H_{T_{k+1}-}^A(x)$ for all
$x\in [-A,A]$ and consider $[a,b]:= D^A_{T_{k+1}-}(X_{k+1})$.
Set $Z^A_{T_{k+1}}(x) = 0$ for all $x\in (a,b)$ and 
$Z^A_{T_{k+1}}(x) = Z^A_{T_{k+1}-}(x)$
for all $x\in [-A,A] \setminus [a,b]$. Set finally  
$Z^A_{T_{k+1}}(a)=0$ if $Z^A_{T_{k+1}-}(a)=1$ and $Z^A_{T_{k+1}}(a)=Z^A_{T_{k+1}-}(a)$ 
if $Z^A_{T_{k+1}-}(a)<1$ and 
$Z^A_{T_{k+1}}(b)=0$ if $Z^A_{T_{k+1}-}(b)=1$ and $Z^A_{T_{k+1}}(b)=Z^A_{T_{k+1}-}(b)$ 
if $Z^A_{T_{k+1}-}(b)<1$.

\vip

$\bullet$ If $Z^A_{T_{k+1}-}(X_{k+1})<1$, set 
$H^A_{T_{k+1}}(X_{k+1})=F_S(Z^A_{T_{k+1}-}(X_{k+1}),V_{k+1})$, put 
$Z^A_{T_{k+1}}(X_{k+1})=Z^A_{T_{k+1}-}(X_{k+1})$ and 
$(Z^A_{T_{k+1}}(x),H^A_{T_{k+1}}(x))
=(Z^A_{T_{k+1}-}(x),H^A_{T_{k+1}-}(x))$ for
all $x\in [-A,A]\setminus \{X_{k+1}\}$.

\vip

$\bullet$ Using the values of $(Z^A_{T_{k+1}}(x),H^A_{T_{k+1}}(x))_{x\in[-A,A]}$,
compute $(D^A_{T_{k+1}}(x))_{x\in[-A,A]}$
as in (\ref{dfcc}).
\end{algo}

\vip

We now state a refined version of Theorems \ref{wpinfty} and \ref{wpbs}.

\begin{prop}\label{loc}
Let $\beta \in \{\infty,BS\}$.
Let $\pi_M$ be a Poisson measure on $[0,\infty)\times\rr\times[0,1]$
with intensity measure $dtdxdv$.

(i) There exists a unique $LFF(\beta)$-process 
$(Z_t(x),D_t(x),H_t(x))_{t\geq 0,x\in \rr}$.

(ii) It can be perfectly simulated on $[0,T]\times[-n,n]$ 
for any $T>0$, any $n>0$.

(iii) For $A>0$, let $(Z_t^A(x),D_t^A(x),H_t^A(x))_{t\geq 0,x\in [-A,A]}$
be the unique $LFF_A(\beta)$-process. There holds
\begin{align}\label{fcfc}
\Pr \Big[(Z_t(x),&D_t(x),H_t(x))_{t\in [0,T],x\in[-A/2,A/2]} \\
&= (Z_t^A(x),D_t^A(x),H_t^A(x))_{t\in [0,T],x\in[-A/2,A/2]} \Big] 
\geq 1 - C_T e^{-\alpha_T A},\nonumber
\end{align}
for some constants $\alpha_T>0$ and $C_T>0$ not depending on $A>0$.
\end{prop}

To prove this result, we need a lower-bound of the length of the
barriers.

\begin{lem}\label{lblb}
Let $\beta \in \{\infty,BS\}$. If $\beta=\infty$, consider $F_S$
as in Remark \ref{i=bs}. If $\beta=BS$, consider $F_S$ as in Definition
\ref{defF}. There exists $v_0\in [0,1)$ such that
for all $z\in [3/4,1)$, all $v\in [v_0,1]$, $F(z,v)\geq 1/2$.
\end{lem}

\begin{proof}
If $\beta=\infty$, the result is obvious with $v_0=0$, since
$F_S(z,v)=z\geq 1/2$ for all $z\in [1/2,1]$, $v\in [0,1]$. Consider
now the case $\beta=BS$. First observe that 
$g_S(t,s)\leq \Pr[N_{T_S(t+s)}^S-N^S_{T_S t}>0]=\nu_S([0,T_S s])$. Hence for 
all $z \in [3/4,1)$,
\begin{align*}
\theta_z([0,1/2))&\leq \nu_S([3T_S/4,T_S])+\frac{\nu_S([3T_S/4,T_S])^2}
{\nu_S([T_S/2,T_S])^2}\nu_S([0,T_S/2])\\
&\leq 
\nu_S([0,T_S/2]\cup [3T_S/4,T_S])=:v_0<1,
\end{align*}
since $\supp \nu_S = [0,T_S]$.
We deduce that for $z\in [3/4,1]$,
$$
\int_0^1 dv \indiq_{\{F_S(z,v)< 1/2\}}
=\theta_z([0,1/2))\leq v_0.
$$
Recalling that $v\mapsto F_S(z,v)$ is nondecreasing, we deduce that
$F_S(z,v)\geq 1/2$ for $v \in [v_0,1]$.
\end{proof}

\begin{preuve} {\it of Proposition \ref{loc}.}
We split the proof into several steps. We work on $[0,T]$.

\vip

{\bf Step 1.} We observe that for a mark $(\tau,X,V)$ of $\pi_M$ with
$X\in [-A,A]$ and  $V\geq v_0$ (see Lemma \ref{lblb}), 
we have $H^A_{t}(X)>0$ or $Z^A_t(X)<1$
for all $t \in [\tau,\tau+1/4]$ (and the same result applies 
to the $LFF(\beta)$-process if it exists). 

\vip

Indeed,
assume first that $Z^A_{\tau-}(X)\in [0,3/4)$. 
Then $Z^A_t(X)=Z^A_{\tau-}(X) +t-\tau <1$ for $t \in [\tau,\tau+1/4]$.

\vip

Assume next that  $Z^A_{\tau-}(X)\in [3/4,1)$. Then 
$H_{\tau}(X)=F_S(Z^A_{\tau-}(X),V)\geq 1/2$ due to Lemma \ref{lblb},
so that $H_t(X)=H_{\tau}(X)-t+\tau >0$ for $t \in [\tau,\tau+1/2) \supset 
[\tau,\tau+1/4]$.

\vip

If finally $Z^A_{\tau-}(X)=1$, then $Z^A_{\tau}(X)=0$, whence
$Z^A_t(X)=t-\tau<1$ for $t \in [\tau,\tau+1) \supset  [\tau,\tau+1/4]$.

\vip

{\bf Step 2.} For $a \in \rr$, we consider the event $\Omega_a$ defined as
follows: for $\{(T_k,X_k,V_k)\}_{k=1,\dots,n}$ the marks of $\pi_M$
restricted to $[0,T]\times [a,a+1) \times [v_0,1]$ ordered chronologically,
for $T_0=0$, $T_{n+1}=T$, we put $\Omega_a=\{\max_{i=0,\dots,n} 
(T_{i+1}-T_i) < 1/4\}$. 

\vip

We immediately deduce from Step 1 that for any $a\in \rr$,
any $A > |a|+1$, 
$$
\Omega_a \subset \left\{ 
\forall\; t\in [0,T], \; \exists \; x \in (a,a+1), H^A_t(x)>0 \hbox{ or }
Z^A_t(x)<1\right\}.
$$
Thus on $\Omega_a$, clusters on the left of $a$ cannot be
connected to clusters on the right of $a+1$ during $[0,T]$.
Hence matches falling at the right of $a+1$ (resp. on the left of $a$)
do not affect the zone $(-\infty,a)$ (resp. $(a+1,\infty)$) 
during $[0,T]$.

\vip

{\bf Step 3.} Obviously, $q_T:=\Pr(\Omega_a)$ is positive
and does not depend on $a$. Furthermore, $\Omega_a$ is independent of
$\Omega_b$ for all $a,b \in \zz$ with $a\ne b$.
Hence there are a.s. infinitely many $a\in \zz$ such that $\Omega_a$
is realized.  

\vip

Then it is routine to deduce the well-posedness of the
$LFF(\beta)$-process.
The perfect simulation algorithm on a
finite-box $[0,T]\times [-n,n]$ is also easy: simulate
$\pi_M$ on $[0,T]\times [a_1,a_2]$ in such a way that
$\Omega_{a_1}\cap \Omega_{a_2}$ is realized and that $a_1+1<-n<n<a_2$.
Then apply the same rules as for the $LFF_A(\beta)$-process.
This will give the true $LFF(\beta)$-process inside $[a_1+1,a_2]\supset 
[-n,n]$, 
because matches falling outside $[a_1,a_2+1]$ 
have no effect on the process in the box 
$[a_1+1,a_2]$ during $[0,T]$.

\vip

Finally, we can clearly bound from below the left hand side of
(\ref{fcfc}) by 
\begin{align*}
\Pr\left[\left(\cup_{a \in [-A,-A/2-1]\cap \zz}\Omega_a \right) \cap  
\left(\cup_{a \in [A/2, A-1]\cap \zz}\Omega_a \right)\right] 
&\geq 1 - 2 \Pr[\Omega_0^c]^{\lfloor A \rfloor -\lfloor A/2 \rfloor -2} \\
&\geq 1 - 2 (1-q_T)^{A/2-4},
\end{align*}
whence (\ref{fcfc}) with $C_T=2/(1-q_T)^4$ and $ \alpha_T= -\log(1-q_T)/2$.
\end{preuve}

\section{Localization of the discrete processes when $\beta\in\{\infty,BS\}$}
\label{locd}
\setcounter{equation}{0}

We recall that $\ba_\la$, $\bn_\la$ and $\bm_\la$ are defined
in (\ref{ala}), (\ref{nla}) and (\ref{mla}).
For $A>0$, we set $A_\la = \lfloor A \bn_\la \rfloor$ and 
$I_A^\la=\lb -A_\la,A_\la\rb$.
For $i\in \zz$, we set $i_\la=[i/\bn_\la, (i+1)/\bn_\la)$.

\vip

For $\eta \in \{0,1\}^{I_A^\la}$ and $i \in I_A^\la$, we define the occupied 
connected component around $i$ as
$$
C_A(\eta,i)=\left\{ \begin{array}{lll}
\emptyset & \hbox{ if } & \eta(i)=0, \\
\lb l_A(\eta,i),r_A(\eta,i)\rb & \hbox{ if } & \eta(i)=1,
\end{array}\right.
$$ 
where $l_A(\eta,i)=(-A_\la) \lor (\sup\{k< i:\; \eta(k)=0\}+1)$ and
$r_A(\eta,i)=A_\la \land (\inf\{k > i:\; \eta(k)=0\}-1)$.

\begin{defin}\label{gffA}
Assume $(H_M)$ and $(H_S(\beta))$ with $\beta\in \{\infty,BS\}$.
Let $\la\in (0,1]$ and $A>0$ be fixed.
For each $i\in I_A^\la$, we consider a $SR(\mu_S)$-process 
$(N^S_t(i))_{t\geq 0}$ and a $SR(\mu_M^\la)$-process $(N^{M,\la}_t(i))_{t\geq 0}$, 
all these processes being
independent. Consider a $\{0,1\}$-valued process 
$(\eta_t^{\la,A}(i))_{i \in I_A^\la, t\geq 0}$ 
such that a.s., for all $i\in I_A^\la$, $(\eta^{\la,A}_t(i))_{t\geq 0}$ 
is c\`adl\`ag. We say that  $(\eta_t^{\la,A}(i))_{i \in I_A^\la, t\geq 0}$
is a  $FF_A(\mu_S,\mu_M^\la)$-process if a.s., for all $i \in I_A^\la$, 
all $t\geq 0$,
$$
\eta_t^{\la,A}(i)=\intot \indiq_{\{\eta_\sm^{\la,A}(i)=0\}} dN^S_s(i) 
- \sum_{j\in I_A^\la}\intot \indiq_{\{j\in C_A(\eta_\sm^{\la,A},i)\}} dN^{M,\la}_s(j)
$$
For $x\in [-A,A]$ and $t\geq 0$, we introduce
\begin{align}
D^{\la,A}_t(x)=&\frac{1}{\bn_\la} C_A(\eta^{\la,A}_{\ba_\la t},
\lfloor \bn_\la x \rfloor) \subset [-A_\la/\bn_\la,A_\la/\bn_\la]
\simeq [-A,A]
,\label{dlambdaA} \\
K^{\la,A}_t(x)=&
\frac{\left|\left\{i \in \lb \lfloor \bn_\la x \rfloor - \bm_\la 
, \lfloor \bn_\la x\rfloor + \bm_\la  \rb \cap I_A^\la\,:\;   
\eta^{\la,A}_{\ba_\la t} (x)=1
\right\}\right|}{\left|\lb \lfloor \bn_\la x \rfloor - \bm_\la 
, \lfloor \bn_\la x\rfloor + \bm_\la  \rb \cap I_A^\la \right|} \in [0,1],\ala
Z^{\la,A}_t(x)=& \frac{\psi_S(K^{\la,A}_t(x) )}{\ba_\la} \land 1 \in [0,1].
\label{zlambdaA}
\end{align}
\end{defin}

We generalize \cite[Proposition 11]{bf}, with a consequently
less intricate proof.

\begin{prop}\label{plocla}
Assume $(H_M)$ and $(H_S(\beta))$, for some $\beta \in \{\infty, BS\}$.
Let $T>0$ and $\la\in (0,1)$. 
For each $i\in \zz$, 
we consider a $SR(\mu_S)$-process 
$(N^S_t(i))_{t\geq 0}$ and a $SR(\mu_M^\la)$-process $(N^{M,\la}_t(i))_{t\geq 0}$, 
all these processes being
independent. Let $(\eta^\la_t(i))_{t\geq 0,i\in\zz}$ be the corresponding 
$FF(\mu_S,\mu_M^{\la})$-process,
and for each $A>0$, let 
$(\eta^{\la,A}_t(i))_{t\geq 0,i\in I_A^\la}$ be the corresponding 
$FF_A(\mu_S,\mu_M^{\la})$-process.
Recall (\ref{dlambda})-(\ref{zlambda}) and (\ref{dlambdaA})-(\ref{zlambdaA}).
There are some constants $\alpha_T>0$ and $C_T>0$ such that for 
all $A\geq 1$, all $\la \in (0,1]$ small enough,
\begin{align*}
&\Pr\Big[(\eta^\la_t(i))_{t\in [0,\ba_\la T],i\in I_{A/2}^\la}
=(\eta^{\la,A}_t(i))_{t\in [0,\ba_\la T],i\in I_{A/2}^\la}, \\
& \hskip1cm (Z^\la_t(x),D^\la_t(x))_{t\in [0,T],x \in [-A/2,A/2]}
=(Z^{\la,A}_t(x),D^{\la,A}_t(x))_{t\in [0,T],x \in [-A/2,A/2]}\Big] \\
\geq& 1 - C_T e^{-\alpha_T A}.
\end{align*}
\end{prop}

\begin{preuve} {\it in the case where $\beta=\infty$.}  
It of course suffices to prove the result for all $A$ large enough
(we will assume that $A>8T$).
We consider the true $FF(\mu_S,\mu_M^\la)$-process
$(\eta_t^\la(i))_{t\geq 0, i \in \zz}$.  For $a \in \rr$, we introduce 
$$
J_a^\la := \lb \lfloor a \bn_\la\rfloor,
\lfloor (a+1) \bn_\la\rfloor -1 \rb.
$$

\vip

{\bf Step 1.} We show here that for all $a\in \rr$, there exists an event
$\Omega^\la_{a,0}$, depending only on $(N_s^S(i),N_s^{M,\la}(i))_{i \in J_a^\la,
s\in [0,3\ba_\la/4]}$ such that

\vip

(i) on $\Omega^\la_{a,0}$, a.s., there is $i\in J_a^\la$ such that 
$\eta^\la_{\ba_\la s}(i)=0$ for all $s\in [0,3/4]$;

\vip

(ii) $\lim_{\la \to 0} \Pr[\Omega_{a,0}^\la] =1$.

\vip

This is very easy: consider simply $\Omega_{a,0}^\la = \{\exists\; i \in J_a^\la,
N^S_{3\ba_\la/4}(i)=0\}$. Clearly, point (i) is satisfied, since 
there is a site of $J_a^\la$ on which no seed
falls during $[0,3\ba_\la/4]$. Since
$|J_a^\la| = \bn_\la \sim
1/(\la\ba_\la)=1/\nu_S((\ba_\la,\infty))$, we deduce from
$(H_S(\infty))$ that
\begin{align*}
\Pr[\Omega_{a,0}^\la]=& 1 - \nu_S((0,3\ba_\la/4))^{\bn_\la} 
= 1 - \left(1- \nu_S((3\ba_\la/4,\infty))\right)^{\bn_\la}\\
\simeq& 1 - e^{-\nu_S((3\ba_\la/4,\infty))/\nu_S((\ba_\la,\infty))} \to 1
\end{align*}
as $\la\to 0$, whence (ii).

\vip

{\bf Step 2.} We now check that for all $a\in \rr$, all $t\geq 1/2$, 
there exists an event
$\Omega^\la_{a,t}$, depending only on $(N_s^S(i),N_s^{M,\la}(i))_{i \in J_a^\la,
s\in [(t-1/2)\ba_\la,(t+1/4)\ba_\la]}$ such that

\vip

(i) on $\Omega^\la_{a,t}$, a.s., there is $i\in J_a^\la$ such that 
$\eta^\la_{\ba_\la s}(i)=0$ for all $s\in [t,t+1/4]$;

\vip

(ii) $q_\la := \Pr[\Omega_{a,t}^\la]$ does not depend on $t,a$ and 
$q:=\liminf_{\la \to 0} q_\la>0$.

\vip

This is much more delicate. We put 
$\bk_\la= \lfloor 1/ \nu_S((3\ba_\la/8,\infty))\rfloor$. Observe that
due to $(H_S(\infty))$, $\bk_\la << 
\bn_\la= \lfloor 1/ \nu_S((\ba_\la,\infty))\rfloor$

\vip

We introduce the event $\Omega_{a,t}^\la$ on which 
(see Figure \ref{figOmegalaat}):

\begin{figure}[b] 
\fbox{
\begin{minipage}[c]{0.95\textwidth}
\centering
\includegraphics[width=8cm]{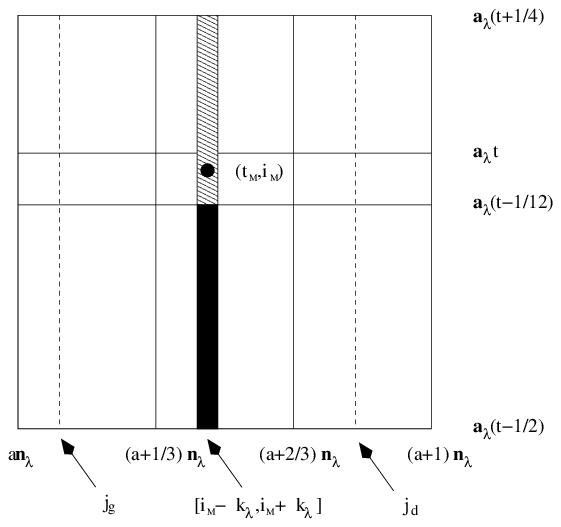}
\caption{The event $\Omega^\la_{a,t}$.}
\label{figOmegalaat}
\vip
\parbox{13.3cm}{
\footnotesize{
A match falls on $i_M$ at time $t_M$, no seed fall on $j_g$ and $j_d$ during
$[\ba_\la (t-1/2), \ba_\la (t+1/4) ]$. All the sites of 
$\lb i_M-\bk_\la,i_M+\bk_\la \rb$ receive
at least one seed during $[\ba_\la (t-1/2), \ba_\la (t-1/12) ]$.
Finally, there is at least one site of $\lb i_M-\bk_\la,i_M+\bk_\la \rb$
on which no seed falls during
$[\ba_\la (t-1/12), \ba_\la (t+1/4)]$.
}}
\end{minipage}
}
\end{figure}

\vip

(a) we have $\Delta N^{M,\la}_{t_M} (i_M)>0$ for some
$i_M \in \lb \lfloor (a+1/3) \bn_\la \rfloor, \lfloor (a+2/3)\bn_\la
\rfloor \rb$, some $t_M \in [(t-1/12)\ba_\la,t \ba_\la]$
and this is the only match
falling in $J_a^\la$ during $[(t-1/2)\ba_\la,t \ba_\la]$;

\vip

(b) there are 
$j_g \in \lb \lfloor a \bn_\la \rfloor, \lfloor (a+1/4)\bn_\la\rfloor \rb$
and $j_d \in \lb \lfloor (a+3/4) \bn_\la \rfloor, 
\lfloor (a+1)\bn_\la-1\rfloor \rb$ such that
$N^S_{\ba_\la (t+1/4)}(j_g)-N^S_{\ba_\la (t-1/2)}(j_g)=
N^S_{\ba_\la (t+1/4)}(j_d)-N^S_{\ba_\la (t-1/2)}(j_d)=0$;

\vip

(c) for all $i \in \lb i_M-\bk_\la,i_M+\bk_\la \rb$, 
$N^S_{\ba_\la (t-1/12)}(i)-N^S_{\ba_\la (t-1/2)}(i)>0$;

\vip

(d) there is $j_0 \in \lb i_M-\bk_\la,i_M+\bk_\la \rb$ such that
$N^S_{\ba_\la (t+1/4)}(j_0)-N^S_{\ba_\la (t-1/12)}(j_0)=0$.

\vip

We first prove point (i), considering two cases.

\vip

$\bullet$ If the zone $\lb i_M-\bk_\la,i_M+\bk_\la \rb$ is completely occupied
at time $t_M-$, then it burns at time $t_M$ and since no seed
falls on $j_0$, which belongs to this zone, during $[t_M,\ba_\la (t+1/4)]
\supset [\ba_\la t,\ba_\la (t+1/4)]$, we deduce that $\eta^\la_{\ba_\la s}(j_0)=0$
for all $s \in [t,t+1/4]$.

\vip

$\bullet$ Assume now that there is $i_0 \in \lb i_M-\bk_\la,i_M+\bk_\la \rb$
that is vacant at time $t_M-$. Recall that there is no fire in
$J_a^\la$ during $[\ba_\la(t-1/2),t_M)$ and that on each site of 
$\lb i_M-\bk_\la,i_M+\bk_\la \rb$,
at least one seed falls during $[\ba_\la (t-1/2), \ba_\la(t-1/12)]\subset
[\ba_\la (t-1/2),t_M)$. 
Then necessarily, a fire starting at some $i_M' \notin J_a^\la$ at some time
$t_M' \in [\ba_\la(t-1/2),t_M)$ has made vacant $i_0$.
Assume e.g. that $i_M'< \lfloor a \bn_\la
\rfloor$ and observe that $i_M'< j_g < i_0$. The fire $(t'_M,i'_M)$ has then
also necessarily made vacant $j_g$. Since no seed falls on $j_g$ during
$[\ba_\la (t-1/2), \ba_\la (t+1/4)]$, we deduce that $j_g$ remains vacant
during $[t'_M,\ba_\la (t+1/4)] \supset [\ba_\la t,\ba_\la (t+1/4)]$.

\vip

We now prove (ii). The quantity $\Pr[\Omega_{a,t}^\la]$ does obviously
not depend on $a \in \rr$ nor on $t\geq 1/2$ by invariance by
spatial translation and by time stationarity.
We infer from Proposition 
\ref{coupling1} that for $\pi_M(ds,dx)$ a Poisson measure on
$[0,\infty)\times \rr$ with intensity measure 
$dsdx$, the probability of (a) tends,
as $\la\to 0$, to 
\begin{align*}
q:=\Pr \Big( &\pi_M([t-1/12,t]\times [a+1/3,a+2/3])=1,  \\
&\pi_M(([t-1/2,t]\times [a,a+1])\backslash ([t-1/12,t]\times [a+1/3,a+2/3]) 
)=0 \Big),
\end{align*}
which is clearly positive. Next, the probability of (b) tends to $1$. Indeed,
treating e.g. the case of $j_g$, there holds, recalling that
$\bn_\la \simeq 1/\nu_S((\ba_\la,\infty))$,
\begin{align*}
\Pr\left[\exists \; j \in \lb \lfloor a \bn_\la \rfloor, 
\lfloor (a+1/4)\bn_\la\rfloor \rb, N^S_{\ba_\la(t+1/4)}(j)=N^S_{\ba_\la(t-1/2)}(j)  
\right] \simeq 1- \nu_S((0,3\ba_\la/4))^{\bn_\la/4}\\
\simeq 1 - e^{-\nu_S((3\ba_\la/4,\infty))/[4\nu_S((\ba_\la,\infty))] },
\end{align*}
which tends to $1$ as $\la \to 0$ due to $(H_S(\infty))$.
The probability 
of (c) (conditionally on $(a)$) also tends to $1$. Indeed, its value is
nothing but
$$
\nu_S((0,5\ba_\la/12))^{2\bk_\la +1} 
\simeq e^{-2 \nu_S((5\ba_\la/12,\infty))/\nu_S((3\ba_\la/8,\infty))}
$$
which tends to $1$ due to $(H_S(\infty))$, since $5/12>3/8$. Finally,
the probability of (d) (conditionally on $(a)$) also tends to $1$, since
it equals
$$
1 - \left(\nu_S((0,\ba_\la/3)) \right)^{2\bk_\la +1} 
\simeq 1 - e^{-2 \nu_S((\ba_\la/3,\infty))/\nu_S((3\ba_\la/8,\infty))},
$$
which tends to 1 due to $(H_S(\infty))$, since $1/3<3/8$.

\vip

{\bf Step 3.} Let now $T>0$ be fixed. Set $K=\lfloor 4T \rfloor$. 
For $a \in \rr$, we set
$$
\tOmega_{a,T}^\la = \Omega_{a,0}^\la \cap \bigcap_{k=2}^K \Omega_{a+(k-1),k/4}^\la.
$$
Then it is clear from Steps 1 and 2 (observe that 
$(K/4+1/4\geq T)$) that

\vip

(i) on $\tOmega_{a,T}^\la$, for all $t\in [0,T]$ there is 
$i \in \lb \lfloor a \bn_\la \rfloor, \lb \lfloor (a +K)\bn_\la -1\rfloor\rb$ 
such that $\eta_{\ba_\la t}^\la (i)=0$;

\vip

(ii) $p_\la=\Pr[\tOmega_{a,T}^\la]$ does not depend on $a$ and
$p:=\liminf_{\la \to 0} p_\la \geq q^{K-1}>0$;

\vip

(iii) $\tOmega_{a,T}^\la$ depends only on 
$(N^S_{\ba_\la t}(i),N^{M,\la}_{\ba_\la t}(i))_{t\in [0,T+1], 
i \in \lb \lfloor a \bn_\la \rfloor,\lfloor (a+K) \bn_\la \rfloor -1 \rb}$.

\vip

{\bf Step 4.} We deduce that for all $a\in\zz$,
conditionally on $\Omega_{a,T}^\la$, 
clusters on the left of $\lfloor a \bn_\la \rfloor-1$ are never connected 
(during $[0,\ba_\la T]$)
to clusters on the right of $\lfloor (a+K)\bn_\la \rfloor$. 
Thus on $\Omega_{a,T}^\la$,
fires starting on the left of
$\lfloor a\bn_\la \rfloor-1$ do not affect the zone 
$[\lfloor (a+K)\bn_\la \rfloor ,\infty)\cap \zz$
and fires starting on the right of 
$\lfloor (a+K)\bn_\la \rfloor $  do not affect the 
zone $(-\infty,\lfloor a\bn_\la\rfloor-1]\cap \zz$.

\vip

We deduce that for $A\geq 2K$, the $FF_A(\mu_S,\mu_M^\la)$-process and the
$FF(\mu_S,\mu_M^\la)$-process
coincide on $I^\la_{A/2}$ during $[0,\ba_\la T]$
as soon as there are $a_1 \in [-A,-A/2-K]$ and $a_2\in [A/2,A-K]$
with $\Omega_{a_1,T}^\la \cap \Omega_{a_2,T}^\la$ realized.
Furthermore, $\Omega_{a,T}^\la$ is independent
of $\Omega_{b,T}^\la$ for all $a,b\in\zz$ with $|a-b|>K$.
Thus we can bound the probabilities of the statement from below,
for $A\geq 2 K$ and $\la>0$ small enough (so that
$\Pr[\Omega_{a,T}^\la]\geq p/2$), by
\begin{align*}
&1- \Pr\left[\bigcap_{l=1}^{\lfloor A/(2K)\rfloor} 
(\Omega_{\lfloor -A\rfloor+lK,T}^\la)^c
\right]
- \Pr\left[\bigcap_{l=1}^{\lfloor A/(2K)\rfloor} 
(\Omega_{\lfloor A/2 \rfloor+lK,T}^\la)^c
\right] \\
\geq& 1 - 2 (1-p/2)^{\lfloor A/(2K) \rfloor} \\
\geq & 1 - 2(1-p/2)^{A/(2K) -1}.
\end{align*}
This concludes the proof: choose $C_T=2/(1-p/2)>0$
($p$ depends only on $T$) and
$\alpha_T=-\log(1-p/2)/(2K)>0$.
\end{preuve}

When $\beta=BS$, the proof is similar, but consequently simpler.

\begin{preuve} {\it when $\beta=BS$.}
Recall that $\ba_\la=T_S$ and consider the true $FF(\mu_S,\mu_M^\la)$-process
$(\eta^\la_t(i))_{t\geq 0,i\in \zz}$. 
For $a \in \rr$, let $J_{a}^\la= \lb \lfloor a \bn_\la \rfloor , 
\lfloor (a + 1)  \bn_\la \rfloor -1 \rb$. 

\vip

{\bf Step 1.} We show here that for all $a\in \rr$, there exists an event
$\Omega^\la_{a,0}$, depending only on $(N_s^S(i),N_s^{M,\la}(i))_{i \in J_a^\la,
s\in [0,3\ba_\la/4]}$ such that

\vip

(i) on $\Omega^\la_{a,0}$, a.s., there is $i\in J_a^\la$ such that 
$\eta^\la_{\ba_\la s}(i)=0$ for all $s\in [0,3/4]$;

\vip

(ii) $\lim_{\la \to 0} \Pr[\Omega_{a,0}^\la] =1$.

\vip

This is done as in the case where $\beta=\infty$.
Consider simply $\Omega_{a,0}^\la = \{\exists\; i \in J_a^\la,
N^S_{3\ba_\la/4}(i)=0\}$. Clearly, (i) is satisfied. To check (ii),
recall that $|J_a^\la| = \bn_\la \to \infty$,
whence $\Pr[\Omega_{a,0}^\la]= 1 - \nu_S((0,3T_S/4))^{\bn_\la}  \to 1$,
because $\nu_S((0,3T_S/4))<1$ (recall that $\supp \nu_S =[0,T_S]$).

\vip

{\bf Step 2.} We now check that for all $a\in \rr$, all $t\geq 1/2$, 
there exists an event
$\Omega^\la_{a,t}$, depending only on $(N_s^S(i),N_s^{M,\la}(i))_{i \in J_a^\la,
s\in [(t-1/4)\ba_\la,(t+1/4)\ba_\la]}$ such that

\vip

(i) on $\Omega^\la_{a,t}$, a.s., there is $i\in J_a^\la$ such that 
$\eta^\la_{\ba_\la s}(i)=0$ for all $s\in [t,t+1/4]$;

\vip

(ii) $q_\la := \Pr[\Omega_{a,t}^\la]$ does not depend on $t,a$ and 
$q:=\liminf_{\la \to 0} q_\la>0$.

\vip

This is much easier than in the case where $\beta=\infty$: simply set
$$
\Omega_{a,t}^\la = \left\{ \exists\; i_0 \in J_a^\la, 
N^S_{\ba_\la(t+1/4)}(i_0)=N^S_{\ba_\la(t-1/4)}(i_0),N^{M,\la}_{\ba_\la t}(i_0)
>N^{M,\la}_{\ba_\la(t-1/4)}(i_0) \right\}.
$$
Point (i) is obviously checked since no seed fall on $i_0$ during
$[\ba_\la(t-1/4),\ba_\la(t+1/4) ]$ and a match falls on $i_0$ during 
$[\ba_\la(t-1/4),\ba_\la t]$. Next,
$\Pr\left(\Omega_{a,t}^\la \right) =  1 - r_\la^{|J_a^\la|}$, where
(for any $i\in \zz$, any $t\geq 1/4$)
\begin{align*}
r_\la:=&\Pr\left[N^S_{\ba_\la(t+1/4)}(i_0)>N^S_{\ba_\la(t-1/4)}(i_0) \hbox{ or }
N^{M,\la}_{\ba_\la t}(i_0)=N^{M,\la}_{\ba_\la(t-1/4)}(i_0) \right]
\\
=&\Pr\left[N^S_{\ba_\la /2}(i)>0 \hbox{ or } N^M_{\ba_\la /4}(i)=0 
\right]\\
=& \nu_S([0,T_S/2]) + \nu^\la_M((T_S/4,\infty)) -  
\nu_S([0,T_S/2])\nu^\la_M((T_S/4,\infty)).
\end{align*}
Due to $(H_M)$, 
$\nu_M^\la((T_S/4,\infty))= 1- \la \int_0^{T_S/4}\mu_M^1(\la t,\infty) dt
= 1- \la T_S (1+\e(\la))/4$, for some function $\e$ such that
$\lim_{\la\to 0} \e(\la) =0$. Setting $\alpha=\nu_S([0,T_S/2]) \in (0,1)$, 
we deduce that 
\begin{align*}
r_\la=& \alpha +  1- \la T_S (1+\e(\la))/4 - \alpha(1- \la T_S (1+\e(\la))/4)
\\
=&1 - \la (1-\alpha)T_S (1+\e(\la))/4.
\end{align*}
Recalling that $|J_a^\la|\simeq \bn_\la\simeq 1/(\la T_S)$, 
we finally conclude that
$$
\Pr\left(\Omega_{a,t}^\la \right)\simeq 1 - r_\la^{1/(\la T_S)} 
\to 1-e^{-(1-\alpha)/4}=:q>0.
$$
\vip

{\bf Steps 3 and 4} are exactly the same as when $\beta=\infty$.
\end{preuve}

\section{Localization of the results when $\beta\in\{\infty,BS\}$}

We recall that $\ba_\la$, $\bn_\la$ are defined
in (\ref{ala}) and (\ref{nla}).
For $A>0$, we set as usual $A_\la = \lfloor A \bn_\la \rfloor$ and 
$I_A^\la=\lb -A_\la,A_\la\rb$.
For $i\in \zz$, we set $i_\la=[i/\bn_\la, (i+1)/\bn_\la)$.

\vip

In the next sections, we will prove the following localized version
of Theorems \ref{converge1} and \ref{convergebs}, separating the cases
$\beta=\infty$ and $\beta=BS$.

\begin{prop}\label{converge1A}
Let $\beta \in \{\infty, BS\}$.
Assume $(H_M)$ and $(H_S(\beta))$. Let $A>0$ be fixed.
Consider, for each $\la\in (0,1]$, the process
$(Z^{\la,A}_t(x),D^{\la,A}_t(x))_{t\geq 0,x\in[-A,A]}$ associated with
the $FF_A(\mu_S,\mu_M^\la)$-process and the $LFF_A(\beta)$-process
$(Z_t^A(x),D_t^A(x),H_t^A(x))_{t\geq 0,x\in [-A,A]}$.

(a) For any $T>0$, any $\{x_1,\dots,x_p\}\subset [-A,A]$, 
$(Z^{\la,A}_t(x_i),D^{\la,A}_t(x_i))_{t\in [0,T],i=1,\dots,p}$ goes in law to 
$(Z_t^A(x_i),D_t^A(x_i))_{t\in [0,T],i=1,\dots,p}$, in 
$\dd([0,T], \rr\times\cI\cup\{\emptyset\})^p$, 
as $\la$ tends to $0$. Here
$\dd([0,\infty), \rr\times \cI\cup\{\emptyset\})$ 
is endowed with the distance $\bd_T$.

(b) For any $\{(t_1,x_1),\dots,(t_p,x_p)\}\subset 
[0,\infty)\times [-A,A]$ (assume also that $t_k \ne 1$ for $k=1,\dots, p$
if $\beta=\infty$),
$(Z^{\la,A}_{t_i}(x_i),D^{\la,A}_{t_i}(x_i))_{i=1,\dots,p}$ goes in law to 
$(Z_{t_i}^A(x_i),D_{t_i}^A(x_i))_{i=1,\dots,p}$ 
in $(\rr\times\cI\cup\{\emptyset\})^p$.
Here $\cI\cup\{\emptyset\}$ is endowed with $\bdelta$.

(c)-(i) Assume first that $\beta=\infty$. For all $t>0$, 
$$
\left(\frac{\psi_S\left(1-1/|C_A(\eta^{\la,A}_{\ba_\la t},0)| \right)}{\ba_\la} 
\indiq_{\{|C_A(\eta^{\la,A}_{\ba_\la t},0)|\geq 1\}}\right) \land 1 
$$
goes in law to $Z^A_t(0)$ as $\la\to 0$.

(c)-(ii) Assume next that $\beta=BS$. 
For any $t\geq 0$, any $k \in \nn$, there holds
$$
\lim_{\la\to 0} \Pr\left[|C_A(\eta^{\la,A}_{T_S t},0)|=k \right] = 
\E \left[q_k(Z^A_t(0)) \right],
$$
where $q_k(z)$ was defined, for $k\geq 0$ and $z\in [0,1]$, in (\ref{defqk}).
\end{prop}

Assuming for a moment that this proposition holds true, we
conclude the proofs of Theorems \ref{converge1} and \ref{convergebs}.

\begin{preuve} {\it of Theorem \ref{converge1}.} Let us first 
prove (a). Consider a continuous bounded functional 
$\Psi:\dd([0,T], \rr\times\cI\cup\{\emptyset\})^p \mapsto \rr$. 
We have to prove that $\lim_{\la \to 0} G_\la(\Psi)=0$, where
$$
G_\la(\Psi)=
\E\left[\Psi\left((Z^{\la}_t(x_i),D^{\la}_t(x_i))_{t\in [0,T],i=1,\dots,p} \right)
\right]-
\E\left[\Psi\left((Z_t(x_i),D_t(x_i))_{t\in [0,T],i=1,\dots,p} \right)
\right].
$$
Using now Propositions \ref{loc} and \ref{plocla}, we observe that for any
$A>2 \max_{i=1,\dots,p} |x_i|$, for all $\la\in (0,1]$ small enough,
\begin{align*}
&|G_\la(\Psi)|\\
\leq& 2||\Psi||_\infty\Pr\left[(Z^{\la,A}_t(x),D^{\la,A}_t(x))_{t\in 
[0,T],x\in [-A/2,A/2]} \ne (Z^{\la}_t(x),D^{\la}_t(x))_{t\in 
[0,T],x\in [-A/2,A/2]}  \right]\\
&+2||\Psi||_\infty \Pr\left[(Z^{A}_t(x),D^{A}_t(x))_{t\in 
[0,T],x\in [-A/2,A/2]} \ne (Z_t(x),D_t(x))_{t\in 
[0,T],x\in [-A/2,A/2]}  \right]\\
&+\left|\E\left[\Psi\left((Z^{\la,A}_t(x_i),D^{\la,A}_t(x_i))_{t\in 
[0,T],i=1,\dots,p} \right) \right]-
\E\left[\Psi\left((Z_t^A(x_i),D_t^A(x_i))_{t\in [0,T],i=1,\dots,p} \right)
\right]\right|\\
\leq& 4 ||\Psi||_\infty C_T e^{-\alpha_T A}\\
&+\left|\E\left[\Psi\left((Z^{\la,A}_t(x_i),D^{\la,A}_t(x_i))_{t\in 
[0,T],i=1,\dots,p} \right) \right]-
\E\left[\Psi\left((Z_t^A(x_i),D_t^A(x_i))_{t\in [0,T],i=1,\dots,p} \right)
\right]\right|.
\end{align*}
Thus Proposition \ref{converge1A}-(a) implies that
$\limsup_{\la \to 0}|G_\la(\Psi)|
\leq 4 ||\Psi||_\infty C_T e^{-\alpha_T A}$. We conclude by making $A$ tend
to infinity.

\vip

Point (b) is checked similarly. The proof of (c) is also similar,
since $D^\la_t(0)=D^{\la,A}_t(0)$ implies that
$C(\eta^\la_{\ba_\la t},0)= C_A(\eta^{\la,A}_{\ba_\la t},0)$.
\end{preuve}

\begin{preuve} {\it of Theorem \ref{convergebs}.} 
It is deduced from Propositions \ref{loc}, \ref{plocla} and 
\ref{converge1A} exactly as Theorem \ref{converge1}.
\end{preuve}

\section{Convergence proof when $\beta = BS$}\label{prbs}
\setcounter{equation}{0}

The aim of this section is to prove
Proposition \ref{converge1A} in the case where $\beta=BS$ and
this will conclude the proof of Theorem \ref{convergebs}. In the whole
section, we thus assume $(H_M)$ and $(H_S(BS))$.
The parameters $A>0$ and $T>0$ are fixed and we
omit the subscript/superscript $A$ in the whole proof.

\vip

We recall that $\ba_\la$, $\bn_\la$ and $\bm_\la$ are defined
in (\ref{ala}), (\ref{nla}) and (\ref{mla}).
For $A>0$, we set as usual $A_\la = \lfloor A \bn_\la \rfloor$ and 
$I_A^\la=\lb -A_\la,A_\la\rb$.
For $i\in \zz$, we set $i_\la=[i/\bn_\la, (i+1)/\bn_\la)$.
For $[a,b]$ an interval of $[-A,A]$ and $\la\in(0,1)$, we introduce,
assuming that $-A<a<b<A$,
\begin{align}
[a,b]_\la =&\lb \left\lfloor  \bn_\la a +\bm_\la
\right\rfloor +1 ,  \left\lfloor \bn_\la b-\bm_\la \right\rfloor -1\rb 
\subset \zz, \label{xxbb} \\
{[-A,b]}_\la =&\lb -A_\la ,  
\left\lfloor \bn_\la b-\bm_\la \right\rfloor -1\rb 
\subset \zz, \nonumber \\
{[a,A]}_\la =&\lb \left\lfloor  \bn_\la a +\bm_\la
\right\rfloor +1,  A_\la \rb 
\subset \zz, \nonumber
\end{align}
For $x\in (-A,A)$ and $\la\in(0,1)$, we introduce
\begin{align}\label{xla}
x_\la=&\lb \left\lfloor   \bn_\la x-\bm_\la 
\right\rfloor , \left\lfloor \bn_\la x + \bm_\la \right\rfloor 
\rb \subset \zz.
\end{align}

\subsection{Height of the barriers}\label{BSHB}

We need the following lemma. It describes the time
needed for a destroyed (microscopic) cluster to be regenerated. Below, 
we assume that the zone around $0$ is completely vacant at time
$T_S t_0$. Then we consider the situation where a match falls on the site 
$0$ at some time  $T_S t_1 \in (T_St_0,T_S(t_0+1))$ and we compute
the law of $\Theta_{t_0,t_1}$, which is the delay needed for the destroyed
cluster to be fully regenerated (divided by $T_S$).

\begin{lem}\label{deftheta}
Consider a family of i.i.d. $SR(\mu_S)$-processes 
$(N^{S}_{t}(i))_{t\geq 0,i\in\zz}$. Let $0\leq t_0 <t_1 < t_0 +1$ be fixed.
Put $\zeta_{t_0,t}(i)=\min(N^{S}_{T_S(t_0+t)}(i)-N^{S}_{T_S t_0}(i),1)$ 
and  $\zeta_{t_1,t}(i)=\min(N^{S}_{T_S(t_1+t)}(i)-N^{S}_{T_S t_1}(i),1)$ 
for all $t>0$ 
and $i\in \zz$.
Define
$$
\Theta_{t_0,t_1}=\inf \left\{ t>0\; : \; \forall\; i \in 
C(\zeta_{t_0,t_1-t_0},0),\; 
\zeta_{t_1,t}(i)=1 \right\} \in [0,1].
$$
The the law of $\Theta_{t_0,t_1}$ is $\theta_{t_1-t_0}$, recall
Definition \ref{defF}.
\end{lem}

\begin{proof}
We can assume that $t_0=0$ by stationarity. We 
put $u=t_1=t_1-t_0$ and write, for $h \in [0,1]$,
\begin{align*}
\Pr\left[\Theta_{t_0,t_1} \leq h \right]=& \Pr\left[N^S_{T_S u}(0)=0 \right]
+ \sum_{k\geq 1} \sum_{j=0}^{k-1} \Pr\Big[N^S_{T_S u}(j-k)=N^S_{T_S u}(j+1)=0,\\
&\hskip2.5cm \forall\; i \in \lb j-k+1 , j \rb, \; \; 
 N^S_{T_S u}(i)>0, \; N^S_{T_S (u+h)}(i)>N^S_{T_S u}(i) \Big].
\end{align*}
This yields, since $g_S(u,h)=\Pr[N^S_{T_S u}>0,N^S_{T_S (u+h)}>N^S_{T_S u}]$,
\begin{align*}
\Pr\left[\Theta_{t_0,t_1} \leq h \right]=& \nu_S([T_Su,T_S])
+ \sum_{k\geq 1} k [\nu_S([T_Su,T_S])]^2 [g_S(u,h)]^k \\
=& \nu_S([T_Su,T_S]) + \frac{[\nu_S([T_Su,T_S])]^2}{[1-g_S(u,h)]^2} 
g_S(u,h)  = \theta_u([0,h]),
\end{align*}
recall Definition \ref{defF}.
\end{proof}

\subsection{Persistent effect of microscopic fires}\label{sspp}

Here we study the effect of microscopic fires. First, they produce a barrier,
and then, if there are alternatively macroscopic fires on the left and right,
they still have an effect. 
This phenomenon is illustrated on Figure \ref{figpp} in the case of the limit 
process.

\begin{figure}[b] 
\fbox{
\begin{minipage}[c]{0.95\textwidth}
\centering
\includegraphics[width=11cm]{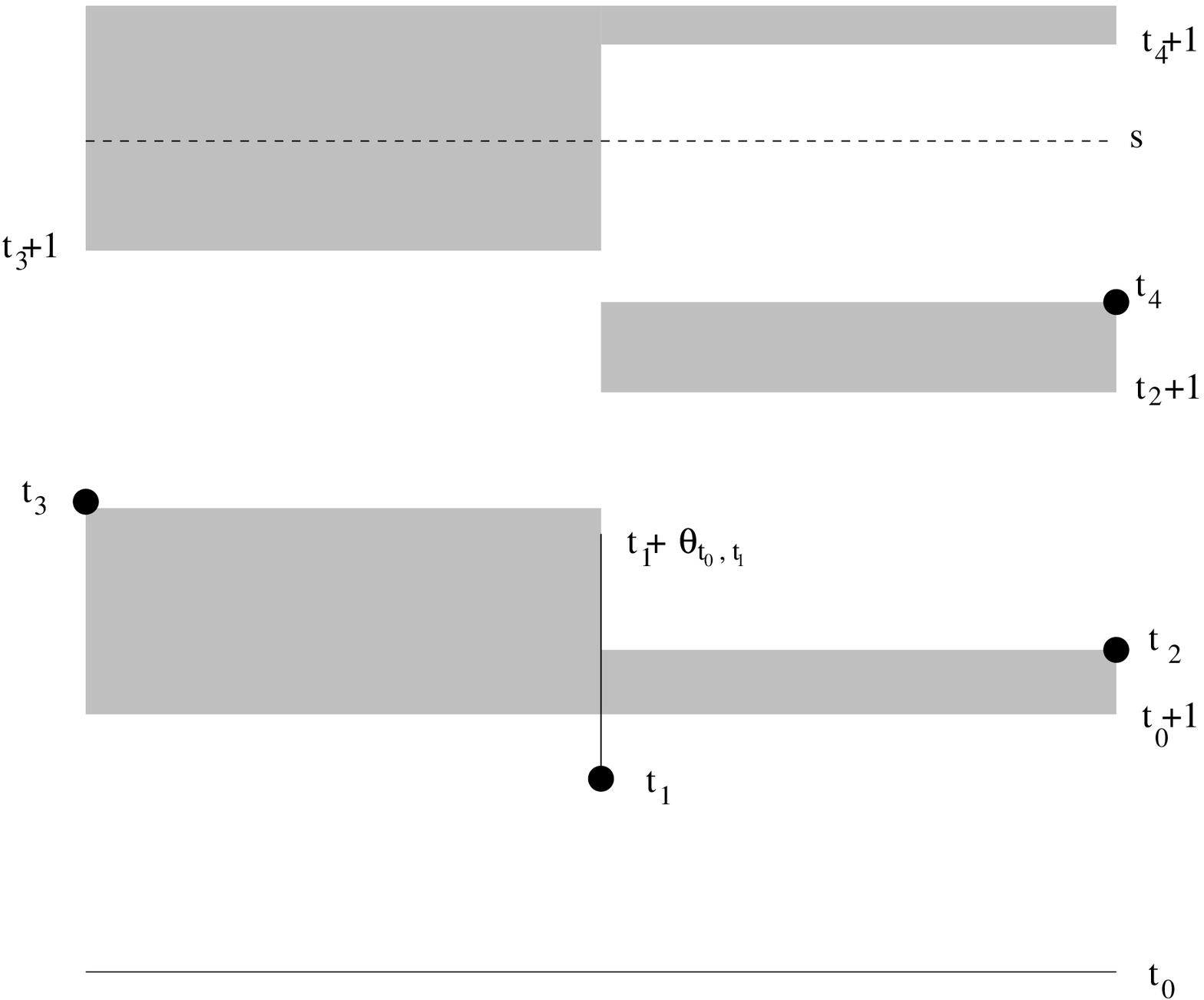}
\caption{Persistent effect of a microscopic fire. Here 
$\cR=(1;t_0,t_1,t_2,t_3,t_4;s)$.}
\label{figpp}
\end{minipage}
}
\end{figure}

\vip

We say that $\cR=(\e; t_0,t_1,\dots,t_K;s)$ satisfies $(PP)$ 
(like ping-pong) if 

\vip

(i) $K \geq 2$, $\e \in \{-1,1\}$,

\vip

(ii) $0< t_0 < t_1 < \cdots < t_K<s<t_K+1$, 

\vip

(iii) for all $k=0,\dots,K-1$, $t_{k+1}-t_k<1$,

\vip

(iv) $t_2-t_0>1$ and for all $k=2,\dots,K-2$, $t_{k+2}-t_k>1$.

\vip

We then set $\e_k= (-1)^k \e$ for $k\geq 0$. 

\vip

Consider a family of i.i.d. $SR(\mu_S)$-processes 
$(N^S_t(i))_{t\geq 0, i\in \zz}$. 

\vip

We introduce, for each $\la\in (0,1)$, 
the process $(\zeta^{\la,\cR}_t(i))_{t\geq t_0, i \in \lb -\bm_\la,\bm_\la\rb}$ 
defined as follows:

\vip

$\bullet$ for all $t\in [t_0,t_1)$, all $i \in \lb -\bm_\la,\bm_\la\rb$,
$\zeta_t^{\la,\cR}(i)=\min(N^{S}_{T_S t}(i) -N^{S}_{T_S t_0}(i),1)$, 

\vip

$\bullet$ for all $i \in \lb -\bm_\la,\bm_\la\rb$, 
$\zeta_{t_1}^{\la,\cR}(i) = \zeta_{t_1-}^{\la,\cR}(i) 
\indiq_{\{i \not\in C(\zeta^{\la,\cR}_{t_1-},0)\}}$,

\vip

$\bullet$ for $k=1, \ldots, K-1$, 

$(*)$ for all $t \in (t_k,t_{k+1})$,  $i \in \lb -\bm_\la,\bm_\la\rb$,  
$\zeta_t^{\la,\cR}(i)= \min{(\zeta_{t_k}^{\la,\cR}(i) + N^{S}_{T_S t}(i) 
-N^{S}_{T_S t_k}(i),1)}$,

$(*)$ for all $i \in \lb -\bm_\la,\bm_\la\rb$,   
$\zeta_{t_{k+1}}^{\la,\cR}(i) = \zeta_{t_{k+1}-}^{\la,\cR}(i) 
\indiq_{i \not \in C(\zeta_{t_{k+1}-}^{\la,\cR}, \epsilon_k \bm_\la)}$,

\vip

$\bullet$ for all  $t \in (t_K,\infty)$, $i \in \lb -\bm_\la,\bm_\la\rb$,  
$\zeta_t^{\la,\cR}(i)= \min{(\zeta_{t_K}^{\la,\cR}(i) + 
N^{S}_{T_S t}(i) -N^{S}_{T_S t_K}(i),1)}$.

\vip

Roughly, we start at time $T_S t_0$ with an empty configuration
and seeds fall according to $(N^S_t(i))_{t\geq 0, i\in \zz}$. At time
$T_S t_1$, there is a (microscopic) fire at $0$.
Then alternatively on the left and right, far away from $0$
(at $-\bm_\la$ or at $\bm_\la$),
there is a (macroscopic) fire at time $T_S t_k$.

\vip

Consider the event
$$
\Omega^S_{\cR}(\la)=\left\{
\exists -\bm_\la<i_1<i_2<i_3<\bm_\la\; : \; 
\zeta^{\la,\cR}_s(i_1)=\zeta^{\la,\cR}_s(i_3)=0, \zeta^{\la,\cR}_s(i_2)=1
\right\}.
$$

\begin{lem}\label{pingpong} Let $\cR=(\e; t_0,t_1,\dots,t_K;s)$ 
satisfy $(PP)$. 
Consider $\Theta_{t_0,t_1}$ defined in Lemma \ref{deftheta} and 
$(\zeta_t^{\la,\cR}(i))_{t\geq t_0, i \in \lb -\bm_\la,\bm_\la\rb}$ 
defined above. There holds
$$
\lim_{\la\to 0} \Pr\left(\left. \Omega^S_{\cR}(\la)  
\right| \Theta_{t_0,t_1}>t_2-t_1 \right) =1.
$$
\end{lem}

\begin{proof}
We assume that $\e=1$ and that $K$ is even for simplicity.
Fix $\alpha = 1/K$.

\vip

{\it First fire.} We put $C=C(\zeta^{\la,\cR}_{t_1-},0)$. Since $t_1-t_0<1$ 
(so that each site is vacant with probability $\nu_S((T_S(t_1-t_0),T_S))>0$ 
at time $t_1$),
the probability
that $C \subset \lb - \lfloor \alpha \bm_\la \rfloor, 
\lfloor \alpha \bm_\la \rfloor   
\rb$ clearly tends to $1$. Thus the match falling at time 
$t_1$ at $0$ destroys nothing outside 
$\lb - \lfloor \alpha \bm_\la \rfloor, \lfloor \alpha \bm_\la \rfloor \rb$
(with probability tending to $1$).
\vip

{\it Second fire.} 
Since $t_2-t_0>1$ (so that $T_S(t_2-t_0)>T_S$), at least one seed has
fallen, during $[t_0,t_2)$ 
on each site of $\lb \lfloor \alpha \bm_\la \rfloor+1, \bm_\la \rb$.
Thus the fire at time $t_2$ destroys completely this zone, but does
not affect $\lb -\bm_\la, -\lfloor \alpha \bm_\la \rfloor-1 \rb$,
because $t_2<t_1+\Theta_{t_0,t_1}$ and because by 
definition of $\Theta_{t_0,t_1}$, there is an empty site in 
$C\subset \lb - \lfloor \alpha \bm_\la \rfloor, 
\lfloor \alpha \bm_\la \rfloor \rb$ during 
$[t_1,t_1+\Theta_{t_0,t_1}]$.
\vip

{\it Third fire.} Since $t_3-t_2<1$, the probability that there is a 
vacant site in 
$\lb\lfloor \alpha \bm_\la \rfloor+1, \lfloor 2\alpha \bm_\la \rfloor \rb$ 
at time $t_3$ tends to $1$ as $\la\to 0$.

\vip

Next, all the sites of $\lb -\bm_\la, -\lfloor \alpha \bm_\la \rfloor-1 \rb$
are occupied at time $t_3-$ (because they have not been affected by a fire
and because $t_3-t_0>t_2-t_0>1$). Thus the fire at time $t_3$ destroys the zone
$\lb -\bm_\la, -\lfloor \alpha \bm_\la \rfloor-1 \rb$ and does not
affect the zone  $\lb \lfloor 2\alpha \bm_\la \rfloor , \bm_\la \rb$.
\vip

{\it Fourth fire.} Since $t_4-t_3<1$, the probability that there is 
(at least) a
vacant site in $\lb -\lfloor 2 \alpha \bm_\la\rfloor, 
-\lfloor \alpha \bm_\la \rfloor-1 \rb$
at time $t_4$ tends to $1$ as $\la\to 0$.

\vip

Next, all the sites of $\lb \lfloor 2\alpha \bm_\la \rfloor , \bm_\la \rb$
are occupied at time $t_4-$ (because they have not been affected by a fire
during $(t_2,t_4)$ with $t_4-t_2>1$). 
Thus the fire at time $t_4$ destroys the zone
$\lb \lfloor 2\alpha \bm_\la \rfloor , \bm_\la \rb$
and does not affect the zone  
$\lb -\bm_\la, -\lfloor 2\alpha \bm_\la \rfloor \rb$.
\vip

{\it Last fire and conclusion.} Iterating the procedure, we see that
with a probability tending to $1$ as $\la\to 0$, 
the fire at time $t_K$ destroys the zone
$\lb \lfloor (K\alpha/2) \bm_\la \rfloor , \bm_\la \rb =
\lb \lfloor \bm_\la/2 \rfloor , \bm_\la \rb$.

\vip

Then one easily concludes: since $0<s-t_K<1$, the probability that there
is at least one site in
$\lb \lfloor \bm_\la/2 \rfloor , \lfloor 2\bm_\la/3 \rfloor \rb$
with no seed falling during $[t_K,s]$  tends to $1$,
the probability that there is at least one site in
$\lb \lfloor 2\bm_\la/3 \rfloor +1, \lfloor 5\bm_\la/6 \rfloor \rb$
with at least one seed falling during $[t_K,s]$  tends to $1$,
and the probability that there is at least one site in
$\lb \lfloor 5\bm_\la/6 \rfloor +1, \bm_\la \rb$
with no seed falling during $[t_K,s]$  tends to $1$.
\end{proof}

\subsection{The coupling}\label{coucou}
We are going to construct a coupling between the $FF_A(\mu_S,\mu_M^\la)$-process
(on the time interval $[0,T_S T]$) and the $LFF_A(BS)$-process (on $[0,T]$). 

\vip

First, we couple a family
of i.i.d. $SR(\mu^\la_M)$-processes $(N^{M,\la}_t(i))_{t\geq 0, 
i \in \zz}$ with a Poisson measure $\pi_M(dt,dx)$ on $[0,T]\times [-A,A]$
with intensity measure $dtdx$  as in Proposition \ref{coupling1}.

\vip

We call $n:=\pi_M([0,T]\times[-A,A])$
and we consider the marks $(T_q,X_q)_{q=1,\dots,n}$ of $\pi_M$ ordered 
in such a way that $0<T_1<\dots<T_n<T$.

\vip

Next, we introduce some i.i.d. families of i.i.d. $SR(\mu_S)$-processes
$(N^{S,q}_t(i))_{t\geq 0,i\in \zz}$, for $q=0,1,\dots$, independent of 
$\pi_M$ and $(N^{M,\la}_t(i))_{t\geq 0, i \in \zz}$.

\vip

Then we build a family of i.i.d. $SR(\mu_S)$-processes, independent
of $(N^{M,\la}_t(i))_{t\geq 0, i \in \zz}$ and of $\pi_M$, as follows.

\vip

$\bullet$ For $q \in \{1, \ldots, n\}$, for all $i \in (X_q)_\la$,
set $(N^{S,\la}_t(i))_{t\geq 0}=
(N^{S,q}_t(i-\lfloor \bn_\la X_q \rfloor))_{t\geq 0}$
(if $i$ belongs to  $(X_q)_\la\cap (X_r)_\la$ for some $q<r$, set e.g. 
$(N^{S,\la}_t(i))_{t\geq 0}=(N^{S,q}_t(i-\lfloor \bn_\la X_q \rfloor))_{t\geq 0}$.
This will occur with a very small probability, so that this choice is not
important).

\vip

$\bullet$ For all other $i \in \zz$ 
set $(N^{S,\la}_t(i))_{t\geq 0} = (N^{S,0}_t(i))_{t\geq 0}$.    

\vip

The $FF_A(\mu_S,\mu_M^\la)$-process 
$(\eta^\la_t(i))_{t\geq 0, i \in I_A^\la}$ is built upon the seed processes
$(N^{S,\la}_t(i))_{t\geq 0, i \in \zz}$ and match
processes $(N^{M,\la}_t(i))_{t\geq 0, i \in \zz}$.

\vip

The advantage of the previous construction is the following.
When a match falls at some $X_q$ for the $LFF_A(BS)$-process,
it will fall at $\lfloor \bn_\la X_q \rfloor$ in the discrete process,
and thus if it is {\it microscopic}, it will  involve the same seed processes
for all values of $\la$.

It also considerably 
simplifies the dependence/independence considerations.

\vip

Finally, we build the $LFF_A(BS)$-process. We consider 
the Poisson measure $\pi_M$ previously introduced, 
and for all $0<t_0<t_1<t_0+1$, for all $q=1,\dots,n$, we consider 
$\Theta_{t_0,t_1}^q$ defined from $(N^{S,q}_t(i))_{t\geq 0,i\in \zz}$ 
as in Lemma \ref{deftheta}. 
We define $(Z_t(x),D_t(x),H_t(x))_{t\in [0,T],x\in [-A,A]}$
as follows:

\begin{algo}\label{algo2} \rm
Consider the marks $(T_k,X_k)_{k=1,\dots,n}$ of
$\pi_M$ in $[0,T]\times[-A,A]$, ordered
chronologically and set $T_0=0$. 

\vip

{\bf Step 0.} Put $Z_0(x)=H_0(x)=0$ and $D_0(x)=\{x\}$ 
for all $x\in [-A,A]$. 

\vip

Assume that for some $k\in \{0,\dots,n-1\}$,  
$(Z_t(x),D_t(x),H_t(x))_{t\in[0,T_k],x\in[-A,A]}$ has been built.

\vip

{\bf Step k+1.} For $t\in (T_k,T_{k+1})$ and $x\in [-A,A]$,
put $Z_t(x)=\min(1,Z_{T_k}(x)+t-T_k)$, set 
$H_t(x)=\max(0,H_{T_k}(x)-t+T_k)$ and then define $D_t(x)$
as in (\ref{dfcc}).
Finally, build $(Z_{T_{k+1}}(x),D_{T_{k+1}}(x),H_{T_{k+1}}(x))$ as follows.

\vip

(i) If $Z_{T_{k+1}-}(X_{k+1})=1$, set $H_{T_{k+1}}(x) =H_{T_{k+1}-}(x)$ for all
$x\in [-A,A]$ and consider $[a,b]:= D_{T_{k+1}-}(X_{k+1})$.
Set $Z_{T_{k+1}}(x) = 0$ for all $x\in (a,b)$ and 
$Z_{T_{k+1}}(x) = Z_{T_{k+1}-}(x)$
for all $x\in [-A,A] \setminus [a,b]$. Set finally  
$Z_{T_{k+1}}(a)=0$ if $Z_{T_{k+1}-}(a)=1$ and $Z_{T_{k+1}}(a)=Z_{T_{k+1}-}(a)$ 
if $Z_{T_{k+1}-}(a)<1$ and 
$Z_{T_{k+1}}(b)=0$ if $Z_{T_{k+1}-}(b)=1$ and $Z_{T_{k+1}}(b)=Z_{T_{k+1}-}(b)$ 
if $Z_{T_{k+1}-}(b)<1$.

\vip

(ii) If $Z_{T_{k+1}-}(X_{k+1})<1$, set 
$H_{T_{k+1}}(X_{k+1})=\Theta^{k+1}_{T_{k+1}- Z_{T_{k+1}-}(X_{k+1}),T_{k+1}}$,
$Z_{T_{k+1}}(X_{k+1})=Z_{T_{k+1}-}(X_{k+1})$ and 
$(Z_{T_{k+1}}(x),H_{T_{k+1}}(x))
=(Z_{T_{k+1}-}(x),H_{T_{k+1}-}(x))$ for
all $x\in [-A,A]\setminus \{X_{k+1}\}$.

\vip

(iii) Using the values of $(Z_{T_{k+1}}(x),H_{T_{k+1}}(x))_{x\in[-A,A]}$,
compute $(D_{T_{k+1}}(x))_{x\in[-A,A]}$
as in (\ref{dfcc}).
\end{algo}

\begin{lem} \label{couplagelimite}
The process $(Z_t(x),D_t(x),H_t(x))_{t\in [0,T],x\in [-A,A]}$ built in
Algorithm \ref{algo2}
is a $LFF_A(BS)$-process.
\end{lem}

\begin{proof}
The only difference between algorithms \ref{algo1} and \ref{algo2}
is that in Step $k+1$, point (ii), we use 
$\Theta_{T_{k+1}- Z_{T_{k+1}-}(X_{k+1}),T_{k+1}}^{k+1}$ instead of
$F_S(Z_{T_{k+1}-}(X_{k+1}), V_{k+1})$. But due
to Lemma \ref{deftheta} and Definition \ref{defF}, these two variables have 
the same law $\theta_{Z_{T_{k+1}-}(X_{k+1})}$ 
(conditionally on  $T_{k+1}$, $X_{k+1}$ and
$(Z_t(x),D_t(x),H_t(x))_{t\in [0,T_{k+1}), x\in [-A,A]}$). Indeed, 
it suffices to use that in Algorithm \ref{algo1}, 
$V_{k+1}$ is independent of $Z_{T_{k+1}-}(X_{k+1})$,
while in Algorithm \ref{algo2}, 
the family $(N^{S,k+1}_t(i))_{t\geq 0,i\in \zz}$ is
independent of $(T_{k+1},Z_{T_{k+1}-}(X_{k+1}))$.
\end{proof}

Finally, we observe that  $(Z_t(x),D_t(x),H_t(x))_{t\in[0,T],x\in[-A,A]}$ 
depends only on $\pi_M$ and on $((N^{S,q}_t(i))_{t\in [0,T], i \in \zz})_{q\geq 1}$.
It is independent of $(N^{S,0}_t(i))_{t\in [0,T], i \in \zz}$.

\subsection{A favorable event}

First, we know from Proposition \ref{coupling1} that
$$
\Omega^M_{A,T}(\la):=\left\{\forall t\in [0,T],\;
\forall i\in I_A^\la, \;
\Delta N^{M,\la}_{T_S t}(i) \ne 0 \; \hbox{iff} \; 
\pi_M(\{t\}\times i_\la) \ne 0
\right\}
$$
satisfies $\lim_{\la \to 0} \Pr[\Omega^M_{A,T}(\la)]=1$.
Next, we recall that the marks of $\pi_M$ are called
$(T_1,X_1),$ $\dots,$ $(T_n,X_n)$ and are ordered chronologically.
We introduce $\cT_M=\{0,T_1,\dots,T_n\}$,  
$\cB_M=\{X_1,\dots,X_n\}$, as well as the set
$\cC_M$ of connected components of  $[-A,A]\setminus \cB_M$ (sometimes 
referred to as {\it cells}).  

For $\alpha>0$, we consider the event
$$
\Omega_M(\alpha)=\left\{\min_{s,t\in \cT_M,s\ne t}|t-s|\geq \alpha, 
\min_{x,y \in \cB_M \cup\{-A,A\},x\ne y}|x-y|\geq \alpha, 
\right\},
$$
which clearly satisfies $\lim_{\alpha \to 0} \Pr[\Omega_M(\alpha)]=1$.
Observe that for any given $\alpha>0$, there is $\la_\alpha>0$
such that for all $\la\in(0,\la_\alpha)$, on $\Omega_M(\alpha)$,
there holds that

\vip

$\bullet$ for all $x,y\in \cB_M\cup\{-A,A\}$ with $x\ne y$, 
$x_\la\cap y_\la =\emptyset$, 

\vip

$\bullet$ the family $\{c_\la, c\in \cC_M\}\cup \{x_\la, x \in 
\cB_M\}$ is a partition of $I_A^\la$ (recall (\ref{xxbb}) and (\ref{xla})).

\vip

Indeed, it suffices that $\sup_{(0,\la_\alpha)} [\bm_{\la}/\bn_{\la}] <\alpha/4$.

\vip

Let $q\in \{1,\dots,n\}$.
We call $\cU_q$ the set of all possible $\cR=(\e,t_0,\dots,t_K;s)$
satisfying $(PP)$ with $\e\in \{-1,1\}$, with $\{t_0,\dots,t_K,s\}
\subset \cT_M$ and with $\Theta_{t_0,t_1}^q>t_2-t_1$.
We introduce, for $q=1,\dots,n$ and $\cR\in \cU_q$, 
the event $\Omega^{S,q}_{\cR}(\la)$ defined as in Subsection \ref{sspp}
with the $SR(\mu_S)$-processes $(N^{S,q}_t(i))_{t\geq 0, i\in \zz}$.
Then we put 
$$
\Omega^S_1(\la)=\cap_{q=1}^n \cap_{\cR\in\cU_q} \Omega^{S,q}_{\cR}(\la),
$$
which satisfies $\lim_{\la \to 0} \Pr\left( \Omega^S_1 (\la) \right)=1$
thanks to Lemma \ref{pingpong} (since for each $q$,
$(N^{S,q}_t(i))_{t\geq 0, i \in \zz}$ is independent of $\pi_M$
and since conditionally on $\pi_M$, the set $\cU_q$ is finite).

\vip

We also consider the event $\Omega^S_2(\la)$ on which the following 
conditions hold:
for all $t_1,t_2 \in \cT_M$ with  $0<t_2-t_1<1$, for all $q=1,\dots,n$,
there are  
$$-\bm_\la<i_1<i_2 < -\bm_\la/2 <i_3<0<i_4< \bm_\la/2 <i_5<i_6<\bm_\la$$ 
such that

\vip

$\bullet$ for $j=1,3,4,6$, $N^{S,q}_{T_S t_2}(i_j)- N^{S,q}_{T_S t_1}(i_j)=0$,

\vip

$\bullet$ for $j=2,5$, $N^{S,q}_{T_S t_2}(i_j)- N^{S,q}_{T_S t_1}(i_j)>0$.

\vip

There holds $\lim_{\la \to 0} \Pr\left( \Omega^S_2 (\la) \right)=1$.
Indeed, it suffices to prove that almost surely, $\lim_{\la\to 0} 
\Pr\left(\Omega^S_2 (\la) \vert \pi_M\right)=1$. 
Since there are a.s. finitely
many possibilities for $q,t_1,t_2$ and since $\pi_M$ is independent
of  $(N^{S,q}_t(i))_{t\geq 0, i\in \zz}$, it suffices to work with 
a fixed $q\in \{1,\dots,n\}$ and some fixed $0<t_2-t_1<1$.

\vip

Observe that for each $i$, 
$\Pr(N^{S,q}_{T_S t_2}(i)- N^{S,q}_{T_S t_1}(i)=0)=\nu_S((T_S(t_2-t_1),T_S))<1$
and $\Pr(N^{S,q}_{T_S t_2}(i)- N^{S,q}_{T_S t_1}(i)>0)=\nu_S((0,T_S(t_2-t_1)))<1$
by definition of $T_S$ and since $t_2-t_1<1$. Recall
also that $\bm_\la$ tends to infinity. Thus during $[T_St_1,T_St_2]$, 
the probability that 
a seed falls on each site of $\lb -\bm_\la+1, -\lfloor \bm_\la/4\rfloor \rb$ 
tends to $0$, the probability that no seed at all falls on 
$\lb -\lfloor \bm_\la/4\rfloor +1, -\lfloor \bm_\la/2\rfloor-1 \rb$
tends to $0$, the probability a seed falls on each site of
$\lb -\lfloor \bm_\la/2 \rfloor, -1\rb$ tends to $0$, etc.

\vip

We finally introduce the event 
$$
\Omega(\alpha,\la)= \Omega^M_{A,T}(\la)\cap\Omega_M(\alpha)\cap 
\Omega^S_1(\la)\cap\Omega^S_2(\la).
$$
We observe that $\Omega(\alpha,\la)$ is independent of 
$(N^{S,0}_t(i))_{t\geq 0, i \in \zz}$ and that for any
$\e>0$, choosing $\alpha>0$ small enough, $\Pr[\Omega(\alpha,\la)]>1-\e$
for all $\la>0$ small enough.

\subsection{Heart of the proof}\label{hpBS}
We now handle the main part of the proof.

\vip

Consider the $LFF_A(BS)$-process.
Observe that by construction, we have, for $c\in \cC_M$ and $x,y\in c$,
$Z_t(x)=Z_t(y)$ for all $t\in [0,T]$, thus we can introduce
$Z_t(c)$. 

\vip

If $x \in \cB_M$, it is at the boundary of two cells 
$c_-,c_+ \in \cC_M$ and then we set $Z_t(x_-)=Z_t(c_-)$ and $Z_t(x_+)=Z_t(c_+)$ 
for all $t\in [0,T]$. 

\vip

If $x \in (-A,A)\setminus \cB_M$, we put  $Z_t(x_-)=Z_t(x_+)=Z_t(x)$ for all
$t\in [0,T]$.

\vip

For $x\in \cB_M$ and $t\geq 0$ we set 
$\tH_t(x)=\max(H_t(x),1-Z_t(x),1-Z_t(x_-),1-Z_t(x_+))$. 
Observe that $x$ is {\it microscopic} or
{\it acts like a barrier} at time $t$
if and only if $\tH_t(x)>0$.

\vip

Actually $Z_t(x)$ always equals either $Z_t(x_-)$ or $Z_t(x_+)$ 
and these can be distinct only at a point where has occurred  a 
microscopic fire (that is if $x=X_q$ for some $q\in \{1,\dots,n\}$,
if $Z_{T_q-}(X_q)<1$ and if $t > T_q$).

\vip

For all $x\in (-A,A)$ and $t \in [0,T]$, we put 
$$
\tau_t(x)= \sup \left\{s\leq t\;:\; Z_s(x_+)=Z_s(x_-)=Z_s(x)=0\right\} \in [0,t]
\cap \cT_M.
$$
For $c\in \cC_M$ and $t\in [0,T]$, it clearly holds that
$\tau_t(x)=\tau_t(y)$ for all $x,y\in c$,
so that we can also define $\tau_t(c)$.

\vip

Observe, using Algorithm \ref{algo2}, that
\begin{align}
&\hbox{for }x\notin \cB_M, \; \; Z_t(x) = \min{(t-\tau_t(x),1)} \hbox{ for all }
t\in[0,T], \label{tauzC}\\
&\hbox{for }q=1,\dots,n,\;\;  Z_t(X_q) = \min{(t-\tau_t(X_q),1)} 
\hbox{ for all }
t\in[0,T_{q}).\label{tauzB}
\end{align}

Indeed, $\tau_t(x)$ stands for the last time before $t$ where $x$ was involved 
in a macroscopic fire (with the convention that a macroscopic fire occurs
at time $0$). Thus for $x \notin \cB_M$, 
if $t-\tau_t(x)\geq 1$, $Z_t(x)=1$,
and if $t-\tau_t(x)<1$, $Z_t(x)=t-\tau_t(x)$. For $x=X_q$, the same 
reasoning holds during $[0,T_q)$.

\vip

We also define for all $t\in [0,T]$, all $c \in \cC_M$ and all $x\in (-A,A)$
here ($c_\la$ is defined by (\ref{xxbb}) and $x_\la$ by (\ref{xla}))
\begin{align*}
\tau_t^\la(c)=& \sup \left\{s\leq t\;:\; \forall i \in c_\la, 
\eta_{T_S t-}^\la(i)=1 \hbox{ and } \eta_{T_S t}^\la(i)=0   \right\}\in [0,t],\\
\rho_t^\la(c)=& \sup \left\{s\leq t\;:\; \exists i \in c_\la, 
\eta_{T_S t-}^\la(i)=1 \hbox{ and } \eta_{T_S t}^\la(i)=0   \right\}\in [0,t],\\
\tau_t^\la(x)=& \sup \left\{s\leq t\;:\; \forall i \in x_\la, 
\eta_{T_S t-}^\la(i)=1 \hbox{ and } \eta_{T_S t}^\la(i)=0 \right\}\in [0,t]
\end{align*}
with the convention that $\eta_{0-}^\la(i)=1$ for all $i\in I_A^\la$.
Observe that on $\Omega^M_{A,T}(\la)$, there holds that 
$\tau_t^\la(c),\rho_t^\la(c),\tau_t^\la(x) \in [0,t]\cap \cT_M$
for all $t\in [0,T]$, all $c \in \cC_M$ and all $x\in (-A,A)$.

\vip

For $t \in [0,T]$, consider the event  
$$
\Omega_t^\la =  \left\{ \forall s\in [0,t], \forall c \in\cC_M, 
\tau^{\la}_{s}(c) = \rho^\la_s(c)=\tau_{s}(c) \hbox{ and } \forall x \in \cB_M,
\tau^{\la}_{s}(x) =\tau_{s}(x)\right\}.
$$ 
We define $\Omega_{t-}^\la$ similarly, replacing $[0,t]$
by $[0,t)$. The aim of the subsection is to prove the following
result.

\begin{lem}\label{ggg}
For any $\alpha>0$, any $\la\in(0,\la_\alpha)$, $\Omega^\la_T$ a.s.
holds on $\Omega(\alpha,\la)$.
\end{lem}

\begin{proof}
We work on  $\Omega(\alpha,\la)$ and assume that $\la\in(0,\la_\alpha)$.
Clearly, $\tau_0(x)=\tau_0^\la(x)=0$ and $\tau_0(c)=\tau_0^\la(c)
=\rho^\la_0(c)=0$ 
for all $x\in\cB_M$, all $c\in\cC_M$, so that $\Omega^\la_0$ a.s. holds.
We will show that for $q=0,\dots,n-1$, $\Omega^\la_{T_q}$ implies
$\Omega^\la_{T_{q+1}}$. This will prove that $\Omega^\la_{T_{n}}$ holds.
The extension to $\Omega^\la_{T}$ will be straightforward (see Step 1 below).

\vip

We thus fix  $q\in\{0,\dots,n-1\}$ and assume $\Omega^\la_{T_q}$. 
We repeatedly use below that
on the time interval $(T_q,T_{q+1})$, there are no fires at all (in $[-A,A]$)
for the $LFF_A(BS)$-process and no fires at all (in $I_A^\la$) during 
$(T_S T_{q},T_S T_{q+1})$ for the $FF_A(\mu_S,\mu_M^{\la})$-process
(use $\Omega^{M}_{A,T}(\la)$).

\vip

{\bf Step 1.} To start with, we observe that since there are no fires 
between $T_S T_q$ and $T_S T_{q+1}$, we have 
$\tau^\la_{t}(x)=\tau^\la_{T_q}(x)$, $\tau^\la_{t}(c)=\tau^\la_{T_q}(c)$ and
$\rho^\la_{t}(c)=\rho^\la_{T_q}(c)$
for all $x\in \cB_M$, all $c\in\cC_M$, all $t\in [T_q,T_{q+1})$
(because $\eta_{T_S t}^\la(i)$ is nondecreasing on $[T_q,T_{q+1})$
for all $i\in I_A^\la$). By the same way, $\tau_{t}(x)=\tau_{T_q}(x)$ 
and $\tau_{t}(c)=\tau_{T_q}(c)$
for all $x\in \cB_M$, all $c\in\cC_M$, all $t\in [T_q,T_{q+1})$
(because $Z_t(x),Z_{t}(x_+),Z_t(x_-)$ are nondecreasing on $[T_q,T_{q+1})$
for all $x \in [-A,A]$).
Hence for $t \in [T_q,T_{q+1})$, 
$\Omega_t^\la = \Omega^\la_{T_q}$. Thus $\Omega^\la_{T_q}$ implies
$\Omega^\la_{T_{q+1}-}$.

\vip
{\bf Step 2.} 
Let $c \in \cC_M$. Observe that on $\Omega^\la_{T_{q+1}-}$, there holds,
for all $i \in c_\la$,  
\begin{align}
\label{danslescellules}
\eta^\la_{T_S T_{q+1}-}(i)
=\min\left(N^{S,0}_{T_S T_{q+1}-}(i) -N^{S,0}_{T_S \tau_{T_{q}}(c)}(i),1
\right).
\end{align}
Indeed, seeds are falling on $i$ according to $(N^{S,0}_{t}(i))_{t\geq 0}$.
Furthermore, we know from Step 1 that $\rho^\la_{T_{q+1}-}(c)=\tau^\la_{T_{q+1}-}(c)
=\tau_{T_{q+1}-}(c)=\tau_{T_{q}}(c)$. By definition of $\tau^\la_{T_{q+1}-}(c)$,
$\eta^\la_{T_S\tau_{T_{q}}}(i)=0$ for all $i\in c_\la$.
And by definition of $\rho^\la_{T_{q+1}-}(c)$,
no fire affects $c_\la$ during $(T_S\rho^\la_{T_{q+1}-}(c),T_S T_{q+1})$.

\vip

{\bf Step 3.} We show here that if $Z_{T_{q+1}-}(X_{q+1})<1$, there exist
$j_1,j_2,j_3,j_4 \in (X_{q+1})_\la$ such that 
$j_1<j_2<\lfloor \bn_\la X_{q+1} \rfloor <j_3<j_4$ and 
$\eta^\la_{T_S T_{q+1}-}(j_2)=\eta^\la_{T_S T_{q+1}-}(j_3)=0$ and
$\eta^\la_{T_S T_{q+1}-}(j_1)= \eta^\la_{T_S T_{q+1}-}(j_4)=1$.

\vip

Recall that for $i\in (X_{q+1})_\la$, the seeds fall according to 
$(N^{S,q+1}_t(i- \lfloor \bn_\la X_{q+1} \rfloor))_{t\geq 0}$.
Recall also that $\tau^\la_{T_{q+1}-}(X_{q+1})=\tau_{T_{q+1}-}(X_{q+1})$ 
(by Step 1), so that by definition, $(X_{q+1})_\la$ is completely vacant
at time $T_S \tau_{T_{q+1}-}(X_{q+1})$.
Recall finally that $\tau_{T_{q+1}-}(X_{q+1}) \in \cT_M$ (and so does $T_{q+1}$).

\vip

Observe that by (\ref{tauzB}), $Z_{T_{q+1}-}(X_{q+1})<1$ implies
that  $T_{q+1}-\tau_{T_{q+1}-}(X_{q+1})<1$.
Since we work on $\Omega^S_2(\la)$,
we know that there are some sites $i_1<i_2<i_3<\lfloor \bn_\la X_{q+1} \rfloor
<i_4<i_5<i_6$ in $(X_{q+1})_\la$ such that
at least one seed has fallen on $i_2$ and $i_5$ and 
no seed has fallen on $i_1,i_3,i_4,i_6$ during
$[T_S \tau_{T_{q+1}-}(X_{q+1}),T_ST_{q+1})$. All this
implies that $\eta^\la_{T_ST_{q+1}-}(i_2)=\eta^\la_{T_ST_{q+1}-}(i_5)=1$ and 
$\eta^\la_{T_ST_{q+1}-}(i_3)=\eta^\la_{T_ST_{q+1}-}(i_4)= 0$
(because the vacant sites $i_1,i_6$ protect the occupied sites
$i_2,i_4$ from fires falling outside $(X_{q+1})_\la$ and because
no fire falls on $(X_{q+1})_\la$ during $[0,T_ST_{q+1})$).

\vip

{\bf Step 4.} Next we check that if $Z_{T_{q+1}-}(c)=1$ for some $c\in\cC_M$,
then $\eta^\la_{T_S T_{q+1}-}(i)=1$ for all $i\in c_\la$.

\vip

Recalling (\ref{tauzC}), we see that $Z_{T_{q+1}-}(c)=1$ implies that
$T_{{q+1}}-\tau_{T_{q+1}-}(c)\geq 1$ and thus  $T_{{q+1}}-\tau_{T_{q}}(c)\geq 1$
by Step 1.
Using (\ref{danslescellules}), we conclude that for all 
$i\in c_\la$, $\eta_{T_ST_{q+1}-}^\la(i)=\min(N^{S,0}_{T_S T_{q+1}-}(i)
-N^{S,0}_{T_S \tau_{T_q}(c)(i)}  ,1)=1$
(at least one seed falls on each site
during a time interval of length greater than $T_S$).

\vip

{\bf Step 5.} We now prove that if $\tH_{T_{q+1}-}(x)=0$ for some $x \in \cB_M$,
then for all $i\in x_\la$, $\eta^\la_{T_S T_{q+1}-}(i)=1$.

\vip

{\it Preliminary considerations.}
Let $k\in \{1,\dots,n\}$
such that $x=X_k$, which is at the boundary of two
cells $c_-,c_+ \in \cC_M$. We know that 
$\tH_{T_{q+1}-}(x)=0$, whence $H_{T_{q+1}-}(x)=0$ and 
$Z_{T_{q+1}-}(x)=Z_{T_{q+1}-}(c_+)=Z_{T_{q+1}-}(c_-)=1$. This implies
that $T_{q+1}\geq 1$, because $Z_t(x)=t$ for all $t<1$, all $x\in [-A,A]$.

\vip

No fire has concerned
$(c_-)_\la$ during
$(T_S\rho^\la_{T_{q+1}-}(c_-), T_S T_{q+1})$
(by definition of $\rho^\la_{T_{q+1}-}(c_-)$).
But Step 1 implies that 
$\rho^\la_{T_{q+1}-}(c_-)=\tau_{T_{q+1}-}(c_-) \leq T_{q+1}-1$,
because $Z_{T_{q+1}-}(c_-)=1$, see (\ref{tauzC}). 
Using a similar argument for $c_+$, we conclude that 
no match falling outside $(X_k)_\la$ can affect $(X_k)_\la$ during
$(T_S(T_{q+1}-1), T_ST_{q+1})$ (because to affect $(X_k)_\la$, a match falling
outside $(X_k)_\la$ needs to cross $c_-$ or $c_+$). 

\vip

{\it Case 1.} First assume that $k\geq q+1$. Then we know that
no fire has fallen on $(X_k)_\la$ during $[0,T_ST_{q+1})$.
Due to the preliminary considerations, we deduce that 
no fire at all has concerned $(X_k)_\la$ during 
$(T_S(T_{q+1}-1), T_ST_{q+1})$. This time interval is of
length greater than $T_S$. Thus  $(X_k)_\la$ is completely
occupied at time $T_S T_{q+1}-$.

\vip

{\it Case 2.} Assume that $k \leq q$ and $Z_{T_k-}(X_k)=1$, 
so that there already has been a macroscopic fire
in $(X_k)_\la$ (at time $\ba_\la T_k$).
Since then $Z_{T_k}(X_k)=0$ and
$Z_{T_{q+1}-}(X_k)=1$, we deduce
that $T_{q+1}-T_k\geq 1$. We conclude as in Case 1 that 
no fire at all has concerned $(X_k)_\la$ during 
$(T_S(T_{q+1}-1), T_ST_{q+1})$, which implies the claim.

\vip

{\it Case 3.} Assume that $k \leq q$ and $Z_{T_k-}(X_k)<1$ and 
$T_{q+1}-T_k\geq 1$.
Then there already has been a microscopic fire
in $(X_k)_\la$ (at time $T_ST_k$). But there are no fire in 
$(X_k)_\la$ during $(T_ST_k,T_S T_{q+1})$ and we conclude as in Case 2.

\vip

{\it Case 4.} Assume finally 
that $k \leq q$ and $Z_{T_k-}(X_k)<1$ and $T_{q+1}-T_k<1$.
There has been a microscopic fire in $(X_k)_\la$ (at time $T_ST_k$).
Since $H_{T_{q+1}-}(X_k)=0$, we deduce (see Algorithm \ref{algo2})
that $T_k+\Theta^k_{T_k-Z_{T_k-}(X_k),T_k} \leq T_{q+1}$. 

Consider the zone $C(\eta^\la_{T_ST_k-},\lfloor \bn_\la X_k \rfloor)$ destroyed
by the match falling at time $T_ST_k$. This zone is completely occupied
at time $T_S(T_k+\Theta_{T_k-Z_{T_k-}(X_k),T_k})\leq T_ST_{q+1}$ by definition of
$\Theta_{T_k-Z_{T_k-}(X_k),T_k}$, see Lemma \ref{deftheta}, using here
again the preliminary considerations.

\vip

We deduce that
$C(\eta^\la_{T_ST_k-},\lfloor \bn_\la X_k \rfloor)$ is completely occupied
at time $T_ST_{q+1}-$. 

\vip

Consider now
$i\in (X_k)_\la \setminus C(\eta^\la_{T_ST_k-},\lfloor \bn_\la X_k \rfloor)$.
Then $i$ has not been killed by the fire falling on
$\lfloor \bn_\la X_k \rfloor$. Thus $i$ cannot have been killed
during $(T_S(T_{q+1}-1),T_ST_{q+1})$ (due to the preliminary considerations)
and is thus occupied at time $T_ST_{q+1}-$. This implies the claim.

\vip

{\bf Step 6.} Let us now prove 
that if $\tH_{T_{q+1}-}(x)>0$ and $Z_{T_{q+1}-}(x_+)=1$
for some $x \in \cB_M$, there are $i_1,i_2 \in x_\la$ such that $i_1<i_2$
and $\eta^\la_{T_S T_{q+1}-}(i_1)=1$, $\eta^\la_{T_S T_{q+1}-}(i_2)=0$.

\vip

Recall that $x$ is at the boundary of two cells $c_-,c_+$.
We have either $H_{T_{q+1}-}(x)>0$
or  $Z_{T_{q+1}-}(c_-)<1$ (because $Z_{T_{q+1}-}(c_+)=1$ by assumption). Clearly,
$x=X_k$ for some $k\leq q$, with $Z_{T_k-}(X_k)<1$
(else, we would have  $H_{t}(x)=0$ and $Z_t(c_-)=Z_t(c_+)$ for all $t\in 
[0,T_{q+1})$).
Thus, recalling (\ref{tauzB}), $T_k-Z_{T_k-}(X_k)=\tau_{T_k-}(X_k)
=\tau^\la_{T_k-}(X_k)$, so that $(X_k)_\la$ is completely empty at time
$T_S(T_k-Z_{T_k-}(X_k))$.
\vip

{\it Case 1.} Assume first that $H_{T_{q+1}-}(x)>0$. Then by construction,
see Algorithm \ref{algo2}, $T_k+\Theta^k_{T_k-Z_{T_k-}(X_k),T_k} > T_{q+1}>T_k$.  

\vip

Consider $C=C(\eta^\la_{T_ST_k-},\lfloor \bn_\la X_k \rfloor)$.
By  $\Omega^S_2(\la)$, we have 
$$
C \subset
\lb \lfloor \bn_\la X_k -\bm_\la/2\rfloor, \lfloor \bn_\la X_k +\bm_\la/2 
\rfloor\rb,
$$
because $T_k-Z_{T_k-}(X_k)$ and $T_k$ belong to $\cT_M$ and $0<Z_{T_k-}(X_k)<1$.

\vip

The component $C$ is destroyed at time $T_ST_k$. 
By Definition of
$\Theta^k_{T_k-Z_{T_k-}(X_k),T_k}$, see Lemma \ref{deftheta}, we deduce that
$C$ is not completely occupied at time $T_ST_{q+1} < 
T_S( T_k+\Theta^k_{T_k-Z_{T_k-}(X_k),T_k})$. 
Consequenty, there is $i_2 \in \lb \lfloor \bn_\la X_k -\bm_\la/2\rfloor, 
\lfloor \bn_\la X_k +\bm_\la/2 \rfloor\rb$ such that $\eta^\la_{T_ST_{q+1}-}(i_2)
=0$.

\vip

Finally, using again
$\Omega^S_2(\la)$  there is necessarily (at least)
one seed falling on a site in $\lb \lfloor \bn_\la X_k -\bm_\la +1 \rfloor, 
\lfloor \bn_\la X_k - \bm_\la/2 -1\rfloor
\rb \subset (X_k)_\la$ during
$(T_ST_q,T_ST_{q+1})$. This shows the result.

\vip

{\it Case 2.} Assume next that $H_{T_{q+1}-}(x)=0$ and that 
$T_{q+1}- [T_k-Z_{T_k-}(X_k)]<1$.  Recall that $(X_k)_\la$ is completely
empty at time $T_S(T_k-Z_{T_k-}(X_k))$. Since
$T_k-Z_{T_k-}(X_k)$ and $T_{q+1}$ belong to $\cT_M$ and since their
difference is smaller than $1$ by assumption,
$\Omega^S_2(\la)$ guarantees us the existence of $i_1<i_2<i_3$, all in
$(X_k)_\la$, such that (at least) one seed falls on $i_2$ and no
seed fall on $i_1$ nor on $i_3$ during $(T_S(T_k-Z_{T_k-}(X_k)),T_ST_{q+1})$.
One easily concludes that $i_2$ is occupied and $i_3$ is vacant
at time $T_ST_{q+1}-$, as desired.

\vip

{\it Case 3.} Assume finally that $H_{T_{q+1}-}(x)=0$ and that 
$T_{q+1}- [T_k-Z_{T_k-}(X_k)] \geq 1$. 
Since $H_{T_{q+1}-}(x)=0$, there holds 
$Z_{T_{q+1}-}(c_-)< 1 = Z_{T_{q+1}-}(c_+)$ and
$T_k+\Theta^k_{T_k-Z_{T_k-}(X_k),T_k} \leq T_{q+1}$.
We aim to use
the event $\Omega^S_1(\la)$. We introduce 
$$
t_0=T_k-Z_{T_k-}(X_k)=\tau_{T_k-}(X_k)=\tau^\la_{T_k-}(X_k). 
$$
Observe that
$\tau_{T_k-}(c_-)=\tau_{T_k-}(c_+)=\tau_{T_k-}(x)$
because there is no match falling (exactly) on $x$ during $[0,T_k)$. 
Thus
$Z_{t_0}(x)=Z_{t_0}(c_-)=Z_{t_0}(c_+)=0$.

\vip

Set now $t_1=T_k$ and $s=T_{q+1}$. Observe that $0<t_1-t_0<1$
(because $Z_{T_k-}(X_k)<1$).
Necessarily, $Z_t(c_-)$ has jumped to $0$ at least one time
between $t_0$ and $T_{q+1}-$ (else, one would have  $Z_{T_{q+1}-}(c_-)=1$,
since $T_{q+1}-t_0\geq 1$ by assumption) and this jump occurs
after $t_0+1>t_1$ (since a jump of  $Z_t(c_-)$ requires that $Z_t(c_-)=1$,
and since for all $t\in [t_0,t_0+1)$, $Z_t(c_-)=t-t_0<1$).

\vip

We thus may denote by
$t_2<t_3<\dots<t_K$, for some $K\geq 2$, the successive times of jumps of
the process $(Z_t(c_-),Z_t(c_+))$  during $(t_0+1,s)$. 
We also put $\e=1$ if $t_2$ is a jump of $Z_t(c_+)$ and $\e=-1$ else.
Then we observe that $Z_t(c_-)$ and $Z_t(c_+)$ do never jump to $0$ at the
same time during $(t_0,s]$ (else, it would mean that they are killed by the 
same fire at some time $u$, whence necessarily, 
$H_r(u)=0$ and $Z_r(c_-)=Z_r(c_+)$ for all $r\in (u,s]$).

\vip

Furthermore, there is always at least one jump of $(Z_t(c_-),Z_t(c_+))$
in any time interval of length $1$ (during $[t_0+1,s)$), because else,
$Z_t(c_+)$ and $Z_t(c_-)$ would both become equal to $1$ and thus would
remain equal forever. 

\vip

Finally, 
observe that two jumps of $Z_t(c_-)$ 
cannot occur in a time interval of length $1$
(since a jump of  $Z_t(c_-)$ requires that $Z_t(c_-)=1$)
and the same thing holds for $c_+$. 

\vip

Consequently, the family 
$\cR=\{\e,t_0,\dots,t_K;s\}$ necessarily satisfies the condition $(PP)$.

\vip

Next, there holds that $t_2-t_1 <\Theta^k_{T_k-Z_{T_k-}(X_k),T_k}=\Theta^k_{t_0,t_1}$,
because else, we would have $H_{t_2-}(X_k)=0$ and thus the fire destroying
$c_+$ (or $c_-$) at time $t_2$ would also destroy $c_-$ (or $c_+$),
we thus would have $Z_{t_2}(c_+)=Z_{t_2}(c_-)=0$, so that 
$Z_t(c_+)$ and $Z_t(c_-)$ would remain equal forever.

\vip

Finally, we check that 
$(\eta^\la_{T_St}(i))_{t\geq t_0, i \in x_\la}=
(\zeta^{\la,\cR,k}_t(i+\lfloor \bn_\la x \rfloor))_{t\geq t_0, i \in x_\la}$,
this last process being built with the family of seed processes
$(N^{S,k}_{T_St}(i))_{t\geq t_0, i \in x_\la}$ as in Subsection \ref{sspp}.
Both are empty at time $t_0$. Seeds fall according to the same processes.
In both cases, a first match falls on $\lfloor \bn_\la x \rfloor$
at time $t_1$. In both cases (say that $\e=1$) 
a fire destroys the occupied connected component containing
$\lfloor \bn_\la x \rfloor+\bm_\la$  at time $t_2$ 
(by definition for $\zeta^{\la,\cR}$ and since
$Z_{t_2-}(c_+)=1$ implies, exactly as in Step 4, that $\eta^\la_{T_S t_2-}(i)=1$
for all $i$ in $(c_+)_\la$, so that the fire destroying $c_+$ at time
$t_2$ also destroys the occupied connected component around 
$\lfloor \bn_\la x \rfloor+\bm_\la$, which is at the boundary of $c_+$).
And so on.

\vip

We thus can use $\Omega^S_1(\la)$ and conclude that there are some sites
$i_1<i_2$ in $x_\la$ with  $\eta^\la_{T_ST_{q+1-}}(i_1)=1$ and
$\eta^\la_{T_ST_{q+1-}}(i_2)=0$ as desired.

\vip

{\bf Step 7.} We finally conclude the proof.
We put $z:=Z_{T_{q+1}-}(X_{q+1})$ and consider separately the
cases where $z\in (0,1)$ and $z=1$. 
Observe that $z=0$ never happens,
since by construction, 
$Z_{T_{q+1}-}(X_{q+1})=\min(Z_{T_q}(X_{q+1})+(T_{q+1}-T_q),1)>0$
and since $T_{q+1}>T_q$.

\vip

{\it Case $z\in (0,1)$.} Then in the $LFF_A(BS)$-process,
see Algorithm \ref{algo2}, $Z_{T_{q+1}}(x)=Z_{T_{q+1}-}(x)>0$ for all $x\in [-A,A]$,
whence $\tau_{T_{q+1}}(x)= \tau_{T_{q+1}-}(x)$ and 
$\tau_{T_{q+1}}(c)= \tau_{T_{q+1}-}(c)$ for all $x\in\cB_M$, all $c\in\cC_M$.

\vip

Using Step 3, we see that the match
falling on  $\lfloor \bn_\la X_{q+1} \rfloor$ at time $T_S T_{q+1}$
destroys nothing outside $\lb j_2+1,j_3-1 \rb$. As a conclusion, we
obviously have $\tau^\la_{T_{q+1}}(x)=\tau^\la_{T_{q+1}-}(x)$ and
$\rho^\la_{T_{q+1}}(c)=\tau^\la_{T_{q+1}}(c)=\tau^\la_{T_{q+1}-}(c)$ 
for all $x\in \cB_M\setminus\{X_{q+1}\}$ and all $c \in \cC_M$. 
There also holds 
$\tau^\la_{T_{q+1}}(X_{q+1})=\tau^\la_{T_{q+1}-}(x)$
because $j_1$ (see Step 3), 
which is occupied at time $T_ST_{q+1}-$ and not killed at
time $T_ST_{q+1}$ (thanks to $j_2$), does belong to $(X_{q+1})_\la$.

\vip

We conclude that when $z \in (0,1)$, $\Omega^\la_{T_{q+1}-}$
implies $\Omega^\la_{T_{q+1}}$. Using Step 1, we deduce that $\Omega^\la_{T_q}$
implies $\Omega^\la_{T_{q+1}}$ when $z\in (0,1)$.

\vip

{\it Case $z=1$.} Then there are $a,b \in \cB_M\cup \{-A,A\}$ such that
$D_{T_{q+1}-}(X_{q+1})=[a,b]$. We assume that  $a,b \in \cB_M$, the other
cases being treated similarly. Recalling Algorithm \ref{algo2}, we
know that for all $c \in \cC_M$ with $c \subset (a,b)$, $Z_{T_{q+1}-}(c)=1$,
for all $x \in \cB_M\cap (a,b)$, $\tH_{T_{q+1}-}(x)=0$, while finally
$\tH_{T_{q+1}-}(a)>0$ and $\tH_{T_{q+1}-}(b)>0$.  For the $LFF_A(BS)$-process,
we have 

\vip

(i) $\tau_{T_{q+1}}(c)= T_{q+1}$ for all $c \in\cC_M$ with 
$c \subset (a,b)$,

\vip

(ii) $\tau_{T_{q+1}}(x)= T_{q+1}$ for all $x \in\cB_M \cap (a,b)$,

\vip

(iii) $\tau_{T_{q+1}}(c)=\tau_{T_{q+1}-}(c)$ for all $c \in\cC_M$ with 
$c \cap (a,b)=\emptyset$,

\vip

(iv) $\tau_{T_{q+1}}(x)=\tau_{T_{q+1}-}(x)$ for all $x \in\cB_M \setminus (a,b)$.

\vip

Next, using Steps 4, 5, using Step 6 for $a$ (and a very similar 
result for $b$), we immediately check that
the fire occurring at $\lfloor \bn_\la X_{q+1}  \rfloor$ at time
$T_S T_{q+1}$ 

\vip

$\bullet$ destroys completely all the cells $c\in \cC_M$ with 
$c \subset (a,b)$, 

\vip

$\bullet$ destroys completely all the zones 
$x_\la$ with $x\in \cB_M \cap (a,b)$,

\vip

$\bullet$ does not destroy at all the cells $c\in \cC_M$ with 
$c \cap (a,b)=\emptyset$ and the 
zones $x_\la$ with $x\in \cB_M \setminus [a,b]$,

\vip

$\bullet$ does not destroy completely $a_\la$ nor $b_\la$.

\vip

Consequently, we have 

\vip

(i) $\rho^\la_{T_{q+1}}(c)=\tau^\la_{T_{q+1}}(c)= T_{q+1}$ for all $c \in\cC_M$ with 
$c \subset (a,b)$,

\vip

(ii) $\tau^\la_{T_{q+1}}(x)= T_{q+1}$ for all $x \in\cB_M \cap (a,b)$,

\vip

(iii) $\rho^\la_{T_{q+1}}(c)=\rho^\la_{T_{q+1}-}(c)$ and 
$\tau^\la_{T_{q+1}}(c)=\tau^\la_{T_{q+1}-}(c)$ for all $c \in\cC_M$ with 
$c \cap (a,b)=\emptyset$,

\vip

(iv) $\tau^\la_{T_{q+1}}(x)=\tau^\la_{T_{q+1}-}(x)$ for all 
$x \in\cB_M \setminus (a,b)$.

\vip

We conclude that when $z=1$, $\Omega^\la_{T_{q+1}-}$
implies $\Omega^\la_{T_{q+1}}$. Using Step 1, we deduce that $\Omega^\la_{T_q}$
implies $\Omega^\la_{T_{q+1}}$ when $z=1$.
\end{proof}

\subsection{Conclusion}

To achieve the proof, we will need the following result.

\begin{lem}\label{densitiesbs}
Let $(N^S_t(i))_{t\geq 0,i \in \zz}$ be a family of i.i.d. $SR(\mu_S)$-processes.

(i) Put $K_t^\la=(2\bm_\la+1)^{-1} 
|\{i \in \lb -\bm_\la,\bm_\la\rb \; : \; N^S_{T_St}(i)>0\}|$ and 
$$
U^\la_t
= \left(\frac{\psi_S(K^\la_t)}{T_S}\right)\land 1,
$$ 
recall Notation \ref{phipsi}.
Then for any $T>0$, $\sup_{[0,T]}|U^\la_t - t\land 1 |$ tends a.s. to $0$
as $\la$ tends to $0$.

(ii) For any $k\geq 0$, $\Pr[|C(\min(N^S_{T_St},1),0)|=k] = q_k(t\land 1)$,
where $q_k(z)$ was defined (\ref{defqk}).
\end{lem}

\begin{proof} We start with (i). First
observe that $t\mapsto U^\la_t$ and $t \mapsto t\land 1$ are
nondecreasing and  $t \mapsto t\land 1$ is continuous. By the Dini
Theorem, it suffices to prove that for all $t\in [0,T]$, a.s.,
$\lim U^\la_t = t\land 1$. To do so, observe that $(2\bm_\la+1)K^\la_t$
has a binomial distribution with parameters $2\bm_\la +1$
and $\nu_S((0,T_St))$. Thus $K^\la_t$ tends a.s. to $\nu_S((0,T_St))$.
Hence $U^\la_t$ tends a.s. to $(\psi_S(\nu_S((0,T_St)))/T_S)\land 1=t\land 1$
by definition of $\psi_S$.

\vip

We now check (ii). If $t\geq 1$, then obviously, $\min(N^S_{T_St}(i),1)=1$
for all $i\in\zz$, whence $|C(\min(N^S_{T_St},1),0)|=\infty$ a.s. 
Consequently, $\Pr[|C(\min(N^S_{T_St},1),0)|=k]=0=q_k(1)$. 

\vip

For $t<1$, the result relies on a simple
computation involving the i.i.d. random variables $\min(N^S_{T_St}(i),1)$,
which have a Bernoulli distribution with parameter $\nu_S((0,T_St))$:
if $k=0$, there holds 
$$\Pr[|C(\min(N^S_{T_St},1),0)|=0]=\Pr[N_{T_St}(i)=0]
=\nu_S((T_St,T_S))=q_0(t). 
$$
For $k\geq 1$,
\begin{align*}
&\Pr[|C(\min(N^S_{T_St},1),0)|=k]\\
=& \sum_{j=0}^{k-1}
\Pr\left[N^S_{T_St}(j-k)=N^S_{T_St}(j+1)=0, \forall i \in \lb j-k+1,j \rb,
N^S_{T_St}(i)=1 \right]\\
=& k [\nu_S((T_St,T_S))]^2 [\nu_S((0,T_St))]^k=q_k(t),
\end{align*}
which ends the proof.
\end{proof}

We finally give the

\begin{preuve} {\it of Proposition \ref{converge1A} when $\beta=BS$.}
Let us fix $x_0 \in (-A,A)$, $t_0\in (0,T]$ and $\e>0$. We will prove that 
with our coupling (see Subsection \ref{coucou}), there holds

\vip

(a) $\lim_{\la\to 0} \Pr\left[ \bdelta(D^\la_{t_0}(x_0),D_{t_0}(x_0) )
>\e\right] =0$;

\vip

(b) $\lim_{\la\to 0} \Pr\left[ \bdelta_T(D^\la(x_0),D(x_0) )>\e\right] =0$;

\vip

(c) $\lim_{\la\to 0} \Pr\left[ \sup_{[0,T]} |Z_t^\la(x_0)-Z^\la_t(x_0)|\geq \e
\right] =0$;

\vip

(d) $\lim_{\la \to 0} \Pr\left[ |C(\eta^\la_{T_St_0},\lfloor \bn_\la x_0 \rfloor)|
=k \right] = \E[q_k(Z_t(x_0))]$.

\vip

Recall that $q_k(z)$ was defined, for $k\geq 0$ and
$z\in [0,1]$ in (\ref{defqk}). These points will clearly imply the result.

\vip

First, we introduce, for $\zeta>0$,  
the event $\Omega^{x_0}_{A,T}(\zeta)$ on which
$x_0 \notin \cup_{q=1}^n [X_q - \zeta, X_q + \zeta]$. The probability
of this event obviously tends to $1$ as $\zeta \to 0$.

\vip

On  $\Omega^{x_0}_{A,T}(\zeta)$, it holds that for $\la>0$ small enough
(say, small enough such that $4 \bm_\la/\bn_\la < \zeta$),
$\lfloor \bn_\la x_0 \rfloor \notin 
\cup_{q=1}^n (X_q)_\la$.  We then call $c_0\in \cC_M$
the cell containing $x_0$.  

\vip

{\bf Step 1.} We first show that (a) (which holds for an arbitrary
value of $t_0\in (0,T]$) implies (b). Indeed, we have by construction,
for any $t\in [0,T]$, $\bdelta(D^\la_{t}(x_0),D_{t}(x_0) )< 4A$. Hence
by dominated convergence, (a) implies that 
$\lim_{\la \to 0} \E \left[\bdelta(D^\la_{t}(x_0),D_{t}(x_0) ) \right] =0$, 
whence again by dominated convergence, 
$\lim_{\la \to 0} \E \left[\bdelta_T(D^\la(x_0),D(x_0) ) \right] =0$.

\vip

{\bf Step 2.} Due to Lemma \ref{ggg}, we know that on 
$\Omega(\alpha,\la) \cap \Omega^{x_0}_{A,T}(\zeta)$, there holds that
$\tau^\la_t(c_0)=\rho^\la_t(c_0)=\tau_t(x_0)$ for all $t\in [0,T]$.
This implies that for all $i\in (c_0)_\la$, for all $t \in [0,T]$, 
$$
\eta^\la_{T_St}(i)=\min(N^{S,0}_{T_St}(i)-N^{S,0}_{T_S \tau_t(x_0)}(i), 1).
$$
We also recall that by construction, $(\tau_t(x_0))_{t\geq 0}$ is independent of
$(N^{S,0}_{t}(i))_{t\geq 0, i \in \zz}$.

\vip

{\bf Step 3.} Here we prove (d), for some fixed $k\geq 0$.
Let $\delta>0$ be fixed. We first consider $\alpha_0>0$, $\zeta_0>0$
and $\la_0>0$ such that for all $\la \in (0,\la_0)$, 
$\Pr\left[ \Omega(\alpha_0,\la) \cap \Omega^{x_0}_{A,T}(\zeta_0)\right] 
> 1-\delta$.
Then we consider $\la_k\leq \la_0$ in such a way
that for $\la \in (0,\la_k)$,  $\lb \lfloor \bn_\la x_0 \rfloor -k-1, 
\lfloor \bn_\la x_0 \rfloor +k+1 \rb \subset (c_0)_\la$ on 
$\Omega^{x_0}_{A,T}(\zeta_0)$ (it suffices that $2k < \zeta_0 \bn_\la$ for all
$\la\in(0,\la_k)$).

\vip

We easily conclude: for $\la \in (0,\la_k)$, recalling 
(\ref{tauzC}), using Lemma \ref{densitiesbs}-(ii) together with 
a (spatial and temporal) stationarity argument, using Step 2
and that $(N^{S,0}_{t}(i))_{t\geq 0, i \in \zz}$ is independent of 
$\Omega^{x_0}_{A,T}(\zeta)\cap \Omega(\alpha,\la)$ and $\tau_t(x_0) $, we
obtain
\begin{align*}
&\left|\Pr\left[|C(\eta^\la_{T_St},\lfloor \bn_\la x_0 \rfloor)|=k \right] - 
\E\left[q_k(Z_t(x_0)) \right] \right|\\
=& \left|\Pr\left[|C(\eta^\la_{T_St},\lfloor \bn_\la x_0 \rfloor)|=k \right] - 
\E\left[q_k(\min(t-\tau_t(x_0),1) ) \right] \right|\\
\leq& \Pr \left[ \left(\Omega(\alpha,\la) \cap \Omega^{x_0}_{A,T}(\zeta)
\right)^c\right] < \delta.
\end{align*}
This concludes the proof of (d).

\vip

{\bf Step 4.} We next prove (c). For $\delta>0$
fixed, we consider $\alpha_0>0$, $\zeta_0>0$
and $\la_0>0$ be as in Step 3. 
Consider the successive
values $0=s_0<s_1<\dots<s_l<T$ of $(\tau_t(x_0))_{t\in [0,T]}$.  
Set also $s_{l+1}=T$.
Recall the definition of $Z^\la_t(x)$, see (\ref{zlambda}), and compare
to Lemma \ref{densitiesbs}-(i).

\vip

Let $k\in \{0,\dots,l\}$ be fixed. Denote by $(U^{k,\la}_t)_{t\geq 0}$
the process defined as in Lemma \ref{densitiesbs}-(i) with the seed
process $(N^{S,0}_{s_k/T_S+t}(i-\lfloor \bn_\la x_0 \rfloor)
-N^{S,0}_{s_k/T_S}(i-\lfloor \bn_\la x_0 \rfloor))_{t\geq 0,i \in \zz}$
(this is indeed a family of $SR(\mu_S)$-processes by stationarity
and since $s_1,\dots,s_l$ are independent of 
$(N^{S,0}_t(i))_{t\geq 0, i \in \zz}$).
Then due to Lemma \ref{densitiesbs}-(i), for all $\la>0$ small enough,
say $\la \in (0,\la_1)$,
$$
\Pr\left(\sup_{[s_k,s_{k+1})} |U^{k,\la}_{t-s_k}- (t-s_k)\land 1| \geq \e\right)
\leq \delta.
$$
But on $\Omega(\alpha,\la) \cap \Omega^{x_0}_{A,T}(\zeta)$, it holds that
$Z^\la_{t}(x_0)=U^{k,\la}_{t-s_k}$ for all $t\in [s_k,s_{k+1})$, see Step 2. 
It also holds, recall (\ref{tauzC}), that $Z_t(x_0)=(t-s_k)\land 1$ for
$t\in [s_k,s_{k+1})$. As a conclusion, for all $\la>0$ small enough,
\begin{align*}
\Pr\left( \sup_{[s_k,s_{k+1})} |Z^\la_t(x_0)-Z_t(x_0)|\geq \e \right)
\leq& \Pr\left( (\Omega(\alpha,\la) \cap \Omega^{x_0}_{A,T}(\zeta) )^c\right)\\
&+ \Pr\left(\sup_{[s_k,s_{k+1})} |U^{k,\la}_{t-s_k}- (t-s_k)\land 1| \geq \e\right) 
\leq 2\delta.
\end{align*}
Observing finally that $l\leq \pi_M([0,T]\times(-A,A))$ and that
$\E[ \pi_M([0,T]\times(-A,A))]=2TA$, we easily deduce that 
for all $\la>0$ small enough,
$$
\Pr\left( \sup_{[0,T]} |Z^\la_t(x_0)-Z_t(x_0)|\geq \e \right)
\leq 2TA\delta.
$$
Point (c) immediately follows.

\vip

{\bf Step 5.} It remains to prove (a). Let $\delta>0$.
Put $\cT_M^*=\cT_M\cup\{t_0\}$.
Define the events $\Omega_M^*(\alpha)$, $\Omega^{S,*}_1(\la)$ and
$\Omega^{S,*}_2(\la)$ as $\Omega_M(\alpha)$, $\Omega^S_1(\la)$ and
$\Omega^{S}_2(\la)$, replacing $\cT_M$ by $\cT_M^*$. 
Define also $\Omega^*(\la,\alpha)=\Omega^M_{A,T}(\la)\cap\Omega^*_M(\alpha)
\cap \Omega^{S,*}_1(\la)\cap \Omega^{S,*}_2(\la)$. Clearly,
choosing $\alpha_1>0$ and $\zeta_1>0$ small enough, 
we have $\Pr[\Omega^*(\la,\alpha_1) 
\cap \Omega^{x_0}_{A,T}(\zeta_1)]\geq
1-\delta$ for all $\la>0$ small enough, say $\la\in (0,\la_2)$.

\vip

On $\Omega^*(\la,\alpha_1) \cap \Omega^{x_0}_{A,T}(\zeta_1)$, we can argue
exactly as in the proof of Lemma \ref{ggg} to check that

\vip

$(i)$ if $Z_{t_0}(x_0)<1$, then $D_{t_0}(x_0)=\{x_0\}$ and
$C(\eta^\la_{T_St},\lfloor \bn_\la x_0 \rfloor)\subset (x_0)_\la$ 
(see Step 3 of the proof of Lemma \ref{ggg}), whence
$D_{t_0}^\la (x_0) \subset [x_0- \bm_\la/\bn_\la,x_0+\bm_\la/\bn_\la]$.
We deduce that 
$\bdelta(D_{t_0}(x_0),D_{t_0}^\la (x_0)) \leq 2 \bm_\la/\bn_\la$;

\vip

$(ii)$ if $Z_{t_0}(x_0)=1$ and $D_{t_0}(x_0)=[a,b]$ for some
$a,b \in \cB_M \cup\{-A,A\}$, then 

\vip

$\bullet$ for all $c\in \cC_M$ with
$c\subset [a,b]$, $\eta^\la_{T_St}(i)=1$ for all $i \in c_\la$
(see Step 4 of the preceding proof);

\vip

$\bullet$ for all $x\in \cB_M\cap (a,b)$, $\eta^\la_{T_St}(i)=1$ for all 
$i \in x_\la$ (see Step 5 of the preceding proof);

\vip

$\bullet$ there are $i\in a_\la$ and $j \in b_\la$ such that  $\eta^\la_{T_St}(i)=
\eta^\la_{T_St}(j)=0$ (see Step 6 of the preceding proof);

\vip

so that 
$$
\lb \lfloor \bn_\la a \rfloor +\bm_\la+1, \lfloor \bn_\la b \rfloor -\bm_\la-1  
\rb \subset C(\eta^\la_{T_St},\lfloor \bn_\la x_0 \rfloor )
\subset \lb \lfloor \bn_\la a \rfloor -\bm_\la, \lfloor \bn_\la b \rfloor 
+\bm_\la\rb,
$$
and thus $[a+\bm_\la/\bn_\la,b-\bm_\la/\bn_\la]\subset D^\la_{t_0}(x_0)
\subset [a-\bm_\la/\bn_\la,b+\bm_\la/\bn_\la]$, whence as previously,
$\bdelta(D_{t_0}(x_0),D_{t_0}^\la (x_0)) \leq 2 \bm_\la/\bn_\la$.

\vip
Thus for all $\la \in (0,\la_2)$,
on $\Omega^*(\la,\alpha_1) \cap \Omega^{x_0}_{A,T}(\zeta_1)$, we always
have  $\bdelta(D_{t_0}(x_0),D_{t_0}^\la (x_0)) \leq 2 \bm_\la/\bn_\la$.
We conclude that for $\delta>0$, for all $\la\in(0,\la_2)$ small enough
(so that $2 \bm_\la/\bn_\la<\e$), there holds 
$$
\Pr\left[\bdelta(D_{t_0}(x_0),D_{t_0}^\la (x_0)) >\e \right] 
\leq \Pr[(\Omega^*(\la,\alpha) \cap \Omega^{x_0}_{A,T}(\zeta))^c] <\delta.
$$ 
This concludes the proof.
\end{preuve}

\section{Convergence proof when $\beta = \infty$}\label{pri}
\setcounter{equation}{0}

The aim of this section is to prove
Proposition \ref{converge1A} in the case where $\beta=\infty$ and
this will conclude the proof of Theorem \ref{convergebs}. 
This section generalizes consequently \cite[Section 4]{bf}
and the proof we present here is quite different and slightly
less intricate. We follow essentially the ideas of the previous section.
Some points are easier (because the height of the barriers are 
deterministic in the limit process), but some other points are
more complicated (in particular, the height of the barriers are not constant
as a function of $\la$).

\vip

In the whole
section, we assume $(H_M)$ and $(H_S(\infty))$.
The parameters $A>0$ and $T>0$ are fixed and we
omit the subscript/superscript $A$ in the whole proof.

\vip

We recall that $\ba_\la$, $\bn_\la$ and $\bm_\la$ are defined
in (\ref{ala}), (\ref{nla}) and (\ref{mla}).
For $A>0$, we set as usual $A_\la = \lfloor A \bn_\la \rfloor$ and 
$I_A^\la=\lb -A_\la,A_\la\rb$.
For $i\in \zz$, we set $i_\la=[i/\bn_\la, (i+1)/\bn_\la)$.
For $[a,b]$ an interval of $[-A,A]$ and $\la\in(0,1)$, we introduce,
assuming that $-A<a<b<A$,
\begin{align}
[a,b]_\la =&\lb \left\lfloor  \bn_\la a +\bm_\la
\right\rfloor +1 ,  \left\lfloor \bn_\la b-\bm_\la \right\rfloor -1\rb 
\subset \zz, \label{xxbb2} \\
{[-A,b]}_\la =&\lb -A_\la ,  
\left\lfloor \bn_\la b-\bm_\la \right\rfloor -1\rb 
\subset \zz, \nonumber \\
{[a,A]}_\la =&\lb \left\lfloor  \bn_\la a +\bm_\la
\right\rfloor +1,  A_\la \rb 
\subset \zz, \nonumber
\end{align}
For $x\in (-A,A)$ and $\la\in(0,1)$, we introduce as usual
\begin{align}\label{xla2}
x_\la=&\lb \left\lfloor   \bn_\la x-\bm_\la 
\right\rfloor , \left\lfloor \bn_\la x + \bm_\la \right\rfloor 
\rb \subset \zz.
\end{align}

\subsection{Speed of occupation}

We start with some easy estimates.

\begin{lem}\label{occupation}
Consider a family of i.i.d. $SR(\mu_S)$-processes 
$(N^{S}_{t}(i))_{t\geq 0,i\in\zz}$. Let $a<b$.

(i) For $t<1$, $\lim_{\la\to 0} \Pr[\forall i \in \lb 
\lfloor a\bm_\la \rfloor, \lfloor b\bm_\la \rfloor
\rb, N^S_{\ba_\la t}(i)>0]=0$.

(ii) For $t\geq 1$, 
$\lim_{\la\to 0} \Pr[\forall i \in \lb \lfloor a\bm_\la \rfloor , 
\lfloor b\bm_\la \rfloor
\rb, N^S_{\ba_\la t}(i)>0]=1$.

(iii) For $t<1$, $\lim_{\la\to 0} 
\Pr[\forall i \in \lb \lfloor a\bn_\la \rfloor, 
\lfloor b\bn_\la \rfloor
\rb, N^S_{\ba_\la t}(i)>0]=0$.

(iv) For $t > 1$, 
$\lim_{\la\to 0} \Pr[\forall i \in \lb \lfloor a\bn_\la \rfloor, 
\lfloor b\bn_\la \rfloor
\rb, N^S_{\ba_\la t}(i)>0]=1$.

(v) For $t>0$, $\lim_{\la\to 0} \Pr[\exists i \in \lb \lfloor a\bm_\la \rfloor, 
\lfloor b\bm_\la \rfloor
\rb, N^S_{\ba_\la t}(i)>0]=1$.
\end{lem}

\begin{proof}
To check points (i) and (ii), it suffices to note that
$$
\nu_S((0,\ba_\la t))^{(b-a)\bm_\la} \sim e^{-(b-a) 
\bm_\la \nu_S((\ba_\la t,\infty))},
$$
which tends to $0$ if $t<1$ (see (\ref{mla})) and to $1$ if $t\geq 1$
(because then $\bm_\la \nu_S((\ba_\la t,\infty))\leq 
\bm_\la \nu_S((\ba_\la,\infty))\simeq \bm_\la/\bn_\la\to 0$).
To check points (iii) and (iv), observe that
$$
\nu_S((0,\ba_\la t))^{(b-a)\bn_\la} \sim e^{-(b-a) 
\bn_\la \nu_S((\ba_\la t,\infty))}
\sim e^{-(b-a) \nu_S((\ba_\la t,\infty))/\nu_S((\ba_\la,\infty))  }
$$
tends to $0$ if $t<1$ and to $1$
if $t>1$ due to $(H_S(\infty))$. Finally, (v) follows from the
fact that
$$
1- \nu_S((\ba_\la t,\infty))^{(b-a)\bm_\la }
$$
obviously tends to $1$.
\end{proof}

\subsection{Height of the barriers}

We describe here the time
needed for a destroyed (microscopic) cluster to be regenerated. Roughly, 
we assume that the zone around $0$ is completely vacant at time
$\ba_\la t_0$. Then we consider the situation where a match falls on the site 
$0$ at some time  $\ba_\la t_1 \in (\ba_\la t_0,\ba_\la(t_0+1))$ and we denote
by $\Theta_{t_0,t_1}^\la$ the delay needed for the destroyed
cluster to be fully regenerated (divided by $\ba_\la$). We show
that $\Theta_{t_0,t_1}^\la\simeq t_1-t_0$ when $\la$ is small.

\begin{lem}\label{convtheta}
Consider a family of i.i.d. $SR(\mu_S)$-processes 
$(N^{S}_{t}(i))_{t\geq 0,i\in\zz}$. Let $0\leq t_0 <t_1 < t_0 +1$ be fixed.
Put $\zeta_{t_0,t}^\la(i)=\min(N^{S}_{\ba_\la(t_0+t)}(i)-N^{S}_{\ba_\la t_0}(i),1)$ 
and  $\zeta_{t_1,t}^\la(i)=\min(N^{S}_{\ba_\la(t_1+t)}(i)-N^{S}_{\ba_\la t_1}(i),1)$ 
for all $t>0$ 
and $i\in \zz$.
Define
$$
\Theta_{t_0,t_1}^\la=\inf \left\{ t>0\; : \; \forall\; i \in 
C(\zeta_{t_0,t_1-t_0}^\la,0),\; 
\zeta_{t_1,t}^\la(i)=1 \right\} \in [0,1].
$$
Then for all $\delta>0$,
$$
\lim_{\la \to 0} \Pr \left[ 
|\Theta_{t_0,t_1}^\la - (t_1-t_0)|\geq \delta\right] = 0.
$$
\end{lem}

\begin{proof}
We can assume that $t_0=0$ by stationarity. We 
put $u=t_1=t_1-t_0$. Exactly as in the case where $\beta=BS$ (see Subsection
\ref{BSHB}), we obtain, for  $h >0$,
$$
\Pr\left[ \Theta_{t_0,t_1}^\la \leq h \right] = \nu_S((\ba_\la u, \infty))
+ \frac{[\nu_S((\ba_\la u, \infty))]^2}{[1-g_S^\la(u,h)]^2}g^\la_S(u,h),
$$
where
$$
g^\la_S(u,h)= \Pr\left[N^S_{\ba_\la u}(0)>0,N^S_{\ba_\la (u+h)}(0)>N^S_{\ba_\la u}(0)  
\right].
$$

For $h>u$, we observe that $g_S^\la(u,h)
\geq 1 - \nu_S((\ba_\la h,\infty))- \nu_S((\ba_\la u, \infty))$, whence
$$
\Pr\left[ \Theta_{t_0,t_1}^\la \leq h \right] \geq \left(
\frac{\nu_S((\ba_\la u, \infty))}{\nu_S((\ba_\la h,\infty))
+\nu_S((\ba_\la u, \infty))  }\right)^2
[1 - \nu_S((\ba_\la h,\infty))- \nu_S((\ba_\la u, \infty))],
$$
which tends to $1$ as $\la\to 0$ due to $(H_S(\infty))$, since $\ba_\la$ 
increases to infinity and since $h>u$.

\vip

For $h<u$, there holds  $g_S^\la(u,h) \leq 1 - \nu_S((\ba_\la h, \infty))$, 
so that
$$
\Pr\left[ \Theta_{t_0,t_1}^\la \leq h \right] \leq 
\nu_S((\ba_\la u, \infty))+ \left(
\frac{\nu_S((\ba_\la u, \infty))}{\nu_S((\ba_\la h,\infty))}\right)^2
[1 - \nu_S((\ba_\la h,\infty))],
$$
which tends to $0$ due to $(H_S(\infty))$ and since $h<u$. 
This concludes the proof.
\end{proof}

\subsection{Persistent effect of microscopic fires}\label{sspp2}

We handle a study similar to subsection \ref{sspp}.

\vip

Recall that $\cR=(\e; t_0,t_1,\dots,t_K;s)$ satisfies $(PP)$ if 

\vip

(i) $K \geq 2$, $\e \in \{-1,1\}$,

\vip

(ii) $0< t_0 < t_1 < \cdots < t_K<s<t_K+1$, 

\vip

(iii) for all $k=0,\dots,K-1$, $t_{k+1}-t_k<1$,

\vip

(iv) $t_2-t_0>1$ and for all $k=2,\dots,K-2$, $t_{k+2}-t_k>1$,

\vip

and that we then set $\e_k= (-1)^k \e$ for $k\geq 0$. 

\vip

For a family of i.i.d. $SR(\mu_S)$-processes 
$(N^S_t(i))_{t\geq 0, i\in \zz}$,
we introduce, for each $\la\in (0,1)$, 
the process $(\zeta^{\la,\cR}_t(i))_{t\geq t_0, i \in \lb -\bm_\la,\bm_\la\rb}$ 
defined as follows:

\vip

$\bullet$ for all $t\in [t_0,t_1)$, all $i \in \lb -\bm_\la,\bm_\la\rb$,
$\zeta_t^{\la,\cR}(i)=\min(N^{S}_{\ba_\la t}(i) -N^{S}_{\ba_\la t_0}(i),1)$, 

\vip

$\bullet$ for all $i \in \lb -\bm_\la,\bm_\la\rb$, 
$\zeta_{t_1}^{\la,\cR}(i) = \zeta_{t_1-}^{\la,\cR}(i) 
\indiq_{\{i \not\in C(\zeta^{\la,\cR}_{t_1-},0)\}}$,

\vip

$\bullet$ for $k=1, \ldots, K-1$, 

$(*)$ for all $t \in (t_k,t_{k+1})$,  $i \in \lb -\bm_\la,\bm_\la\rb$,  
$\zeta_t^{\la,\cR}(i)= \min{(\zeta_{t_k}^{\la,\cR}(i) + N^{S}_{\ba_\la t}(i) 
-N^{S}_{\ba_\la t_k}(i),1)}$,

$(*)$ for all $i \in \lb -\bm_\la,\bm_\la\rb$,   
$\zeta_{t_{k+1}}^{\la,\cR}(i) = \zeta_{t_{k+1}-}^{\la,\cR}(i) 
\indiq_{i \not \in C(\zeta_{t_{k+1}-}^{\la,\cR}, \epsilon_k \bm_\la)}$,

\vip

$\bullet$ for all  $t \in (t_K,\infty)$, $i \in \lb -\bm_\la,\bm_\la\rb$,  
$\zeta_t^{\la,\cR}(i)= \min{(\zeta_{t_K}^{\la,\cR}(i) + 
N^{S}_{\ba_\la t}(i) -N^{S}_{\ba_\la t_K}(i),1)}$.

\vip

Consider the event
$$
\Omega^S_{\cR}(\la)=\left\{
\exists -\bm_\la<i_1<i_2<i_3<\bm_\la\; : \; 
\zeta^{\la,\cR}_s(i_1)=\zeta^{\la,\cR}_s(i_3)=0, \zeta^{\la,\cR}_s(i_2)=1
\right\}.
$$

\begin{lem}\label{pingpong2} Let $\cR=(\e; t_0,t_1,\dots,t_K;s)$ 
satisfy $(PP)$. For each $\la\in (0,1]$, consider the process
$(\zeta_t^{\la,\cR}(i))_{t\geq t_0, i \in \lb -\bm_\la,\bm_\la\rb}$ 
defined above. If $t_2-t_1<t_1-t_0$, there holds
$$
\lim_{\la\to 0} \Pr\left(\Omega^S_{\cR}(\la)  \right) =1.
$$
\end{lem}

Compare to Lemma \ref{pingpong}: the
condition $\Theta_{t_0,t_1}>t_2-t_1$ is replaced by the condition
$t_1-t_0>t_2-t_1$. This is very natural, in view of Lemma \ref{convtheta}.

\begin{proof}
In view of Lemma \ref{occupation}, 
the proof is very similar to that of Lemma \ref{pingpong}.
We assume that $\e=1$ and that $K$ is even for simplicity.
Fix $\alpha = 1/K$.

\vip

{\it First fire.} We put $C=C(\zeta^{\la,\cR}_{t_1-},0)$. Since $t_1-t_0<1$,
$C \subset \lb - \lfloor \alpha \bm_\la \rfloor, 
\lfloor \alpha \bm_\la \rfloor   
\rb$ with probability tending to $1$ (use Lemma \ref{occupation}-(i)
and space/time stationarity).  
Thus the match falling at time 
$t_1$ destroys nothing outside 
$\lb - \lfloor \alpha \bm_\la \rfloor, \lfloor \alpha \bm_\la \rfloor \rb$.
\vip

{\it Second fire.} 
Since $t_2-t_0>1$, at least one seed has
fallen, during $[t_0,t_2)$ 
on each site of $\lb \lfloor \alpha \bm_\la \rfloor+1, \bm_\la \rb$ with
probability tending to $1$ (use Lemma \ref{occupation}-(ii)
and space/time stationarity).
Thus the fire at time $t_2$ destroys completely this zone, but does
not affect $\lb -\bm_\la, -\lfloor \alpha \bm_\la \rfloor-1 \rb$
with probability tending to $1$, because 
because $t_2<t_1+\Theta^\la_{t_0,t_1}$ with probability tending to $1$
(by Lemma \ref{convtheta}, $\Theta^\la_{t_0,t_1}\simeq t_1-t_0$,
and $t_2-t_1<t_1-t_0$ by assumption)
and because by 
definition of $\Theta_{t_0,t_1}^\la$, there is an empty site in 
$C\subset \lb - \lfloor \alpha \bm_\la \rfloor, 
\lfloor \alpha \bm_\la \rfloor   
\rb$ during 
$[t_1,t_1+\Theta_{t_0,t_1}^\la]$.
\vip

{\it Third fire.} Since $t_3-t_2<1$, the probability that there is a 
vacant site in 
$\lb\lfloor \alpha \bm_\la \rfloor+1, \lfloor 2\alpha \bm_\la \rfloor \rb$ 
at time $t_3$ tends to $1$ as $\la\to 0$ (use  Lemma \ref{occupation}-(i)
and space/time stationarity).

\vip

Next, all the sites of $\lb -\bm_\la, -\lfloor \alpha \bm_\la \rfloor-1 \rb$
are occupied at time $t_3-$ with probability tending to $1$
(because they have not been affected by a fire during $[t_0,t_3)$
and because $t_3-t_0>t_2-t_0>1$, see Lemma \ref{occupation}-(ii)). 
Thus the fire at time $t_3$ destroys the zone
$\lb -\bm_\la, -\lfloor \alpha \bm_\la \rfloor-1 \rb$ and does not
affect the zone  $\lb \lfloor 2\alpha \bm_\la \rfloor , \bm_\la \rb$.

\vip

{\it Last fire and conclusion.} Iterating the procedure, we see that
with a probability tending to $1$ as $\la\to 0$, 
the fire at time $t_K$ destroys the zone
$\lb \lfloor (K\alpha/2) \bm_\la \rfloor , \bm_\la \rb =
\lb \lfloor \bm_\la/2 \rfloor , \bm_\la \rb$.

\vip

Then one easily concludes: since $0<s-t_K<1$, the probability that there
is at least one site in
$\lb \lfloor \bm_\la/2 \rfloor , \lfloor 2\bm_\la/3 \rfloor \rb$
with no seed falling during $[t_K,s]$  tends to $1$ 
(by Lemma \ref{occupation}-(i)),
the probability that there is at least one site in
$\lb \lfloor 2\bm_\la/3 \rfloor +1, \lfloor 5\bm_\la/6 \rfloor \rb$
with at least one seed falling during $[t_K,s]$  tends to $1$
(by Lemma \ref{occupation}-(v)),
and the probability that there is at least one site in
$\lb \lfloor 5\bm_\la/6 \rfloor +1, \bm_\la \rb$
with no seed falling during $[t_K,s]$  tends to $1$
(by Lemma \ref{occupation}-(i)).
\end{proof}

\subsection{The coupling}\label{coucou2}
We are going to construct a coupling between the $FF_A(\mu_S,\mu_M^\la)$-process
(on the time interval $[0,\ba_\la T]$) and the $LFF_A(\infty)$-process 
(on $[0,T]$). 

\vip

First, we couple a family
of i.i.d. $SR(\mu^\la_M)$-processes $(N^{M,\la}_t(i))_{t\geq 0, 
i \in \zz}$ with a Poisson measure $\pi_M(dt,dx)$ on $[0,T]\times [-A,A]$
with intensity measure $dtdx$  as in Proposition \ref{coupling1}.

\vip

We call $n:=\pi_M([0,T]\times[-A,A])$
and we consider the marks $(T_q,X_q)_{q=1,\dots,n}$ of $\pi_M$ ordered 
in such a way that $0<T_1<\dots<T_n<T$.

\vip

Next, we introduce some i.i.d. families of i.i.d. $SR(\mu_S)$-processes
$(N^{S,q}_t(i))_{t\geq 0,i\in \zz}$, for $q=0,1,\dots$, independent of 
$\pi_M$ and $(N^{M,\la}_t(i))_{t\geq 0, i \in \zz}$.

\vip

Then we build a family of i.i.d. $SR(\mu_S)$-processes (independent
of $(N^{M,\la}_t(i))_{t\geq 0, i \in \zz}$ and $\pi_M$) as follows.

\vip

$\bullet$ For $q \in \{1, \ldots, n\}$, for all $i \in (X_q)_\la$,
set $(N^{S,\la}_t(i))_{t\geq 0}=
(N^{S,q}_t(i-\lfloor \bn_\la X_q \rfloor))_{t\geq 0}$
(if $i$ belongs to  $(X_q)_\la\cap (X_r)_\la$ for some $q<r$, set e.g. 
$(N^{S,\la}_t(i))_{t\geq 0}=(N^{S,q}_t(i-\lfloor \bn_\la X_q \rfloor))_{t\geq 0}$.
This will occur with a very small probability, so that this choice is not
important).

\vip

$\bullet$ For all other $i \in \zz$ 
set $(N^{S,\la}_t(i))_{t\geq 0} = (N^{S,0}_t(i))_{t\geq 0}$.    

\vip

The $FF_A(\mu_S,\mu_M^\la)$-process 
$(\eta^\la_t(i))_{t\geq 0, i \in I_A^\la}$ is built from the seed processes
$(N^{S,\la}_t(i))_{t\geq 0, i \in \zz}$ and 
from the match processes $(N^{M,\la}_t(i))_{t\geq 0, i \in \zz}$.

\vip

Finally, we build the $LFF_A(\infty)$-process  
$(Z_t(x),D_t(x),H_t(x))_{t\in[0,T],x\in[-A,A]}$ from $\pi_M$ 
(use Algorithm \ref{algo2} replacing
$\Theta^{k+1}_{T_{k+1}- Z_{T_{k+1}-}(X_{k+1}),T_{k+1}}$ by $Z_{T_{k+1}-}(X_{k+1})$)
and
observe that  it is independent of 
$(N^{S,q}_t(i))_{t\in [0,T], i \in \zz, q\geq 0}$.

\subsection{A favorable event}

First, we know from Proposition \ref{coupling1} that
$$
\Omega^M_{A,T}(\la):=\left\{\forall t\in [0,T],\;
\forall i\in I_A^\la, \;
\Delta N^{M,\la}_{\ba_\la t}(i) \ne 0 \; \hbox{iff} \; 
\pi_M(\{t\}\times i_\la) \ne 0
\right\}
$$
satisfies $\lim_{\la \to 0} \Pr[\Omega^M_{A,T}(\la)]=1$.
Next, we recall that the marks of $\pi_M$ are called
$(T_1,X_1),$ $\dots,$ $(T_n,X_n)$ and are ordered chronologically.
We introduce $\cT_M=\{0,T_1,\dots,T_n\}$,  
$\cB_M=\{X_1,\dots,X_n\}$, as well as the set
$\cC_M$ of connected components of  $[-A,A]\setminus \cB_M$ (sometimes 
referred to as {\it cells}).  

\vip

We also introduce $\cS_M=\{2t-s\; : s,t\in\cT_M, s<t \}$, which has
to be seen as the possible limit values of $t+\Theta_{s,t}^\la \simeq
t+t-s$, recall Lemma \ref{convtheta}.

\vip

For $\alpha>0$, we consider the event 
\begin{align*}
\Omega_M(\alpha)=\Big\{\min_{s,t\in \cT_M \cup \cS_M,s\ne t}|t-s|\geq \alpha, 
\min_{s,t\in \cT_M \cup \cS_M,s\ne t}|t-(s+1)|\geq \alpha,\\
\min_{x,y \in \cB_M \cup\{-A,A\},x\ne y}|x-y|\geq \alpha
\Big\},
\end{align*}
which clearly satisfies $\lim_{\alpha \to 0} \Pr[\Omega_M(\alpha)]=1$.
As in the case $\beta=BS$, for any given $\alpha>0$, there is $\la_\alpha>0$
such that for all $\la\in(0,\la_\alpha)$, on $\Omega_M(\alpha)$,
there holds that

\vip

$\bullet$ for all $x,y\in \cB_M\cup\{-A,A\}$, with $x\ne y$, 
$x_\la\cap y_\la =\emptyset$, 

\vip

$\bullet$ the family $\{c_\la, c\in \cC_M\}\cup \{x_\la, x \in 
\cB_M\}$ is a partition of $I_A^\la$ (recall (\ref{xxbb2}) and (\ref{xla2})).

\vip

Let $q\in \{1,\dots,n\}$.
We call $\cU_q$ the set of all possible $\cR=(\e,t_0,\dots,t_K;s)$
satisfying $(PP)$ with $\e\in \{-1,1\}$, with $\{t_0,\dots,t_K,s\}
\subset \cT_M$ and with $t_1-t_0 > t_2-t_1$.
We introduce, for $q=1,\dots,n$ and $\cR\in \cU_q$, 
the event $\Omega^{S,q}_{\cR}(\la)$ defined as in Subsection \ref{sspp2}
with the $SR(\mu_S)$-processes $(N^{S,q}_t(i))_{t\geq 0, i\in \zz}$.
Then we put 
$$
\Omega^S_1(\la)=\cap_{q=1}^n \cap_{\cR\in\cU_q} \Omega^{S,q}_{\cR}(\la),
$$
which satisfies $\lim_{\la \to 0} \Pr\left( \Omega^S_1 (\la) \right)=1$
thanks to Lemma \ref{pingpong2}.

\vip

We also consider the event $\Omega^S_2(\la)$ on which the following 
conditions hold:
for all $t_1,t_2 \in \cT_M \cup \cS_M$ 
with  $0<t_2-t_1<1$, for all $q=1,\dots,n$,
there are  
$$-\bm_\la<i_1<i_2 < -\bm_\la/2 <i_3<0<i_4< \bm_\la/2 <i_5<i_6<\bm_\la$$ 
such that

\vip

$\bullet$ for $j=1,3,4,6$, $N^{S,q}_{\ba_\la t_2}(i_j)- N^{S,q}_{\ba_\la t_1}(i_j)=0$,

\vip

$\bullet$ for $j=2,5$, $N^{S,q}_{\ba_\la t_2}(i_j)- N^{S,q}_{\ba_\la t_1}(i_j)>0$.

\vip

There holds $\lim_{\la \to 0} \Pr\left( \Omega^S_2 (\la) \right)=1$.
Indeed, it suffices to prove that almost surely, $\lim_{\la\to 0} 
\Pr\left(\Omega^S_2 (\la) \vert \pi_M\right)=1$. 
Since there are a.s. finitely
many possibilities for $q,t_1,t_2$ and since $\pi_M$ is independent
of  $(N^{S,q}_t(i))_{t\geq 0, i\in \zz}$, it suffices to work with 
a fixed $q\in \{1,\dots,n\}$ and some fixed $0<t_2-t_1<1$.
The result then follows from Lemma \ref{occupation}-(i)-(v) together
with space/time stationarity.

\vip

Next we introduce the event $\Omega^S_3(\la)$ on which the following 
conditions hold:
for all $t_1,t_2 \in \cT_M \cup \cS_M$, 

\vip

$\bullet$ for all $c\in \cC_M$, if $0<t_2-t_1<1$, there is $i\in c_\la$ 
with $N^{S,\la}_{\ba_\la t_2}(i)- N^{S,\la}_{\ba_\la t_1}(i)=0$;

\vip

$\bullet$ for all $x \in \cB_M$, if $0<t_2-t_1<1$, there is $i\in x_\la$ 
with $N^{S,\la}_{\ba_\la t_2}(i)- N^{S,\la}_{\ba_\la t_1}(i)=0$;

\vip

$\bullet$ if $t_2-t_1>1$, for all $c\in \cC_M$,  for all 
$i\in c_\la$, $N^{S,\la}_{\ba_\la t_2}(i)- N^{S,\la}_{\ba_\la t_1}(i)>0$.

\vip

$\bullet$ if $t_2-t_1>1$, for all $x\in \cB_M$,  for all 
$i\in x_\la$, $N^{S,\la}_{\ba_\la t_2}(i)- N^{S,\la}_{\ba_\la t_1}(i)>0$.

\vip

There holds $\lim_{\la \to 0} \Pr\left( \Omega^S_3 (\la) \right)=1$.
As previously, it suffices to work with some fixed $t_1,t_2$, $x\in (-A,A)$
and $c=(a,b) \subset (-A,A)$.
Observing that $|x_\la|\sim 2\bm_\la$ and that 
$|c_\la|\sim (b-a)\bn_\la$,
Lemma \ref{occupation} and space/time stationarity
shows the result.

\vip

We also need $\Omega^S_4(\gamma,\la)$, defined for $\gamma>0$ as follows:
for all $q=1,\dots,n$, for all $t_0,t_1 \in \cT_M$
with $t_0<t_1<t_0+1$, there holds that $|\Theta^{q,\la}_{t_0,t_1} - 
(t_1-t_0)|<\gamma$. Here $\Theta^{q,\la}_{t_0,t_1}$ is defined as in 
Lemma \ref{convtheta} with the seed processes family
$(N^{S,q}_t(i))_{t\geq 0, i\in\zz} $.
Lemma \ref{convtheta} directly implies that for any $\gamma>0$,
$\lim_{\la\to 0} \Pr[\Omega^S_4(\gamma,\la)]=1$.

\vip

We finally introduce the event 
$$
\Omega(\alpha,\gamma,\la)= \Omega^M_{A,T}(\la)\cap\Omega_M(\alpha)\cap 
\Omega^S_1(\la)\cap\Omega^S_2(\la) \cap \Omega^S_3(\la) \cap
\Omega^S_4(\gamma,\la).
$$
We have shown that for any $\e>0$, there exists 
$\alpha>0$ such that for any $\gamma>0$, there holds
$\liminf_{\la\to 0} \Pr[\Omega(\alpha,\gamma,\la)]>1-\e$.

\subsection{Heart of the proof}
We now handle the main part of the proof, following closely Subsection
\ref{hpBS}.

\vip

Consider the $LFF_A(\infty)$-process.
Observe that by construction, we have, for $c\in \cC_M$ and $x,y\in c$,
$Z_t(x)=Z_t(y)$ for all $t\in [0,T]$, thus we can introduce
$Z_t(c)$. 

\vip

If $x \in \cB_M$, it is at the boundary of two cells 
$c_-,c_+ \in \cC_M$ and then we set $Z_t(x_-)=Z_t(c_-)$ and $Z_t(x_+)=Z_t(c_+)$ 
for all $t\in [0,T]$. 

\vip

If $x \in (-A,A)\setminus \cB_M$, we put  $Z_t(x_-)=Z_t(x_+)=Z_t(x)$ for all
$t\in [0,T]$.

\vip

For $x\in \cB_M$ and $t\geq 0$ we set 
$\tH_t(x)=\max(H_t(x),1-Z_t(x),1-Z_t(x_-),1-Z_t(x_+))$. 

\vip

Actually $Z_t(x)$ always equals either $Z_t(x_-)$ or $Z_t(x_+)$ 
and these can be distinct only at a point where has occurred  a 
microscopic fire (that is if $x=X_q$ for some $q\in\{1,\dots,n\}$
with $T_q<t$ and $Z_{T_q-}(X_q)<1$).

\vip

For all $x\in (-A,A)$ and $t \in [0,T]$, we put 
$$
\tau_t(x)= \sup \left\{s\leq t\;:\; Z_s(x_+)=Z_s(x_-)=Z_s(x)=0\right\} \in [0,t]
\cap \cT_M.
$$
For $c\in \cC_M$ and $t\in [0,T]$, we can define 
$\tau_t(c)$ as usual.

\vip

Observe, using Algorithm \ref{algo2}, that as when $\beta=BS$,
\begin{align}
&\hbox{for }x\notin \cB_M, \; \; Z_t(x) = \min{(t-\tau_t(x),1)} 
\hbox{ for all }
t\in[0,T], \label{tauzC2}\\
&\hbox{for }q=1,\dots,n,\;\;  Z_t(X_q) = \min{(t-\tau_t(X_q),1)} 
\hbox{ for all }
t\in[0,T_{q}).\label{tauzB2}
\end{align}

\vip

We also define for all $t\in [0,T]$, all $c \in \cC_M$ and all $x\in (-A,A)$
\begin{align*}
\tau_t^\la(c)=& \sup \left\{s\leq t\;:\; \forall i \in c_\la, 
\eta_{\ba_\la t-}^\la(i)=1 \hbox{ and } \eta_{\ba_\la t}^\la(i)=0   
\right\}\in [0,t],\\
\rho_t^\la(c)=& \sup \left\{s\leq t\;:\; \exists i \in c_\la, 
\eta_{\ba_\la t-}^\la(i)=1 \hbox{ and } \eta_{\ba_\la t}^\la(i)=0   
\right\}\in [0,t],\\
\tau_t^\la(x)=& \sup \left\{s\leq t\;:\; \forall i \in x_\la, 
\eta_{\ba_\la t-}^\la(i)=1 \hbox{ and } \eta_{\ba_\la t}^\la(i)=0 \right\}\in [0,t]
\end{align*}
with the convention that $\eta_{0-}^\la(i)=1$ for all $i\in I_A^\la$.
Observe that on $\Omega^M_{A,T}(\la)$, there holds that 
$\tau_t^\la(c),\rho_t^\la(c),\tau_t^\la(x) \in [0,t]\cap \cT_M$
for all $t\in [0,T]$, all $c \in \cC_M$ and all $x\in (-A,A)$.

\vip

For $t \in [0,T]$, consider the event  
$$
\Omega_t^\la =  \left\{ \forall s\in [0,t], \forall c \in\cC_M, 
\tau^{\la}_{s}(c) = \rho^\la_s(c)=\tau_{s}(c) \hbox{ and }
\forall x \in \cB_M,
\tau^{\la}_{s}(x) =\tau_{s}(x)\right\}.
$$ 

\begin{lem}\label{ggg2}
Let $\alpha>\gamma>0$. For any $\la\in(0,\la_\alpha)$, $\Omega^\la_T$ a.s.
holds on $\Omega(\alpha,\gamma,\la)$.
\end{lem}

\begin{proof}
We work on  $\Omega(\alpha,\gamma,\la)$ and assume that $\la\in(0,\la_\alpha)$.
Clearly, $\tau_0(x)=\tau_0^\la(x)=0$ and 
$\tau_0(c)=\tau_0^\la(c)=\rho^\la_0(c)=0$
for
all $x\in\cB_M$, all $c\in\cC_M$, so that $\Omega^\la_0$ a.s. holds.
We will show that for $q=0,\dots,n-1$, $\Omega^\la_{T_q}$ implies
$\Omega^\la_{T_{q+1}}$. This will prove that $\Omega^\la_{T_{n}}$ holds.
The extension to $\Omega^\la_{T}$ will be straightforward (see Step 1 below).

\vip

We thus fix  $q\in\{0,\dots,n-1\}$ and assume $\Omega^\la_{T_q}$. 
We repeatedly use below that
on the time interval $(T_q,T_{q+1})$, there are no fires at all (in $[-A,A]$)
for the $LFF_A(BS)$-process and no fires at all (in $I_A^\la$) during 
$(\ba_\la T_{q}, \ba_\la T_{q+1})$ for the $FF_A(\mu_S,\mu_M^{\la})$-process
(use $\Omega^{M}_{A,T}(\la)$).

\vip

{\bf Step 1.} Exactly as in the proof of Lemma \ref{ggg}-Step 1,
$\Omega^\la_{T_q}$ implies
$\Omega^\la_{T_{q+1}-}$.

\vip
{\bf Step 2.} Exactly as in the proof of Lemma \ref{ggg}-Step 2, we
observe that for $c \in \cC_M$, on $\Omega^\la_{T_{q+1}-}$, there holds,
for all $i \in c_\la$,  
\begin{align}
\label{danslescellules2}
\eta^\la_{\ba_\la T_{q+1}-}(i)=\min\left(N^{S,0}_{\ba_\la T_{q+1}-}(i) -
N^{S,0}_{\ba_\la \tau_{T_{q}}(c)}(i),1\right).
\end{align}

\vip

{\bf Step 3.} If $Z_{T_{q+1}-}(X_{q+1})<1$, there exist
$j_1,j_2,j_3,j_4 \in (X_{q+1})_\la$ such that 
$j_1<j_2<\lfloor \bn_\la X_{q+1} \rfloor <j_3<j_4$ and 
$\eta^\la_{\ba_\la T_{q+1}-}(j_2)=\eta^\la_{\ba_\la T_{q+1}-}(j_3)=0$ and
$\eta^\la_{\ba_\la T_{q+1}-}(j_1)= \eta^\la_{\ba_\la T_{q+1}-}(j_4)=1$.
The proof is the same as Lemma \ref{ggg}-Step 3.

\vip

{\bf Step 4.} Next we check that if $Z_{T_{q+1}-}(c)=1$ for some $c\in\cC_M$,
then $\eta^\la_{\ba_\la T_{q+1}-}(i)=1$ for all $i\in c_\la$.

Recalling (\ref{tauzC2}), we see that $Z_{T_{q+1}-}(c)=1$ implies that
$T_{{q+1}}-\tau_{T_{q+1}-}(c)\geq 1$ and thus  $T_{{q+1}}-\tau_{T_{q}}(c)\geq 1$
by Step 1.
Using $\Omega_M(\alpha)$ and that $T_{{q+1}},\tau_{T_{q}}(c) \in \cT_M$,
we deduce that $T_{{q+1}}-\tau_{T_{q}}(c)> 1$.
Using (\ref{danslescellules2}), we conclude that for all 
$i\in c_\la$, $\eta_{\ba_\la T_{q+1}-}^\la(i)=\min(N^{S,0}_{\ba_\la T_{q+1}-}(i)
-N^{S,0}_{\ba_\la \tau_{T_q}(c)(i)}  ,1)=1$
by $\Omega^S_3(\la)$.

\vip

{\bf Step 5.} We now prove that if $\tH_{T_{q+1}-}(x)=0$ for some $x \in \cB_M$,
then for all $i\in x_\la$, $\eta^\la_{\ba_\la T_{q+1}-}(i)=1$.

\vip

{\it Preliminary considerations.}
Let $k\in \{1,\dots,n\}$
such that $x=X_k$, which is at the boundary of two
cells $c_-,c_+ \in \cC_M$. We know that 
$\tH_{T_{q+1}-}(x)=0$, whence $H_{T_{q+1}-}(x)=0$ and 
$Z_{T_{q+1}-}(x)=Z_{T_{q+1}-}(c_+)=Z_{T_{q+1}-}(c_-)=1$. This implies that
$T_{q+1}\geq 1$ (because $Z_t(x)=t$ for all $t<1$ and all $x\in [-A,A]$)
and thus $T_{q+1}\geq 1+\alpha$ due to $\Omega_M(\alpha)$.

\vip

No fire has concerned
$(c_-)_\la$ during
$(\ba_\la \rho^\la_{T_{q+1}-}(c_-), \ba_\la T_{q+1})$
(by definition of $\rho^\la_{T_{q+1}-}(c_-)$).
But Step 1 implies that $\rho^\la_{T_{q+1}-}(c_-)=\tau_{T_{q+1}-}(c_-) 
\leq T_{q+1}-1$ because $Z_{T_{q+1}-}(c_-)=1$, see (\ref{tauzC2}). Recalling
$\Omega_M(\alpha)$, we deduce that  $\rho^\la_{T_{q+1}-}(c_-)< T_{q+1}-1-\alpha$.
Using a similar argument for $c_+$, we conclude that 
no match falling outside $(X_k)_\la$ can affect $(X_k)_\la$ during
$(\ba_\la(T_{q+1}-1-\alpha), \ba_\la T_{q+1})$ 
(because to affect $(X_k)_\la$, a match falling
outside $(X_k)_\la$ needs to cross $c_-$ or $c_+$). 

\vip

{\it Case 1.} First assume that $k\geq q+1$. Then we know that
no fire has fallen on $(X_k)_\la$ during $[0,\ba_\la T_{q+1})$.
Due to the preliminary considerations, we deduce that 
no fire at all has concerned $(X_k)_\la$ during 
$(\ba_\la (T_{q+1}-1-\alpha), \ba_\la T_{q+1})$. 
Using $\Omega^S_3(\la)$, we conclude that 
$(X_k)_\la$ is completely
occupied at time $\ba_\la T_{q+1}-$.

\vip

{\it Case 2.} Assume that $k \leq q$ and $Z_{T_k-}(X_k)=1$, 
so that there already has been a macroscopic fire
in $(X_k)_\la$ (at time $\ba_\la T_k$).
Since $Z_{T_k}(X_k)=0$ and
$Z_{T_{q+1}-}(X_k)=1$, we deduce
that $T_{q+1}-T_k\geq 1$, whence $T_{q+1}-T_k\geq 1+\alpha$ as usual. 
We conclude as in Case 1 that 
no fire at all has concerned $(X_k)_\la$ during 
$(T_S(T_{q+1}-1-\alpha), T_ST_{q+1})$, which implies the claim
by $\Omega^S_3(\la)$.

\vip

{\it Case 3.} Assume that $k \leq q$ and $Z_{T_k-}(X_k)<1$ and 
$T_{q+1}-T_k\geq 1$,
whence  $T_{q+1}-T_k\geq 1+\alpha$ due to $\Omega_M(\alpha)$.
Then there already has been a microscopic fire
in $(X_k)_\la$ (at time $\ba_\la T_k$). But there are no fire in 
$(X_k)_\la$ during $(\ba_\la T_k,\ba_\la T_{q+1})\supset
(T_S(T_{q+1}-1-\alpha), T_ST_{q+1})$ and we conclude as in Case 2.

\vip

{\it Case 4.} Assume finally 
that $k \leq q$ and $Z_{T_k-}(X_k)<1$ and $T_{q+1}-T_k<1$, whence
$T_{q+1}-T_k\leq 1-\alpha$ due to $\Omega_M(\alpha)$.
There has been a microscopic fire in $(X_k)_\la$ (at time $\ba_\la T_k$).
Since $H_{T_{q+1}-}(X_k)=0$, we deduce (see Algorithm \ref{algo2}
and recall that $\Theta^k_{T_k-Z_{T_k-}(X_k),T_k}$ is replaced by
$Z_{T_k-}(X_k)$) that 
$T_k+Z_{T_k-}(X_k)\leq T_{q+1}$, whence $T_k+Z_{T_k-}(X_k)\leq T_{q+1}-\alpha$
by $\Omega_M(\alpha)$ ($\cS_M$ was designed for that purpose).

\vip

Consider the zone $C=C(\eta^\la_{\ba_\la T_k-},\lfloor \bn_\la X_k \rfloor)$ 
destroyed
by the match falling at time $\ba_\la T_k$. This zone is completely occupied
at time $\ba_\la (T_k+\Theta^{k,\la}_{T_k-Z_{T_k-}(X_k),T_k})$: this follows from
the definition of
$\Theta^{k,\la}_{T_k-Z_{T_k-}(X_k),T_k}$, see Lemma \ref{deftheta} 
and from the preliminary considerations.
Using $\Omega^S_4(\gamma,\la)$, we deduce that 
$T_k+\Theta^{k,\la}_{T_k-Z_{T_k-}(X_k),T_k} \leq T_k+Z_{T_k-}(X_k)+\gamma < T_{q+1}$,
since $\gamma<\alpha$.
Hence $C$ is completely occupied at time $\ba_\la T_{q+1}-$.

\vip

Consider now
$i\in (X_k)_\la \setminus C$.
Then $i$ has not been killed by the fire starting at 
$\lfloor \bn_\la X_k \rfloor$. Thus $i$ cannot have been killed
during $(\ba_\la(T_{q+1}-1-\alpha),\ba_\la T_{q+1})$ 
(due to the preliminary considerations) and we conclude,
using $\Omega^S_3(\la)$, that $i$ is occupied at time $\ba_\la T_{q+1}-$. 
This implies the claim.

\vip

{\bf Step 6.} Let us now prove 
that if $\tH_{T_{q+1}-}(x)>0$ and $Z_{T_{q+1}-}(x_+)=1$
for some $x \in \cB_M$, there are $i_1,i_2 \in x_\la$ such that $i_1<i_2$
and $\eta^\la_{\ba_\la T_{q+1}-}(i_1)=1$, $\eta^\la_{\ba_\la T_{q+1}-}(i_2)=0$.
Recall that $x$ is at the boundary of two cells $c_-,c_+$.

\vip

We have either $H_{T_{q+1}-}(x)>0$
or  $Z_{T_{q+1}-}(c_-)<1$ (because $Z_{T_{q+1}-}(c_+)=1$ by assumption). Clearly,
$x=X_k$ for some $k\leq q$, with $Z_{T_k-}(X_k)<1$
(else, we would have  $H_{t}(x)=0$ and $Z_t(c_-)=Z_t(c_+)$ for 
all $t\in[0,T_{q+1})$). 
Thus, recalling (\ref{tauzB2}), $T_k-Z_{T_k-}(X_k)=\tau_{T_k-}(X_k)
=\tau^\la_{T_k-}(X_k)$, so that $(X_k)_\la$ is completely empty at time
$\ba_\la (T_k-Z_{T_k-}(X_k))$.
\vip

{\it Case 1.} Assume first that $H_{T_{q+1}-}(x)>0$. Then by construction,
see Algorithm \ref{algo2} (with $\Theta_{T_k-Z_{T_k-}(X_k),T_k}$ replaced by
$Z_{T_k-}(X_k)$), there holds 
$T_k+Z_{T_k-}(X_k) > T_{q+1}>T_k$, whence by $\Omega_M(\alpha)$,
$T_k+Z_{T_k-}(X_k) > T_{q+1} + \alpha >T_k + 2\alpha$.

\vip

Consider $C=C(\eta^\la_{\ba_\la T_k-},\lfloor \bn_\la X_k \rfloor)$.
By  $\Omega^S_2(\la)$, we have 
$C \subset
\lb \lfloor \bn_\la X_k -\bm_\la/2\rfloor, \lfloor \bn_\la X_k +\bm_\la/2 
\rfloor\rb$
(because $(X_k)_\la$ is completely empty at time
$\ba_\la (T_k-Z_{T_k-}(X_k))$, because 
$T_k-Z_{T_k-}(X_k)$ and $T_k$ belong to $\cT_M$ and because $0<Z_{T_k-}(X_k)<1$).

\vip

The component $C$ is destroyed at time $T_ST_k$. 
By Definition of
$\Theta^{k,\la}_{T_k-Z_{T_k-}(X_k),T_k}$, see Lemma \ref{convtheta}, we deduce that
$C$ is not completely occupied at time $\ba_\la(T_k+ 
\Theta^{k,\la}_{T_k-Z_{T_k-}(X_k),T_k})$. But by $\Omega^S_4(\gamma,\la)$
we see that $\Theta^{k,\la}_{T_k-Z_{T_k-}(X_k),T_k}\geq Z_{T_k-}(X_k)-\gamma$,
whence $T_k+ \Theta^{k,\la}_{T_k-Z_{T_k-}(X_k),T_k} \geq T_k+Z_{T_k-}(X_k)-\gamma
>T_{q+1}$ since $\gamma<\alpha$. All
this implies that $C$ is not completely occupied
at time  $\ba_\la T_{q+1}-$.

\vip

Finally, using again
$\Omega^S_2(\la)$  there is necessarily (at least)
one seed falling on a site in $\lb \lfloor \bn_\la X_k -\bm_\la +1 \rfloor, 
\lfloor \bn_\la X_k - \bm_\la/2 -1\rfloor
\rb \subset (X_k)_\la$ during
$(\ba_\la T_q,\ba_\la T_{q+1})$. This shows the result.

\vip

{\it Case 2.} Assume next that $H_{T_{q+1}-}(x)=0$ and that 
$T_{q+1}- [T_k-Z_{T_k-}(X_k)]<1$.
Recall that $(X_k)_\la$ is completely
empty at time $\ba_\la (T_k-Z_{T_k-}(X_k))$. Since
$T_k-Z_{T_k-}(X_k)$ and $T_{q+1}$ belong to $\cT_M$ and since their
difference is smaller than $1$ by assumption,
$\Omega^S_2(\la)$ guarantees us the existence of $i_1<i_2<i_3$, all in
$(X_k)_\la$, such that (at least) one seed falls on $i_2$ and no
seed fall on $i_1$ nor on $i_3$ during $(\ba_\la (T_k-Z_{T_k-}(X_k)),
\ba_\la T_{q+1})$.
One easily concludes that $i_2$ is occupied and $i_3$ is vacant
at time $\ba_\la T_{q+1}-$, as desired.

\vip

{\it Case 3.} Assume finally that $H_{T_{q+1}-}(x)=0$ and that 
$T_{q+1}- [T_k-Z_{T_k-}(X_k)] \geq 1$, whence $T_{q+1}- [T_k-Z_{T_k-}(X_k)] 
\geq 1+\alpha$ by $\Omega_M(\alpha)$.
Since $H_{T_{q+1}-}(x)=0$, there holds 
$Z_{T_{q+1}-}(c_-)< 1 = Z_{T_{q+1}-}(c_+)$ and
$T_k+ Z_{T_k-}(X_k) \leq T_{q+1}$, so that 
$T_k+ Z_{T_k-}(X_k) \leq T_{q+1}-\alpha$.

\vip

We aim to use
the event $\Omega^S_1(\la)$. We introduce 
$t_0=T_k-Z_{T_k-}(X_k)=\tau_{T_k-}(X_k)=\tau^\la_{T_k-}(X_k)$. Observe that
$\tau_{T_k-}(c_-)=\tau_{T_k-}(c_+)=\tau_{T_k-}(x)$
because there has been no fire (exactly) at $x$ during $[0,T_k)$. 
Thus 
$Z_{t_0-}(x)=Z_{t_0-}(x_-)=Z_{t_0-}(x_+)=1$ and 
$Z_{t_0}(x)=Z_{t_0}(c_-)=Z_{t_0}(c_+)=0$.

\vip

Set now $t_1=T_k$ and $s=T_{q+1}$. Observe that $0<t_1-t_0<1$.
Necessarily, $Z_t(c_-)$ has jumped to $0$ at least one time
between $t_0$ and $T_{q+1}-$ (else, one would have  $Z_{T_{q+1}-}(c_-)=1$,
since $T_{q+1}-t_0\geq 1$ by assumption) and this jump occurs
after $t_0+1>t_1$ (since a jump of  $Z_t(c_-)$ requires that $Z_t(c_-)=1$,
and since for all $t\in [t_0,t_0+1)$, $Z_t(c_-)=t-t_0<1$).

\vip

We thus may denote by
$t_2<t_3<\dots<t_K$, for some $K\geq 2$, the successive times of jumps of
the process $(Z_t(c_-),Z_t(c_+))$  during $(t_0+1,s)$. 
We also put $\e=1$ if $t_2$ is a jump of $Z_t(c_+)$ and $\e=-1$ else.
Then we prove exactly as in Lemma \ref{ggg}-Step 6-Case 3
that 
$\cR=\{\e,t_0,\dots,t_K;s\}$ necessarily satisfies the condition $(PP)$.

\vip

Next, there holds that $t_2-t_1 <Z_{T_k-}(X_k)=t_1-t_0$,
because else, we would have $H_{t_2-}(X_k)=0$ and thus the fire destroying
$c_+$ (or $c_-$) at time $t_2$ would also destroy $c_-$ (or $c_+$),
we thus would have $Z_{t_2}(c_+)=Z_{t_2}(c_-)=0$, so that 
$Z_t(c_+)$ and $Z_t(c_-)$ would remain equal forever.

\vip

Finally, we check as in Lemma \ref{ggg}-Step 6-Case 3 that 
$(\eta^\la_{\ba_\la t}(i))_{t\geq t_0, i \in x_\la}=
(\zeta^{\la,\cR,k}_t(i+ \lfloor \bn_\la x \rfloor ))_{t\geq t_0, i \in x_\la}$,
this last process being built upon the family 
$(N^{S,k}_{t}(i))_{t\geq t_0, i \in x_\la}$ as in Subsection \ref{sspp2}.

\vip

We thus can use $\Omega^S_1(\la)$ and conclude that there are some sites
$i_1<i_2$ in $x_\la$ with  $\eta^\la_{T_ST_{q+1-}}(i_1)=1$ and
$\eta^\la_{T_ST_{q+1-}}(i_2)=0$ as desired.

\vip

{\bf Step 7.} The conclusion follows from the previous steps exactly as
in the proof of Lemma \ref{ggg}-Step 7: it suffices to replace everywhere
$T_S$ by $\ba_\la$.
\end{proof}

\subsection{Conclusion}

To achieve the proof, we will need the following result.

\begin{lem}\label{densitiesinfty}
Let $(N^S_t(i))_{t\geq 0,i \in \zz}$ be a family of i.i.d. $SR(\mu_S)$-processes,
and define $\zeta^\la_t(i)=\min(N^S_{\ba_\la t}(i),1)$.

(i) Put $K_t^\la=(2\bm_\la+1)^{-1} 
|\{i \in \lb -\bm_\la,\bm_\la\rb \; : \; \zeta^\la_{t}(i)>0\}|$ and 
$$
U^\la_t= \left(\frac{\psi_S(K^\la_t)}{\ba_\la}\right)\land 1, 
$$
recall Notation \ref{phipsi}.
Then for any $\e>0$, any $T>0$, 
$$
\lim_{\la\to 0} \Pr\left[
\sup_{[0,T]}|U^\la_t - t\land 1 | >\e \right] = 0.
$$

(ii) Put also $C^\la_t=C(\zeta^\la_t,0)$ and define
$$
V_t^\la=\left(\ba_\la^{-1} \psi_S( 1-1/|C^\la_t| ) \indiq_{\{|C^\la_t|>0\}}
\right)\land 1. 
$$
Then 
for any $\e>0$,  for all $t\in [0,1)$,
$$
\lim_{\la\to 0} \Pr \left[ 
C^\la_t \subset   \lb -\bm_\la,\bm_\la \rb, |V^\la_t - t| <\e\right]=1.
$$
\end{lem}

\begin{proof} We split the proof into three steps. 

\vip
 
{\bf Step 1.} Here we show that for $t\geq 0$ fixed, 
$\lim_{\la\to 0} \Pr\left[|U^\la_t - t\land 1 | >\e \right] = 0$.

\vip

{\it Case 1.} Assume first that $t\geq 1$. Then Lemma \ref{occupation}-(ii)
implies that $\lim_{\la\to 0} \Pr[K^\la_t=1]=1$.
But $K^\la_t=1$ implies that $U^\la_t= [\psi_S(1)/\ba_\la]\land 1=1$
(because $\psi_S(1)=\infty)$.

\vip

{\it Case 2.} Assume next that $t<1$. 
Then the random variable $X^\la_t=(2\bm_\la+1)K^\la_t$
has a binomial distribution with parameters $2\bm_\la +1$
and $\nu_S((0,\ba_\la t))$. Let $\e\in (0,t)$ be fixed. Then,
using Bienaym\'e-Chebyshev's inequality,
\begin{align*}
\Pr[K^\la_t \leq &\nu_S((0,\ba_\la (t-\e)))]=
\Pr[X^\la_t \leq (2\bm_\la+1)\nu_S((0,\ba_\la (t-\e)))]\\
&\leq \Pr[|X^\la_t -(2\bm_\la +1)\nu_S((0,\ba_\la t))| \geq 
(2\bm_\la+1)\nu_S((\ba_\la (t-\e),\ba_\la t))]\\
&\leq \frac{(2\bm_\la+1) \nu_S((0,\ba_\la t)) \nu_S((\ba_\la t,\infty))}
{ (2\bm_\la+1)^2\nu_S^2((\ba_\la (t-\e),\ba_\la t))}\\
&\leq  \frac{\nu_S((\ba_\la t,\infty))}
{ (2\bm_\la+1) \nu_S^2((\ba_\la (t-\e),\ba_\la t))}.
\end{align*}
This last quantity tends to $0$. Indeed, $(H_S(\infty))$ implies that
$\nu_S((\ba_\la (t-\e),\ba_\la t)) \sim \nu_S((\ba_\la (t-\e),\infty)) 
\geq \nu_S((\ba_\la t,\infty))$ and it suffices to use that
$\bm_\la  \nu_S((\ba_\la t,\infty)) \to \infty$ by (\ref{mla}), since $t<1$.

\vip

By the same way, for $\e>0$, 
\begin{align*}
\Pr[K^\la_t \geq &\nu_S((0,\ba_\la (t+\e)))]=
\Pr[X^\la_t \geq (2\bm_\la+1)\nu_S((0,\ba_\la (t+\e)))]\\
&\leq \Pr[|X^\la_t -(2\bm_\la +1)\nu_S((0,\ba_\la t))| \geq 
(2\bm_\la+1)\nu_S((\ba_\la t,\ba_\la (t+\e)))]\\
&\leq \frac{(2\bm_\la+1) \nu_S((0,\ba_\la t)) \nu_S((\ba_\la t,\infty))}
{ (2\bm_\la+1)^2\nu_S^2((\ba_\la t,\ba_\la (t+\e)))}\\
&\leq  \frac{\nu_S((\ba_\la t,\infty))}
{ (2\bm_\la+1) \nu_S^2((\ba_\la t,\ba_\la (t+\e)))},
\end{align*}
which also tends to $0$, because $(H_S(\infty))$ implies that
$\nu_S((\ba_\la t,\ba_\la (t+\e))) \sim \nu_S((\ba_\la t,\infty))$,
and because $\bm_\la \nu_S((\ba_\la t, \infty))\to \infty$, since
$t<1$.

\vip

To conclude the step it suffices to note that for $0<t-\e<t<t+\e<1$, 
$K^\la_t \in (\nu_S((0,\ba_\la (t-\e))), \nu_S((0,\ba_\la (t+\e))))$
implies that $U^\la_t \in (t-\e, t+\e)$ by definition of $\psi_S$.

\vip

{\bf Step 2.} Using a well suited version of the Dini theorem, we conclude
the proof of (i). Indeed, let $\e>0$ and consider a subdivision 
$0=t_0<t_1<\dots < t_l=T$, with $t_{i+1}-t_i<\e/2$. Using Step 1,
we see that $\lim_{\la\to 0} \Pr[ \max_{i=0,\dots,l} 
|U^\la_{t_i}- t_i\land 1|>\e/2]=0$.  Observe that $t\mapsto U^\la_t$
is a.s. nondecreasing and that $t \mapsto t\land 1$ is nondecreasing
and Lipschitz continuous. We deduce that 
$\sup_{[0,T]} |U^\la_t - t\land 1| \leq \e/2 + \max_{i=0,\dots,l} 
|U^\la_{t_i}- t_i\land 1|$. One immediately concludes.

\vip

{\bf Step 3.} It remains to prove (ii). Let thus $t<1$ and $\e>0$
be fixed. We can of course assume that $0<t-\e<t<t+\e<1$.

\vip

First,
$\lim_{\la \to 0} \Pr[C^\la_t \subset \lb -\bm_\la,\bm_\la \rb]=1$ 
due to Lemma \ref{occupation}-(i). 

\vip

Next, each site
is vacant with probability $\nu_S((\ba_\la t, \infty))$.
It is thus classical that as $\la\to 0$,
$\nu_S((\ba_\la t, \infty)) |C^\la_t|$ goes in law to a random variable
$X$ with density $xe^{-x}\indiq_{x>0}$. Indeed, 

\vip

$\bullet$ for $Y_\delta$ a geometric random variable with parameter
$\delta$, the random variable $\delta Y_\delta$ goes in law,
as $\delta\to 0$,  to an exponentially distributed
random variable with parameter $1$; 

\vip

$\bullet$ $|C^\la_t|$ is the sum of two independent geometric random variables,
both with parameter $\nu_S((\ba_\la t, \infty))$;
\vip

$\bullet$ $xe^{-x}\indiq_{x>0}$ is the density of the sum of two independent
exponentially distributed random variables with parameter $1$.

\vip

For $\delta>0$, consider $0<a<1<b$
such that $\Pr[X\in (a,b))]\geq 1-\delta$.
Then 
$$
\lim_{\la\to 0} \Pr[ |C^\la_t| \in (a/\nu_S((\ba_\la t, \infty)),
b/\nu_S((\ba_\la t, \infty))] \geq 1-\delta.
$$
But due to $(H_S(\infty))$, $|C^\la_t| \in (a/\nu_S((\ba_\la t, \infty)),
b/\nu_S((\ba_\la t, \infty))$ implies, if $\la$ is small enough, 
that  
$|C^\la_t| \in (1/\nu_S((\ba_\la (t-\e), \infty)),
1/\nu_S((\ba_\la (t+\e), \infty))$,
whence finally
$$
V^\la_t \in (\ba_\la^{-1} \psi_S(\nu_S((0,\ba_\la (t-\e)))), 
\ba_\la^{-1} \psi_S(\nu_S((0,\ba_\la (t+\e))))=(t-\e,t+\e).
$$
We have proved that for all $\delta>0$, 
$\liminf_{\la\to 0} \Pr[|V^\la_t-t|<\e ] \geq 1-\delta$, 
which concludes the proof.
\end{proof}

We finally give the

\begin{preuve} {\it of Proposition \ref{converge1A} when $\beta=\infty$.}
Let us fix $x_0 \in (-A,A)$, $t_0\in (0,T]\setminus \{1\}$ 
and $\e>0$. We will prove that 
with our coupling (see Subsection \ref{coucou}), there holds

\vip

(a) $\lim_{\la\to 0} 
\Pr\left[ \bdelta(D^\la_{t_0}(x_0),D_{t_0}(x_0) )>\e\right] =0$;

\vip

(b) $\lim_{\la\to 0} \Pr\left[ \bdelta_T(D^\la(x_0),D(x_0) )>\e\right] =0$;

\vip

(c) $\lim_{\la\to 0} \Pr\left[ \sup_{[0,T]} |Z_t^\la(x_0)-Z_t(x_0)|\geq \e
\right] =0$;

\vip

(d) $\lim_{\la \to 0} \Pr\left[|W^\la_{t_0}(x_0) - Z_{t_0}(x_0)|>\e \right] = 0$,
where 
$$
W^\la_{t_0}(x_0)=\left(\frac{\psi_S( 1-1/|C_A(\eta^\la_{\ba_\la t_{0}}, 
\lfloor \bn_\la x_0 \rfloor)| ) 
\indiq_{\{|C_A(\eta^\la_{\ba_\la t_{0}}, 
\lfloor \bn_\la x_0 \rfloor)|>0\}}}{\ba_\la} \right)\land 1 .
$$

These points will clearly imply the result.

\vip

First, we introduce, for $\zeta>0$,  
the event $\Omega^{x_0}_{A,T}(\zeta)$ on which
$x_0 \notin \cup_{q=1}^n [X_q - \zeta, X_q + \zeta]$.
The probability
of this event obviously tends to $1$ as $\zeta \to 0$.

\vip

On  $\Omega^{x_0}_{A,T}(\zeta)$, it holds that for $\la>0$ small enough
(say, such that $4 \bm_\la/\bn_\la < \zeta$),
$\lfloor \bn_\la x_0 \rfloor \notin 
\cup_{q=1}^n (X_q)_\la$.  We then call $c_0\in \cC_M$
the cell containing $x_0$.  

\vip

{\bf Step 1.} As in the case where $\beta=BS$, (a) implies (b)
(the fact that $t_0=1$ is excluded in (a) is of course not a problem,
because $\{1\}$ is Lebesgue-negligible).

\vip

{\bf Step 2.} Due to Lemma \ref{ggg2}, we know that if $0<\gamma<\alpha$ on 
$\Omega(\alpha,\gamma,\la) \cap \Omega^{x_0}_{A,T}(\zeta)$, there holds that
$\tau^\la_t(c_0)=\rho^\la_t(c_0)=\tau_t(x_0)$ for all $t\in [0,T]$.
This implies that for all $i\in (c_0)_\la$, for all $t \in [0,T]$, 
$$
\eta^\la_{T_St}(i)=\min(N^{S,0}_{\ba_\la t}(i)-N^{S,0}_{\ba_\la  \tau_t(x_0)}(i), 1).
$$
We also recall that by construction, $(\tau_t(x_0))_{t\geq 0}$ is independent of
$(N^{S,0}_{t}(i))_{t\geq 0, i \in \zz}$.

\vip

{\bf Step 3.} Here we prove (d).
Let $\delta>0$ be fixed. We first consider $\alpha_0>0$, 
$\gamma_0\in (0,\alpha_0)$, $\zeta_0>0$
and $\la_0>0$ such that for all $\la \in (0,\la_0)$, 
$\Pr\left[ \Omega(\alpha_0,\gamma_0,\la) 
\cap \Omega^{x_0}_{A,T}(\zeta_0)\right] > 1-\delta$.
Then we consider $\la_1\leq \la_0$ in such a way
that for $\la \in (0,\la_1)$,  $\lb \lfloor \bn_\la x_0 \rfloor -\bm_\la, 
\lfloor \bn_\la x_0 \rfloor + \bm_\la \rb \subset (c_0)_\la$ (this can be
done properly by using $\Omega^{x_0}_{A,T}(\zeta)$ and the fact
that $\bm_\la/\bn_\la \to 0$). 

\vip

Introduce $C^\la_t$ and $V^\la_t$ as in Lemma \ref{densitiesinfty}-(ii),
using the seed processes 
$(N^{S,\la}_{t + \tau_t(x_0)/\ba_\la}(i+\lfloor \bn_\la x_0 \rfloor)-
N^S_{\tau_t(x_0)/\ba_\la}(i+\lfloor \bn_\la x_0 \rfloor))_{t\geq 0, i \in \zz}$.

\vip

Then by Step 2, we observe that 
$C^\la_{t_0-\tau_t(x_0)}\subset \lb -\bm_\la,\bm_\la\rb$
implies that, 
on $\Omega(\alpha,\gamma,\la) \cap \Omega^{x_0}_{A,T}(\zeta)$
and for $\la<\la_1$, $C_A(\eta^\la_{\ba_\la t_0},\lfloor \bn_\la x_0 \rfloor)
=\{i+\lfloor \bn_\la x_0 \rfloor\; : \; i \in C^\la_{t_0-\tau_t(x_0)} \}$,
whence $W^\la_{t_0}(x_0)=V^\la_{t_0-\tau_{t_0}(x_0)}$.
All this implies, using Lemma  \ref{densitiesinfty}-(ii), that 
$$
\liminf_{\la\to 0} \Pr\left[\left.|W^\la_{t_0}(x_0)-(t_0-\tau_{t_0}(x_0)) 
\vert <\e \;
\right\vert \; t_0-\tau_{t_0}(x_0)<1 \right]\geq 1-\delta.
$$
Recalling finally (\ref{tauzC2}), we deduce that
$$
\liminf_{\la\to 0} \Pr\left[\left.|W^\la_{t_0}(x_0)-Z_{t_0}(x_0)\vert <\e \;
\right\vert \; t_0-\tau_{t_0}(x_0)<1 \right]\geq 1-\delta.
$$
If now $t_0-\tau_{t_0}(x_0)> 1$, then Step 2
and $\Omega^S_3(\la)$ imply that $(c_0)_\la$ is completely occupied
at time $\ba_\la t_0$. Hence 
$|C(\eta^\la_{\ba_\la t_0},\lfloor \bn_\la x_0 \rfloor)|
\geq |(c_0)_\la| \simeq |c| \bn_\la \geq \alpha \bn_\la$ by $\Omega_M(\alpha)$.
Consequently, 
$W^\la_{t_0}(x_0)\geq [\ba_\la^{-1}\psi_S(1-1/(\alpha\bn_\la))]\land 1
\simeq [\ba_\la^{-1}\psi_S(1- \nu_S((\ba_\la,\infty))/\alpha)]\land 1$.
For $\e>0$, there holds that $\nu_S((\ba_\la,\infty))/\alpha
\leq \nu_S(((1-\e)\ba_\la,\infty))$ for all $\la$ small enough:
use $(H_S(\infty))$.

\vip

Thus for all $\la$ small enough, 
on $\Omega(\alpha,\gamma,\la) \cap \Omega^{x_0}_{A,T}(\zeta)$, 
we have 
$W^\la_{t_0}(x_0) \geq  
[\ba_\la^{-1}\psi_S(1- \nu_S(((1-\e)\ba_\la,\infty)))]\land 1
=1-\e$ by definition of $\psi_S$. Thus 
$$
\liminf_{\la\to 0} \Pr\left[\left. W^\la_{t_0}(x_0) \in (1-\e,1] \;
\right\vert \; t_0-\tau_{t_0}(x_0)>1 \right]\geq 1-\delta.
$$
Recalling (\ref{tauzC2}), we deduce that
$$
\liminf_{\la\to 0} \Pr\left[\left.|W^\la_{t_0}(x_0)-Z_{t_0}(x_0)\vert <\e \;
\right\vert \; t_0-\tau_{t_0}(x_0)>1 \right]\geq 1-\delta.
$$
Finally, we observe that a.s.,  $t_0-\tau_{t_0}(x_0)\ne 1$.
Indeed, we have excluded $t_0=1$ and the only value charged with
positive probability by $\tau_{t_0}(x_0)$ is $0$.
Thus there holds that
$$
\liminf_{\la\to 0} \Pr\left[|W^\la_{t_0}(x_0)-Z_{t_0}(x_0)\vert <\e \right]
\geq 1-\delta.
$$
Since this holds for any $\delta>0$, this concludes the proof
of (d).

\vip

{\bf Step 4.} Next, (c) is proved exactly as when $\beta=BS$ 
(change the beginning: let first $\delta>0$, 
$\alpha_0>0$, $\zeta_0 \in (0,\alpha_0)$
and $\la_0>0$ be as in Step 3; replace everywhere $T_S$ by $\ba_\la$;
and make use of Lemma \ref{densitiesinfty} instead of Lemma 
\ref{densitiesbs}). 

\vip

{\bf Step 5.} Finally, (a) is also proved as when $\beta=BS$.
The only difference is that when put $\cT_M^*=\cT_M\cup\{t_0\}$, 
we need that $t_0\ne 1$
(because $0\in\cT_M$ and $\Omega_M^*(\alpha)$ will thus require that
for $|t_0-1|>\alpha$).
\end{preuve}

\section{Cluster-size distribution when $\beta\in\{\infty,BS\}$}\label{prcsd}
\setcounter{equation}{0}

The aim of this section is to prove Corollaries \ref{co1} and \ref{cobs}.

\subsection{Study of the $LFF(\infty)$ and $LFF(BS)$-processes}

We first extend \cite[Lemma 17]{bf}.

\begin{lem}\label{zunif}
Let $\beta \in \{\infty,BS\}$. Let $(Z_t(x),D_t(x),H_t(x))_{t\geq 0, x\in \rr}$
be a $LFF(\beta)$-process.
There are some constants $0<c_1<c_2$ and $0 < \kappa_1<\kappa_2$ 
such that the following estimates hold.

(i)  For any $t\in (1,\infty)$, any $x\in \rr$, any
$z\in [0,1)$, $\Pr[Z_t(x)=z]=0$.

(ii) For any $t\in[0,\infty)$, any $B>0$, any $x\in\rr$, $P[|D_t(x)|=B]=0$.

(iii) For all
$t\in [0,\infty)$, all $x\in \rr$, all $B>0$,
$\Pr[|D_t(x)|\geq B] \leq c_2 e^{-\kappa_1 B}$.

(iv) For all
$t\in [3/2,\infty)$, all $x\in \rr$, all $B>0$, 
$\Pr[|D_t(x)|\geq B] \geq c_1 e^{-\kappa_2 B}$.

(v) For all $t\in[5/2,\infty)$, all
$0\leq a < b < 1$, all $x\in \rr$, 
$c_1(b-a) \leq \Pr(Z_t(x)\in [a,b])   \leq c_2 (b-a)$.
\end{lem}

\begin{proof} By invariance by translation,
it suffices to treat the case $x=0$. When 
$\beta=BS$, the function $F_S$ was defined in Definition \ref{defF}.
Recall that the $LFF(\infty)$-process
can be viewed as a $LFF(BS)$-process with the function
$F_S(z,v)=z$, see Remark \ref{i=bs}.

\vip

We consider a Poisson measure $\pi_M(dt,dx,dv)$ on 
$[0,\infty)\times \rr \times [0,1]$ with intensity measure $dtdxdv$. 
We also denote by $\pi_M(dt,dx)=\int_{v\in [0,1]} \pi_M(dt,dx,dv)$.

\vip

{\bf Point (i).} For $t\in [0,1]$,
we have a.s. $Z_t(0)=t$. But for $t> 1$ and $z\in [0,1)$,
$Z_t(0)=z$ implies that the cluster containing $0$
has been killed at time $t-z$, so that necessarily 
$\pi_M(\{t-z\}\times \rr)>0$. This happens with probability $0$. 

\vip

{\bf Point (ii).} 
For any $t>0$,
$|D_t(0)|$ is either $0$ or of the form $|X_i-X_j|$ (with $i\ne j$),
where $(T_i,X_i)_{i\geq 1}$ are the marks of the 
Poisson measure $\pi_M(ds,dx)$ restricted to $[0,t]\times \rr$. 
We easily conclude as previously that for $B>0$,
$\Pr(|D_t(0)|=B)=0$.

\vip

{\bf Point (iii).}
First if $t \in [0,1)$, we have a.s. $|D_t(0)|=0$ and the result is obvious. 
Recall now that $v_0\in [0,1)$ was defined in Lemma \ref{lblb}
and that for $(\tau,X,V)$ a mark of $\pi_M$ such that
$V\geq v_0$, we have $H_t(X)>0$ or $Z_t(X)<1$
for all $t\in [\tau,\tau+1/4]$ (see
the proof of Proposition \ref{loc}-Step 1).
This implies that for $t\geq 1$,
\begin{align*}
\{D_t(0)\geq B\} \subset& \{[0,B/2] \hbox{ 
is connected at time $t$ or $[-B/2,0]$
is connected at time $t$} \}\\
\subset& \left\{\pi_M([t-1/4,t]\times [0,B/2]\times [v_0,1])=0 \right\}\\
&\cup \left\{\pi_M([t-1/4,t]\times [-B/2,0]\times [v_0,1])=0\right\}.
\end{align*}
Consequently, $\Pr[|D_t(0)|\geq B] \leq 2 e^{-(1-v_0)B/8}$ as desired.

\vip

{\bf Point (iv).}
Fix $t\geq 3/2$ and $B>0$. Consider the event
$\Omega_{t,B}=\Omega^1_{t,B} \cap\Omega^2_{t}\cap\Omega^3_{t,B}$,
illustrated by Figure \ref{figOmegatB},
where 

\vip

$\bullet$ $\Omega^{1}_{t,B}=\{\pi_M([t-3/2,t]\times [0,B]\times [0,1])=0\}$;

\vip

$\bullet$ $\Omega^2_{t}$ is the event that in the box $[t-3/2,t]\times[-1,0]
\times [0,1]$,
$\pi_M$ has exactly $5$ marks $(S_i,Y_i,V_i)_{i=1,...5}$ 
with $Y_5<Y_4<Y_3<Y_2<Y_1$,
$\min_{i=1,\dots,5} V_i>v_0$ 
and $t-3/2 < S_1 < t-1$, $S_1<S_2 < S_1+1/4$, $S_2<S_3< S_2+1/4$,
$S_3<S_4< S_3+1/4$, $S_4<S_5< S_4+1/4$ and $S_5+1/4>t$.

\vip

$\bullet$ $\Omega^3_{t,B}$ is the event that in the box 
$[t-3/2,t]\times[B,B+1] \times [0,1]$,
$\pi_M$ has exactly $5$ marks $(\tS_i,\tY_i,\tV_i)_{i=1,...,5}$ with 
$\tY_1<\tY_2<\tY_3<\tY_4<\tY_5$, $\min_{i=1,\dots,5} \tV_i>v_0$ 
and $t-3/2 < \tS_1 < t-1$, $\tS_1<\tS_2 < \tS_1+1/4$, $\tS_2<\tS_3< \tS_2+1/4$,
$\tS_3<\tS_4< \tS_3+1/4$, $\tS_4<\tS_5< \tS_4+1/4$ and  $\tS_5+1/4>t$.

\begin{figure}[b] 
\fbox{
\begin{minipage}[c]{0.95\textwidth}
\centering
\includegraphics[width=12cm]{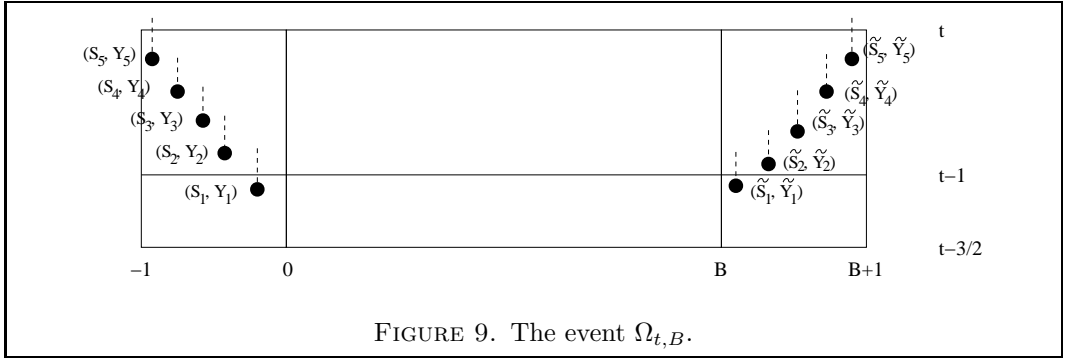}
\caption{The event $\Omega_{t,B}$.}
\label{figOmegatB}
\end{minipage}
}
\end{figure}

\vip

We of course have $p:=\Pr(\Omega^2_{t})=\Pr(\Omega^3_{t,B})>0$,
and this probability does not depend on $t\geq 3/2$ nor on $B>0$.
Furthermore, $\Pr(\Omega^1_{t,B})=e^{-3B/2}$. These three
events being independent, we conclude that 
$\Pr(\Omega_{t,B}) \geq p^2 e^{-3B/2}$. To conclude the proof of (iv),
it thus suffices to check that $\Omega_{t,B}\subset \{[0,B]\subset D_t(0)\}$.
But on $\Omega_{t,B}$, using the same arguments as in Point (iii),
we observe that:

\vip

$\bullet$ the fire starting at $(S_2,Y_2)$ can not affect $[0,B]$,
because since $S_2 \in [S_1, S_1 +1/4)$, $H_{S_2-}(Y_1)>0$ or
$Z_{S_2-}(Y_1)>0$, with $Y_2<Y_1<0$;

\vip

$\bullet$ then the fire starting at $(S_3,Y_3)$ can not affect $[0,B]$, because
since $S_3 \in [S_2,S_2+1/4)$, $H_{S_3-}(Y_2)>0$ or
$Z_{S_3-}(Y_2)>0$, with $Y_3<Y_2<0$;

\vip

$\bullet$ then the fire starting at $(S_4,Y_4)$ can not affect $[0,B]$, because
since $S_4 \in [S_3,S_3+1/4)$, $H_{S_4-}(Y_3)>0$ or
$Z_{S_4-}(Y_3)>0$, with $Y_4<Y_3<0$;

\vip

$\bullet$ then the fire starting at $(S_5,Y_5)$ can not affect $[0,B]$, because
since $S_5 \in [S_4,S_4+1/4)$, $H_{S_5-}(Y_4)>0$ or
$Z_{S_5-}(Y_4)>0$, with $Y_5<Y_4<0$;

\vip

$\bullet$ furthermore, the fires starting on the left at $-1$ during $(S_1,t]$
cannot affect $[0,B]$, because for all $t\in (S_1,t]$, there is
always a site $x_t \in  \{Y_1,Y_2,Y_3,Y_4\} \subset [-1,0]$ 
with $H_t(x_t)>0$ or $Z_t(x_t)<1$;

\vip

$\bullet$ the same arguments apply on the right of $B$.

\vip

As a conclusion, the zone $[0,B]$ is not affected by any fire
during $(S_1 \lor \tS_1,t]$. Since the length of this time interval is 
greater than $1$, we deduce that for all $x \in [0,B]$,
$Z_t(x)=\min(Z_{S_1 \lor \tS_1}(x) + t- S_1 \lor \tS_1,1)
\geq \min(t- S_1 \lor \tS_1,1)=1$ 
and $H_t(x)=\max(H_{S_1 \lor \tS_1}(x) - (t- S_1 \lor \tS_1),0) 
\leq \max(1 - (t- S_1 \lor \tS_1),0)=0$, whence $[0,B]\subset D_t(0)$.

\vip

{\bf Point (v).}
For $0\leq a < b < 1$ and $t\geq 1$,
we have $Z_t(0)\in [a,b]$ if and only there is $\tau\in [t-b,t-a]$
such that $Z_\tau(0)=0$. And this happens if and only if
$X_{t,a,b}:=\int_{t-b}^{t-a}\int_\rr \indiq_{\{y\in D_\sm(0)\}}\pi_M(ds,dy)\geq 1$.
We deduce that
$$
\Pr\left(Z_t(0)\in [a,b]\right)
=\Pr\left(X_{t,a,b}\geq 1\right)
\leq \E\left[X_{t,a,b}\right]
= \int_{t-b}^{t-a}\E[|D_s(0)|]ds \leq C(b-a),
$$
where we used Point (iii) for the last inequality.

\vip

Next, we have $\{\pi_M([t-b,t-a]\times D_{t-b}(0))\geq 1 \}\subset
\{X_{t,a,b}\geq 1\}$: it suffices to note that a.s.,
$\{X_{t,a,b}=0\} \subset \{X_{t,a,b}=0, D_{t-b}(0)\subset D_s(0)$ for all $s\in
[t-b,t-a] \} \subset \{\pi_M([t-b,t-a]\times D_{t-b}(0))=0\}$.
Now since $D_{t-b}(0)$ is independent of $\pi_M(ds,dx)$ restricted to
$(t-b,\infty)\times \rr$, we deduce that for $t\geq 5/2$
\begin{align*}
\Pr\left(Z_t(0)\in [a,b]\right)
&\geq \Pr\left[\pi_M((t-b,t-a]\times D_{t-b}(0))\geq 1\right]\\
&\geq \Pr\left[ |D_{t-b}(0)|\geq 1 \right] (1-e^{-(b-a)}) \geq c (1-e^{-(b-a)}),
\end{align*}
where we used Point (iv) (here $t-b\geq 3/2$) to get the last inequality.
This concludes the proof, since $1-e^{-x}\geq x/2$ for all $x\in [0,1]$.
\end{proof}

\subsection{The case $\beta=\infty$}

We can now handle the

\begin{preuve} {\it of Corollary \ref{co1}.}
We thus assume $(H_M)$ and $(H_S(\infty))$ and
consider, for each $\la>0$, a $FF(\mu_S,\mu_M^\la)$-process 
$(\eta^\la_t(i))_{t\geq 0,i\in\zz}$.
Let also  $(Z_t(x),D_t(x),H_t(x))_{t\geq 0, x\in \rr}$ be a $LFF(\infty)$-process.

\vip

{\bf Point (ii).} Using Lemma \ref{zunif}-(iii)-(iv)
and recalling that $|C(\eta_{\ba_\la t}^\la,0)|/\bn_\la
=|D^\la_t(0)|$ by (\ref{dlambda}), it suffices
to check that for all $t\geq 3/2$, all $B>0$,
$\lim_{\la\to 0} \Pr\left[|D^\la_t(0)|\geq B \right] = 
\Pr\left[|D_t(0)|\geq B \right]$.
This follows from Theorem \ref{converge1}-(b), which implies
that $|D^\la_t(0)|$ goes in law to $|D_t(0)|$ and from Lemma \ref{zunif}-(ii).

\vip

{\bf Point (i).} Due to Lemma \ref{zunif}-(v)
we only need that for all $0 <a < b < 1$, all $t\geq 5/2$, 
$$
\lim_{\la \to 0} \Pr\left(|C(\eta^\la_{\ba_\la t},0)|\in 
[1/\nu_S((\ba_\la a,\infty)), 1/\nu_S((\ba_\la b,\infty))]
\right)
=\Pr\left(Z_t(0)\in [a,b]\right).
$$
But using Theorem \ref{converge1}-(c) and Lemma \ref{zunif}-(i), we know that
$$
\lim_{\la\to 0} \Pr\left[\psi_S\left(1-1/|C(\eta^\la_{\ba_\la t},0)| \right)
\indiq_{\{|C(\eta^\la_{\ba_\la t},0)| \geq 1\}}\in[\ba_\la a,
\ba_\la b] \right] 
= \Pr\left(Z_t(0)\in [a,b]\right).
$$
Using finally the definition of $\psi_S$ (see Notation
\ref{phipsi}-(ii)), we see that for all $c\in \nn$, all $0<\alpha < \beta$,
$$
\psi_S\left(1-1/c \right) \indiq_{\{c\geq 1\}}\in[\alpha, \beta] 
\hbox{ if and only if }
c \in [1/\nu_S((\alpha,\infty)), 1/\nu_S((\beta,\infty))].
$$
One immediately concludes.
\end{preuve}

\subsection{The case $\beta=BS$}

We finally give the

\begin{preuve} {\it of Corollary \ref{cobs}.}
We thus assume $(H_M)$ and $(H_S(BS))$ and
consider, for each $\la>0$, a $FF(\mu_S,\mu_M^\la)$-process 
$(\eta^\la_t(i))_{t\geq 0,i\in\zz}$.
Let also  $(Z_t(x),D_t(x),H_t(x))_{t\geq 0, x\in \rr}$ 
be a $LFF(BS)$-process.

\vip

{\bf Point (ii).}  
Using Lemma \ref{zunif}-(iii)-(iv)
and recalling that $|C(\eta_{\ba_\la t}^\la,0)|/\bn_\la
=|D^\la_t(0)|$ by (\ref{dlambda}), it suffices
to check that for all $t\geq 3/2$, all $B>0$,
$\lim_{\la\to 0} \Pr\left[|D^\la_t(0)|\geq B \right] = 
\Pr\left[|D_t(0)|\geq B \right]$.
This follows from Theorem \ref{convergebs}-(b), which implies
that $|D^\la_t(0)|$ goes in law to $|D_t(0)|$ and from Lemma \ref{zunif}-(ii).

\vip

{\bf Point (i).} Theorem \ref{convergebs}-(c) asserts 
that for all $t\geq 0$, all $k\geq 0$, $\lim_{\la \to 0}
\Pr [ |C(\eta_{T_S t}^\la,0)|=k] =\E[q_k(Z_t(0))]$, where
$q_k(z)$ was defined in (\ref{defqk}).
Using next Lemma \ref{zunif}-(v) and recalling that
$Z_t(0)\in [0,1]$ a.s., we see that for
$t\geq 5/2$, the law of $Z_t(0)$ is of the form 
$$
g_t(z)\indiq_{\{0\leq z \leq 1\}}dz + \alpha_t \delta_1(dz),
$$
for some function
$g_t:[0,1]\mapsto \rr_+$ satisfying $c\leq g_t \leq C$, where
the constants $0<c<C$ do not depend on $t\geq 5/2$. One immediately
deduces that for any $k\geq 0$, $\E[q_k(Z_t(0))] \in [c q_k,Cq_k]$.
Indeed, there holds $q_k=\int_0^1 q_k(z)dz$ and $q_k(1)=0$.
This concludes the proof.
\end{preuve}

\part{Numerical simulations}

\section{Simulations}\label{simu}
\setcounter{equation}{0}

\begin{figure}[b]
\fbox{
\begin{minipage}[c]{0.95\textwidth}
\centering
\includegraphics[width=7.5cm]{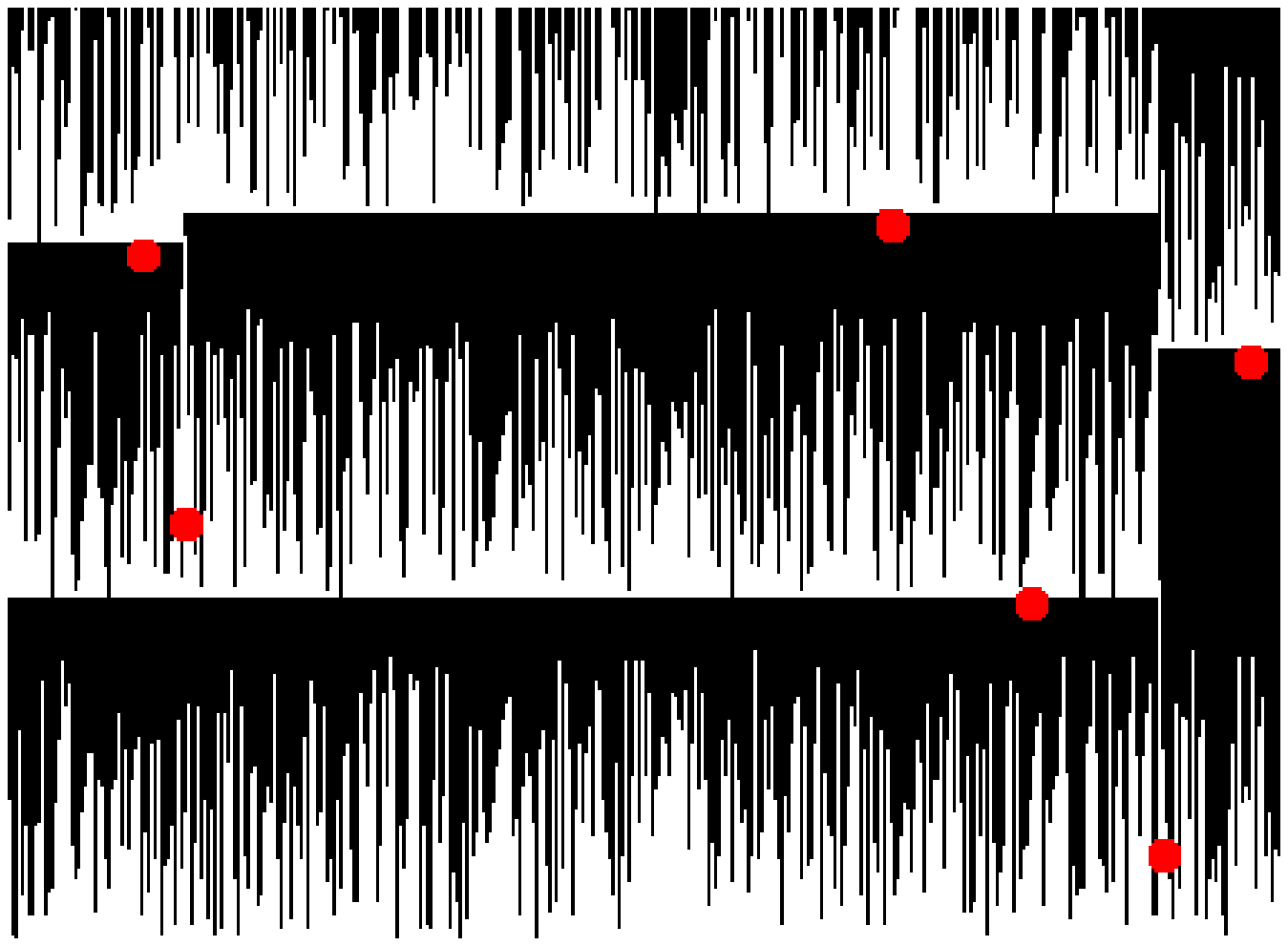}
\caption{Simulation with $\beta=BS$.}
\label{SimBS}
\vip
\parbox{13.3cm}{
\footnotesize{
We used $\mu_S=\delta_1$, 
$\nu_S(dt)=\indiq_{\{t\in [0,1]\}}dt$, $\ba_\la=T_S=1$ and $\la=10^{-3}$.
Here everything happens, roughly, as described by the limit process.
The first fire is microscopic and limits the length of the
second fire, which is macroscopic. Then there is another microscopic fire, etc.
Observe that the effect of the first microscopic fire persists for quite
a long time.
}}
\end{minipage}
}
\end{figure}

\begin{figure}[b]
\fbox{
\begin{minipage}[c]{0.95\textwidth}
\centering
\includegraphics[width=7.5cm]{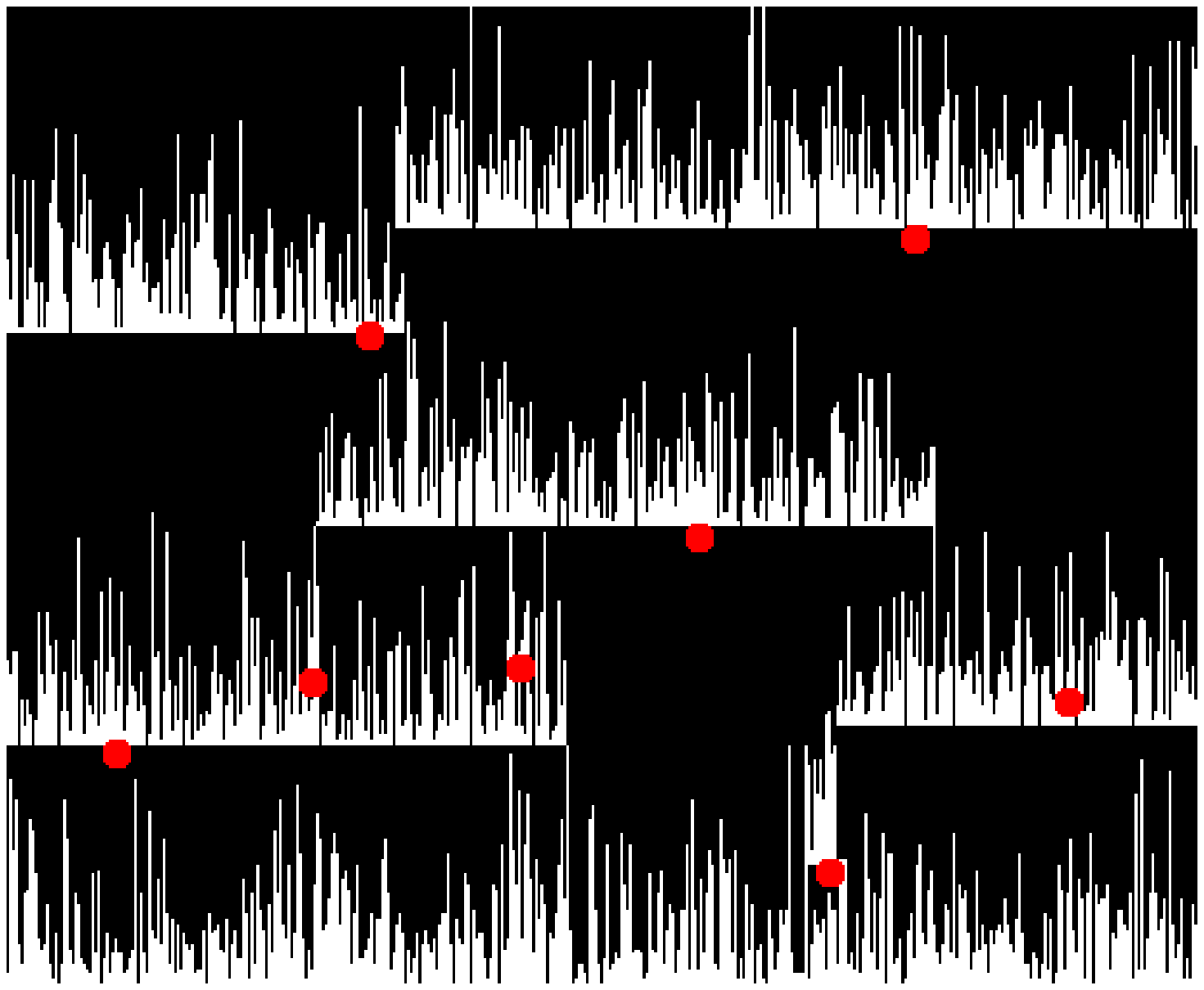}
\caption{Simulation with $\beta=\infty$.}
\label{SimInf}
\vip
\parbox{13.3cm}{
\footnotesize{Here $\mu_S((t,\infty))=e^{-t^2/2}$, 
$\nu_S(dt)=(\sqrt{2/\pi})e^{-t^2/2}\indiq_{\{t\geq 0\}}dt$ and
$\la=10^{-3}$. We used the approximate value
$\ba_\la\simeq \sqrt{2 \log(1/\la)}$.
The picture is not so far from the limit process, 
but there are some defaults.
The first fire is 
rather microscopic, but has however quite a large length. 
The second fire, which is clearly macroscopic, is limited
not by a previous microscopic fire, but by a site where the first seed 
has needed an unusual large time to fall.
}}
\end{minipage}
}
\end{figure}

\begin{figure}[b]
\fbox{
\begin{minipage}[c]{0.95\textwidth}
\centering
\includegraphics[width=7.5cm]{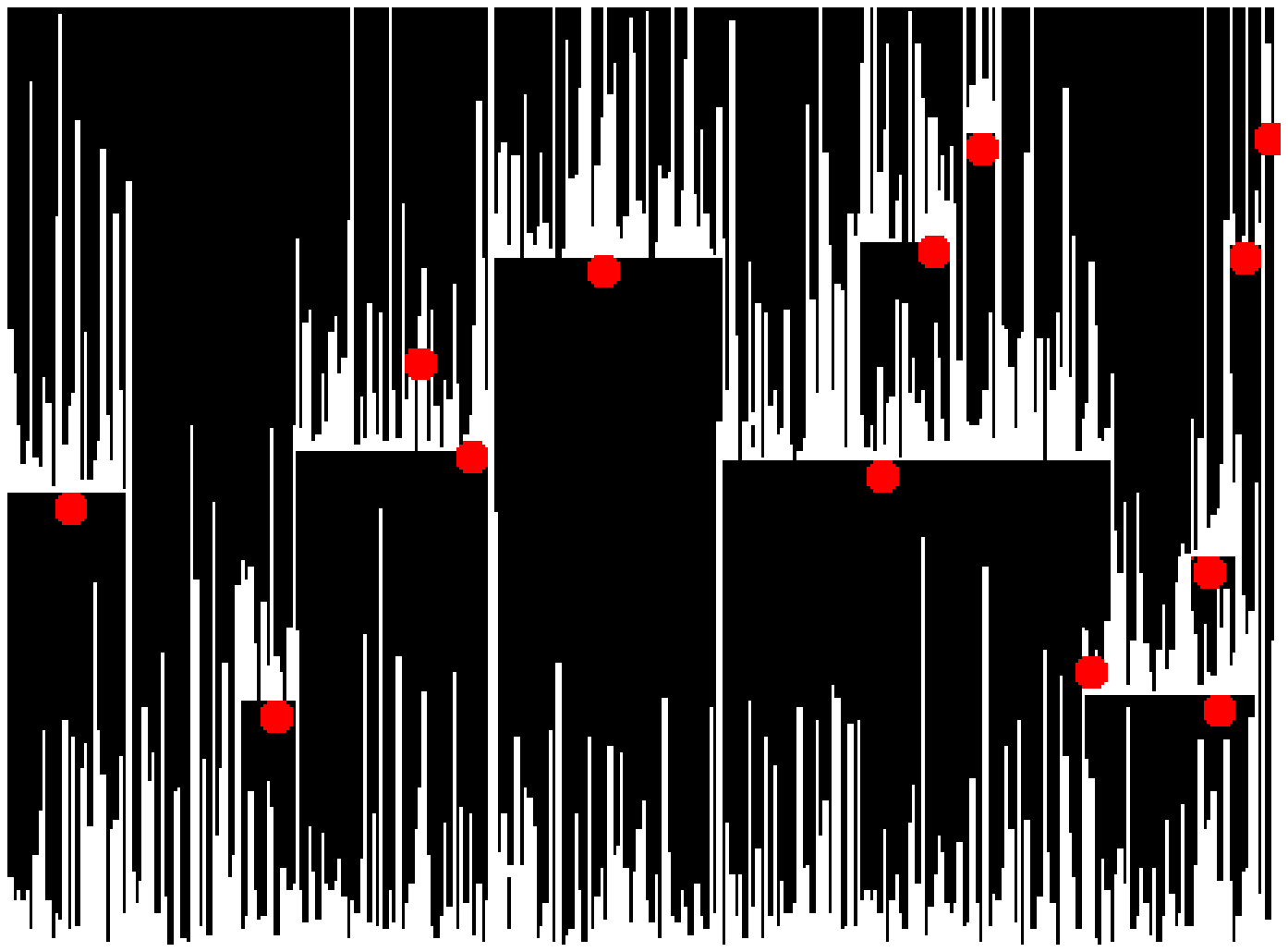}
\caption{Simulation with $\beta=5$.}
\label{Sim5}
\vip
\parbox{13.3cm}{
\footnotesize{We considered
$\mu_S((t,\infty))=(1+t/\beta)^{-\beta-1}$
and $\nu_S((t,\infty))=(1+t/\beta)^{-\beta}$ with $\beta=5$, $\la=5.10^{-3}$.
We used the approximate value $\ba_\la\simeq (1/\la)^{1/(\beta+1)}$. 
This picture resembles much the limit process: almost all
the fires are macroscopic.
}}
\end{minipage}
}
\end{figure}

\begin{figure}[b]
\fbox{
\begin{minipage}[c]{0.95\textwidth}
\centering
\includegraphics[width=7.5cm]{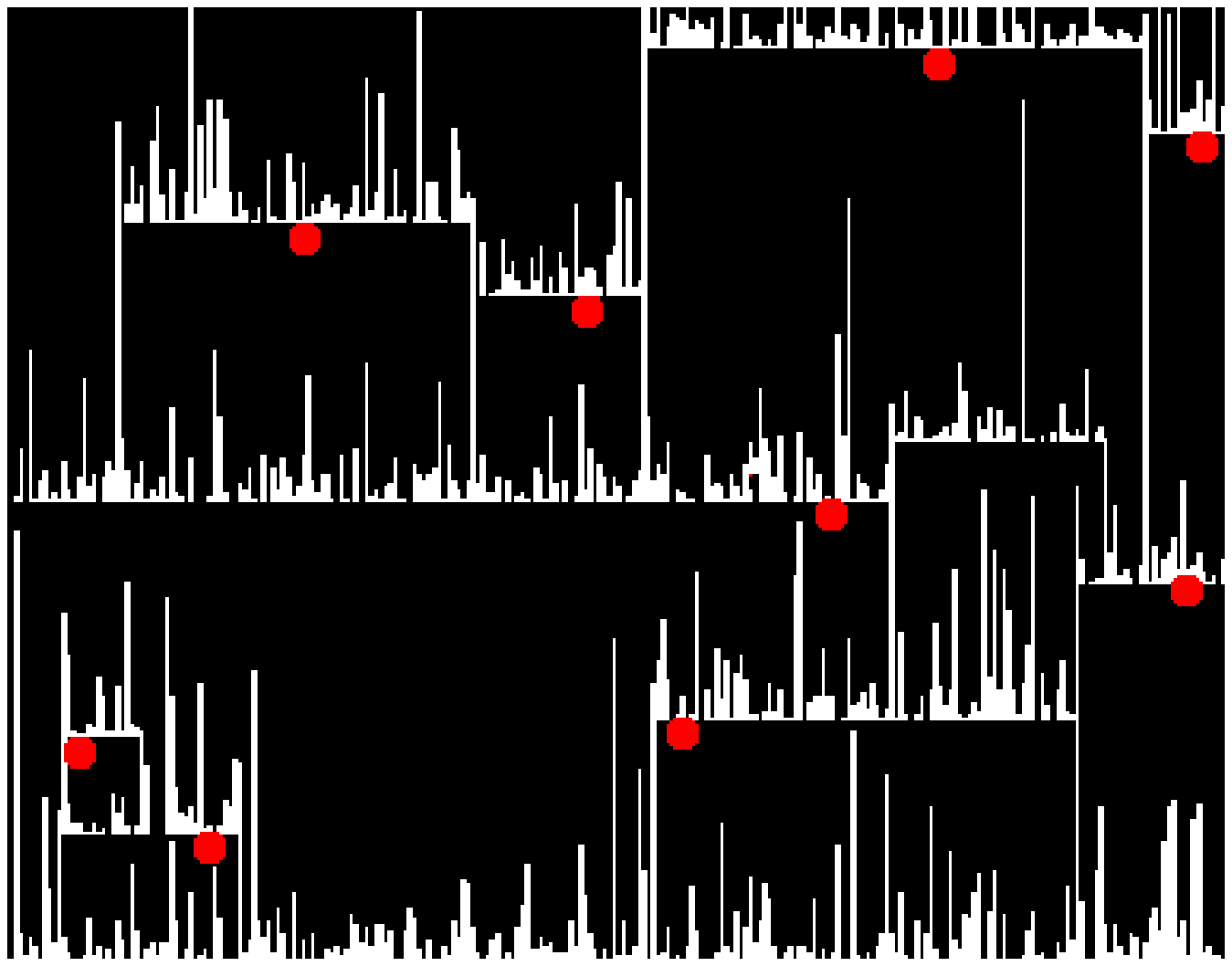}
\caption{Simulation with $\beta=2$.}
\label{Sim2}
\vip
\parbox{13.3cm}{
\footnotesize{
Same thing as Figure \ref{Sim5} with $\beta=2$ and $\la=10^{-3}$.
This picture is in perfect adequacy with 
the limit process, at least from a qualitative point of view.
}}
\end{minipage}
}
\end{figure}

\begin{figure}[b]
\fbox{
\begin{minipage}[c]{0.95\textwidth}
\centering
\includegraphics[width=7.5cm]{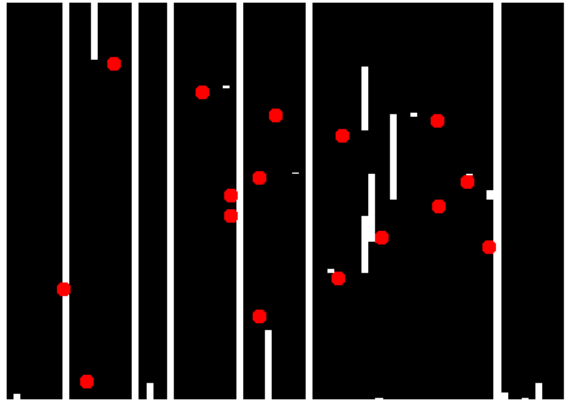}
\caption{Simulation with $\beta=0$.}
\label{Sim0}
\vip
\parbox{13.3cm}{
\footnotesize{We used $\mu_S((t,\infty))=e (e+t)^{-1} 
[\log(e+t)]^{-2}$, $\nu_S((t,\infty))=[\log(e+t)]^{-1}$ and $\la=10^{-7}$.
We used the approximate value $\ba_\la\simeq 1/[\la \log(1/\la)]$.
This picture is quite satisfactory:
there are six sites where the first seed never falls and the fires
have quite a low effect.
}}
\end{minipage}
}
\end{figure}

We would like to present some simulations of the discrete forest fire process.
In all the simulations below, we choose
$\mu_M^\la(dt)=\la e^{-\la t}\indiq_{t\geq 0}dt$ and we consider
different laws $\mu_S$. We simulate the $FF_A(\mu_S,\mu^\la_M)$ process
with $A=2.5$, for some given value of $\la$. Since there are too much
concerned sites, it is not possible to draw the whole picture.
We thus extract a zone in which some interesting events occur.
The vacant (resp. occupied) zones are drawn in white (resp. black).
Matches are represented by bullets.

\part{Appendix}

\section{Appendix}\label{ap}
\setcounter{equation}{0}

\subsection{Regularly varying functions}
The proof below is closely related to the theory 
of regularly varying functions and is probably completely standard.

\begin{lem}\label{hsimplieshsbeta}
Assume $(H_S)$. Then either $(H_S(BS))$ holds or there exists
$\beta\in [0,\infty)\cup\{\infty\}$ such that  $(H_S(\beta))$ holds.
\end{lem}

\begin{proof}
We thus assume $(H_S)$ and that the support of $\mu_S$ is unbounded.
Hence, for all $t>0$,
$$
\varphi(t) := \lim_{x\to \infty} \frac{\nu_S((x,\infty))}{\nu_S((xt,\infty))}
\in [0,\infty)\cup\{\infty\}
$$
exists.
The function $\varphi$
is clearly nondecreasing and satisfies $\varphi(1)=1$.

\vip

{\bf Step 1.} We first show that for all $t>0$, $\varphi(1/t)=1/\varphi(t)$,
with the convention that $1/0=\infty$ and $1/\infty=0$.
This is not hard:
$$
\varphi(1/t)=\lim_{x \to \infty} \frac{\nu_S((x,\infty))}{\nu_S((x/t,\infty))}
=\lim_{y \to \infty} \frac{\nu_S((yt,\infty))}{\nu_S((y,\infty))}=1/\varphi(t).
$$

{\bf Step 2.} By the same way, one easily checks that for 
$0<s\leq t$, one has 
$\varphi(st)=\varphi(s)\varphi(t)$ as soon as $\varphi(s)>0$ 
or $\varphi(t)<\infty$. It suffices to write
$$
\varphi(st)=\lim_{x \to \infty} \frac{\nu_S((x,\infty))}{\nu_S((xst,\infty))}
=\lim_{x \to \infty} \frac{\nu_S((x,\infty))}{\nu_S((xs,\infty))}
\frac{\nu_S((xs,\infty))}{\nu_S((xst,\infty))}
=\varphi(s)\varphi(t).
$$

{\bf Step 3.} We assume first that $\varphi(s)>0$ for all $s\in (0,1)$. By Step
1, one easily deduces that $\varphi(s)\in (0,\infty)$ for all $s>0$.
We thus have a nondecreasing function $\varphi: (0,\infty) \mapsto (0,\infty)$
such that $\varphi(st)=\varphi(s)\varphi(t)$ for all $0<s\leq t$ and
such that $\varphi(1)=1$. One classically concludes that 
there exists $\beta \in [0,\infty)$ such that $\varphi(t)=t^\beta$.

\vip

{\bf Step 4.} We now assume that $\varphi(\alpha)=0$ for some 
$\alpha \in (0,1)$. We want to show that if so, then $\varphi(t)=0$
for all $t\in (0,1)$. This will imply that
$\varphi(t)=\infty$ for $t>1$ by Step 1, whence $\varphi(t)=t^\infty$.

\vip

Let thus $\alpha_*=\sup\{\alpha>0:\; 
\varphi(\alpha)=0\}$. Suppose by contradiction that 
$\alpha_*\in (0,1)$. By monotonicity, we have
$\varphi(\alpha)=0$ for all $\alpha \in (0,\alpha_*)$.
By Step 1,
we know that $\varphi(s)\in (0,\infty)$ for all $s\in (\alpha_*,1/\alpha_*)$.
Due to Step 2, we deduce that for all small $\e>0$,
$\varphi((\alpha_*-\e)(1/\alpha_*-\e))=0$. But for
$\e>0$ small enough, we have $(\alpha_*-\e)(1/\alpha_*-\e)>\alpha_*$
(because $\alpha_*<1$). This contradicts the definition of $\alpha_*$.
\end{proof}

Next, we prove the existence of the scale $\bm_\la$ satisfying (\ref{mla}).

\begin{lem}\label{mlaexist}
Assume $(H_S(\infty))$. Recall (\ref{ala}), (\ref{nla}). There exists
a function $\bm_\la : (0,1] \mapsto \nn$ satisfying (\ref{mla}).
\end{lem}

\begin{proof}
Recalling that $\lim_{\la\to 0} \ba_\la = \infty$ and using
$(H_S(\infty))$, we observe that for any $n\geq 1$, 
$\lim_{\la\to 0} \nu_S(((1-1/n)\ba_\la,\infty))/\nu_S((\ba_\la,\infty))=\infty$.
Thus there exists $\la_n \in (0,1]$ such that for all
$\la\in (0,\la_n]$, $\nu_S(((1-1/n)\ba_\la,\infty))/\nu_S((\ba_\la,\infty))
\geq n$. We of course may choose $\la_1=1$ and choose 
the sequence $(\la_n)_{n\geq 1}$ decreasing to $0$. Then we define 
$\e_\la: (0,1]\mapsto (0,1]$ by setting,
for all $n\geq 1$,  $\e_\la=1/n$ for $\la \in (\la_{n+1},\la_n]$.
There holds $\lim_{\la \to 0} \e_\la = 0$. Finally, we put
$\bm_\la = \lfloor 1/\nu_S((\ba_\la(1-\e_\la),\infty)) \rfloor$. This function
is obviously non-increasing. Next, recalling that $\bn_\la =
\lfloor 1/ \nu_S((\ba_\la,\infty)) \rfloor$, we see that for all $n\geq 1$,
all $\la\in (\la_{n+1},\la_n)$,
$$
\frac{\bm_\la}{\bn_\la} \simeq \frac{\nu_S((\ba_\la,\infty))}
{\nu_S((\ba_\la(1-\e_\la),\infty))} = \frac{\nu_S((\ba_\la,\infty))}
{\nu_S((\ba_\la(1-1/n),\infty))} \leq 1/n,
$$
whence $\lim_{\la \to 0} (\bm_\la/\bn_\la)=0$. Finally, fix
$z\in (0,1)$ and consider $n$ large enough, so that $1-1/n>z$. 
Then for $\la \in (0, \la_n)$, there holds $\e_\la \leq 1/n$, whence
$$
\nu_S((\ba_\la z,\infty)) \bm_\la \simeq \frac{\nu_S((\ba_\la z,\infty))}
{\nu_S((\ba_\la(1-\e_\la),\infty))} \geq \frac{\nu_S((\ba_\la z,\infty))}
{\nu_S((\ba_\la(1-1/n),\infty))} \to \infty
$$
as $\la \to 0$ due to $(H_S(\infty))$, since $z<1-1/n$. 
\end{proof}

\subsection{Coupling}
Finally, we recall some well-known facts about coupling.

\begin{lem}\label{gcou}
(i) Let $(p_k)_{k\geq 0}$ and $(q_k)_{k\geq 0}$ be two probability 
laws on $\{0,1,\dots\}$. One can couple
$X\sim (p_k)_{k\geq 0}$ and $Y\sim (q_k)_{k\geq 0}$ such that for
all $k\geq 0$, $\Pr[X=Y=k]\geq p_k\land q_k$.

(ii) For $f,g$ two probability densities on $\rr$, one can couple
$X\sim f(x)dx$ and $Y\sim g(x)dx$ in such a way that 
$\Pr[X=Y] \geq \int_\rr \min(f(x),g(x))dx$. 

(iii) If we have a sequence of laws $\mu_n$ on some Polish space, converging
weakly to some law $\mu$, then it is possible to find 
some random variables $X_n \sim \mu_n$ and $X \sim \mu$ such that a.s.,
$\lim_{n \to \infty} X_n=X$.
\end{lem}

\begin{proof}
First observe that (iii) is nothing but the Skorokhod representation Theorem.

\vip

To prove (i), set $r_k=p_k \land q_k$ and $r=\sum_0^\infty r_k$.
Consider a Bernoulli r.v. $C$ with parameter $r$,
a $(r_k/r)_{k\geq 0}$-distributed r.v. $Z$, a 
$((p_k-r_k)/(1-r))_{k\geq 0}$-distributed r.v. $U$ and  a 
$((q_k-r_k)/(1-r))_{k\geq 0}$-distributed r.v. $V$. Assume that all
these objects are independent and put $(X,Y)=C (Z,Z) + (1-C)(U,V)$.
Some immediate computations show that $X\sim (p_k)_{k\geq 0}$ and 
$Y\sim (q_k)_{k\geq 0}$ and for $k\geq 0$, $\Pr[X=Y=k]\geq r_k$.

\vip

The proof of (ii) is similar: put $h=\min(f,g)$ and 
$r=\int_\rr h(x)dx$.
Consider a Bernoulli r.v. $C$ with parameter $r$,
a r.v. $Z$ with density $h/r$, a r.v. $U$ with density
$(f-h)/(1-r)$ and  a r.v. $V$ with density
$(g-h)/(1-r)$. Assume that all
these objects are independent and put $(X,Y)=C (Z,Z) + (1-C)(U,V)$.
Some immediate computations show that $X\sim f(x)dx$,
$Y\sim g(y)dy$ and $\Pr[X=Y]\geq r$.
\end{proof}

\end{document}